\documentclass[12pt]{amsart}

\usepackage{
 a4wide, 
  amsmath,
  amsfonts,
  amssymb,
  graphicx, 
  times,
  color,
  hyphenat, 
  appendix, stmaryrd, datetime, bbm, tikz-cd, thmtools, verbatim, dirtytalk, multicol
}
\usepackage[boxsize=1.2ex]{ytableau}

\usepackage[shortlabels]{enumitem}

\usepackage[scr=boondox,  
            cal=esstix]   
           {mathalpha}

\usepackage{dynkin-diagrams}
\usepackage{mathdots}
\usepackage[all]{xy}
\usepackage{euscript,mathrsfs}
\usepackage{bbm,xspace}
\usepackage[final]{hyperref, bookmark}
\usepackage{cleveref}
\hypersetup{colorlinks=true, pdfstartview=FitV, linkcolor=blue, citecolor=blue, urlcolor=blue}

\usepackage{dirtytalk}

\usepackage{tikz}
\usetikzlibrary{cd}
\usetikzlibrary{decorations}
\usetikzlibrary{decorations.markings}
\usetikzlibrary{decorations.pathreplacing}
\usetikzlibrary{decorations.pathmorphing}
\usetikzlibrary{arrows.meta,shapes,positioning,matrix,calc}
\usetikzlibrary{shapes.callouts}
\tikzstyle{densely dotted}=[dash pattern=on \pgflinewidth off 0.5pt]
\tikzset{anchorbase/.style={baseline={([yshift=-0.5ex]current bounding box.center)}},
tinynodes/.style={font=\tiny,text height=0.25ex,text depth=0.05ex},
smallnodes/.style={font=\scriptsize,text height=0.75ex,text depth=0.15ex},
usual/.style={line width=0.9,color=black},
dusual/.style={line width=0.9,color=spinach,densely dashed},
pole/.style={line width=3.0,color=specialgray},
crossline/.style={preaction={draw=white,line width=5.0pt,-},preaction={draw=black,line width=0.9pt,-}},
crosspole/.style={preaction={draw=white,line width=6.0pt,-},preaction={draw=specialgray,line width=3.0pt,-}},
mor/.style={line width=0.75,color=black,fill=cream},
blob/.style={circle,fill,minimum size=5.0pt,inner sep=0pt,outer sep=0pt},
blobed/.style n args={3}{decoration={markings,post length=0.5mm,pre length=0.5mm,
mark=at position #1 with {\node[blob,#3,label=left:$#2\!$]at (0,0){};}
},postaction={decorate}},
rblobed/.style n args={3}{decoration={markings,post length=0.5mm,pre length=0.5mm,
mark=at position #1 with {\node[blob,#3,label=right:$\!#2$]at (0,0){};}
},postaction={decorate}},
}

\tikzstyle{tikzdot}=[fill, circle, inner sep=2pt]
\tikzstyle{smoltikzdot}=[fill=white, draw=black, circle, inner sep=1pt]

\tikzset{
    partial ellipse/.style args={#1:#2:#3}{
        insert path={+ (#1:#3) arc (#1:#2:#3)}
    }
}

\input xy
\xyoption{all}
\allowdisplaybreaks

\definecolor{myblue}{rgb}{0,.5,1}
\definecolor{myred}{rgb}{0.9,0,0}
\definecolor{mygreen}{rgb}{0,0.667,0.267}
\definecolor{mypurple}{rgb}{0.5,0,0.5}
\definecolor{myyellow}{rgb}{1,0.8,0}
\definecolor{mycyan}{rgb}{0,1,0.8}
\definecolor{myorange}{rgb}{1,0.4,0}

\newcommand{\eqskip}{{\vskip2mm\noindent}} 
\newcommand{\nn}{\nonumber}

\DeclareMathOperator{\ev}{ev}

\newcommand{\Ev}{\mathcal{Ev}}

\newcommand{\affu}[1]{\tilde{\mathcal{U}}_{\triangle}({#1})}

\newcommand{\naffu}[1]{\tilde{\mathcal{U}}({#1})}

\newcommand{\taffu}[1]{\mathcal{U}({#1})}

\newcommand{\oneid}{\mathbf{1}} 

\newcommand{\E}{\mathscr{E}}
\newcommand{\F}{\mathscr{F}}

\newcommand{\sln}{{\mathfrak{sl}_{n}}}
\newcommand{\gln}{{\mathfrak{gl}_{n}}}
\newcommand{\asln}{\widehat{\mathfrak{sl}}_{n}}

\newcommand{\agln}{\widehat{\mathfrak{gl}}_{n}}

\newcommand{\dotuagl}[1]{\dot{\bf U}(\widehat{\mathfrak{gl}}_{#1})}

\newcommand{\dotugl}[1]{\dot{\bf U}(\mathfrak{gl}_{#1})}

\newcommand{\dotu}[1]{\dot{\bf U}({#1})}
\newcommand{\dotau}[1]{\dot{\bf U}_{\triangle}({#1})}

\newcommand{\olambda}{\overline{\lambda}}

\newcommand{\levl}{{\vert\lambda\vert}}

\DeclareMathOperator{\End}{End}

\newtheorem{thm}{Theorem}[section]
\newtheorem{lem}[thm]{Lemma}

\newtheorem{prop}[thm]{Proposition}

\theoremstyle{definition}
\newtheorem{defn}[thm]{Definition}

\newtheorem{rem}[thm]{Remark}

\newcommand{\bZ}{\mathbb{Z}}
\newcommand{\bQ}{\mathbb{Q}}

\newcommand{\cC}{\mathcal{C}}

\newcommand{\cK}{K}

\DeclareMathOperator{\Hom}{Hom}

\catcode`\@=11
\long\def\@makecaption#1#2{%
    \vskip 10pt
    \setbox\@tempboxa\hbox{%
\small{#1: }\ignorespaces #2}%
    \ifdim \wd\@tempboxa >\captionwidth {%
        \rightskip=\@captionmargin\leftskip=\@captionmargin
        \unhbox\@tempboxa\par}%
      \else
        \hbox to\hsize{\hfil\box\@tempboxa\hfil}%
    \fi}
\newdimen\@captionmargin\@captionmargin=2\parindent
\newdimen\captionwidth\captionwidth=\hsize
\catcode`\@=12
\def\makeautorefname#1#2{\expandafter\def\csname#1autorefname\endcsname{#2}}
\makeautorefname{lem}{Lemma}%
\makeautorefname{prop}{Proposition}%
\makeautorefname{rem}{Remark}%
\makeautorefname{section}{Section}%
\makeautorefname{subsection}{Section}%
\makeautorefname{subsubsection}{Section}%
\makeautorefname{equation}{Eq.}%

\title{Evaluation 2-functors for Kac--Moody 2-categories of type $A_2$}
\author{Marco Mackaay}
\address{M.M.: Centro de An\'{a}lise Matem\'{a}tica, Geometria e Sistemas Din\^{a}micos, Departamento de Matem\'{a}tica, Instituto Superior T\'{e}cnico, 1049-001 Lisboa, PORTUGAL \&
Departamento de Matem\'{a}tica, FCT, Universidade do Algarve, Campus de Gambelas,
8005-139 Faro, PORTUGAL \& Center for Research and Development in Mathematics and Applications (CIDMA), Department of Mathematics, University of Aveiro, 3810-193 Aveiro, PORTUGAL, \newline \href{https://fct.ualg.pt/bio/mmackaay}{https://fct.ualg.pt/bio/mmackaay}, \href{https://orcid.org/0000-0001-9807-6991}{ORCID 0000-0001-9807-6991}} 
\email{mmackaay@ualg.pt}
\author{James Macpherson}
\address{J.M.: entro de An\'{a}lise Matem\'{a}tica, Geometria e Sistemas Din\^{a}micos, Departamento de Matem\'{a}tica, Instituto Superior T\'{e}cnico, 1049-001 Lisboa, PORTUGAL,\newline 
\href{https://orcid.org/0009-0006-4091-8619}{ORCID 0009-0006-4091-8619}}
\email{jdam2@cantab.net}
\author{Pedro Vaz}
\address{P.V.: Institut de Recherche en Math{\'e}matique et Physique, 
Universit{\'e} Catholique de Louvain, Chemin du Cyclotron 2,  
1348 Louvain-la-Neuve, Belgium, \newline \href{https://perso.uclouvain.be/pedro.vaz}{https://perso.uclouvain.be/pedro.vaz}, \href{https://orcid.org/0000-0001-9422-4707}{ORCID 0000-0001-9422-4707}}
\email{pedro.vaz@uclouvain.be}

\usepackage{pinlabel}
\usepackage[pagewise]{lineno}
\begin{document}
%
%

\begin{abstract} 
We construct a 2-functor from the Kac--Moody 2-category for the extended quantum affine $\mathfrak{sl}_3$ to the homotopy 2-category of bounded chain complexes with values in the Kac--Moody 2-category for quantum $\mathfrak{gl}_3$, categorifying the evaluation map between the corresponding quantum Kac--Moody algebras. Our approach establishes and exploits a categorical analogue of the well-known relation between the evaluation map and Lusztig's internal braid group action for quantum $\mathfrak{gl}_3$.
\end{abstract}

\maketitle

{\hypersetup{hidelinks}
\tableofcontents 
}

\pagestyle{myheadings}
\markboth{\footnotesize M.~Mackaay,~J.~Macpherson,~P.~Vaz }{\footnotesize Evaluation 2-functors for Kac--Moody 2-categories }
%
%

\section{Introduction}

In the 90s, Chari and Pressley launched a systematic study of finite-dimensional representations of quantum affine algebras, starting with affine $\mathfrak{sl}_2$ in~\cite{chari-pressley91}. Since then, these representations have been studied intensively and continue to be an active research topic with important open questions and interesting links to other research areas, e.g. mathematical physics and cluster algebras, see e.g.~\cite{hernandez-leclerc20} for more information. 

In affine type $A$, there is a special class of irreducible finite-dimensional representations, the so-called {\em evaluation representations}. These are obtained by pulling back irreducible representations of finite type $A$ through a so-called {\em evaluation map}, which is an algebra homomorphism $\ev_{a,n} \colon {\mathbf U}_{\triangle}(n)\to \mathbf{U}(n)$, where $a\in \mathbb{C}^\times$ is a scalar, ${\mathbf U}_{\triangle}(n)$ is the so-called {\em extended quantum affine $\sln$} and $\mathbf{U}(n)={\mathbf U}_q(\gln)$, see~\autoref{Sec:Decategorified}. As we will recall in more detail in that section, the level zero weight lattice of the former quantum affine Kac--Moody algebra can be identified with the $\gln$-weight lattice. The fact that we have to pass from $\sln$ to $\gln$ is important and has a categorical counterpart, as we will explain below. For more information on evaluation maps and  evaluation representations in general, see e.g.~\cite{chari-pressley94, chari-pressley942, Du-Fu}. Since we wish to categorify this construction, we must pass to the idempotented forms of the above algebras, which can be considered as categories, and the evaluation map can therefore be considered as a functor.

Quantum Kac--Moody algebras were {\em categorified} by Khovanov and Lauda~\cite{Kh-L}, and independently by Rouquier~\cite{rouquier2008}. We call these 2-categories {\em Kac--Moody 2-categories} after \cite{brundan16on}. The ones of interest to us in this paper are $\affu{n}$ and $\naffu{n}$, which categorify ${\mathbf U}_{\triangle}(n)$ and $\mathbf{U}(n)$, respectively. The tilde indicates that our choice of signs in their definition differs from Khovanov and Lauda's original choices, see below for more comments on this. In finite Dynkin types, all irreducible finite-dimensional representations can be categorified by certain quotients of the Kac--Moody 2-categories, which nowadays go under the name of {\em cyclotomic KLR algebras}. In other Dynkin types, e.g. affine Dynkin types, this is not true. In particular, evaluation representations in affine type $A$ cannot be categorified by cyclotomic KLR-algebras, because the latter categorify highest weight representations and evaluation representations do not have a highest weight. However, we conjecture that the evaluation map (considered as an evaluation functor) $\ev_n^{\,t}:=\ev_{q^t,n}$, for any $t\in \mathbb{Z}$ and $n\in \mathbb{N}_{> 2}$, can be categorified by an evaluation 2-functor $\Ev_n^{\;t}\colon \affu{n}\to \cK^b(\naffu{n})$, which can be used to define evaluation 2-representations (i.e. categorified evaluation representations) of $\affu{n}$ by pulling back \say{irreducible} 2-representations (i.e. cyclotomic KLR algebras) of $\naffu{n}$. Here $\cK^b(\naffu{n})$ denotes the homotopy $2$-category of bounded complexes in $\naffu{n}$, so the $1$-morphisms of $\affu{n}$ act by composing with bounded complexes in $\naffu{n}$. As a matter of fact, we not only conjecture $\Ev_n^{\;t}$ to exist but also an extension 
of it to $\cK^b(\affu{n})$. 

In this paper, we prove the first conjecture for $\affu{3}$ and hope that it serves as the base case for an inductive proof for $\affu{n}$, when $n > 3$, in a forthcoming paper. Proving that there is no obstruction to extending $\Ev_n^{\;t}$ to $\cK^b(\affu{n})$ is not easy and certainly beyond the scope of this paper and its sequel. 

There are two good reasons for publishing the case $n=3$ separately. Firstly, in this case there is a close relation with 
the categorification of the internal braid group action on ${\mathbf U}_q(\mathfrak{gl}_3)$ in \cite{abram2022categorification} (strictly speaking, in that paper they consider ${\mathbf U}_q(\mathfrak{sl}_3)$, so part of our work consists in adapting their results to our setting - see the following paragraph for more details). This is the categorical analogue of a relation between the evaluation map and the braid group action on the decategorified level, which is certainly known to experts, although we couldn't find a reference in the literature. We therefore spell it out in \autoref{sec:DecatT}, because it is not completely straightforward. Its categorification is conceptually clear, but requires solving multiple non-trivial sign problems, which we do by using certain $2$-isomorphisms. This is also why we define two versions of the evaluation $2$-functor, denoted $\Ev$ and $\Ev'$, respectively. The former uses relatively nice sign conventions, whereas the signs in the definition of $\Ev'$ are much more complicated. However, the latter are easier to match with the signs in the categorified internal braid group action (for our choice of signs in $\affu{3}$ and $\naffu{3}$), which is necessary to prove that $\Ev'$ is well-defined in our approach, see  \autoref{thm:BigThmPrime} and its proof in \autoref{sec:BigThPrime}. The relation between $\Ev$ and $\Ev'$, given in \autoref{lem:EvRel}, guarantees that well-definedness of the latter implies well-definedness of the former. In principle, all of this should also work for $n>3$, but only if the categorified braid group action extends to $\cK^b(\naffu{n})$ (to include the action of longer braids), which has been conjectured to be the case but not yet proved (see \cite[Conjecture 1.2]{abram2022categorification}). This is why our approach for $n>3$ will be completely different. We hope that presenting the base case $n=3$ here will prepare the ground for the general case and also keep the size of the forthcoming paper within reasonable bounds. 

The second reason for publishing this case separately, is that it reveals the need to pass from $\sln$ to $\gln$ once more, but now on the categorical level. Recall that the definition of a Kac--Moody 2-category depends on a choice of invertible scalars and compatible bubble parameters, see e.g.~\cite{lauda2020parameters}. In finite type $A$ all choices yield essentially the same $2$-category, i.e. up to $2$-isomorphism, but in affine type $A$ they don't. In particular, Khovanov and Lauda's original affine type $A$ {\em unsigned} Kac--Moody 2-category in~\cite{Kh-L}, with all scalars and bubble parameters equal to one, and the Kac--Moody 2-category defined in~\cite{mt-affine-schur}, with non-trivial bubble parameters depending on level zero $\widehat{\mathfrak{gl}}_n$-weights (instead of level zero $\widehat{\mathfrak{sl}}_n$-weights), are not 2-isomorphic when $n$ is odd. This was mentioned in \cite{Kh-L-2} without proof and, therefore, we prove it in \autoref{thm:NoIsomorphism}. Although this does not by itself imply that there is no evaluation 2-functor for trivial scalars and bubble parameters when $n=3$, we failed to find one. More generally, it seems that one is forced to use the scalars and level zero $\widehat{\mathfrak{gl}}_n$-bubble parameters from~\cite{mt-affine-schur} when $n$ is odd. When $n$ is even, everything is simpler because in that case both choices of scalars and bubble parameters yield essentially the same Kac--Moody 2-category, see \autoref{sec:No2Iso}.

There is an analogous story for the affine Hecke algebra and its finite-dimensional representations. The categorification of the corresponding evaluation map was carried out in~\cite{mmv22} and was technically less challenging than the categorification of the evaluation map for the affine type A Kac--Moody algebra. In both cases, the target (2-)category of the evaluation (2-)functor is a homotopy category of bounded complexes and, as was argued in~\cite{mmv22}, one motivation for defining and studying evaluation 2-representations is that they might provide some important clues for the development of triangulated 2-representation theory, which at the moment is very poorly understood, even at the most basic level. For example, it was shown in~\cite{mmv22} that every evaluation 2-representation of extended affine Soergel bimodules has a {\em finitary cover}, somehow relating finitary and triangulated 2-representations. The same might hold for the evaluation 2-representations of $\affu{n}$, but that question is outside the scope of this paper. Also in both cases, one would like to categorify tensor products of evaluation representations, which play a fundamental role in the finite-dimensional representation theory of the affine quantum algebras in question, see e.g.~\cite[Chapter 12, Section 2C]{chari-pressley942} for 
the case of $\affu{n}$. However, it is far from clear how to do that at this point. Perhaps it is possible to somehow adapt Webster's tensor algebras of Stendhal diagrams~\cite{webster17} in that case. We hope to address these and some other interesting questions about evaluation 2-representations in the future.

The structure of the paper is as follows. \autoref{Sec:Decategorified} reviews the evaluation map/functor $\ev^t_3$ and \autoref{Sec:CategorifiedDefn} presents the definitions of the affine and finite type A Kac--Moody 2-categories $\affu{3}$ and $\naffu{3}$ that we will be working with. In \autoref{Sec:Eval2Fr} we define the two evaluation 2-functors $\Ev$ and $\Ev'$ and prove their relationship to each other. We translate the categorified braid group actions to our choice of scalars in \autoref{sec:BraidAction}, and then in \autoref{sec:MainProof} we prove \autoref{thm:BigThmPrime}, that $\Ev'$ is a well-defined 2-functor that decategorifies to $\ev^t_3$, from which \autoref{thm:BigThmMain} follows. We finish the paper with \autoref{sec:No2Iso}, where we justify our choice of the scalars and bubble parameters in the definition of $\affu{3}$ over a choice in \cite{Kh-L} by proving in \autoref{thm:NoIsomorphism} that the two choices are not related by a 2-isomorphism that fixes objects and 1-morphisms.

\vspace{0.1in}

\noindent{\bf Acknowledgments.} M.M. and J.M. were supported by Funda\c{c}\~{a}o para a Ci\^{e}ncia e a Tecnologia (Portugal) (\href{https://ror.org/00snfqn58}{https://ror.org/00snfqn58}), project UIDB/04459/2020 (Center for Mathematical Analysis, Geometry and Dynamical Systems - CAMGSD) with DOI identifier 10-54499/UIDP/ 04459/2020. M.M. was additionally funded by CIDMA (\href{https://ror.org/05pm2mw36}{https://ror.org/05pm2mw36}) under the FCT (Portugal) grant UID/04106/2025 (\href{https://doi.org/10.54499/UID/04106/2025}{https://doi.org/10.54499/UID/04106/2025}) and UID/PRR/ 04106/2025, and J.M. was additionally funded by FCT (Portugal) through project BL94/2022-IST-ID. P.V. was supported by the Fonds de la Recherche Scientifique-FNRS (Belgium) under Grant no. J.0189.23. 

We thank the referee for the suggestion of making the relation between the categorified braid group action and the evaluation 2-functor more explicit.

\vspace{0.1in}

\labelmargin-{-8pt}

\section{The decategorified setting}\label{Sec:Decategorified}
Our main reference for this section is \cite{Du-Fu}, though the evaluation map was first considered in \cite{jimbo1986q}. Note that we are interested in the idempotented version of some of the quantum algebras in that paper, so we have to adapt Du and Fu's definitions. We use the idempotented versions because these are the ones that are categorified by Kac--Moody 2-categories.

\subsection{Finite type \texorpdfstring{$\dotu{n}$}{un} and affine type \texorpdfstring{$\dotau{n}$}{aun} of level zero}\label{sec:decat}

Throughout this paper we identify both the (integral) $\gln$-weight lattice and the level-zero (integral) $\agln$-weight lattice with $\bZ^{n}$, denoting either sort of (integral) weight by e.g. $\lambda=(\lambda_1,\ldots,\lambda_n)\in \mathbb{Z}^n$. The simple $\agln$-roots $\alpha_1,\dotsc,\alpha_{n}$ are then given by
\[
\alpha_i =
\begin{cases}
(0,\dots,0,1,-1,0,\dots,0) & 1\leq i\leq n-1,
\\
(-1,0,\dots,0,1) & i=n,
\end{cases}
\]
where the $1$ is always the $i$th entry. Note that $\alpha_1,\ldots,\alpha_{n-1}$ are the simple $\gln$-roots.

Under the above identification, the bilinear form 
on these weight lattices corresponds to the Euclidean inner product on $\mathbb{Z}^n$. Its restriction to the root lattices then reads  
\[
(\alpha_i,\alpha_j) = 
\begin{cases}
2, & \text{if}\; i=j, \\
-1, & \text{if}\; i \equiv j\pm 1 \bmod n,\\
0 & \text{else}, 
\end{cases}
\]
for $1\leq i,j\leq n$. Note that, in the affine case, the indices $1,\ldots, n$ are interpreted as representatives of the residue classes modulo $n$. From now on, we will always tacitly use this interpretation of the indices of affine weights and roots. We also recall the standard notation $i\cdot j:= (\alpha_i,\alpha_j)$, which we will often use below.

Finally, given $\lambda\in \mathbb{Z}^n$, define $\olambda=(\olambda_1,\ldots,\olambda_n)\in \mathbb{Z}^n$, where $\olambda_i=\lambda_i-\lambda_{i+1}$ for all $i=1,\ldots, n$. By the above convention for affine weights, we have $\olambda_n=\lambda_n-\lambda_1$, so $\olambda_1+\ldots+ \olambda_n=0$. In other words, $\olambda$ belongs to a rank $n-1$ sublattice of $\mathbb{Z}^n$, which can be identified with the level-zero integral $\asln$-weight lattice. The element $(\olambda_1,\ldots,\olambda_{n-1})\in \mathbb{Z}^{n-1}$ 
can then be identified with an (integral) $\mathfrak{sl}_n$-weight. 

For the definition below, recall that the quantum integer $[m]$, for $m\in \mathbb{Z}$, is defined as

\[
[m]=\dfrac{q^m-q^{-m}}{q-q^{-1}}.
\]
\begin{defn}\label{d:dotaun}
    The \emph{idempotented extended quantum affine $\sln$}, denoted by $\dotau{n}$, is the associative idempotented $\bQ(q)$-algebra generated by $1_\lambda$, $E_i1_\lambda$ and  $F_i1_\lambda$, for $\lambda\in\mathbb{Z}^n$ and $i=1,\dotsc ,n$, subject to the relations: 
\begingroup\allowdisplaybreaks
\begin{align*} 
    1_\lambda 1_{\mu} &= \delta_{\lambda,\mu} 1_\lambda ,  && 
\\[1ex]
E_i1_\lambda 1_{\lambda'}&=\delta_{\lambda,\lambda'} E_i1_\lambda , &&\\[1ex]
F_i1_\lambda 1_{\lambda'}&=\delta_{\lambda,\lambda'} F_i1_\lambda , &&\\[1ex]
1_{\mu} E_i1_\lambda &=\delta_{\mu,\lambda + \alpha_i} E_i1_\lambda , &&\\[1ex]
1_{\mu} F_i1_\lambda &=\delta_{\mu,\lambda-\alpha_i} F_i1_\lambda , &&\\[1ex]
E_iF_j1_\lambda - F_jE_i1_\lambda &= \delta_{i,j}[\olambda_i]1_\lambda , && 
\\[1ex]
E_i E_j 1_\lambda &= E_j E_i 1_\lambda  && \text{ if } i\cdot j =0,
\\[1ex]
F_i F_j 1_\lambda &= F_j F_i 1_\lambda && \text{ if } i\cdot j=0 ,
\\[1ex]
E_i^2 E_j 1_\lambda + E_jE_i^2 1_\lambda &= [2] E_i E_j E_i 1_\lambda && \text{ if } i\cdot j =-1 ,
\\[1ex]
F_i^2 F_j 1_\lambda + F_jF_i^2 1_\lambda &= [2] F_i F_j F_i 1_\lambda && \text{ if } i\cdot j =-1 .
\end{align*}    
\endgroup
 \end{defn}
Note that $E_i 1_\lambda = 1_{\lambda + \alpha_i} E_i 1_\lambda$, so we can use the notation $E_i E_j 1_\lambda := E_i 1_{\lambda +\alpha_j} \cdot E_j1_\lambda$ without ambiguity. Similarly, we will use the notation $1_{\mu} E_i = 1_{\mu} E_i 1_{\mu - \alpha_i}$ and $1_\mu F_i = 1_\mu F_i 1_{\mu + \alpha_i}$, so that $E_i 1_\lambda = 1_{\lambda + \alpha_i} E_i$ and $F_i 1_\lambda = 1_{\lambda - \alpha_i} F_i$.

\begin{defn}\label{d:dotun}
The \emph{idempotented quantum $\gln$}, denoted by $\dotu{n}$, is the idempotented subalgebra of $\dotau{n}$ generated by $1_{\lambda}$, $E_i1_{\lambda}$ and $F_i1_{\lambda}$, for $i=1,\ldots ,n-1$ and $\lambda\in\mathbb{Z}^n$.
\end{defn}
Note that $\dotau{n}$ and $\dotu{n}$ share the same idempotents, but, whereas $\dotu{n}=\dotugl{n}$, the idempotented algebra $\dotau{n}$ is only an idempotented subalgebra of $\dotuagl{n}$, which is why it is called 
the idempotented {\em extended} quantum affine $\sln$ and 
not the idempotented quantum affine $\gln$, see \cite[Section 2]{Du-Fu} for more details.

\begin{rem}\label{rem:IdempCat}
    Recall that these idempotented algebras can be seen as linear categories whose object sets are given by the sets of weights and whose hom-spaces are given by e.g. $$\Hom_{\dotu{n}}(\lambda,\mu)=1_\mu\dotu{n}1_\lambda$$ with composition corresponding to multiplication. This is why these idempotented algebras are categorified by 2-categories rather than categories.
\end{rem}

\subsection{Evaluation maps} Fix $t\in \mathbb{Z}$ and let $[X,Y]_{q^{\pm 1}} = XY-q^{\pm 1} YX$ be the $q^{\pm 1}$-commutator. From now on we will always assume that $n>2$.

\begin{defn}\label{d:evala}
The \emph{evaluation map} $\ev_{n}^{\,t}\colon \dotau{n}\to\dotu{n}$ is the homomorphism of idempotented algebras defined by
\begin{align}
    \ev_{n}^{\,t}(1_{\lambda}) &=   1_{\lambda} , 
\\[1ex] 
    \ev_{n}^{\,t}(E_i1_\lambda) &= E_i1_\lambda \quad \text{for }i\neq n,
    \\[1ex] 
 \ev_{n}^{\,t}(F_i1_\lambda) &= F_i1_\lambda \quad \text{for }i\neq n,
\\[1ex] \label{eq:evmape}
    \ev_{n}^{\,t}(E_n 1_\lambda) &= q^{\lambda_1+\lambda_n+t-1} [\dotsm [[F_1,F_2]_q,F_3]_q\dotsm]_q, F_{n-1}]_q 1_\lambda,
\\[1ex] \label{eq:evmapf}
    \ev_{n}^{\,t}(F_n 1_\lambda) &= q^{-\lambda_1-\lambda_n-t+1}[E_{n-1},[E_{n-2},[\dotsm [E_{2},E_1]_{q^{-1}}]_{q^{-1}}\dotsm]_{q^{-1}} 1_\lambda.    
\end{align}
\end{defn}

\begin{rem} Two quick observations:
\begin{enumerate}[wide,labelindent=0pt,itemsep=5pt,label=(\alph*)]

\item Note that 
\[
[\dotsm [[F_1,F_2]_q,F_3]_q\dotsm]_q, F_{n-1}]_q 1_\lambda = 
1_{\lambda - \alpha_1 - \dotsm -\alpha_{n-1}} [\dotsm [[F_1,F_2]_q,F_3]_q\dotsm]_q, F_{n-1}]_q, 
\]
so $\ev_{n}^{\,t}(E_n 1_\lambda)=\ev_{n}^{\,t}(1_{\lambda+\alpha_n}E_n 1_\lambda)$ is well defined, because $\alpha_1+\ldots+\alpha_{n-1}+\alpha_n=0$. The same is true for $\ev_{n}^{\,t}(F_n1_\lambda)=\ev_{n}^{\,t}(1_{\lambda-\alpha_n}F_n1_\lambda)$.
\item When we consider the idempotented algebras as categories as in \autoref{rem:IdempCat}, $\ev_n^t$ becomes a linear functor. This is why it is categorified by a 2-functor rather than a functor.
\end{enumerate}
\end{rem}

\vspace{0.1in}

 The expressions for $\ev_n^{\,t}(E_n 1_\lambda)$ and $\ev_n^{\,t}(F_n 1_\lambda)$ in~\eqref{eq:evmape} and ~\eqref{eq:evmapf} can be written 
as alternating sums, which will be important later on. 
For $\xi=(\xi_1,\dotsc,\xi_{n-2})\in\{0,1\}^{n-2}$ set  
\begin{align}
\label{eq:exi}
    E_\xi 1_\lambda &:= E_{n-1}^{1-\xi_{n-2}}E_{n-2}^{1-\xi_{n-3}}\dotsm E_2^{1-\xi_1}E_1 E_2^{\xi_{1}}\dotsm E_{n-2}^{\xi_{n-3}}
    E_{n-1}^{\xi_{n-2}}1_\lambda ,
    \\[1ex]
\label{eq:fxi}
    F_\xi 1_\lambda &:= F_{n-1}^{\xi_{n-2}}F_{n-2}^{\xi_{n-3}}\dotsm F_2^{\xi_1}F_1F_2^{1-\xi_{1}}\dotsm F_{n-2}^{1-\xi_{n-3}}
    F_{n-1}^{1-\xi_{n-2}}1_\lambda .
\end{align}
and let $\vert\xi\vert=\xi_1+\dotsm+\xi_{n-2}$. The following can be obtained by direct computation.
\begin{lem}
We have
\begin{align}\label{eq:evmapae}
\ev_n^{\,t}(E_n 1_\lambda) &=q^{\lambda_1+\lambda_n+t-1} \sum\limits_{\xi\in \{0,1\}^{n-2}}(-q)^{\vert\xi\vert}F_\xi 1_\lambda ,
\\[1ex] \label{eq:evmapaf}
\ev_n^{\,t}(F_n 1_\lambda) &= q^{-\lambda_1-\lambda_n-t+1}\sum\limits_{\xi\in \{0,1\}^{n-2}}(-q)^{-\vert\xi\vert}E_\xi 1_\lambda .    
\end{align}
\end{lem}
\noindent For more details on the evaluation map, see \cite[Section 5]{Du-Fu}. 

\subsubsection{Connection with the braid group action for \texorpdfstring{$n=3$}{n=3}}\label{sec:DecatT}
For each $i=1,\ldots, n-1$ and $e=\pm 1$, Lusztig defined algebra automorphisms $T'_{i,e}$ 
and $T''_{i,e}$ of $\dotu{\mathfrak{sl}_n}$, see e.g.~\cite[Section 37.1]{lusztig2010introduction} for their definition, which we can adapt to $\dotu{n}$ without issue. The two automorphisms 
are related by the equation $(T'_{i,e})^{-1}=T''_{i,-e}$ (see \cite[Proposition 37.1.2]{lusztig2010introduction}) and, for a fixed choice 
of $e$, the $T'_{1,e},\ldots, T'_{n-1,e}$, resp. the $T''_{1,e},\ldots, T''_{n-1,e}$, satisfy the 
braid relations (see~\cite[Theorem 39.4.3]{lusztig2010introduction}) and, therefore, define two actions of the braid group $B_n$ on $\dotu{n}$, called the {\em internal braid group actions}. 

Let $n=3$ and $t\in\mathbb{Z}$, and set $\ev=\ev^t_3$. Comparison of the expressions in~\cite[Subsection 37.1.3]{lusztig2010introduction} with the ones in \autoref{d:evala} shows that $\ev$ can be partially expressed in terms of the above algebra automorphisms.
For $i=1,3$ and $\lambda=(\lambda_1, \lambda_2,\lambda_3)\in \mathbb{Z}^3$, we have 
\begin{eqnarray*}\label{eq:evbraid13}
\ev(E_3 1_{\lambda})&=&q^{\lambda_1+\lambda_3+t-1}\, T'_{1,-1}(F_2 1_{s_1(\lambda)}),\\
\ev(F_3 1_{\lambda})&=&q^{-\lambda_1-\lambda_3-t+1}\, T'_{1,-1}(E_2 1_{s_1(\lambda)}),\\
\ev(E_1 1_{\lambda})&=&-q^{\lambda_1-\lambda_2}\, T'_{1,-1}(F_1 1_{s_1(\lambda)}),\\
\ev(F_1 1_{\lambda})&=&-q^{-\lambda_1+\lambda_2+2}\, T'_{1,-1}(E_1 1_{s_1(\lambda)}),
\end{eqnarray*}
where $s_1(\lambda)=(\lambda_2, \lambda_1,\lambda_3)$. 
For $i=2,3$ and $\lambda=(\lambda_1, \lambda_2,\lambda_3)\in \mathbb{Z}^3$, we have 
\begin{eqnarray*}\label{eq:evbraid23}
\ev(E_3 1_{\lambda})&=&q^{\lambda_1+\lambda_3+t-1}\, T''_{2,1}(F_1 1_{s_2(\lambda)}),\\
\ev(F_3 1_{\lambda})&=&q^{-\lambda_1-\lambda_3-t+1}\, T''_{2,1}(E_1 1_{s_2(\lambda)}),\\
\ev(E_2 1_{\lambda})&=&-q^{-\lambda_2+\lambda_3+2}\, T''_{2,1}(F_2 1_{s_2(\lambda)}),\\
\ev(F_2 1_{\lambda})&=&-q^{\lambda_2-\lambda_3}\, T''_{2,1}(E_2 1_{s_2(\lambda)}),
\end{eqnarray*}
where $s_2(\lambda)=(\lambda_1, \lambda_3,\lambda_2)$. Using the fact that 
$T'_{1,-1}$ and $T'_{2,1}$ are well-defined algebra automorphisms of $\dotu{3}$, it is easy to prove that  $\ev\colon \dotau{3}\to \dotu{3}$ is a well-defined algebra homomorphism. Specifically, the fact that $T'_{1,-1}$ is an algebra automorphisms implies that $\ev$ preserves the relations in \autoref{d:dotaun} for $i=1,3$, the fact that $T''_{2,1}$ is an algebra automorphisms implies that $\ev$ preserves the relations in \autoref{d:dotaun} for $i=2,3$, and $\ev$ preserves the relations in \autoref{d:dotaun} for $i=1,2$ by definition, of course. Since all relations in \autoref{d:dotaun} involve either one colour $i$ or two colours $i,j$, and it's very easy to check that $\ev$ preserves the one-colour relations directly, we see that $\ev$ preserves all relations in $\dotau{3}$ and is therefore a well-defined algebra homomorphism. 

Of course, one can also prove that $\ev$ preserves the relations in $\dotau{3}$ directly, but that is besides the point. To show that the evaluation 2-functor $\Ev$ preserves the relations in $\affu{3}$, we will follow the same reasoning as above for all one- and two-colour KLR relations, taking advantage of the categorification of $T'_{i,1}$ in~\cite{abram2022categorification}. For the three-colour KLR relations, the results in that paper cannot be used and we will give a direct proof. 

\section{Kac--Moody 2-categories}\label{Sec:CategorifiedDefn}

We will move on to recalling in detail the $2$-categories $\naffu{n}$ and $\affu{n}$ as defined in ~\cite[Definition 3.1]{msv-schur} and~\cite[Definition 3.19]{mt-affine-schur}, respectively. These decategorify
to $\dotu{n}$ and $\dotau{n}$. 
 
\subsection{Definition} We define $\affu{n}$ and $\naffu{n}$ simultaneously, because only the range of the indices of the $1$-morphisms and of the colours of the $2$-morphisms differ. For concreteness, we will work over $\bQ$, but any field of characteristic zero would serve equally well.

\begin{defn}\label{defn:KLR-2cats} The $2$-category $\affu{n}$ (resp. $\naffu{n}$) is the graded $\bQ$-linear $2$-category with:
\begin{itemize}[wide,labelindent=0pt,itemsep=5pt]
\item Objects: $\lambda\in \mathbb{Z}^n$,
\item $1$-morphisms: formal direct sums of shifts of 
\[
\oneid_\lambda,\quad \E_i\oneid_\lambda=\oneid_{\lambda+\alpha_i} \E_i\oneid_{\lambda}=\oneid_{\lambda+\alpha_i}\E_i,\quad \E_i\oneid_\lambda=\oneid_{\lambda-\alpha_i} \E_i\oneid_{\lambda}=\oneid_{\lambda-\alpha_i}\E_i,
\]
for $\lambda\in \mathbb{Z}^n$ and for $i\in\{1,\dots,n\}$ (resp. $i\in\{1,\dots,n-1\}$),
\item $2$-morphisms: equivalence classes of $\bQ$-linear combinations of diagrams obtained by horizontally concatenating and vertically gluing the generators below. By convention, a $2$-morphism   
$\alpha\colon X\langle r\rangle\to Y\langle s\rangle$, for $r,s\in \mathbb{Z}$, is given by a linear combination of homogeneous diagrams 
of degree $s-r$, as defined in \cite{Kh-L}. 
\begingroup\allowdisplaybreaks
\begin{align*}
    \xy (0,0)*{
			\labellist
			\small\hair 2pt
                \pinlabel \scalebox{0.7}{$i$} at 2 -5
			\pinlabel \scalebox{0.7}{$\lambda+\alpha_i$} at -15 12
			\pinlabel \scalebox{0.9}{$\lambda$} at 12 12
			\endlabellist 
			\centering 
			\includegraphics[scale=1.3]{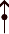}
	}\endxy 
 \quad\; 
 &
 \colon \E_i\oneid_{\lambda}\to \E_i\oneid_{\lambda}\langle 2\rangle ,
 &  
\xy (0,0)*{\labellist
			\small\hair 2pt
                 \pinlabel \scalebox{0.7}{$i$} at 2 -5
			\pinlabel \scalebox{0.7}{$\lambda-\alpha_i$} at -15 12
			\pinlabel \scalebox{0.9}{$\lambda$} at 12 12
			\endlabellist 
             \centering 
			\includegraphics[scale=1.3]{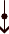}
	}\endxy 
 \quad\; 
 &\colon \F_i\oneid_{\lambda}\to \F_i\oneid_{\lambda}\langle 2\rangle , 
 \\[5ex]
  \xy (0,0)*{
			\labellist
			\small\hair 2pt
   \pinlabel \scalebox{0.7}{$i$} at -2 -5
   \pinlabel \scalebox{0.7}{$j$} at 15 -5
			\pinlabel \scalebox{0.9}{$\lambda$} at 22 9
			\endlabellist 
			\centering 
			\includegraphics[scale=1.3]{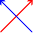}
	}\endxy 
 \quad\; 
 &
 \colon \E_i\E_j\oneid_{\lambda}\to \E_j\E_i\oneid_{\lambda}\langle -i\cdot j\rangle ,
 &  
\xy (0,0)*{
			\labellist
			\small\hair 2pt
   \pinlabel \scalebox{0.7}{$i$} at -2 -5
   \pinlabel \scalebox{0.7}{$j$} at 15 -5
			\pinlabel \scalebox{0.9}{$\lambda$} at 22 9
			\endlabellist 
			\centering 
			\includegraphics[scale=1.3]{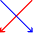}
	}\endxy 
 \quad\; 
 &
 \colon \F_iF_j\oneid_{\lambda}\to \F_i\F_j\oneid_{\lambda}\langle -i\cdot j\rangle ,
 \\[5ex]
  \xy (0,0)*{
			\labellist
			\small\hair 2pt
   \pinlabel \scalebox{0.7}{$i$} at -2 -5
   \pinlabel \scalebox{0.7}{$j$} at 15 -5
			\pinlabel \scalebox{0.9}{$\lambda$} at 22 9
			\endlabellist 
			\centering 
			\includegraphics[scale=1.3]{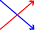}
	}\endxy 
 \quad\; 
 &
 \colon \E_i\F_j\oneid_{\lambda}\to \F_j\E_i\oneid_{\lambda} ,
 & 
\xy (0,0)*{
			\labellist
			\small\hair 2pt
   \pinlabel \scalebox{0.7}{$i$} at -2 -5
   \pinlabel \scalebox{0.7}{$j$} at 15 -5
			\pinlabel \scalebox{0.9}{$\lambda$} at 22 9
			\endlabellist 
			\centering 
			\includegraphics[scale=1.3]{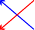}
	}\endxy 
 \quad\; 
 &
 \colon \F_iE_j\oneid_{\lambda}\to \E_i\F_j\oneid_{\lambda} ,
 \\[5ex]
  \xy (0,0)*{
			\labellist
			\small\hair 2pt
   \pinlabel \scalebox{0.7}{$i$} at 8 -5
			\pinlabel \scalebox{0.9}{$\lambda$} at 22 0
			\endlabellist 
			\centering 
			\includegraphics[scale=1.3]{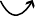}
	}\endxy 
 \quad\; 
 &\colon \oneid_{\lambda}\to \F_i\E_i\oneid_{\lambda}
 \langle 1+\overline{\lambda}_i \rangle ,
 & 
\xy (0,0)*{
			\labellist
			\small\hair 2pt
   \pinlabel \scalebox{0.7}{$i$} at 8 -5
			\pinlabel \scalebox{0.9}{$\lambda$} at 22 0
			\endlabellist 
			\centering 
			\includegraphics[scale=1.3]{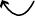}
	}\endxy 
 \quad\; &\colon \oneid_{\lambda}\to \E_i\F_j\oneid_{\lambda} 
 \langle 1-\overline{\lambda}_i \rangle ,
 \\[5ex]
  \xy (0,0)*{
			\labellist
			\small\hair 2pt
   \pinlabel \scalebox{0.7}{$i$} at 8 13
			\pinlabel \scalebox{0.9}{$\lambda$} at 22 9
			\endlabellist 
			\centering 
			\includegraphics[scale=1.3]{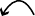}
	}\endxy 
 \quad\; 
 &
 \colon \F_i\E_1\oneid_{\lambda}\to \oneid_{\lambda} 
 \langle 1+\overline{\lambda}_i\rangle ,
 & 
\xy (0,0)*{
			\labellist
			\small\hair 2pt
   \pinlabel \scalebox{0.7}{$i$} at 8 13
			\pinlabel \scalebox{0.9}{$\lambda$} at 22 9
			\endlabellist 
			\centering 
			\includegraphics[scale=1.3]{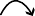}
	}\endxy 
 \quad\; 
 &
 \colon \E_iF_i\oneid_{\lambda}\to \oneid_{\lambda}\langle 1-\overline{\lambda}_i\rangle .
\end{align*}
\endgroup

\end{itemize}

The equivalence relation is defined by the equations below.

\begin{enumerate}[wide,labelindent=0pt,label=(KM\arabic*)]

\item Right and left adjunction:
\begin{equation}\label{eq:adjunction}
\begin{array}{ll}
 \xy (0,1)*{
			\labellist
			\small\hair 2pt
			\pinlabel \scalebox{0.7}{$i$} at 0 -5
			\pinlabel \scalebox{0.9}{$\lambda$} at 19 12
			\endlabellist 
			\centering 
			\includegraphics[scale=1.4]{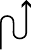}
	}\endxy
 \quad = \;\,
 \xy (0,0)*{
			\labellist
			\small\hair 2pt
			\pinlabel \scalebox{0.7}{$i$} at 2 -5
			\pinlabel \scalebox{0.9}{$\lambda$} at 9 12
			\endlabellist 
			\centering 
			\includegraphics[scale=1.4]{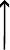}
	}\endxy
 \quad = \;\,
 \xy (0,0)*{
			\labellist
			\small\hair 2pt
			\pinlabel \scalebox{0.7}{$i$} at 15 -5
			\pinlabel \scalebox{0.9}{$\lambda$} at 19 12
			\endlabellist 
			\centering 
			\includegraphics[scale=1.4]{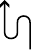}
	}\endxy
 \qquad
 &
\qquad 
\xy (0,0)*{
			\labellist
			\small\hair 2pt
			\pinlabel \scalebox{0.7}{$i$} at 13 -5
			\pinlabel \scalebox{0.9}{$\lambda$} at 19 12
			\endlabellist 
			\centering 
			\includegraphics[scale=1.4]{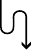}
	}\endxy
 \quad = \;\,
 \xy (0,0)*{
			\labellist
			\small\hair 2pt
			\pinlabel \scalebox{0.7}{$i$} at 0 -5
			\pinlabel \scalebox{0.9}{$\lambda$} at 9 12
			\endlabellist 
			\centering 
			\includegraphics[scale=1.4]{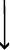}
	}\endxy
 \quad = \;\,
 \xy (0,0)*{
			\labellist
			\small\hair 2pt
			\pinlabel \scalebox{0.7}{$i$} at 2 -5
			\pinlabel \scalebox{0.9}{$\lambda$} at 19 12
			\endlabellist 
			\centering 
			\includegraphics[scale=1.4]{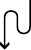}
	}\endxy
 \end{array}\vspace{2ex}
\end{equation}

\item Dot cyclicity:
\begin{equation}\label{eq:dotcyclicity}
 \xy (0,0)*{
			\labellist
			\small\hair 2pt
			\pinlabel \scalebox{0.7}{$i$} at 13 -5
			\pinlabel \scalebox{0.9}{$\lambda$} at 19 12
			\endlabellist 
			\centering 
			\includegraphics[scale=1.4]{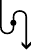}
	}\endxy
 \quad = \;\,
 \xy (0,0)*{
			\labellist
			\small\hair 2pt
			\pinlabel \scalebox{0.7}{$i$} at 2 -5
			\pinlabel \scalebox{0.9}{$\lambda$} at 9 12
			\endlabellist 
			\centering 
			\includegraphics[scale=1.4]{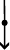}
	}\endxy
 \quad = \;\,
 \xy (0,0)*{
			\labellist
			\small\hair 2pt
			\pinlabel \scalebox{0.7}{$i$} at 1 -5
			\pinlabel \scalebox{0.9}{$\lambda$} at 19 12
			\endlabellist 
			\centering 
			\includegraphics[scale=1.4]{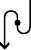}
	}\endxy \vspace{2ex}
\end{equation}

\item Crossing cyclicity:
\begin{equation}\label{eq:crossingcyclicity1}
 \xy (0,0)*{
			\labellist
			\small\hair 2pt
			\pinlabel \scalebox{0.7}{$\textcolor{red}{i}$} at 40 -5
                \pinlabel \scalebox{0.7}{$\textcolor{blue}{j}$} at 50 -5
			\pinlabel \scalebox{0.9}{$\lambda$} at 58 25
			\endlabellist 
			\centering 
			\includegraphics[scale=1.3]{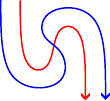}
	}\endxy
 \quad = \;\,
 \xy (0,0)*{
			\labellist
			\small\hair 2pt
			\pinlabel \scalebox{0.7}{$\textcolor{red}{i}$} at 0 -5
                \pinlabel \scalebox{0.7}{$\textcolor{blue}{j}$} at 14 -5
			\pinlabel \scalebox{0.9}{$\lambda$} at 20 8
			\endlabellist 
			\centering 
			\includegraphics[scale=1.3]{./RdlBdr}
	}\endxy
 \quad = \;\,
 \xy (0,0)*{
			\labellist
			\small\hair 2pt
			\pinlabel \scalebox{0.7}{$\textcolor{red}{i}$} at 0 -5
                \pinlabel \scalebox{0.7}{$\textcolor{blue}{j}$} at 11 -5
			\pinlabel \scalebox{0.9}{$\lambda$} at 58 25
			\endlabellist 
			\centering 
			\includegraphics[scale=1.3]{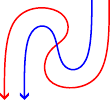}
	}\endxy
 \end{equation}

\begin{equation}\label{eq:crossingcyclicity2}
 \xy (0,0)*{
			\labellist
			\small\hair 2pt
			\pinlabel \scalebox{0.7}{$\textcolor{red}{i}$} at 37 -5
                \pinlabel \scalebox{0.7}{$\textcolor{blue}{j}$} at 49 -5
			\pinlabel \scalebox{0.9}{$\lambda$} at 53 25
			\endlabellist 
			\centering 
			\includegraphics[scale=1.3]{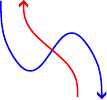}
	}\endxy
 \quad = \;\,
 \xy (0,0)*{
			\labellist
			\small\hair 2pt
			\pinlabel \scalebox{0.7}{$\textcolor{red}{i}$} at 0 -5
                \pinlabel \scalebox{0.7}{$\textcolor{blue}{j}$} at 14 -5
			\pinlabel \scalebox{0.9}{$\lambda$} at 20 8
			\endlabellist 
			\centering 
			\includegraphics[scale=1.3]{./RurBdr}
	}\endxy
 \quad = \;\,
 \xy (0,0)*{
			\labellist
			\small\hair 2pt
			\pinlabel \scalebox{0.7}{$\textcolor{red}{i}$} at 0 -5
                \pinlabel \scalebox{0.7}{$\textcolor{blue}{j}$} at 12 -5
			\pinlabel \scalebox{0.9}{$\lambda$} at 52 25
			\endlabellist 
			\centering 
			\includegraphics[scale=1.3]{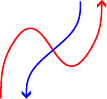}
	}\endxy
 \end{equation}

\begin{equation}\label{eq:crossingcyclicity3}
 \xy (0,0)*{
			\labellist
			\small\hair 2pt
			\pinlabel \scalebox{0.7}{$\textcolor{red}{i}$} at 37 -5
                \pinlabel \scalebox{0.7}{$\textcolor{blue}{j}$} at 49 -5
			\pinlabel \scalebox{0.9}{$\lambda$} at 53 25
			\endlabellist 
			\centering 
			\includegraphics[scale=1.3]{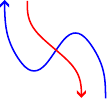}
	}\endxy
 \quad = \;\,
 \xy (0,0)*{
			\labellist
			\small\hair 2pt
			\pinlabel \scalebox{0.7}{$\textcolor{red}{i}$} at 0 -5
                \pinlabel \scalebox{0.7}{$\textcolor{blue}{j}$} at 14 -5
			\pinlabel \scalebox{0.9}{$\lambda$} at 20 8
			\endlabellist 
			\centering 
			\includegraphics[scale=1.3]{./RdlBul}
	}\endxy
 \quad = \;\,
 \xy (0,0)*{
			\labellist
			\small\hair 2pt
			\pinlabel \scalebox{0.7}{$\textcolor{red}{i}$} at 0 -5
                \pinlabel \scalebox{0.7}{$\textcolor{blue}{j}$} at 12 -5
			\pinlabel \scalebox{0.9}{$\lambda$} at 52 25
			\endlabellist 
			\centering 
			\includegraphics[scale=1.3]{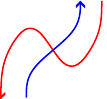}
	}\endxy \vspace{2ex}
 \end{equation}
\eqskip

\item Quadratic KLR:
\begin{equation}\label{eq:R2a}
\xy (0,0)*{
			\labellist
			\small\hair 2pt
			\pinlabel \scalebox{0.7}{$\textcolor{red}{i}$} at 0 -5
			\pinlabel \scalebox{0.7}{$\textcolor{blue}{j}$} at 16 -5
			\pinlabel \scalebox{0.9}{$\lambda$} at 18 16
			\endlabellist 
			\centering 
			\includegraphics[scale=1.3]{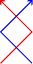}
	}\endxy
\quad =\begin{cases} 
\; 0 &\quad i=j , \\
\; \xy (0,0)*{
			\labellist
			\small\hair 2pt
			\pinlabel \scalebox{0.7}{$\textcolor{red}{i}$} at 0 -5
			\pinlabel \scalebox{0.7}{$\textcolor{blue}{j}$} at 13 -5
			\pinlabel \scalebox{0.9}{$\lambda$} at 17 16
			\endlabellist 
			\centering 
			\includegraphics[scale=1.3]{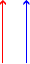}
	}\endxy
 & \quad i\cdot j=0 , 
 \\[5ex]
 \; \varepsilon(i,j)\left(\,\xy (0,0)*{
			\labellist
			\small\hair 2pt
			\pinlabel \scalebox{0.7}{$\textcolor{red}{i}$} at 0 -5
			\pinlabel \scalebox{0.7}{$\textcolor{blue}{j}$} at 13 -5
			\pinlabel \scalebox{0.9}{$\lambda$} at 17 16
			\endlabellist 
			\centering 
			\includegraphics[scale=1.3]{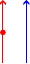}
	}\endxy
\ \ \  - \
 \xy (0,0)*{
			\labellist
			\small\hair 2pt
			\pinlabel \scalebox{0.7}{$\textcolor{red}{i}$} at 0 -5
			\pinlabel \scalebox{0.7}{$\textcolor{blue}{j}$} at 13 -5
			\pinlabel \scalebox{0.9}{$\lambda$} at 17 16
			\endlabellist 
			\centering 
			\includegraphics[scale=1.3]{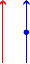}
	}\endxy\;\; \right)
 &\quad  i\cdot j=-1 ,
 \end{cases}\vspace{2ex}
 \end{equation}
 where $\varepsilon(i,j)=\begin{cases}
     1 & i = j+1 (\operatorname{mod} n)\\
     -1 & i = j-1 (\operatorname{mod} n)\\
     0 & \text{ else}
 \end{cases}$ 
 
\item Dot slide: 
 \begin{equation}\label{eq:NilHecke}
 \xy (0,0)*{
			\labellist
			\small\hair 2pt
			\pinlabel \scalebox{0.7}{$\textcolor{red}{i}$} at 0 -5
			\pinlabel \scalebox{0.7}{$\textcolor{blue}{j}$} at 16 -5
			\pinlabel \scalebox{0.9}{$\lambda$} at 18 7
			\endlabellist 
			\centering 
			\includegraphics[scale=1.3]{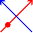}
	}\endxy
 \quad - \ 
    \xy (0,0)*{
			\labellist
			\small\hair 2pt
			\pinlabel \scalebox{0.7}{$\textcolor{red}{i}$} at 0 -5
			\pinlabel \scalebox{0.7}{$\textcolor{blue}{j}$} at 16 -5
			\pinlabel \scalebox{0.9}{$\lambda$} at 18 7
			\endlabellist 
			\centering 
			\includegraphics[scale=1.3]{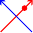}
	}\endxy
 \quad = \ 
 \xy (0,0)*{
			\labellist
			\small\hair 2pt
			\pinlabel \scalebox{0.7}{$\textcolor{red}{i}$} at 0 -5
			\pinlabel \scalebox{0.7}{$\textcolor{blue}{j}$} at 16 -5
			\pinlabel \scalebox{0.9}{$\lambda$} at 18 7
			\endlabellist 
			\centering 
			\includegraphics[scale=1.3]{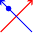}
	}\endxy
 \quad - \ 
 \xy (0,0)*{
			\labellist
			\small\hair 2pt
			\pinlabel \scalebox{0.7}{$\textcolor{red}{i}$} at 0 -5
			\pinlabel \scalebox{0.7}{$\textcolor{blue}{j}$} at 16 -5
			\pinlabel \scalebox{0.9}{$\lambda$} at 18 7
			\endlabellist 
			\centering 
			\includegraphics[scale=1.3]{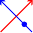}
	}\endxy
 \quad = \ \begin{cases}
   \;  \xy (0,0)*{
			\labellist
			\small\hair 2pt
			\pinlabel \scalebox{0.7}{$\textcolor{red}{i}$} at 12 -5
			\pinlabel \scalebox{0.7}{$\textcolor{red}{i}$} at 0 -5
			\pinlabel \scalebox{0.9}{$\lambda$} at 18 7
			\endlabellist 
			\centering 
			\includegraphics[scale=1.3]{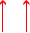}
	}\endxy
 & \quad i=j , \\
  & \\
 \; 0 & \quad i\neq j . 
 \end{cases} \vspace{2ex}
 \end{equation}

\item Cubic KLR:
 \begin{equation}\label{eq:R3}
     \xy (0,0)*{
			\labellist
			\small\hair 2pt
			\pinlabel \scalebox{0.7}{$\textcolor{red}{i}$} at 0 -5
			\pinlabel \scalebox{0.7}{$\textcolor{blue}{j}$} at 16 -5
            \pinlabel \scalebox{0.7}{$\textcolor{mygreen}{k}$} at 32 -5
			\pinlabel \scalebox{0.9}{$\lambda$} at 34 16
			\endlabellist 
			\centering 
			\includegraphics[scale=1.3]{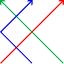}
	}\endxy
 \quad - \, 
 \xy (0,0)*{
			\labellist
			\small\hair 2pt
			\pinlabel \scalebox{0.7}{$\textcolor{red}{i}$} at 0 -5
			\pinlabel \scalebox{0.7}{$\textcolor{blue}{j}$} at 16 -5
            \pinlabel \scalebox{0.7}{$\textcolor{mygreen}{k}$} at 32 -5
			\pinlabel \scalebox{0.9}{$\lambda$} at 34 16
			\endlabellist 
			\centering 
			\includegraphics[scale=1.3]{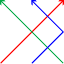}
	}\endxy
 \quad = \ \begin{cases}
    \; \varepsilon(i,j)\xy (0,0)*{
			\labellist
			\small\hair 2pt
			\pinlabel \scalebox{0.7}{$\textcolor{red}{i}$} at 0 -5
			\pinlabel \scalebox{0.7}{$\textcolor{blue}{j}$} at 12 -5
            \pinlabel \scalebox{0.7}{$\textcolor{red}{i}$} at 24 -5
			\pinlabel \scalebox{0.9}{$\lambda$} at 28 16
			\endlabellist 
			\centering 
			\includegraphics[scale=1.3]{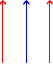}
	}\endxy
 & \quad i=k \text{ and } i\cdot j=-1 ,\\
  & \\
 \; 0 & \quad i\neq k \text{ or } i\cdot j\neq -1 . 
 \end{cases}\vspace{2ex}
 \end{equation}
Before we list more relations, first a useful piece of notation:

\begin{equation}\label{eq:bubbleconvention}
    \xy (0,0)*{
			\labellist
			\small\hair 2pt
			\pinlabel \scalebox{0.7}{$\textcolor{black}{i}$} at 0 14
			\pinlabel \scalebox{0.7}{$\textcolor{black}{+m}$} at 18 -1
			\pinlabel \scalebox{0.9}{$\lambda$} at 19 12
			\endlabellist 
			\centering 
			\includegraphics[scale=1.3]{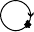}
	}\endxy
 \;\, := \;\,
 \xy (0,0)*{
			\labellist
			\small\hair 2pt
			\pinlabel \scalebox{0.7}{$\textcolor{black}{i}$} at 0 14
			\pinlabel \scalebox{0.7}{$\textcolor{black}{\olambda_i-1+m}$} at 27 -2
			\pinlabel \scalebox{0.9}{$\lambda$} at 19 12
			\endlabellist 
			\centering 
			\includegraphics[scale=1.3]{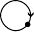}
	}\endxy
 \qquad \qquad
 \xy (0,0)*{
			\labellist
			\small\hair 2pt
			\pinlabel \scalebox{0.7}{$\textcolor{black}{i}$} at 0 14
			\pinlabel \scalebox{0.7}{$\textcolor{black}{+m}$} at 18 -2
			\pinlabel \scalebox{0.9}{$\lambda$} at 19 12
			\endlabellist 
			\centering 
			\includegraphics[scale=1.3]{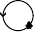}
	}\endxy
 \;\, := \;\, 
 \xy (0,0)*{
			\labellist
			\small\hair 2pt
			\pinlabel \scalebox{0.7}{$\textcolor{black}{i}$} at 0 14
			\pinlabel \scalebox{0.7}{$\textcolor{black}{-\olambda_i-1+m}$} at 27 -2
			\pinlabel \scalebox{0.9}{$\lambda$} at 19 12
			\endlabellist 
			\centering 
			\includegraphics[scale=1.3]{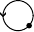}
	}\endxy
 \end{equation}
\eqskip Using this notation, the other relations on diagrams are:

\item Mixed EF:
 \begin{equation}\label{eq:R2b} 
 \xy (0,0)*{
			\labellist
			\small\hair 2pt
			\pinlabel \scalebox{0.7}{$\textcolor{red}{i}$} at 0 -5
			\pinlabel \scalebox{0.7}{$\textcolor{blue}{j}$} at 16 -5
			\pinlabel \scalebox{0.9}{$\lambda$} at 18 16
			\endlabellist 
			\centering 
			\includegraphics[scale=1.3]{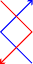}
	}\endxy
 \quad = \ \begin{cases}
     \; \xy (0,0)*{
			\labellist
			\small\hair 2pt
			\pinlabel \scalebox{0.7}{$\textcolor{red}{i}$} at 0 -5
			\pinlabel \scalebox{0.7}{$\textcolor{blue}{j}$} at 12 -5
			\pinlabel \scalebox{0.9}{$\lambda$} at 17 16
			\endlabellist 
			\centering 
			\includegraphics[scale=1.3]{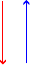}
	}\endxy
 & \quad i\neq j , \\
  & \\
 \; \; \xy (0,0)*{
			\labellist
			\small\hair 2pt
			\pinlabel \scalebox{0.7}{$\textcolor{red}{i}$} at 0 -5
			\pinlabel \scalebox{0.7}{$\textcolor{red}{i}$} at 12 -5
			\pinlabel \scalebox{0.9}{$\lambda$} at 17 18
			\endlabellist 
			\centering 
			\includegraphics[scale=1.3]{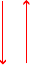}
	}\endxy
 \quad - \ 
 \sum\limits_{a+b+c=-\olambda_i-1} \xy (0,0)*{
			\labellist
			\small\hair 2pt
			\pinlabel \scalebox{0.7}{$\textcolor{red}{i}$} at 0 -5
			\pinlabel \scalebox{0.7}{$\textcolor{red}{i}$} at 0 31
            \pinlabel \scalebox{0.7}{$\textcolor{red}{i}$} at 19 22
            \pinlabel \scalebox{0.7}{$\textcolor{red}{+c}$} at 34 5
			\pinlabel \scalebox{0.9}{$\lambda$} at 37 16
            \pinlabel \scalebox{0.7}{$\textcolor{red}{a}$} at 4 8
            \pinlabel \scalebox{0.7}{$\textcolor{red}{b}$} at 4 18
			\endlabellist 
			\centering 
			\includegraphics[scale=1.3]{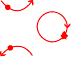}
	}\endxy
 & \quad i=j ,
 \end{cases}
 \end{equation}

 \begin{equation}\label{eq:R2c}
 \xy (0,0)*{
			\labellist
			\small\hair 2pt
			\pinlabel \scalebox{0.7}{$\textcolor{red}{i}$} at 0 -5
			\pinlabel \scalebox{0.7}{$\textcolor{blue}{j}$} at 16 -5
			\pinlabel \scalebox{0.9}{$\lambda$} at 18 16
			\endlabellist 
			\centering 
			\includegraphics[scale=1.3]{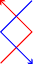}
	}\endxy
 \quad  = \ \begin{cases}
     \; \xy (0,0)*{
			\labellist
			\small\hair 2pt
			\pinlabel \scalebox{0.7}{$\textcolor{red}{i}$} at 0 -5
			\pinlabel \scalebox{0.7}{$\textcolor{blue}{j}$} at 12 -5
			\pinlabel \scalebox{0.9}{$\lambda$} at 17 16
			\endlabellist 
			\centering 
			\includegraphics[scale=1.3]{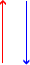}
	}\endxy
 & \quad i\neq j , \\
  & \\
 \; \; \xy (0,0)*{
			\labellist
			\small\hair 2pt
			\pinlabel \scalebox{0.7}{$\textcolor{red}{i}$} at 0 -5
			\pinlabel \scalebox{0.7}{$\textcolor{red}{i}$} at 12 -5
			\pinlabel \scalebox{0.9}{$\lambda$} at 17 18
			\endlabellist 
			\centering 
			\includegraphics[scale=1.3]{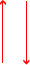}
	}\endxy
 \quad - \ 
 \sum\limits_{a+b+c=\olambda_i-1} \xy (0,0)*{
			\labellist
			\small\hair 2pt
			\pinlabel \scalebox{0.7}{$\textcolor{red}{i}$} at 0 -5
			\pinlabel \scalebox{0.7}{$\textcolor{red}{i}$} at 0 31
            \pinlabel \scalebox{0.7}{$\textcolor{red}{i}$} at 19 22
            \pinlabel \scalebox{0.7}{$\textcolor{red}{+c}$} at 34 5
			\pinlabel \scalebox{0.9}{$\lambda$} at 37 16
            \pinlabel \scalebox{0.7}{$\textcolor{red}{a}$} at 4 8
            \pinlabel \scalebox{0.7}{$\textcolor{red}{b}$} at 4 18
			\endlabellist 
			\centering 
			\includegraphics[scale=1.3]{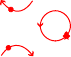}
	}\endxy
 & \quad i=j . 
 \end{cases}\vspace{2ex}
 \end{equation}

\item Bubble relations:
 \begin{equation}\label{eq:bubblerel} 
 \xy (0,0)*{
			\labellist
			\small\hair 2pt
			\pinlabel \scalebox{0.7}{$\textcolor{black}{i}$} at 0 14
			\pinlabel \scalebox{0.7}{$\textcolor{black}{+m}$} at 16 -1
			\pinlabel \scalebox{0.9}{$\lambda$} at 18 12
			\endlabellist 
			\centering 
			\includegraphics[scale=1.3]{./KClockClub}
	}\endxy 
 \ = \ \begin{cases}
     (-1)^{\lambda_{i+1}} & m=0 , \\
     0 & m<0 ,
 \end{cases}
 \quad , \quad
\xy (0,0)*{
			\labellist
			\small\hair 2pt
			\pinlabel \scalebox{0.7}{$\textcolor{black}{i}$} at 0 14
			\pinlabel \scalebox{0.7}{$\textcolor{black}{+m}$} at 16 -1
			\pinlabel \scalebox{0.9}{$\lambda$} at 18 12
			\endlabellist 
			\centering 
			\includegraphics[scale=1.3]{./KAClockClub}
	}\endxy 
 \ = \ \begin{cases}
     (-1)^{\lambda_{i+1}-1} & m=0 , \\
     0 & m<0 .
 \end{cases}
 \end{equation}

\item Infinite Grassmannian relation:
\eqskip
 \begin{equation}\label{eq:Grassmannian}\left(\xy (0,0)*{
			\labellist
			\small\hair 2pt
			\pinlabel \scalebox{0.7}{$\textcolor{black}{i}$} at 0 14
			\pinlabel \scalebox{0.7}{$\textcolor{black}{+0}$} at 16 -2
			\pinlabel \scalebox{0.9}{$\lambda$} at 18 12
			\endlabellist 
			\centering 
			\includegraphics[scale=1.3]{./KAClockClub}
	}\endxy
 \ +
 \xy (0,0)*{
			\labellist
			\small\hair 2pt
			\pinlabel \scalebox{0.7}{$\textcolor{black}{i}$} at 0 14
			\pinlabel \scalebox{0.7}{$\textcolor{black}{+1}$} at 16 -2
			\pinlabel \scalebox{0.9}{$\lambda$} at 18 12
			\endlabellist 
			\centering 
			\includegraphics[scale=1.3]{./KAClockClub}
	}\endxy
 \ \ t + \dots +
 \xy (0,0)*{
			\labellist
			\small\hair 2pt
			\pinlabel \scalebox{0.7}{$\textcolor{black}{i}$} at 0 14
			\pinlabel \scalebox{0.7}{$\textcolor{black}{+m}$} at 16 -2
			\pinlabel \scalebox{0.9}{$\lambda$} at 18 12
			\endlabellist 
			\centering 
			\includegraphics[scale=1.3]{./KClockClub}
	}\endxy \ \ t^m+\dots\right)\!\!
 \left(\xy (0,0)*{
			\labellist
			\small\hair 2pt
			\pinlabel \scalebox{0.7}{$\textcolor{black}{i}$} at 0 14
			\pinlabel \scalebox{0.7}{$\textcolor{black}{+0}$} at 16 -2
			\pinlabel \scalebox{0.9}{$\lambda$} at 18 12
			\endlabellist 
			\centering 
			\includegraphics[scale=1.3]{./KClockClub}
	}\endxy
 \ +
 \xy (0,0)*{
			\labellist
			\small\hair 2pt
			\pinlabel \scalebox{0.7}{$\textcolor{black}{i}$} at 0 14
			\pinlabel \scalebox{0.7}{$\textcolor{black}{+1}$} at 16 -2
			\pinlabel \scalebox{0.9}{$\lambda$} at 18 12
			\endlabellist 
			\centering 
			\includegraphics[scale=1.3]{./KClockClub}
	}\endxy
 \ \ t + \dots +
 \xy (0,0)*{
			\labellist
			\small\hair 2pt
			\pinlabel \scalebox{0.7}{$\textcolor{black}{i}$} at 0 14
			\pinlabel \scalebox{0.7}{$\textcolor{black}{+m}$} at 16 -2
			\pinlabel \scalebox{0.9}{$\lambda$} at 18 12
			\endlabellist 
			\centering 
			\includegraphics[scale=1.3]{./KClockClub}
	}\endxy \ \ t^m+\dots\right)\!=-1 .
 \vspace{2ex}
 \end{equation}

\end{enumerate}
\eqskip
This ends the definition of the $2$-category $\affu{n}$ (resp. $\naffu{n}$).
\end{defn}

\eqskip

\begin{rem}\label{rem:mingens}
Thanks to adjunction and cyclicity (equations~\eqref{eq:adjunction} through \eqref{eq:crossingcyclicity3}), the $2$-morphisms of $\affu{n}$ are already generated by 
\[
\xy (0,0)*{
			\labellist
			\small\hair 2pt
                \pinlabel \scalebox{0.7}{$i$} at 2 -5
			\pinlabel \scalebox{0.7}{$\lambda+\alpha_i$} at -15 12
			\pinlabel \scalebox{0.9}{$\lambda$} at 12 12
			\endlabellist 
			\centering 
			\includegraphics[scale=1.3]{./Kuo}
	}\endxy 
 \qquad \qquad 
  \xy (0,0)*{
			\labellist
			\small\hair 2pt
   \pinlabel \scalebox{0.7}{$i$} at -2 -5
   \pinlabel \scalebox{0.7}{$j$} at 15 -5
			\pinlabel \scalebox{0.9}{$\lambda$} at 22 9
			\endlabellist 
			\centering 
			\includegraphics[scale=1.3]{./RurBul}
	}\endxy 
 \qquad \qquad 
  \xy (0,0)*{
			\labellist
			\small\hair 2pt
   \pinlabel \scalebox{0.7}{$i$} at 8 -5
			\pinlabel \scalebox{0.9}{$\lambda$} at 22 0
			\endlabellist 
			\centering 
			\includegraphics[scale=1.3]{./KrCup1}
	}\endxy 
  \qquad \qquad 
\xy (0,0)*{
			\labellist
			\small\hair 2pt
   \pinlabel \scalebox{0.7}{$i$} at 8 -5
			\pinlabel \scalebox{0.9}{$\lambda$} at 22 0
			\endlabellist 
			\centering 
			\includegraphics[scale=1.3]{./KlCup1}
	}\endxy 
 \qquad \qquad 
  \xy (0,0)*{
			\labellist
			\small\hair 2pt
   \pinlabel \scalebox{0.7}{$i$} at 8 13
			\pinlabel \scalebox{0.9}{$\lambda$} at 22 9
			\endlabellist 
			\centering 
			\includegraphics[scale=1.3]{./KlCap1}
	}\endxy 
\qquad \qquad
\xy (0,0)*{
			\labellist
			\small\hair 2pt
   \pinlabel \scalebox{0.7}{$i$} at 8 13
			\pinlabel \scalebox{0.9}{$\lambda$} at 22 9
			\endlabellist 
			\centering 
			\includegraphics[scale=1.3]{./KrCap1}
	}\endxy 
\]
\eqskip
\noindent for $\lambda\in \mathbb{Z}^n$ and $i,j=1,\ldots, n$ (and similarly for $\naffu{n}$). It therefore suffices to define the evaluation functor on this smaller set of generators.
\end{rem}

\begin{rem}\label{rem:ChOfSc}
    The choice of signs in \autoref{defn:KLR-2cats} is not covered by \cite{abram2022categorification}. This choice of signs is referred to as a \emph{choice of scalars} and \emph{bubble parameters} - see \cite{lauda2020parameters} for an in-depth explanation. Since we will be adapting various proofs from \cite{abram2022categorification} for the proof of our main result, we will therefore need to take care when translating them across the different sign conventions. We will discuss difference choices of scalars and bubble parameters further in \autoref{sec:No2Iso}, since they might have implications for the existence of an evaluation 2-functor.
\end{rem}

\subsection{Some additional relations}
Some well-known consequences of the above relations are listed below. For the proofs, see e.g.~\cite{Beliakova-Habiro-Lauda-Webster16} and references therein.

 \begin{equation}\label{eq:curl}
 \xy (0,0)*{
			\labellist
			\small\hair 2pt
			\pinlabel \scalebox{0.7}{$\textcolor{black}{i}$} at 1 -5
			\pinlabel \scalebox{0.7}{$\textcolor{black}{m}$} at 22 16
			\pinlabel \scalebox{0.9}{$\lambda$} at 20 25
			\endlabellist 
			\centering 
			\includegraphics[scale=1.3]{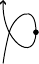}
	}\endxy
 \quad = \; -  
 \sum\limits_{a+b=m-\olambda_i} \xy (0,0)*{
			\labellist
			\small\hair 2pt
			\pinlabel \scalebox{0.7}{$\textcolor{black}{i}$} at 2 -5
			\pinlabel \scalebox{0.7}{$\textcolor{black}{i}$} at 9 24
            \pinlabel \scalebox{0.7}{$\textcolor{black}{a}$} at 5 14
            \pinlabel \scalebox{0.7}{$\textcolor{black}{+b}$} at 23 6
			\pinlabel \scalebox{0.9}{$\lambda$} at 26 24
			\endlabellist 
			\centering 
			\includegraphics[scale=1.3]{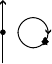}
	}\endxy
 \qquad\qquad \xy (0,0)*{
			\labellist
			\small\hair 2pt
			\pinlabel \scalebox{0.7}{$\textcolor{black}{i}$} at 17 -5
			\pinlabel \scalebox{0.7}{$\textcolor{black}{m}$} at -3 16
			\pinlabel \scalebox{0.9}{$\lambda$} at 20 25
			\endlabellist 
			\centering 
			\includegraphics[scale=1.3]{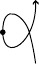}
	}\endxy
 \ =  \
 \sum\limits_{a+b=m+\olambda_i} \xy (0,0)*{
			\labellist
			\small\hair 2pt
			\pinlabel \scalebox{0.7}{$\textcolor{black}{i}$} at 24 -5
			\pinlabel \scalebox{0.7}{$\textcolor{black}{i}$} at 1 24
            \pinlabel \scalebox{0.7}{$\textcolor{black}{a}$} at 28 14
            \pinlabel \scalebox{0.7}{$\textcolor{black}{+b}$} at 15 6
			\pinlabel \scalebox{0.9}{$\lambda$} at 27 24
			\endlabellist 
			\centering 
			\includegraphics[scale=1.3]{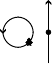}
	}\endxy
 \end{equation}

\begin{equation}\label{eq:bubblesliderla}
\xy (0,0)*{
			\labellist
			\small\hair 2pt
			\pinlabel \scalebox{0.7}{$\textcolor{blue}{j}$} at 1 -5
			\pinlabel \scalebox{0.7}{$\textcolor{red}{i}$} at 9 24
            \pinlabel \scalebox{0.7}{$\textcolor{red}{+m}$} at 23 6
			\pinlabel \scalebox{0.9}{$\lambda$} at 26 24
			\endlabellist 
			\centering 
			\includegraphics[scale=1.3]{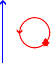}
	}\endxy
 \ = \ \begin{cases}
 \; -\sum\limits_{a+b=m} (a+1)\; \xy (0,0)*{
			\labellist
			\small\hair 2pt
			\pinlabel \scalebox{0.7}{$\textcolor{red}{i}$} at 24 -5
			\pinlabel \scalebox{0.7}{$\textcolor{red}{i}$} at 1 24
            \pinlabel \scalebox{0.7}{$\textcolor{red}{a}$} at 27 14
            \pinlabel \scalebox{0.7}{$\textcolor{red}{+b}$} at 15 6
			\pinlabel \scalebox{0.9}{$\lambda$} at 27 24
			\endlabellist 
			\centering 
			\includegraphics[scale=1.3]{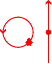}
	}\endxy & \quad i=j , \\
 & \\
 \;-\varepsilon(i,j)\left( \xy (0,0)*{
			\labellist
			\small\hair 2pt
			\pinlabel \scalebox{0.7}{$\textcolor{blue}{j}$} at 31 -5
			\pinlabel \scalebox{0.7}{$\textcolor{red}{i}$} at 1 24
                \pinlabel \scalebox{0.7}{$\textcolor{red}{+m-1}$} at 15 6
			\pinlabel \scalebox{0.9}{$\lambda$} at 35 24
			\endlabellist 
			\centering 
			\includegraphics[scale=1.3]{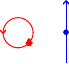}
	}\endxy
 \; - \; 
 \xy (0,0)*{
			\labellist
			\small\hair 2pt
			\pinlabel \scalebox{0.7}{$\textcolor{blue}{j}$} at 23 -5
			\pinlabel \scalebox{0.7}{$\textcolor{red}{i}$} at 1 24
                \pinlabel \scalebox{0.7}{$\textcolor{red}{+m}$} at 15 6
			\pinlabel \scalebox{0.9}{$\lambda$} at 27 24
			\endlabellist 
			\centering 
			\includegraphics[scale=1.3]{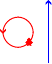}
	}\endxy\;\right) & \quad i\cdot j=-1 , \\
  & \\ 
 \; \xy (0,0)*{
			\labellist
			\small\hair 2pt
			\pinlabel \scalebox{0.7}{$\textcolor{blue}{j}$} at 23 -5
			\pinlabel \scalebox{0.7}{$\textcolor{red}{i}$} at 1 24
            \pinlabel \scalebox{0.7}{$\textcolor{red}{+m}$} at 17 6
			\pinlabel \scalebox{0.9}{$\lambda$} at 27 24
			\endlabellist 
			\centering 
			\includegraphics[scale=1.3]{./RAClockClubBuu}
	}\endxy & \quad i\cdot j=0 , 
 \end{cases}
 \end{equation}
 
\begin{equation}\label{eq:bubbleslidelra} 
\xy (0,0)*{
			\labellist
			\small\hair 2pt
			\pinlabel \scalebox{0.7}{$\textcolor{blue}{j}$} at 23 -5
			\pinlabel \scalebox{0.7}{$\textcolor{red}{i}$} at 1 24
            \pinlabel \scalebox{0.7}{$\textcolor{red}{+m}$} at 15 6
			\pinlabel \scalebox{0.9}{$\lambda$} at 27 24
			\endlabellist 
			\centering 
			\includegraphics[scale=1.3]{./RAClockClubBuu}
	}\endxy
 \ = \ \begin{cases}
    -\; \xy (0,0)*{
			\labellist
			\small\hair 2pt
			\pinlabel \scalebox{0.7}{$\textcolor{red}{i}$} at 2 -5
			\pinlabel \scalebox{0.7}{$\textcolor{red}{i}$} at 9 24
            \pinlabel \scalebox{0.7}{$\textcolor{red}{2}$} at 5 14
            \pinlabel \scalebox{0.7}{$\textcolor{red}{+m-2}$} at 27 6
			\pinlabel \scalebox{0.9}{$\lambda$} at 26 24
			\endlabellist 
			\centering 
			\includegraphics[scale=1.3]{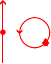}
	}\endxy
 \quad + \ 
 2\; \xy (0,0)*{
			\labellist
			\small\hair 2pt
			\pinlabel \scalebox{0.7}{$\textcolor{red}{i}$} at 2 -5
			\pinlabel \scalebox{0.7}{$\textcolor{red}{i}$} at 9 24
            \pinlabel \scalebox{0.7}{$\textcolor{red}{+m-1}$} at 27 6
			\pinlabel \scalebox{0.9}{$\lambda$} at 26 24
			\endlabellist 
			\centering 
			\includegraphics[scale=1.3]{./RuouRAClockClub}
	}\endxy
 \quad - \ 
 \xy (0,0)*{
			\labellist
			\small\hair 2pt
			\pinlabel \scalebox{0.7}{$\textcolor{red}{i}$} at 2 -5
			\pinlabel \scalebox{0.7}{$\textcolor{red}{i}$} at 9 24
            \pinlabel \scalebox{0.7}{$\textcolor{red}{+m}$} at 23 6
			\pinlabel \scalebox{0.9}{$\lambda$} at 26 24
			\endlabellist 
			\centering 
			\includegraphics[scale=1.3]{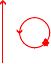}
	}\endxy & \quad  i=j , \\
  & \\
  \;\varepsilon(i,j)\sum\limits_{a+b=m}\;\xy (0,0)*{
			\labellist
			\small\hair 2pt
			\pinlabel \scalebox{0.7}{$\textcolor{blue}{j}$} at 1 -5
			\pinlabel \scalebox{0.7}{$\textcolor{red}{i}$} at 9 24
            \pinlabel \scalebox{0.7}{$\textcolor{blue}{a}$} at 5 14
            \pinlabel \scalebox{0.7}{$\textcolor{red}{+b}$} at 23 6
			\pinlabel \scalebox{0.9}{$\lambda$} at 26 24
			\endlabellist 
			\centering 
			\includegraphics[scale=1.3]{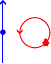}
	}\endxy & \quad i\cdot j=-1 , 
 \end{cases}
 \end{equation}

 \begin{equation}\label{eq:bubbleslidelrb}
 \xy (0,0)*{
			\labellist
			\small\hair 2pt
			\pinlabel \scalebox{0.7}{$\textcolor{blue}{j}$} at 22 -5
			\pinlabel \scalebox{0.7}{$\textcolor{red}{i}$} at 1 24
            \pinlabel \scalebox{0.7}{$\textcolor{red}{+m}$} at 15 6
			\pinlabel \scalebox{0.9}{$\lambda$} at 27 24
			\endlabellist 
			\centering 
			\includegraphics[scale=1.3]{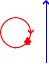}
	}\endxy
 \ = \ \begin{cases}
     \; -\sum\limits_{a+b=m} (a+1)\; \xy (0,0)*{
			\labellist
			\small\hair 2pt
			\pinlabel \scalebox{0.7}{$\textcolor{red}{i}$} at 2 -5
			\pinlabel \scalebox{0.7}{$\textcolor{red}{i}$} at 9 24
            \pinlabel \scalebox{0.7}{$\textcolor{red}{a}$} at 5 14
            \pinlabel \scalebox{0.7}{$\textcolor{red}{+b}$} at 23 6
			\pinlabel \scalebox{0.9}{$\lambda$} at 26 24
			\endlabellist 
			\centering 
			\includegraphics[scale=1.3]{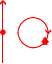}
	}\endxy & \ i=j , \\
  & \\
  -\varepsilon(i,j)\left(\; \xy (0,0)*{
			\labellist
			\small\hair 2pt
			\pinlabel \scalebox{0.7}{$\textcolor{blue}{j}$} at 1 -5
			\pinlabel \scalebox{0.7}{$\textcolor{red}{i}$} at 9 24
            \pinlabel \scalebox{0.7}{$\textcolor{red}{+m-1}$} at 27 6
			\pinlabel \scalebox{0.9}{$\lambda$} at 26 24
			\endlabellist 
			\centering 
			\includegraphics[scale=1.3]{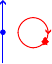}
	}\endxy
 \quad - \ 
 \xy (0,0)*{
			\labellist
			\small\hair 2pt
			\pinlabel \scalebox{0.7}{$\textcolor{blue}{j}$} at 1 -5
			\pinlabel \scalebox{0.7}{$\textcolor{red}{i}$} at 9 24
            \pinlabel \scalebox{0.7}{$\textcolor{red}{+m}$} at 23 6
			\pinlabel \scalebox{0.9}{$\lambda$} at 26 24
			\endlabellist 
			\centering 
			\includegraphics[scale=1.3]{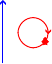}
	}\endxy\;\right)  & \quad  i\cdot j=-1 , \\
 & \\
 \; \xy (0,0)*{
			\labellist
			\small\hair 2pt
			\pinlabel \scalebox{0.7}{$\textcolor{blue}{j}$} at 1 -5
			\pinlabel \scalebox{0.7}{$\textcolor{red}{i}$} at 9 24
            \pinlabel \scalebox{0.7}{$\textcolor{red}{+m}$} at 23 6
			\pinlabel \scalebox{0.9}{$\lambda$} at 26 24
			\endlabellist 
			\centering 
			\includegraphics[scale=1.3]{./BuuRClockClub}
	}\endxy &\quad i\cdot j=0 , 
 \end{cases}
 \end{equation}

 \begin{equation}\label{eq:bubblesliderlb}
 \xy (0,0)*{
			\labellist
			\small\hair 2pt
			\pinlabel \scalebox{0.7}{$\textcolor{blue}{j}$} at 1 -5
			\pinlabel \scalebox{0.7}{$\textcolor{red}{i}$} at 9 24
            \pinlabel \scalebox{0.7}{$\textcolor{red}{+m}$} at 23 6
			\pinlabel \scalebox{0.9}{$\lambda$} at 26 24
			\endlabellist 
			\centering 
			\includegraphics[scale=1.3]{./BuuRClockClub}
	}\endxy
 \ = \ \begin{cases}
     -\; \xy (0,0)*{
			\labellist
			\small\hair 2pt
			\pinlabel \scalebox{0.7}{$\textcolor{red}{i}$} at 30 -5
			\pinlabel \scalebox{0.7}{$\textcolor{red}{i}$} at 1 24
            \pinlabel \scalebox{0.7}{$\textcolor{red}{2}$} at 33 14
            \pinlabel \scalebox{0.7}{$\textcolor{red}{+m-2}$} at 19 6
			\pinlabel \scalebox{0.9}{$\lambda$} at 34 24
			\endlabellist 
			\centering 
			\includegraphics[scale=1.3]{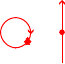}
	}\endxy
 \; + 2\; \xy (0,0)*{
			\labellist
			\small\hair 2pt
			\pinlabel \scalebox{0.7}{$\textcolor{red}{i}$} at 30 -5
			\pinlabel \scalebox{0.7}{$\textcolor{red}{i}$} at 1 24
            \pinlabel \scalebox{0.7}{$\textcolor{red}{+m-1}$} at 19 6
			\pinlabel \scalebox{0.9}{$\lambda$} at 34 24
			\endlabellist 
			\centering 
			\includegraphics[scale=1.3]{./RClockClubRuouWide}
	}\endxy
 \; - \; 
 \xy (0,0)*{
			\labellist
			\small\hair 2pt
			\pinlabel \scalebox{0.7}{$\textcolor{red}{i}$} at 23 -5
			\pinlabel \scalebox{0.7}{$\textcolor{red}{i}$} at 1 24
            \pinlabel \scalebox{0.7}{$\textcolor{red}{+m}$} at 15 6
			\pinlabel \scalebox{0.9}{$\lambda$} at 27 24
			\endlabellist 
			\centering 
			\includegraphics[scale=1.3]{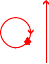}
	}\endxy & \quad i=j , \\
  & \\
  \; \varepsilon(i,j)\sum\limits_{a+b=m} \; \xy (0,0)*{
			\labellist
			\small\hair 2pt
			\pinlabel \scalebox{0.7}{$\textcolor{blue}{j}$} at 22 -5
			\pinlabel \scalebox{0.7}{$\textcolor{red}{i}$} at 1 24
            \pinlabel \scalebox{0.7}{$\textcolor{blue}{a}$} at 27 14
            \pinlabel \scalebox{0.7}{$\textcolor{red}{+b}$} at 15 6
			\pinlabel \scalebox{0.9}{$\lambda$} at 27 24
			\endlabellist 
			\centering 
			\includegraphics[scale=1.3]{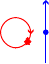}
	}\endxy & \quad i\cdot j=-1 , 
 \end{cases}\end{equation}

 \begin{equation}\label{eq:R3b}
 \xy (0,0)*{
			\labellist
			\small\hair 2pt
			\pinlabel \scalebox{0.7}{$\textcolor{red}{i}$} at 0 -5
			\pinlabel \scalebox{0.7}{$\textcolor{blue}{j}$} at 16 -5
            \pinlabel \scalebox{0.7}{$\textcolor{mygreen}{k}$} at 32 -5
			\pinlabel \scalebox{0.9}{$\lambda$} at 30 16
			\endlabellist 
			\centering 
			\includegraphics[scale=1.3]{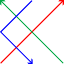}
	}\endxy
 \mspace{5mu} - \mspace{-2mu}
 \xy (0,0)*{
			\labellist
			\small\hair 2pt
			\pinlabel \scalebox{0.7}{$\textcolor{red}{i}$} at 0 -5
			\pinlabel \scalebox{0.7}{$\textcolor{blue}{j}$} at 16 -5
            \pinlabel \scalebox{0.7}{$\textcolor{mygreen}{k}$} at 32 -5
			\pinlabel \scalebox{0.9}{$\lambda$} at 34 16
			\endlabellist 
			\centering 
			\includegraphics[scale=1.3]{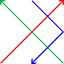}
	}\endxy
 \  \ \ =  \begin{cases}
     \sum\limits_{a+b+c+d=\olambda_i} 
     \xy (0,0)*{
			\labellist
			\small\hair 2pt
			\pinlabel \scalebox{0.7}{$\textcolor{red}{i}$} at 1 -5
			\pinlabel \scalebox{0.7}{$\textcolor{red}{i}$} at 1 35
            \pinlabel \scalebox{0.7}{$\textcolor{red}{i}$} at -2 18
            \pinlabel \scalebox{0.7}{$\textcolor{red}{i}$} at 28 -5
            \pinlabel \scalebox{0.7}{$\textcolor{red}{a}$} at 3 7
            \pinlabel \scalebox{0.7}{$\textcolor{red}{b}$} at 4 23
            \pinlabel \scalebox{0.7}{$\textcolor{red}{c}$} at 31 16
            \pinlabel \scalebox{0.7}{$\textcolor{red}{+d}$} at 16 7
			\pinlabel \scalebox{0.9}{$\lambda$} at 33 22
			\endlabellist 
			\centering 
			\includegraphics[scale=1.3]{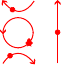}
	}\endxy \quad + \! \sum\limits_{a+b+c+d=-\olambda_i-2} \xy (0,0)*{
			\labellist
			\small\hair 2pt
			\pinlabel \scalebox{0.7}{$\textcolor{red}{i}$} at 12 -5
			\pinlabel \scalebox{0.7}{$\textcolor{red}{i}$} at 12 35
            \pinlabel \scalebox{0.7}{$\textcolor{red}{i}$} at 10 18
            \pinlabel \scalebox{0.7}{$\textcolor{red}{i}$} at 0 -5
            \pinlabel \scalebox{0.7}{$\textcolor{red}{a}$} at 15 7
            \pinlabel \scalebox{0.7}{$\textcolor{red}{b}$} at 15 23
            \pinlabel \scalebox{0.7}{$\textcolor{red}{c}$} at 5 16
            \pinlabel \scalebox{0.7}{$\textcolor{red}{+d}$} at 27 7
			\pinlabel \scalebox{0.9}{$\lambda$} at 33 20
			\endlabellist 
			\centering 
			\includegraphics[scale=1.3]{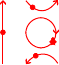}
	}\endxy & \quad  i=j=k , \\
  & \\
 \; 0 & \quad \text{else}.
 \end{cases}
 \end{equation}

 \section{Two versions of the evaluation 2-functor for \texorpdfstring{$n=3$}{n=3}}\label{Sec:Eval2Fr}

In this section, we will define the 2-functors $\Ev$ and $\Ev'$ discussed in the introduction, which are $\bQ$-linear monoidal functors  
\[
\Ev, \Ev'\colon \affu{3}\to \cK^b(\naffu{3}),
\]
defined in the next pages. Note that in this case, \autoref{d:evala} is particularly simple, because \eqref{eq:evmape} and \eqref{eq:evmapf} only involve one $q$-commutator each:
\begin{eqnarray*}
\ev(E_31_\lambda)&=&q^{\lambda_1+\lambda_3+t-1} (F_1F_21_\lambda - qF_2F_1 1_\lambda),\\ 
\ev(F_31_\lambda)&=&q^{-\lambda_1-\lambda_3-t+1} (E_2E_11_\lambda - q^{-1}E_1E_2 1_\lambda),
\end{eqnarray*}
for $\lambda=(\lambda_1,\lambda_2,\lambda_3)\in \mathbb{Z}^3$ (recall that $t\in \mathbb{Z}$ is arbitrary but fixed). For the remainder of this paper, we set $S(\lambda)=\lambda_1+\lambda_3+t-1$, and we suppress the $\lambda$ when there is no confusion. We also 
use the notation $\E_{i_1i_2\ldots i_k}\oneid_\lambda=\E_{i_1} \E_{i_2} \ldots \E_{i_k}\oneid_\lambda$ and $\F_{i_1i_2\ldots i_k}\oneid_\lambda=\F_{i_1} \F_{i_2} \ldots \F_{i_k}\oneid_\lambda$. 

\subsection{The definition of \texorpdfstring{$\Ev$}{E}} 
In the definition below, we underline the $1$-morphism in homological degree zero in each complex. 
\begin{itemize}[wide,labelindent=0pt,itemsep=5pt]
\item For objects, we define $\Ev (\lambda)=\lambda$, where $\lambda\in \mathbb{Z}^3$.
\item For 1-morphisms, we define the action of $\Ev $ on generating 1-morphisms and extend it to all 1-morphisms via composition and direct sums, using the standard composition of complexes. 
\begin{itemize}[\textopenbullet]
	\item $\Ev (\E_1 \oneid_\lambda)=\underline{\E_1 \oneid_\lambda}$,
$\Ev (\E_2 \oneid_\lambda)=\underline{\E_2 \oneid_\lambda}$,
\vskip1mm
 \item $\Ev (\F_1 \oneid_\lambda)=\underline{\F_1 \oneid_\lambda}$,
$\Ev (\F_2 \oneid_\lambda)=\underline{\F_2 \oneid_\lambda}$,
 \item $\Ev (\E_3 \oneid_\lambda)=
 \xymatrix{\underline{\F_{12}\oneid_\lambda}\langle S\rangle \ar[rr]^(0.45){\xy (0,0)*{
		\labellist
		\small\hair 2pt
		\pinlabel \scalebox{0.7}{$\textcolor{red}{2}$} at 0 -5
		\pinlabel \scalebox{0.7}{$\textcolor{blue}{1}$} at 16 -5
		\pinlabel \scalebox{0.9}{$\lambda$} at 16 7
		\endlabellist
		\centering
		\includegraphics[scale=1.3]{./RdlBdr}
	}\endxy} & & \F_{21}\oneid_\lambda\langle S+1\rangle}$,
\item $\Ev (\F_3 \oneid_\lambda)=
 \xymatrix{\E_{12}\oneid_\lambda\langle -S-1\rangle 
 \ar[rr]^(0.52){\xy (0,0)*{
 		\labellist
		\small\hair 2pt
		\pinlabel \scalebox{0.7}{$\textcolor{red}{2}$} at 0 -5
		\pinlabel \scalebox{0.7}{$\textcolor{blue}{1}$} at 16 -5
		\pinlabel \scalebox{0.9}{$\lambda$} at 16 7
		\endlabellist
		\centering
		\includegraphics[scale=1.3]{./RurBul}
	}\endxy} & & \underline{\E_{21}\oneid_\lambda}\langle -S\rangle}$.
\end{itemize} 
\eqskip
\item We set $\levl=\lambda_1+\lambda_2+\lambda_3$ and call it the {\em Schur level} of $\lambda$.
\end{itemize}

For compatibility with later proofs, we give the definition of $\Ev$ on downwards-pointing generating 2-morphisms (that is, the horizontally mirrored versions of the 2-morphisms in \autoref{rem:mingens}). For generating 2-morphisms consisting only of strands between $\E_1\oneid_\lambda$, $\E_2\oneid_\lambda$, $\F_1\oneid_\lambda$ and $\F_2\oneid_\lambda$, the $2$-functor $\Ev $ acts as the identity, with the following exceptions: 
\begin{equation}
\Ev\left(\xy (0,0)*{
				\labellist
				\small\hair 2pt
				\pinlabel \scalebox{0.7}{$\textcolor{blue}{1}$} at 16 -5
				\pinlabel \scalebox{0.7}{$\textcolor{red}{2}$} at 0 -5
               	\pinlabel \scalebox{0.9}{$\lambda$} at 18 8
				\endlabellist
				\centering
				\includegraphics[scale=1.3]{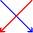}
		}\endxy\;\;\right) =(-1)^{\overline{\lambda}_1(\overline{\lambda}_2+1)}\xy (0,0)*{
				\labellist
				\small\hair 2pt
				\pinlabel \scalebox{0.7}{$\textcolor{blue}{1}$} at 16 -5
				\pinlabel \scalebox{0.7}{$\textcolor{red}{2}$} at 0 -5
               	\pinlabel \scalebox{0.9}{$\lambda$} at 18 8
				\endlabellist
				\centering
				\includegraphics[scale=1.3]{./RdrBdl}
		}\endxy
        \mspace{100mu}
        \Ev\left(\xy (0,0)*{
				\labellist
				\small\hair 2pt
				\pinlabel \scalebox{0.7}{$\textcolor{blue}{1}$} at 0 -5
				\pinlabel \scalebox{0.7}{$\textcolor{red}{2}$} at 16 -5
               	\pinlabel \scalebox{0.9}{$\lambda$} at 18 8
				\endlabellist
				\centering
				\includegraphics[scale=1.3]{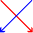}
		}\endxy \;\;\right)=(-1)^{\overline{\lambda}_1(\overline{\lambda}_2+1)}\xy (0,0)*{
				\labellist
				\small\hair 2pt
				\pinlabel \scalebox{0.7}{$\textcolor{blue}{1}$} at 0 -5
				\pinlabel \scalebox{0.7}{$\textcolor{red}{2}$} at 16 -5
               	\pinlabel \scalebox{0.9}{$\lambda$} at 18 8
				\endlabellist
				\centering
				\includegraphics[scale=1.3]{./BdrRdl}
		}\endxy
        \end{equation}
\eqskip
For the remaining generating $2$-morphisms, we define the images as follows, which we emphasise are commutative diagrams:
\begingroup
\allowdisplaybreaks
\begin{equation}
\Ev \left(\xy (0,0)*{
			\labellist
			\small\hair 2pt
			\pinlabel \scalebox{0.7}{$\textcolor{black}{3}$} at 1.5 -5
               	\pinlabel \scalebox{0.9}{$\lambda$} at 10 8
			\endlabellist
			\centering
			\includegraphics[scale=1.1]{./Kdo}
		}\endxy\;\;\right)
= 
  \xymatrix { \E_{12} \oneid_\lambda\langle -S+1\rangle \ar[rr]^-{\xy (0,0)*{
				\labellist
				\small\hair 2pt
				\pinlabel \scalebox{0.7}{$\textcolor{blue}{1}$} at 0 -5
				\pinlabel \scalebox{0.7}{$\textcolor{red}{2}$} at 16 -5
               	\pinlabel \scalebox{0.9}{$\lambda$} at 18 8
				\endlabellist
				\centering
				\includegraphics[scale=1.3]{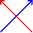}
		}\endxy} & & \underline{\E_{21} \oneid_\lambda}\langle -S+2\rangle \\ & & \\ \E_{12} \oneid_\lambda\langle -S-1\rangle \ar[uu]^{\xy (0,0)*{
		\labellist
		\small\hair 2pt
		\pinlabel \scalebox{0.7}{$\textcolor{blue}{1}$} at 1 -5
		\pinlabel \scalebox{0.7}{$\textcolor{red}{2}$} at 9 -5
       	\pinlabel \scalebox{0.9}{$\lambda$} at 14 8
		\endlabellist
		\centering
		\includegraphics[scale=1.3]{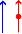}
	}\endxy} \ar[rr]_-{\xy (0,0)*{
		\labellist
		\small\hair 2pt
		\pinlabel \scalebox{0.7}{$\textcolor{blue}{1}$} at 0 -5
		\pinlabel \scalebox{0.7}{$\textcolor{red}{2}$} at 16 -5
       	\pinlabel \scalebox{0.9}{$\lambda$} at 18 8
            \endlabellist
		\centering
		\includegraphics[scale=1.3]{./BurRul}
}\endxy} 
		& & \underline{\E_{21} \oneid_\lambda}\langle -S\rangle \ar[uu]_{\xy (0,0)*{
				\labellist
				\small\hair 2pt
				\pinlabel \scalebox{0.7}{$\textcolor{blue}{1}$} at 9 -5
				\pinlabel \scalebox{0.7}{$\textcolor{red}{2}$} at 1 -5
               	\pinlabel \scalebox{0.9}{$\lambda$} at 14 8
				\endlabellist
				\centering
				\includegraphics[scale=1.3]{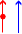}
		}\endxy}}
\end{equation}
\endgroup
\begingroup
\allowdisplaybreaks[0]
\begin{multline}
\Ev \bigl(\,\xy (0,0)*{
	\labellist
	\small\hair 2pt
	\pinlabel \scalebox{0.7}{$\textcolor{black}{3}$} at 0 -5
	\pinlabel \scalebox{0.7}{$\textcolor{black}{3}$} at 16 -5
        \pinlabel \scalebox{0.9}{$\lambda$} at 19 8
	\endlabellist
	\centering
	\includegraphics[scale=1]{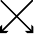}}\endxy\;\,\bigr) =\\[-1ex]
\scalebox{0.95}{ \xymatrix@C=3mm{ & &\E_{2112} \oneid_\lambda\langle -2S-3\rangle \ar[rrrd]^{\xy (0,0)*{
			\labellist
			\small\hair 2pt
			\pinlabel \scalebox{0.55}{$\textcolor{blue}{1}$} at 9 -5
			\pinlabel \scalebox{0.55}{$\textcolor{red}{2}$} at 1 -5
			\pinlabel \scalebox{0.55}{$\textcolor{red}{2}$} at 32 -5
			\pinlabel \scalebox{0.55}{$\textcolor{blue}{1}$} at 16 -5
                \pinlabel \scalebox{0.9}{$\lambda$} at 36 8   
			\endlabellist
			\centering
			\includegraphics[scale=1]{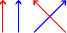}
	}\endxy}
	& & & \\
	\E_{1212} \oneid_\lambda\langle -2S-4\rangle \ar[rru]^{\hspace{-4ex}\xy (0,0)*{
			\labellist
			\small\hair 2pt
			\pinlabel \scalebox{0.55}{$\textcolor{blue}{1}$} at 0 -5
			\pinlabel \scalebox{0.55}{$\textcolor{red}{2}$} at 16 -5
			\pinlabel \scalebox{0.55}{$\textcolor{red}{2}$} at 31 -5
			\pinlabel \scalebox{0.55}{$\textcolor{blue}{1}$} at 23 -5
                \pinlabel \scalebox{0.9}{$\lambda$} at 36 8
                \endlabellist
			\centering
			\includegraphics[scale=1]{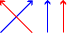}
	}\endxy}
	\ar[rrrd]^(0.35){\xy (0,0)*{
				\labellist
				\small\hair 2pt
                    \pinlabel \scalebox{0.8}{$-$} at -3 7    
				\pinlabel \scalebox{0.55}{$\textcolor{blue}{1}$} at 1 -5
				\pinlabel \scalebox{0.55}{$\textcolor{red}{2}$} at 9 -5
				\pinlabel \scalebox{0.55}{$\textcolor{red}{2}$} at 32 -5
				\pinlabel \scalebox{0.55}{$\textcolor{blue}{1}$} at 16 -5
                    \pinlabel \scalebox{0.9}{$\lambda$} at 36 8    
				\endlabellist
				\centering
				\includegraphics[scale=1]{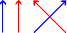}
		}\endxy}
	& & & & & \underline{\E_{2121} \oneid_\lambda}\langle -2S-2\rangle	\\
	& & & \E_{1221} \oneid_\lambda\langle -2S-3\rangle 
	\ar[rru]^(.42){\xy (0,0)*{
			\labellist
			\small\hair 2pt
			\pinlabel \scalebox{0.55}{$\textcolor{blue}{1}$} at 0 -5
			\pinlabel \scalebox{0.55}{$\textcolor{red}{2}$} at 16 -5
			\pinlabel \scalebox{0.55}{$\textcolor{red}{2}$} at 23 -5
			\pinlabel \scalebox{0.55}{$\textcolor{blue}{1}$} at 31 -5
                \pinlabel \scalebox{0.9}{$\lambda$} at 36 8   
			\endlabellist
			\centering
			\includegraphics[scale=1]{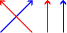}
	}\endxy}
	& &\\
	& & & & &\\
	& &\E_{2112} \oneid_\lambda\langle -2S-1\rangle \ar[rrrd]^(.7){\xy (0,0)*{
			\labellist
			\small\hair 2pt
			\pinlabel \scalebox{0.55}{$\textcolor{blue}{1}$} at 9 -5
			\pinlabel \scalebox{0.55}{$\textcolor{red}{2}$} at 1 -5
			\pinlabel \scalebox{0.55}{$\textcolor{red}{2}$} at 32 -5
			\pinlabel \scalebox{0.55}{$\textcolor{blue}{1}$} at 16 -5
                \pinlabel \scalebox{0.9}{$\lambda$} at 36 8   
			\endlabellist
			\centering
			\includegraphics[scale=1]{./RuBuBurRul}
	}\endxy}
	\ar[uuuu]^(.35){\xy (0,-0.5)*{
			\labellist
			\small\hair 2pt
			\pinlabel \scalebox{0.55}{$\textcolor{blue}{1}$} at 8 -5
			\pinlabel \scalebox{0.55}{$\textcolor{red}{2}$} at 1 -5
			\pinlabel \scalebox{0.55}{$\textcolor{red}{2}$} at 31 -5
			\pinlabel \scalebox{0.55}{$\textcolor{blue}{1}$} at 24 -5
                \pinlabel \scalebox{0.9}{$\lambda$} at 36 8   
			\endlabellist
			\centering
			\includegraphics[scale=1]{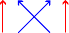}
	}\endxy\;}
	\ar[uur]^(.4){-\xy (0,-0.8)*{
		\labellist
		\small\hair 2pt
		\pinlabel \scalebox{0.55}{$\textcolor{blue}{1}$} at 8 -5
		\pinlabel \scalebox{0.55}{$\textcolor{red}{2}$} at 0 -5
		\pinlabel \scalebox{0.55}{$\textcolor{red}{2}$} at 32 -5
		\pinlabel \scalebox{0.55}{$\textcolor{blue}{1}$} at 24 -5
            \pinlabel \scalebox{0.9}{$\lambda$} at 36 8  
		\endlabellist
		\centering
		\includegraphics[scale=1]{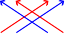}
}\endxy}
	& & & \\
	\E_{1212} \oneid_\lambda\langle -2S-2\rangle \ar[rru]^{\xy (0,0)*{
			\labellist
			\small\hair 2pt
			\pinlabel \scalebox{0.55}{$\textcolor{blue}{1}$} at 0 -5
			\pinlabel \scalebox{0.55}{$\textcolor{red}{2}$} at 16 -5
			\pinlabel \scalebox{0.55}{$\textcolor{red}{2}$} at 31 -5
			\pinlabel \scalebox{0.55}{$\textcolor{blue}{1}$} at 23 -5
                \pinlabel \scalebox{0.9}{$\lambda$} at 36 8   
			\endlabellist
			\centering
			\includegraphics[scale=1]{./BurRulBuRu}
	}\endxy}
	\ar[uuuu]^{\xy (0,0)*{
				\labellist
				\small\hair 2pt
				\pinlabel \scalebox{0.55}{$\textcolor{blue}{1}$} at 0 -5
				\pinlabel \scalebox{0.55}{$\textcolor{red}{2}$} at 8 -5
				\pinlabel \scalebox{0.55}{$\textcolor{red}{2}$} at 32 -5
				\pinlabel \scalebox{0.55}{$\textcolor{blue}{1}$} at 24 -5
                    \pinlabel \scalebox{0.9}{$\lambda$} at 36 8    
				\endlabellist
				\centering
				\includegraphics[scale=1]{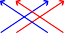}
		}\endxy\;\;}
	\ar[rrrd]_{-\xy (0,0)*{
			\labellist
			\small\hair 2pt
			\pinlabel \scalebox{0.55}{$\textcolor{blue}{1}$} at 1 -5
			\pinlabel \scalebox{0.55}{$\textcolor{red}{2}$} at 9 -5
			\pinlabel \scalebox{0.55}{$\textcolor{red}{2}$} at 32 -5
			\pinlabel \scalebox{0.55}{$\textcolor{blue}{1}$} at 16 -5
                \pinlabel \scalebox{0.9}{$\lambda$} at 36 8   
			\endlabellist
			\centering
			\includegraphics[scale=1]{./BuRuBurRul}
	}\endxy}
	& & & & & \underline{\E_{2121} \oneid_\lambda}\langle -2S\rangle	
	\ar[uuuu]_{\!\!-\xy (0,0)*{
			\labellist
			\small\hair 2pt
			\pinlabel \scalebox{0.55}{$\textcolor{blue}{1}$} at 8 -5
			\pinlabel \scalebox{0.55}{$\textcolor{red}{2}$} at 0 -5
			\pinlabel \scalebox{0.55}{$\textcolor{red}{2}$} at 24 -5
			\pinlabel \scalebox{0.55}{$\textcolor{blue}{1}$} at 32 -5
                \pinlabel \scalebox{0.9}{$\lambda$} at 34 8   
			\endlabellist
			\centering
			\includegraphics[scale=1]{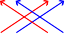}
	}\endxy}\\
	& & & \E_{1221} \oneid_\lambda\langle -2S-1\rangle 
	\ar[rru]_{\xy (0,0)*{
			\labellist
			\small\hair 2pt
			\pinlabel \scalebox{0.55}{$\textcolor{blue}{1}$} at 0 -5
			\pinlabel \scalebox{0.55}{$\textcolor{red}{2}$} at 16 -5
			\pinlabel \scalebox{0.55}{$\textcolor{red}{2}$} at 23 -5
			\pinlabel \scalebox{0.55}{$\textcolor{blue}{1}$} at 31 -5
                \pinlabel \scalebox{0.9}{$\lambda$} at 36 8   
			\endlabellist
			\centering
			\includegraphics[scale=1]{./BurRulRuBu}
	}\endxy}
	\ar[uuuu]_(.75){\hspace{-1ex}\xy (0,0)*{
				\labellist
				\small\hair 2pt
				\pinlabel \scalebox{0.55}{$\textcolor{blue}{1}$} at 1 -5
				\pinlabel \scalebox{0.55}{$\textcolor{red}{2}$} at 8 -5
				\pinlabel \scalebox{0.55}{$\textcolor{red}{2}$} at 24 -5
				\pinlabel \scalebox{0.55}{$\textcolor{blue}{1}$} at 31 -5
                    \pinlabel \scalebox{0.9}{$\lambda$} at 36 8    
				\endlabellist
				\centering
				\includegraphics[scale=1]{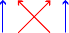}
		}\endxy}
	\ar@/_1pc/[uuuuuul]_(0.8){\hspace{-2ex}-\xy (0,-.8)*{
			\labellist
			\small\hair 2pt 
			\pinlabel \scalebox{0.55}{$\textcolor{blue}{1}$} at 0 -5
			\pinlabel \scalebox{0.55}{$\textcolor{red}{2}$} at 8 -5
			\pinlabel \scalebox{0.55}{$\textcolor{red}{2}$} at 24 -5
			\pinlabel \scalebox{0.55}{$\textcolor{blue}{1}$} at 32 -5
                \pinlabel \scalebox{0.9}{$\lambda$} at 36 8   
			\endlabellist
			\centering
			\includegraphics[scale=1]{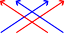}
	}\endxy}
	& &\\
}}
\end{multline}
\endgroup
%
\begingroup\allowdisplaybreaks
\begin{align}
\Ev \left(\xy (0,0)*{
		\labellist
		\small\hair 2pt
		\pinlabel \scalebox{0.7}{$\textcolor{black}{3}$} at 0 -5
		\pinlabel \scalebox{0.7}{$\textcolor{blue}{1}$} at 16 -5
       	\pinlabel \scalebox{0.9}{$\lambda$} at 18 9  
		\endlabellist
		\centering
		\includegraphics[scale=1]{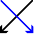}}\endxy\;\right) 
&=
  \xymatrix{\F_1\E_{12} \oneid_\lambda\langle -S\rangle \ar[rr]^-{\xy (0,0)*{
				\labellist
				\small\hair 2pt
				\pinlabel \scalebox{0.7}{$\textcolor{blue}{1}$} at 1 -5
				\pinlabel \scalebox{0.7}{$\textcolor{blue}{1}$} at 8 -5
				\pinlabel \scalebox{0.7}{$\textcolor{red}{2}$} at 24 -5
            	\pinlabel \scalebox{0.9}{$\lambda$} at 25 8
                    \endlabellist
				\centering
				\includegraphics[scale=1.3]{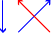}
		}\endxy} & & \underline{\F_1\E_{21} \oneid_\lambda}\langle -S+1\rangle \\
		& & \\
		\E_{12}\F_1 \oneid_\lambda\langle -S\rangle \ar[rr]_{-\xy (0,0)*{
				\labellist
				\small\hair 2pt
				\pinlabel \scalebox{0.7}{$\textcolor{blue}{1}$} at 0 -5
				\pinlabel \scalebox{0.7}{$\textcolor{blue}{1}$} at 23 -5
				\pinlabel \scalebox{0.7}{$\textcolor{red}{2}$} at 16 -5
               	\pinlabel \scalebox{0.9}{$\lambda$} at 27 8
				\endlabellist
				\centering
				\includegraphics[scale=1.3]{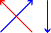}
		}\endxy} 
		\ar[uu]^{-\xy (0,-.2)*{
						\labellist
						\small\hair 2pt
						\pinlabel \scalebox{0.7}{$\textcolor{red}{2}$} at 12 -5
						\pinlabel \scalebox{0.7}{$\textcolor{blue}{1}$} at 0 -5
 						\pinlabel \scalebox{0.7}{$\textcolor{blue}{1}$} at 23 -5
                    	\pinlabel \scalebox{0.9}{$\lambda$} at 25 8
						\endlabellist
						\centering
                        \includegraphics[scale=1.3]{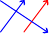}
				}\endxy
			}  & & \underline{\E_{21}\F_1 \oneid_\lambda}\langle -S+1\rangle 
		\ar[uu]_{\xy (0,-.2)*{
						\labellist
						\small\hair 2pt
						\pinlabel \scalebox{0.7}{$\textcolor{red}{2}$} at 0 -5
						\pinlabel \scalebox{0.7}{$\textcolor{blue}{1}$} at 12 -5
						\pinlabel \scalebox{0.7}{$\textcolor{blue}{1}$} at 23 -5
                    	\pinlabel \scalebox{0.9}{$\lambda$} at 25 8
						\endlabellist
						\centering
						\includegraphics[scale=1.3]{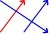}
				}\endxy
			}}
\\[1ex]
\Ev \left(\xy (0,0)*{
			\labellist
			\small\hair 2pt
			\pinlabel \scalebox{0.7}{$\textcolor{black}{3}$} at 16 -5
			\pinlabel \scalebox{0.7}{$\textcolor{blue}{1}$} at 0 -5
          	\pinlabel \scalebox{0.9}{$\lambda$} at 18 9
			\endlabellist
			\centering
			\includegraphics[scale=1]{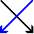}}\endxy\;\right)
&=
\xymatrix{\E_{12}\F_1 \oneid_\lambda\langle -S+1\rangle \ar[rr]^-{-\xy (0,0)*{
					\labellist
					\small\hair 2pt
					\pinlabel \scalebox{0.7}{$\textcolor{blue}{1}$} at 0 -5
					\pinlabel \scalebox{0.7}{$\textcolor{blue}{1}$} at 23 -5
					\pinlabel \scalebox{0.7}{$\textcolor{red}{2}$} at 16 -5
                 	\pinlabel \scalebox{0.9}{$\lambda$} at 27 8
					\endlabellist
					\centering
					\includegraphics[scale=1.3]{./BurRulBd}
			}\endxy} & & \underline{\E_{21}\F_1 \oneid_\lambda}\langle -S+2\rangle \\
			& & \\
			\F_1\E_{12} \oneid_\lambda\langle -S-1\rangle \ar[rr]_-{\xy (0,0)*{
					\labellist
					\small\hair 2pt
					\pinlabel \scalebox{0.7}{$\textcolor{blue}{1}$} at 1 -5
					\pinlabel \scalebox{0.7}{$\textcolor{blue}{1}$} at 8 -5
					\pinlabel \scalebox{0.7}{$\textcolor{red}{2}$} at 24 -5
                 	\pinlabel \scalebox{0.9}{$\lambda$} at 25 8
					\endlabellist
					\centering
					\includegraphics[scale=1.3]{./BdBurRul}
			}\endxy} 
			\ar[uu]^{-\xy (0,0)*{
					\labellist
					\small\hair 2pt
					\pinlabel \scalebox{0.7}{$\textcolor{red}{2}$} at 24 -5
					\pinlabel \scalebox{0.7}{$\textcolor{blue}{1}$} at 0 -5
					\pinlabel \scalebox{0.7}{$\textcolor{blue}{1}$} at 12 -5
                 	\pinlabel \scalebox{0.9}{$\lambda$} at 25 8
					\endlabellist
					\centering
					\includegraphics[scale=1.3]{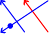}
			}\endxy
            \ \;\, + \xy (0,0)*{
					\labellist
					\small\hair 2pt
					\pinlabel \scalebox{0.7}{$\textcolor{red}{2}$} at 24 -5
					\pinlabel \scalebox{0.7}{$\textcolor{blue}{1}$} at 0 -5
					\pinlabel \scalebox{0.7}{$\textcolor{blue}{1}$} at 12 -5
                 	\pinlabel \scalebox{0.9}{$\lambda$} at 25 8
					\endlabellist
					\centering
					\includegraphics[scale=1.3]{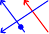}
			}\endxy}  & & \underline{\F_1\E_{21} \oneid_\lambda}\langle -S\rangle
			\ar[uu]_{\xy (0,0)*{
					\labellist
					\small\hair 2pt
					\pinlabel \scalebox{0.7}{$\textcolor{red}{2}$} at 12 -5
					\pinlabel \scalebox{0.7}{$\textcolor{blue}{1}$} at 0 -5
					\pinlabel \scalebox{0.7}{$\textcolor{blue}{1}$} at 23 -5
                 	\pinlabel \scalebox{0.9}{$\lambda$} at 25 8
					\endlabellist
					\centering
					\includegraphics[scale=1.3]{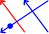}
				}\endxy
                \ \;\, - \xy (0,0)*{
					\labellist
					\small\hair 2pt
					\pinlabel \scalebox{0.7}{$\textcolor{red}{2}$} at 12 -5
					\pinlabel \scalebox{0.7}{$\textcolor{blue}{1}$} at 0 -5
					\pinlabel \scalebox{0.7}{$\textcolor{blue}{1}$} at 23 -5
                 	\pinlabel \scalebox{0.9}{$\lambda$} at 25 8
					\endlabellist
					\centering
					\includegraphics[scale=1.3]{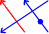}
				}\endxy
		}}
\\[1ex]
\Ev \left(\xy (0,0)*{
		\labellist
		\small\hair 2pt
		\pinlabel \scalebox{0.7}{$\textcolor{black}{3}$} at 16 -5
		\pinlabel \scalebox{0.7}{$\textcolor{red}{2}$} at 0 -5
       	\pinlabel \scalebox{0.9}{$\lambda$} at 18 9
		\endlabellist
		\centering
		\includegraphics[scale=1]{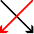}}\endxy\;\right)
  &=
  \xymatrix{\E_{12}\F_2 \oneid_\lambda\langle -S-1\rangle \ar[rr]^-{\xy (0,0)*{
				\labellist
				\small\hair 2pt
				\pinlabel \scalebox{0.7}{$\textcolor{blue}{1}$} at 0 -5
				\pinlabel \scalebox{0.7}{$\textcolor{red}{2}$} at 24 -5
				\pinlabel \scalebox{0.7}{$\textcolor{red}{2}$} at 16 -5
               	\pinlabel \scalebox{0.9}{$\lambda$} at 27 8
				\endlabellist
				\centering
				\includegraphics[scale=1.3]{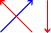}
		}\endxy} & & \underline{\E_{21}\F_2\oneid_\lambda}\langle -S\rangle \\
		& & \\
		\F_2\E_{12} \oneid_\lambda\langle -S-1\rangle \ar[rr]_-{\xy (0,0)*{
					\labellist
					\small\hair 2pt
					\pinlabel \scalebox{0.7}{$\textcolor{blue}{1}$} at 8 -5
					\pinlabel \scalebox{0.7}{$\textcolor{red}{2}$} at 0 -5
					\pinlabel \scalebox{0.7}{$\textcolor{red}{2}$} at 23 -5
                	\pinlabel \scalebox{0.9}{$\lambda$} at 25 8 
					\endlabellist
					\centering
					\includegraphics[scale=1.3]{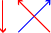}
			}\endxy} 
		\ar[uu]^{\xy (0,0)*{
					\labellist
					\small\hair 2pt
					\pinlabel \scalebox{0.7}{$\textcolor{red}{2}$} at 23 -5
					\pinlabel \scalebox{0.7}{$\textcolor{red}{2}$} at 0 -5
					\pinlabel \scalebox{0.7}{$\textcolor{blue}{1}$} at 12 -5
                 	\pinlabel \scalebox{0.9}{$\lambda$} at 25 8  
					\endlabellist
					\centering
					\includegraphics[scale=1.3]{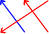}
			}\endxy} 
		& & \underline{\F_2\E_{21} \oneid_\lambda}\langle -S\rangle 
		\ar@<1ex>[uu]_{\xy (0,0)*{
					\labellist
					\small\hair 2pt
					\pinlabel \scalebox{0.7}{$\textcolor{red}{2}$} at 12 -5
					\pinlabel \scalebox{0.7}{$\textcolor{red}{2}$} at 0 -5
					\pinlabel \scalebox{0.7}{$\textcolor{blue}{1}$} at 23 -5
                 	\pinlabel \scalebox{0.9}{$\lambda$} at 25 8
					\endlabellist
					\centering
                    \includegraphics[scale=1.3]{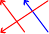}
			}\endxy
	}}
\\[1ex]
\Ev \left(\xy (0,0)*{
		\labellist
		\small\hair 2pt
		\pinlabel \scalebox{0.7}{$\textcolor{black}{3}$} at 0 -5
		\pinlabel \scalebox{0.7}{$\textcolor{red}{2}$} at 16 -5
      	\pinlabel \scalebox{0.9}{$\lambda$} at 18 9
		\endlabellist
		\centering
		\includegraphics[scale=1]{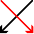}}\endxy\;\right)
  &=
  \xymatrix{\F_2\E_{12} \oneid_\lambda\langle -S\rangle\ar[rr]^-{\xy (0,0)*{
				\labellist
				\small\hair 2pt
				\pinlabel \scalebox{0.7}{$\textcolor{blue}{1}$} at 8 -5
				\pinlabel \scalebox{0.7}{$\textcolor{red}{2}$} at 24 -5
				\pinlabel \scalebox{0.7}{$\textcolor{red}{2}$} at 1 -5
               	\pinlabel \scalebox{0.9}{$\lambda$} at 25 8
				\endlabellist
				\centering
				\includegraphics[scale=1.3]{./RdBurRul}
		}\endxy} & & \underline{\F_2\E_{21} \oneid_\lambda}\langle -S+1\rangle \\
		& & \\
		\E_{12}\F_2 \oneid_\lambda\langle -S-2\rangle \ar[rr]_-{\xy (0,0)*{
				\labellist
				\small\hair 2pt
				\pinlabel \scalebox{0.7}{$\textcolor{blue}{1}$} at 0 -5
				\pinlabel \scalebox{0.7}{$\textcolor{red}{2}$} at 23 -5
				\pinlabel \scalebox{0.7}{$\textcolor{red}{2}$} at 16 -5
               	\pinlabel \scalebox{0.9}{$\lambda$} at 27 8 
				\endlabellist
				\centering
				\includegraphics[scale=1.3]{./BurRulRd}
		}\endxy} 
		\ar[uu]^{\xy (0,0)*{
					\labellist
					\small\hair 2pt
					\pinlabel \scalebox{0.7}{$\textcolor{red}{2}$} at 23 -5
					\pinlabel \scalebox{0.7}{$\textcolor{red}{2}$} at 12 -5
					\pinlabel \scalebox{0.7}{$\textcolor{blue}{1}$} at 0 -5
                 	\pinlabel \scalebox{0.9}{$\lambda$} at 25 8
					\endlabellist
					\centering
					\includegraphics[scale=1.3]{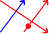}
			}\endxy
            \ \;\, -
            \xy (0,0)*{
					\labellist
					\small\hair 2pt
					\pinlabel \scalebox{0.7}{$\textcolor{red}{2}$} at 23 -5
					\pinlabel \scalebox{0.7}{$\textcolor{red}{2}$} at 12 -5
					\pinlabel \scalebox{0.7}{$\textcolor{blue}{1}$} at 0 -5
                 	\pinlabel \scalebox{0.9}{$\lambda$} at 25 8
					\endlabellist
					\centering
					\includegraphics[scale=1.3]{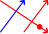}
			}\endxy
            }  & & \underline{\F_{12}\E_1 \oneid_\lambda}\langle -S-1\rangle
		\ar@<1ex>[uu]_{\xy (0,0)*{
				\labellist
				\small\hair 2pt
				\pinlabel \scalebox{0.7}{$\textcolor{red}{2}$} at 23 -5
				\pinlabel \scalebox{0.7}{$\textcolor{red}{2}$} at 0 -5
				\pinlabel \scalebox{0.7}{$\textcolor{blue}{1}$} at 12 -5
               	\pinlabel \scalebox{0.9}{$\lambda$} at 25 8
				\endlabellist
				\centering
				\includegraphics[scale=1.3]{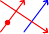}
	}\endxy
    \ \;\, - 
    \xy (0,0)*{
				\labellist
				\small\hair 2pt
				\pinlabel \scalebox{0.7}{$\textcolor{red}{2}$} at 23 -5
				\pinlabel \scalebox{0.7}{$\textcolor{red}{2}$} at 0 -5
				\pinlabel \scalebox{0.7}{$\textcolor{blue}{1}$} at 12 -5
               	\pinlabel \scalebox{0.9}{$\lambda$} at 25 8
				\endlabellist
				\centering
				\includegraphics[scale=1.3]{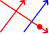}
	}\endxy
    }}
\end{align}
\endgroup

\labelmargin{5.5pt}

\begingroup
\allowdisplaybreaks[0]
\begin{multline}
\Ev \bigl(\xy (0,0)*{
	\labellist
	\small\hair 2pt
	\pinlabel \scalebox{0.7}{$\textcolor{black}{3}$} at 1 12
        \pinlabel \scalebox{0.9}{$\lambda$} at 16 0 
	\endlabellist
	\centering
	\includegraphics[scale=1.1]{./KlCup1}
}\endxy\,\bigr) =\\[-1ex]
\xymatrix@C=1.2mm{ 
	& & \underline{\F_{21}\E_{12} \oneid_\lambda}\langle 1-\overline{\lambda}_3\rangle 
	\ar[rrrd]^{\xy (0,6)*{
			\labellist
			\small\hair 2pt
			\pinlabel \scalebox{0.7}{$\textcolor{blue}{1}$} at 9 -5
			\pinlabel \scalebox{0.7}{$\textcolor{red}{2}$} at 1 -5
			\pinlabel \scalebox{0.7}{$\textcolor{red}{2}$} at 32 -5
			\pinlabel \scalebox{0.7}{$\textcolor{blue}{1}$} at 16 -5
                \pinlabel \scalebox{0.9}{$\lambda$} at 34 8
                \endlabellist
			\centering
			\includegraphics[scale=1.3]{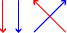}
	}\endxy}
	& & &\\ 
	\F_{12}\E_{12} \oneid_\lambda\langle -\overline{\lambda}_3\rangle 
	\ar[rrrd]_(.4){\xy (0,0)*{
			\labellist
			\small\hair 2pt
			\pinlabel \scalebox{0.7}{$-$} at -2 7
                \pinlabel \scalebox{0.7}{$\textcolor{blue}{1}$} at 1 -5
			\pinlabel \scalebox{0.7}{$\textcolor{red}{2}$} at 9 -5
			\pinlabel \scalebox{0.7}{$\textcolor{red}{2}$} at 32 -5
			\pinlabel \scalebox{0.7}{$\textcolor{blue}{1}$} at 16 -5
                \pinlabel \scalebox{0.9}{$\lambda$} at 34 8
			\endlabellist
			\centering
			\includegraphics[scale=1.3]{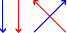}
		}\endxy\mspace{15mu}}
	\ar[rru]^{\xy (0,4)*{
			\labellist
			\small\hair 2pt
			\pinlabel \scalebox{0.7}{$\textcolor{blue}{1}$} at 0 -5
			\pinlabel \scalebox{0.7}{$\textcolor{red}{2}$} at 16 -5
			\pinlabel \scalebox{0.7}{$\textcolor{red}{2}$} at 31 -5
			\pinlabel \scalebox{0.7}{$\textcolor{blue}{1}$} at 23 -5
                \pinlabel \scalebox{0.9}{$\lambda$} at 36 8   
			\endlabellist
			\centering
			\includegraphics[scale=1.3]{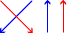}
	\mspace{40mu}}\endxy}
	 & & & & &\F_{21}\E_{21} \oneid_\lambda\langle 2-\overline{\lambda}_3\rangle\\
	& & &\underline{\F_{12}\E_{21} \oneid_\lambda}\langle 1-\overline{\lambda}_3\rangle
	\ar[rru]_{\xy (0,0)*{
			\labellist
			\small\hair 2pt
			\pinlabel \scalebox{0.7}{$\textcolor{blue}{1}$} at 0 -5
			\pinlabel \scalebox{0.7}{$\textcolor{red}{2}$} at 16 -5
			\pinlabel \scalebox{0.7}{$\textcolor{red}{2}$} at 23 -5
			\pinlabel \scalebox{0.7}{$\textcolor{blue}{1}$} at 31 -5
                \pinlabel \scalebox{0.9}{$\lambda$} at 36 8   
			\endlabellist
			\centering
			\includegraphics[scale=1.3]{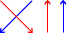}
	}\endxy}  & & \\
		& & & & & \\
	0   \ar[uuu] \ar[rr] & & \underline{\oneid_\lambda}\ar[rrr] \ar@/_1pc/[uur]_(.6){(-1)^{\lambda_3}\xy (0,0)*{
		\labellist
		\small\hair 2pt
		\pinlabel \scalebox{0.7}{$\textcolor{blue}{1}$} at 0 15
		\pinlabel \scalebox{0.7}{$\textcolor{red}{2}$} at 8 15
		\pinlabel \scalebox{0.7}{$\textcolor{red}{2}$} at 24 15
		\pinlabel \scalebox{0.7}{$\textcolor{blue}{1}$} at 32 15
            \pinlabel \scalebox{0.9}{$\lambda$} at 28 2  
		\endlabellist
		\centering
		\includegraphics[scale=1.3]{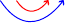}
}\endxy} 
	\ar@/^/[uuuu]^(.15){(-1)^{\lambda_3+1}\xy (0,0)*{
			\labellist
			\small\hair 2pt
			\pinlabel \scalebox{0.7}{$\textcolor{blue}{1}$} at 8 15
			\pinlabel \scalebox{0.7}{$\textcolor{red}{2}$} at 0 15
			\pinlabel \scalebox{0.7}{$\textcolor{red}{2}$} at 32 15
			\pinlabel \scalebox{0.7}{$\textcolor{blue}{1}$} at 24 15
                \pinlabel \scalebox{0.9}{$\lambda$} at 28 2   
			\endlabellist
			\centering
			\includegraphics[scale=1.3]{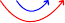}
	}\endxy}
& & & 0 \ar[uuu]}
\end{multline}
\endgroup

\begingroup
\allowdisplaybreaks[0]
\begin{multline}
\Ev \bigl(\xy (0,0)*{
	\labellist
	\small\hair 2pt
	\pinlabel \scalebox{0.7}{$\textcolor{black}{3}$} at 1 12
        \pinlabel \scalebox{0.9}{$\lambda$} at 16 0 
	\endlabellist
	\centering
	\includegraphics[scale=1.1]{./KrCup1}
}\endxy\,\bigr) =\\[-1ex]
\xymatrix@C=1.2mm{ 
	& & \underline{\E_{21}\F_{12} \oneid_\lambda}\langle 1+\overline{\lambda}_3\rangle
	\ar[rrrd]^{\xy (0,6)*{
			\labellist
			\small\hair 2pt
			\pinlabel \scalebox{0.9}{$-$} at -2 7   
			\pinlabel \scalebox{0.7}{$\textcolor{blue}{1}$} at 9 -5
			\pinlabel \scalebox{0.7}{$\textcolor{red}{2}$} at 1 -5
			\pinlabel \scalebox{0.7}{$\textcolor{red}{2}$} at 32 -5
			\pinlabel \scalebox{0.7}{$\textcolor{blue}{1}$} at 16 -5
                \pinlabel \scalebox{0.9}{$\lambda$} at 34 8   
			\endlabellist
			\centering
			\includegraphics[scale=1.3]{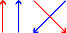}
	}\endxy}
	& & &\\ 
	\E_{12}\F_{12} \oneid_\lambda\langle \overline{\lambda}_3\rangle 
	\ar[rru]^{\xy (0,4)*{
			\labellist
			\small\hair 2pt
			\pinlabel \scalebox{0.7}{$\textcolor{blue}{1}$} at 0 -5
			\pinlabel \scalebox{0.7}{$\textcolor{red}{2}$} at 16 -5
			\pinlabel \scalebox{0.7}{$\textcolor{red}{2}$} at 31 -5
			\pinlabel \scalebox{0.7}{$\textcolor{blue}{1}$} at 23 -5
                \pinlabel \scalebox{0.9}{$\lambda$} at 36 8   
			\endlabellist
			\centering
			\includegraphics[scale=1.3]{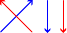}
		\mspace{40mu}}\endxy}
	\ar[rrrd]_(.4){\xy (0,0)*{
		\labellist
		\small\hair 2pt
		\pinlabel \scalebox{0.7}{$\textcolor{blue}{1}$} at 1 -5
		\pinlabel \scalebox{0.7}{$\textcolor{red}{2}$} at 9 -5
		\pinlabel \scalebox{0.7}{$\textcolor{red}{2}$} at 32 -5
		\pinlabel \scalebox{0.7}{$\textcolor{blue}{1}$} at 16 -5
                \pinlabel \scalebox{0.9}{$\lambda$} at 34 8  
		\endlabellist
		\centering
		\includegraphics[scale=1.3]{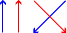}
	}\endxy\mspace{15mu}} & & & & &\E_{21}\F_{21} \oneid_\lambda\langle 2+\overline{\lambda}_3\rangle \ar[ddd]\\
	& & &\underline{\E_{12}\F_{21} \oneid_\lambda}\langle 1+\overline{\lambda}_3\rangle
	\ar[rru]_{\xy (0,0)*{
			\labellist
			\small\hair 2pt
			\pinlabel \scalebox{0.7}{$\textcolor{blue}{1}$} at 0 -5
			\pinlabel \scalebox{0.7}{$\textcolor{red}{2}$} at 16 -5
			\pinlabel \scalebox{0.7}{$\textcolor{red}{2}$} at 23 -5
			\pinlabel \scalebox{0.7}{$\textcolor{blue}{1}$} at 31 -5
                \pinlabel \scalebox{0.9}{$\lambda$} at 36 8   
			\endlabellist
			\centering
			\includegraphics[scale=1.3]{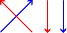}
		}\endxy}  & & \\
	& & & & & \\
	0 \ar[uuu] \ar[rr] & & \underline{\oneid_\lambda}
	\ar@/_1pc/[uur]_(0.65){-\xy (0,0)*{
			\labellist
			\small\hair 2pt
			\pinlabel \scalebox{0.7}{$\textcolor{blue}{1}$} at 0 15
			\pinlabel \scalebox{0.7}{$\textcolor{red}{2}$} at 8 15
			\pinlabel \scalebox{0.7}{$\textcolor{red}{2}$} at 24 15
			\pinlabel \scalebox{0.7}{$\textcolor{blue}{1}$} at 32 15
                \pinlabel \scalebox{0.9}{$\lambda$} at 30 2   
			\endlabellist
			\centering
			\includegraphics[scale=1.3]{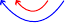}
		}\endxy} 
	\ar@/^/[uuuu]^(.2){\xy (0,0)*{
			\labellist
			\small\hair 2pt
			\pinlabel \scalebox{0.7}{$\textcolor{blue}{1}$} at 8 15
			\pinlabel \scalebox{0.7}{$\textcolor{red}{2}$} at 0 15
			\pinlabel \scalebox{0.7}{$\textcolor{red}{2}$} at 32 15
			\pinlabel \scalebox{0.7}{$\textcolor{blue}{1}$} at 24 15
                \pinlabel \scalebox{0.9}{$\lambda$} at 30 2   
			\endlabellist
			\centering
			\includegraphics[scale=1.3]{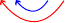}
		}\endxy}
	\ar[rrr] 
	& & & 0 \ar[uuu]}
\end{multline}
\endgroup

\begingroup
\allowdisplaybreaks[0]
\begin{multline}
\Ev \bigl(\xy (0,0)*{
	\labellist
	\small\hair 2pt
	\pinlabel \scalebox{0.7}{$\textcolor{black}{3}$} at 1 -5
   	\pinlabel \scalebox{0.9}{$\lambda$} at 14 10 
	\endlabellist
	\centering
	\includegraphics[scale=1.1]{./KrCap1}
}\endxy\bigr) =\\[-3.7ex]
\xymatrix@C=6.5mm{ 0 \ar[rrr] & & & \underline{\oneid_{\lambda}}\langle 1-\overline{\lambda}_3\rangle\ar[rr] & & 0 \\ & & & & & \\
 & & \underline{\F_{21}\E_{12} \oneid_\lambda} 
\ar[rrrd]^(.6){\xy (0,6)*{
		\labellist
		\small\hair 2pt
		\pinlabel \scalebox{0.7}{$\textcolor{blue}{1}$} at 9 -5
		\pinlabel \scalebox{0.7}{$\textcolor{red}{2}$} at 1 -5
		\pinlabel \scalebox{0.7}{$\textcolor{red}{2}$} at 32 -5
		\pinlabel \scalebox{0.7}{$\textcolor{blue}{1}$} at 16 -5
            \pinlabel \scalebox{0.9}{$\lambda$} at 34 8
            \endlabellist
		\centering
		\includegraphics[scale=1.3]{./RdBdBurRul}
}\endxy}
	\ar@/^1pc/[uur]^(.3){-\xy (0,0)*{
		\labellist
		\small\hair 2pt
		\pinlabel \scalebox{0.7}{$\textcolor{blue}{1}$} at 5 -5
		\pinlabel \scalebox{0.7}{$\textcolor{red}{2}$} at 0 -5
		\pinlabel \scalebox{0.7}{$\textcolor{red}{2}$} at 24 -5
		\pinlabel \scalebox{0.7}{$\textcolor{blue}{1}$} at 19 -5
            \pinlabel \scalebox{0.9}{$\lambda$} at 28 12  
		\endlabellist
		\centering
		\includegraphics[scale=1.3]{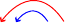}
	}\endxy} 
& & &\\ 
\F_{12}\E_{12} \oneid_\lambda\langle -1\rangle
\ar[rru]^{\xy (0,6)*{
		\labellist
		\small\hair 2pt
		\pinlabel \scalebox{0.7}{$\textcolor{blue}{1}$} at 0 -5
		\pinlabel \scalebox{0.7}{$\textcolor{red}{2}$} at 16 -5
		\pinlabel \scalebox{0.7}{$\textcolor{red}{2}$} at 31 -5
		\pinlabel \scalebox{0.7}{$\textcolor{blue}{1}$} at 23 -5
            \pinlabel \scalebox{0.9}{$\lambda$} at 36 8
            \endlabellist
		\centering
		\includegraphics[scale=1.3]{./BdrRdlBuRu}
}\endxy}
	\ar[uuu]
	\ar[rrrd]_{-\xy (0,0)*{
			\labellist
			\small\hair 2pt
			\pinlabel \scalebox{0.7}{$\textcolor{blue}{1}$} at 1 -5
			\pinlabel \scalebox{0.7}{$\textcolor{red}{2}$} at 9 -5
			\pinlabel \scalebox{0.7}{$\textcolor{red}{2}$} at 32 -5
			\pinlabel \scalebox{0.7}{$\textcolor{blue}{1}$} at 16 -5
                \pinlabel \scalebox{0.9}{$\lambda$} at 34 8
                \endlabellist
			\centering
			\includegraphics[scale=1.3]{./BdRdBurRul}
	}\endxy}
 & & & & &\F_{21}\E_{21} \oneid_\lambda\langle 1\rangle \ar[uuu]\\
			& & &\underline{\F_{12}\E_{21} \oneid_\lambda} 
	\ar@/^/[uuuu]_(.8){\xy (0,0)*{
\labellist
			\small\hair 2pt
			\pinlabel \scalebox{0.7}{$\textcolor{blue}{1}$} at 0 -5
			\pinlabel \scalebox{0.7}{$\textcolor{red}{2}$} at 8 -5
			\pinlabel \scalebox{0.7}{$\textcolor{red}{2}$} at 24 -5
			\pinlabel \scalebox{0.7}{$\textcolor{blue}{1}$} at 32 -5
          	\pinlabel \scalebox{0.9}{$\lambda$} at 28 12  
			\endlabellist
			\centering
			\includegraphics[scale=1.3]{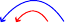}
	}\endxy\mspace{10mu}} 
\ar[rru]_{\xy (0,0)*{
		\labellist
		\small\hair 2pt
		\pinlabel \scalebox{0.7}{$\textcolor{blue}{1}$} at 0 -5
		\pinlabel \scalebox{0.7}{$\textcolor{red}{2}$} at 16 -5
		\pinlabel \scalebox{0.7}{$\textcolor{red}{2}$} at 23 -5
		\pinlabel \scalebox{0.7}{$\textcolor{blue}{1}$} at 31 -5
            \pinlabel \scalebox{0.9}{$\lambda$} at 36 8  
		\endlabellist
		\centering
		\includegraphics[scale=1.3]{./BdrRdlRuBu}
}\endxy}  & & }
\end{multline}
\endgroup

\begingroup
\allowdisplaybreaks[0]
\begin{multline}
\Ev \bigl(\xy (0,0)*{
	\labellist
	\small\hair 2pt
	\pinlabel \scalebox{0.7}{$\textcolor{black}{3}$} at 1 -5
        \pinlabel \scalebox{0.9}{$\lambda$} at 15 10  
	\endlabellist
	\centering
	\includegraphics[scale=1.1]{./KlCap1}
}\endxy\,\bigr) =\\[-3.7ex]
\xymatrix@C=6.5mm{ 0 \ar[rrr] & & & \underline{\oneid_{\lambda}}\langle 1+\overline{\lambda}_3\rangle \ar[rr] & & 0 \\ & & & & & \\
	& & \underline{\E_{21}\F_{12} \oneid_\lambda} 
	\ar[rrrd]^(.6){\xy (0,6)*{
			\labellist
			\small\hair 2pt
   			\pinlabel \scalebox{0.8}{$-$} at -2 7
			\pinlabel \scalebox{0.7}{$\textcolor{blue}{1}$} at 9 -5
			\pinlabel \scalebox{0.7}{$\textcolor{red}{2}$} at 1 -5
			\pinlabel \scalebox{0.7}{$\textcolor{red}{2}$} at 32 -5
			\pinlabel \scalebox{0.7}{$\textcolor{blue}{1}$} at 16 -5
                \pinlabel \scalebox{0.9}{$\lambda$} at 34 8    
			\endlabellist
			\centering
			\includegraphics[scale=1.3]{./RuBuBdrRdl}
	}\endxy}
	\ar@/^1pc/[uur]^(.3){(-1)^{\lambda_3+1}\xy (0,0)*{
			\labellist
			\small\hair 2pt
			\pinlabel \scalebox{0.7}{$\textcolor{blue}{1}$} at 8 -5
			\pinlabel \scalebox{0.7}{$\textcolor{red}{2}$} at 0 -5
			\pinlabel \scalebox{0.7}{$\textcolor{red}{2}$} at 32 -5
			\pinlabel \scalebox{0.7}{$\textcolor{blue}{1}$} at 24 -5
                \pinlabel \scalebox{0.9}{$\lambda$} at 24 12       
			\endlabellist
			\centering
			\includegraphics[scale=1.3]{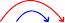}
	}\endxy}
	& & &\\ 
	\E_{12}\F_{12} \oneid_\lambda\langle -1\rangle
	\ar[rru]^{\xy (0,6)*{
			\labellist
			\small\hair 2pt
			\pinlabel \scalebox{0.7}{$\textcolor{blue}{1}$} at 0 -5
			\pinlabel \scalebox{0.7}{$\textcolor{red}{2}$} at 16 -5
			\pinlabel \scalebox{0.7}{$\textcolor{red}{2}$} at 31 -5
			\pinlabel \scalebox{0.7}{$\textcolor{blue}{1}$} at 23 -5
                \pinlabel \scalebox{0.9}{$\lambda$} at 36 8       
			\endlabellist
			\centering
			\includegraphics[scale=1.3]{./BurRulBdRd}
	}\endxy}
	\ar[uuu]
	\ar[rrrd]_{\xy (0,0)*{
			\labellist
			\small\hair 2pt
			\pinlabel \scalebox{0.7}{$\textcolor{blue}{1}$} at 1 -5
			\pinlabel \scalebox{0.7}{$\textcolor{red}{2}$} at 9 -5
			\pinlabel \scalebox{0.7}{$\textcolor{red}{2}$} at 32 -5
			\pinlabel \scalebox{0.7}{$\textcolor{blue}{1}$} at 16 -5
                \pinlabel \scalebox{0.9}{$\lambda$} at 34 8       
			\endlabellist
			\centering
			\includegraphics[scale=1.3]{./BuRuBdrRdl}
	}\endxy} & & & & &\E_{21}\F_{21} \oneid_\lambda\langle 1\rangle \ar[uuu]\\
	& & &\underline{\E_{12}\F_{21} \oneid_\lambda} 
	\ar@/^/[uuuu]_(.8){(-1)^{\lambda_3}\xy (0,0)*{
			\labellist
			\small\hair 2pt
			\pinlabel \scalebox{0.7}{$\textcolor{blue}{1}$} at 0 -5
			\pinlabel \scalebox{0.7}{$\textcolor{red}{2}$} at 8 -5
			\pinlabel \scalebox{0.7}{$\textcolor{red}{2}$} at 24 -5
			\pinlabel \scalebox{0.7}{$\textcolor{blue}{1}$} at 32 -5
                \pinlabel \scalebox{0.9}{$\lambda$} at 28 12    
			\endlabellist
			\centering
			\includegraphics[scale=1.3]{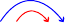}
	}\endxy\mspace{10mu}}
	\ar[rru]_{\xy (0,0)*{
			\labellist
			\small\hair 2pt
			\pinlabel \scalebox{0.7}{$\textcolor{blue}{1}$} at 0 -5
			\pinlabel \scalebox{0.7}{$\textcolor{red}{2}$} at 16 -5
			\pinlabel \scalebox{0.7}{$\textcolor{red}{2}$} at 23 -5
			\pinlabel \scalebox{0.7}{$\textcolor{blue}{1}$} at 31 -5
                \pinlabel \scalebox{0.9}{$\lambda$} at 36 8       
			\endlabellist
			\centering
			\includegraphics[scale=1.3]{./BurRulRdBd}
	}\endxy}  & & }
\end{multline}
\endgroup
\eqskip

The following theorem is the main result of the paper and we will prove in \autoref{sec:EvRelate} that it is a consequence of \autoref{thm:BigThmPrime} and \autoref{lem:EvRel}.

\begin{thm}\label{thm:BigThmMain}
    $\Ev$ is a well-defined 2-functor that decategorifies to $\ev$.
\end{thm}

\vskip0.2cm 
\subsection{The definition of \texorpdfstring{$\Ev'$}{Ev'}}\label{sec:well-def}

To define $\Ev'$, we first introduce some notation. For the rest of the paper, we let $\equiv$ denote $\equiv_4$, that is, congruence modulo 4. Then, we define \begin{equation}\label{eq:kdef} k_i^{a_1,\dots,a_n}(\lambda):=\begin{cases}
    1 & \overline{\lambda}_i \equiv a_1,\dots,a_n\\
    -1 & \text{otherwise}
\end{cases}\end{equation}

\noindent For example, $k_1^{0,1}(\lambda)=\begin{cases}
    1 & \overline{\lambda}_1 \equiv 0,1 \\
    -1 & \overline{\lambda}_1 \equiv 2,3
\end{cases}$. We will often omit the argument when it is $\lambda$. 

\begin{rem}\label{rem:kIdents}
 We have various relations involving the $k_i^{a_1,\dots,a_n}$ which we will be using at various points: 
\begin{gather*}
 k_i^0k_i^2 = (-1)^{\overline{\lambda}_i+1}, \mspace{50mu}   k_i^1k_i^3 = (-1)^{\overline{\lambda}_i} ,  
 \mspace{50mu} k_1^a(\lambda+\alpha_2) = k_1^{a+1}(\lambda),
\\
 k_1^a(s_1(\lambda)) = k_1^{-a}(\lambda), \mspace{50mu} k_1^a(s_2(\lambda))=k_3^{-a}(\lambda) .
\end{gather*}
\end{rem}

We will also be using the more compact notation for 2-morphisms between complexes found in \cite{abram2022categorification}, where they are presented as ordered tuples (most commonly ordered pairs).

We now define $\Ev'$ to be identical to $\Ev$ on objects, 1-morphisms and all generating 2-morphisms except for the following:
\begin{itemize}[wide,labelindent=0pt]
    \item $\Ev' \left(\xy (0,0)*{
		\labellist
		\small\hair 2pt
		\pinlabel \scalebox{0.7}{$\textcolor{black}{3}$} at 16 -5
		\pinlabel \scalebox{0.7}{$\textcolor{red}{2}$} at 0 -5
       	\pinlabel \scalebox{0.9}{$\lambda$} at 18 9
		\endlabellist
		\centering
		\includegraphics[scale=1]{./RdrKdl}}\endxy\;\right)=
        \left(-k_3^{0,3}k_1^{0,3}\xy (0,0)*{
					\labellist
					\small\hair 2pt
					\pinlabel \scalebox{0.7}{$\textcolor{red}{2}$} at 23 -5
					\pinlabel \scalebox{0.7}{$\textcolor{red}{2}$} at 0 -5
					\pinlabel \scalebox{0.7}{$\textcolor{blue}{1}$} at 12 -5
                 	\pinlabel \scalebox{0.9}{$\lambda$} at 25 8  
					\endlabellist
					\centering
					\includegraphics[scale=1.3]{./RdrrBulRul}
			}\endxy\;\;\;, -k_3^{0,3}k_1^{0,1}\xy (0,0)*{
					\labellist
					\small\hair 2pt
					\pinlabel \scalebox{0.7}{$\textcolor{red}{2}$} at 12 -5
					\pinlabel \scalebox{0.7}{$\textcolor{red}{2}$} at 0 -5
					\pinlabel \scalebox{0.7}{$\textcolor{blue}{1}$} at 23 -5
                 	\pinlabel \scalebox{0.9}{$\lambda$} at 25 8
					\endlabellist
					\centering
                    \includegraphics[scale=1.3]{./RdrrRulBul}
			}\endxy
	\;\;\right)$
\item[]
\item[]
    \item $\Ev' \left(\xy (0,0)*{
		\labellist
		\small\hair 2pt
		\pinlabel \scalebox{0.7}{$\textcolor{black}{3}$} at 0 -5
		\pinlabel \scalebox{0.7}{$\textcolor{red}{2}$} at 16 -5
      	\pinlabel \scalebox{0.9}{$\lambda$} at 18 9
		\endlabellist
		\centering
		\includegraphics[scale=1]{./KdrRdl}}\endxy\;\right) = \left(k_3^{0,3}k_1^{0,3}\left(\xy (0,0)*{
					\labellist
					\small\hair 2pt
					\pinlabel \scalebox{0.7}{$\textcolor{red}{2}$} at 23 -5
					\pinlabel \scalebox{0.7}{$\textcolor{red}{2}$} at 12 -5
					\pinlabel \scalebox{0.7}{$\textcolor{blue}{1}$} at 0 -5
                 	\pinlabel \scalebox{0.9}{$\lambda$} at 25 8
					\endlabellist
					\centering
					\includegraphics[scale=1.3]{./BurRurRodll}
			}\endxy
            \ \;\, -
            \xy (0,0)*{
					\labellist
					\small\hair 2pt
					\pinlabel \scalebox{0.7}{$\textcolor{red}{2}$} at 23 -5
					\pinlabel \scalebox{0.7}{$\textcolor{red}{2}$} at 12 -5
					\pinlabel \scalebox{0.7}{$\textcolor{blue}{1}$} at 0 -5
                 	\pinlabel \scalebox{0.9}{$\lambda$} at 25 8
					\endlabellist
					\centering
					\includegraphics[scale=1.3]{./BurRourRdll}
			}\endxy
            \;\;\right.\right), \left.k_3^{0,3}k_1^{0,1}\left(\xy (0,0)*{
				\labellist
				\small\hair 2pt
				\pinlabel \scalebox{0.7}{$\textcolor{red}{2}$} at 23 -5
				\pinlabel \scalebox{0.7}{$\textcolor{red}{2}$} at 0 -5
				\pinlabel \scalebox{0.7}{$\textcolor{blue}{1}$} at 12 -5
               	\pinlabel \scalebox{0.9}{$\lambda$} at 25 8
				\endlabellist
				\centering
				\includegraphics[scale=1.3]{./RurBurRodll}
	}\endxy\ \;\, - 
    \xy (0,0)*{
				\labellist
				\small\hair 2pt
				\pinlabel \scalebox{0.7}{$\textcolor{red}{2}$} at 23 -5
				\pinlabel \scalebox{0.7}{$\textcolor{red}{2}$} at 0 -5
				\pinlabel \scalebox{0.7}{$\textcolor{blue}{1}$} at 12 -5
               	\pinlabel \scalebox{0.9}{$\lambda$} at 25 8
				\endlabellist
				\centering
				\includegraphics[scale=1.3]{./RourBurRdll}
	}\endxy
    \;\;
    \right)
        \right)$
    \item[]
    \item[]
    \item $\Ev' \bigl(\xy (0,0)*{
	\labellist
	\small\hair 2pt
	\pinlabel \scalebox{0.7}{$\textcolor{black}{3}$} at 1 12
        \pinlabel \scalebox{0.9}{$\lambda$} at 16 0 
	\endlabellist
	\centering
	\includegraphics[scale=1.1]{./KlCup1}
}\endxy\,\bigr) =(-1)^{\lambda_3+1}\left(k_1^0\;\xy (0,0)*{
		\labellist
		\small\hair 2pt
		\pinlabel \scalebox{0.7}{$\textcolor{blue}{1}$} at 0 15
		\pinlabel \scalebox{0.7}{$\textcolor{red}{2}$} at 8 15
		\pinlabel \scalebox{0.7}{$\textcolor{red}{2}$} at 24 15
		\pinlabel \scalebox{0.7}{$\textcolor{blue}{1}$} at 32 15
            \pinlabel \scalebox{0.9}{$\lambda$} at 28 2  
		\endlabellist
		\centering
		\includegraphics[scale=1.3]{./BrCup3RrCup1}
}\endxy \;\; -k_1^1\;\xy (0,0)*{
			\labellist
			\small\hair 2pt
			\pinlabel \scalebox{0.7}{$\textcolor{blue}{1}$} at 8 15
			\pinlabel \scalebox{0.7}{$\textcolor{red}{2}$} at 0 15
			\pinlabel \scalebox{0.7}{$\textcolor{red}{2}$} at 32 15
			\pinlabel \scalebox{0.7}{$\textcolor{blue}{1}$} at 24 15
                \pinlabel \scalebox{0.9}{$\lambda$} at 28 2   
			\endlabellist
			\centering
			\includegraphics[scale=1.3]{./RrCup3BrCup1}
	}\endxy\right)$
    \item[]
    \item $\Ev' \bigl(\xy (0,0)*{
	\labellist
	\small\hair 2pt
	\pinlabel \scalebox{0.7}{$\textcolor{black}{3}$} at 1 12
        \pinlabel \scalebox{0.9}{$\lambda$} at 16 0 
	\endlabellist
	\centering
	\includegraphics[scale=1.1]{./KrCup1}
}\endxy\,\bigr) =k_1^0\;\xy (0,0)*{
			\labellist
			\small\hair 2pt
			\pinlabel \scalebox{0.7}{$\textcolor{blue}{1}$} at 0 15
			\pinlabel \scalebox{0.7}{$\textcolor{red}{2}$} at 8 15
			\pinlabel \scalebox{0.7}{$\textcolor{red}{2}$} at 24 15
			\pinlabel \scalebox{0.7}{$\textcolor{blue}{1}$} at 32 15
                \pinlabel \scalebox{0.9}{$\lambda$} at 30 2   
			\endlabellist
			\centering
			\includegraphics[scale=1.3]{./BlCup3RlCup1}
		}\endxy\;\; - k_1^3\;\xy (0,0)*{
			\labellist
			\small\hair 2pt
			\pinlabel \scalebox{0.7}{$\textcolor{blue}{1}$} at 8 15
			\pinlabel \scalebox{0.7}{$\textcolor{red}{2}$} at 0 15
			\pinlabel \scalebox{0.7}{$\textcolor{red}{2}$} at 32 15
			\pinlabel \scalebox{0.7}{$\textcolor{blue}{1}$} at 24 15
                \pinlabel \scalebox{0.9}{$\lambda$} at 30 2   
			\endlabellist
			\centering
			\includegraphics[scale=1.3]{./RlCup3BlCup1}
		}\endxy$
        \item[]
        \item $\Ev' \bigl(\xy (0,0)*{
	\labellist
	\small\hair 2pt
	\pinlabel \scalebox{0.7}{$\textcolor{black}{3}$} at 1 -5
   	\pinlabel \scalebox{0.9}{$\lambda$} at 14 10 
	\endlabellist
	\centering
	\includegraphics[scale=1.1]{./KrCap1}
}\endxy\bigr) = k_1^3\;\xy (0,0)*{
		\labellist
		\small\hair 2pt
		\pinlabel \scalebox{0.7}{$\textcolor{blue}{1}$} at 5 -5
		\pinlabel \scalebox{0.7}{$\textcolor{red}{2}$} at 0 -5
		\pinlabel \scalebox{0.7}{$\textcolor{red}{2}$} at 24 -5
		\pinlabel \scalebox{0.7}{$\textcolor{blue}{1}$} at 19 -5
            \pinlabel \scalebox{0.9}{$\lambda$} at 28 12  
		\endlabellist
		\centering
		\includegraphics[scale=1.3]{./RlCap3BlCap1Big}
	}\endxy \;\; - k_1^2\; \xy (0,0)*{
\labellist
			\small\hair 2pt
			\pinlabel \scalebox{0.7}{$\textcolor{blue}{1}$} at 0 -5
			\pinlabel \scalebox{0.7}{$\textcolor{red}{2}$} at 8 -5
			\pinlabel \scalebox{0.7}{$\textcolor{red}{2}$} at 24 -5
			\pinlabel \scalebox{0.7}{$\textcolor{blue}{1}$} at 32 -5
          	\pinlabel \scalebox{0.9}{$\lambda$} at 28 12  
			\endlabellist
			\centering
			\includegraphics[scale=1.3]{./BlCap3RlCap1Big}
	}\endxy$
    \item[]
    \item $\Ev' \bigl(\xy (0,0)*{
	\labellist
	\small\hair 2pt
	\pinlabel \scalebox{0.7}{$\textcolor{black}{3}$} at 1 -5
        \pinlabel \scalebox{0.9}{$\lambda$} at 15 10  
	\endlabellist
	\centering
	\includegraphics[scale=1.1]{./KlCap1}
}\endxy\,\bigr) = (-1)^{\lambda_3}\left(k_1^1\;\xy (0,0)*{
			\labellist
			\small\hair 2pt
			\pinlabel \scalebox{0.7}{$\textcolor{blue}{1}$} at 8 -5
			\pinlabel \scalebox{0.7}{$\textcolor{red}{2}$} at 0 -5
			\pinlabel \scalebox{0.7}{$\textcolor{red}{2}$} at 32 -5
			\pinlabel \scalebox{0.7}{$\textcolor{blue}{1}$} at 24 -5
                \pinlabel \scalebox{0.9}{$\lambda$} at 28 12       
			\endlabellist
			\centering
			\includegraphics[scale=1.3]{./RrCap3BrCap1}
	}\endxy\;\; -k_1^2\;\xy (0,0)*{
			\labellist
			\small\hair 2pt
			\pinlabel \scalebox{0.7}{$\textcolor{blue}{1}$} at 0 -5
			\pinlabel \scalebox{0.7}{$\textcolor{red}{2}$} at 8 -5
			\pinlabel \scalebox{0.7}{$\textcolor{red}{2}$} at 24 -5
			\pinlabel \scalebox{0.7}{$\textcolor{blue}{1}$} at 32 -5
                \pinlabel \scalebox{0.9}{$\lambda$} at 28 12    
			\endlabellist
			\centering
			\includegraphics[scale=1.3]{./BrCap3RrCap1}
	}\endxy\;\right)$
    \item[]
\end{itemize}

We will prove the following in \autoref{sec:BigThPrime}.
\begin{thm}\label{thm:BigThmPrime}
    $\Ev'$ is a well-defined 2-functor that decategorifies to $\ev$.
\end{thm}

\subsection{Relating \texorpdfstring{$\Ev$}{Ev} and \texorpdfstring{$\Ev'$}{Ev'}}\label{sec:EvRelate} 

We will now show that $\Ev$ can be given by composing $\Ev'$ with 2-isomorphisms. The first such 2-isomorphism is $\gamma$, which acts as the identity on objects, 1-morphisms, and generating 2-morphisms with the exception of:
\begin{align*}
\xy (0,0)*{
			\labellist
			\small\hair 2pt
			\pinlabel \scalebox{0.7}{$\textcolor{blue}{1}$} at 1 -5
			\pinlabel \scalebox{0.9}{$\lambda$} at 18 8
			\endlabellist 
			\centering 
			\includegraphics[scale=1.3]{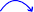}
	}\endxy\;\;
    &\overset{\gamma}{\mapsto}
    -k_1^0\;
     \xy (0,0)*{
			\labellist
			\small\hair 2pt
			\pinlabel \scalebox{0.7}{$\textcolor{blue}{1}$} at 1 -5
			\pinlabel \scalebox{0.9}{$\lambda$} at 18 8
			\endlabellist 
			\centering 
			\includegraphics[scale=1.3]{./BrCap1}
	}\endxy
    &
\xy (0,0)*{
			\labellist
			\small\hair 2pt
			\pinlabel \scalebox{0.7}{$\textcolor{blue}{1}$} at 1 9
			\pinlabel \scalebox{0.9}{$\lambda$} at 18 2
			\endlabellist 
			\centering 
			\includegraphics[scale=1.3]{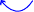}
	}\endxy\;\;
    &\overset{\gamma}{\mapsto}
    -k_1^0\; \xy (0,0)*{
			\labellist
			\small\hair 2pt
			\pinlabel \scalebox{0.7}{$\textcolor{blue}{1}$} at 1 9
			\pinlabel \scalebox{0.9}{$\lambda$} at 18 2
			\endlabellist 
			\centering 
			\includegraphics[scale=1.3]{./BlCup1}
	}\endxy
   \\[1ex]
   \xy (0,0)*{
			\labellist
			\small\hair 2pt
			\pinlabel \scalebox{0.7}{$\textcolor{blue}{1}$} at 1 -5
			\pinlabel \scalebox{0.9}{$\lambda$} at 18 8
			\endlabellist 
			\centering 
			\includegraphics[scale=1.3]{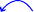}
	}\endxy\;\;
    &\overset{\gamma}{\mapsto}
    -k_1^2 \; \xy (0,0)*{
			\labellist
			\small\hair 2pt
			\pinlabel \scalebox{0.7}{$\textcolor{blue}{1}$} at 1 -5
			\pinlabel \scalebox{0.9}{$\lambda$} at 18 8
			\endlabellist 
			\centering 
			\includegraphics[scale=1.3]{./BlCap1}
	}\endxy
 &
    \xy (0,0)*{
			\labellist
			\small\hair 2pt
			\pinlabel \scalebox{0.7}{$\textcolor{blue}{1}$} at 1 9
			\pinlabel \scalebox{0.9}{$\lambda$} at 18 2
			\endlabellist 
			\centering 
			\includegraphics[scale=1.3]{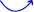}
	}\endxy\;\;&\overset{\gamma}{\mapsto}
    -k_1^2 \; \xy (0,0)*{
			\labellist
			\small\hair 2pt
			\pinlabel \scalebox{0.7}{$\textcolor{blue}{1}$} at 1 9
			\pinlabel \scalebox{0.9}{$\lambda$} at 18 2
			\endlabellist 
			\centering 
			\includegraphics[scale=1.3]{./BrCup1}
	}\endxy
\end{align*}
We abuse notation by also using $\gamma$ to refer to the 2-isomorphism of $K^b(\naffu{3})$ that acts in the same fashion. Similarly to $\beta$, it is straightforward to see that $\gamma$ preserves KM1 through KM9, and is therefore a (pair of) well-defined 2-isomorphism(s).

The second 2-isomorphism is $\delta$, which is again the identity on all objects, 1-morphisms and generating 2-morphisms except:\begin{align*}
    \xy (0,1)*{
		\labellist
		\small\hair 2pt
		\pinlabel \scalebox{0.7}{$\textcolor{black}{3}$} at 0 -5
		\pinlabel \scalebox{0.7}{$\textcolor{red}{2}$} at 16 -5
      	\pinlabel \scalebox{0.9}{$\lambda$} at 18 9
		\endlabellist
		\centering
		\includegraphics[scale=1]{./KdrRdl}}\endxy\;\; \overset{\delta}{\mapsto} -k_3^{0,3}k_1^{0,3}\;\xy (0,0)*{
		\labellist
		\small\hair 2pt
		\pinlabel \scalebox{0.7}{$\textcolor{black}{3}$} at 0 -5
		\pinlabel \scalebox{0.7}{$\textcolor{red}{2}$} at 16 -5
      	\pinlabel \scalebox{0.9}{$\lambda$} at 18 9
		\endlabellist
		\centering
		\includegraphics[scale=1]{./KdrRdl}}\endxy\qquad,\;\;\; &
        \xy (0,1)*{
		\labellist
		\small\hair 2pt
		\pinlabel \scalebox{0.7}{$\textcolor{black}{3}$} at 16 -5
		\pinlabel \scalebox{0.7}{$\textcolor{red}{2}$} at 0 -5
      	\pinlabel \scalebox{0.9}{$\lambda$} at 18 9
		\endlabellist
		\centering
		\includegraphics[scale=1]{./RdrKdl}}\endxy\;\; \overset{\delta}{\mapsto} -k_3^{0,3}k_1^{0,3}\; \xy (0,0)*{
		\labellist
		\small\hair 2pt
		\pinlabel \scalebox{0.7}{$\textcolor{black}{3}$} at 16 -5
		\pinlabel \scalebox{0.7}{$\textcolor{red}{2}$} at 0 -5
      	\pinlabel \scalebox{0.9}{$\lambda$} at 18 9
		\endlabellist
		\centering
		\includegraphics[scale=1]{./RdrKdl}}\endxy
\end{align*}
\eqskip
It is again an easy calculation that $\delta$ preserves KM1 through KM9 and is therefore a well-defined 2-isomorphism.
\begin{lem}\label{lem:EvRel} $\Ev=\gamma\Ev'\delta\gamma$.\end{lem}
\begin{proof} Note that, outside of the generating 2-morphisms that $\Ev'$ sends to a 2-morphism of complexes with a dependency on $\overline{\lambda}_i$ modulo $4$, $\Ev$ and $\gamma\Ev'\delta\gamma$ agree with $\Ev'$. In particular, we have e.g. $$\gamma\Ev'\delta\gamma\left(\xy (0,0)*{
			\labellist
			\small\hair 2pt
			\pinlabel \scalebox{0.7}{$\textcolor{blue}{1}$} at 1 -5
			\pinlabel \scalebox{0.9}{$\lambda$} at 18 8
			\endlabellist 
			\centering 
			\includegraphics[scale=1.3]{./BlCap1}
	}\endxy\;\;\right) = \gamma\Ev'\left( -k_1^2\;\xy (0,0)*{
			\labellist
			\small\hair 2pt
			\pinlabel \scalebox{0.7}{$\textcolor{blue}{1}$} at 1 -5
			\pinlabel \scalebox{0.9}{$\lambda$} at 18 8
			\endlabellist 
			\centering 
			\includegraphics[scale=1.3]{./BlCap1}
	}\endxy\;\;\right) = (k_1^2)^2 \xy (0,0)*{
			\labellist
			\small\hair 2pt
			\pinlabel \scalebox{0.7}{$\textcolor{blue}{1}$} at 1 -5
			\pinlabel \scalebox{0.9}{$\lambda$} at 18 8
			\endlabellist 
			\centering 
			\includegraphics[scale=1.3]{./BlCap1}
	}\endxy \mspace{8mu}= \xy (0,0)*{
			\labellist
			\small\hair 2pt
			\pinlabel \scalebox{0.7}{$\textcolor{blue}{1}$} at 1 -5
			\pinlabel \scalebox{0.9}{$\lambda$} at 18 8
			\endlabellist 
			\centering 
			\includegraphics[scale=1.3]{./BlCap1}
	}\endxy \mspace{8mu}= \Ev'\left(\xy (0,0)*{
			\labellist
			\small\hair 2pt
			\pinlabel \scalebox{0.7}{$\textcolor{blue}{1}$} at 1 -5
			\pinlabel \scalebox{0.9}{$\lambda$} at 18 8
			\endlabellist 
			\centering 
			\includegraphics[scale=1.3]{./BlCap1}
	}\endxy\;\;\right)$$

    For the 2-morphisms that differ we have:
\begingroup\allowdisplaybreaks
        \begin{align}
        \gamma\Ev'\delta\gamma\left(\xy (0,0)*{
			\labellist
			\small\hair 2pt
			\pinlabel \scalebox{0.7}{$\textcolor{black}{3}$} at 1 -5
			\pinlabel \scalebox{0.9}{$\lambda$} at 16 8
			\endlabellist 
			\centering 
			\includegraphics[scale=1.3]{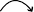}
	}\endxy\;
        \right) & = \gamma\left(k_1^3(\lambda)\xy (0,0)*{
		\labellist
		\small\hair 2pt
		\pinlabel \scalebox{0.7}{$\textcolor{blue}{1}$} at 8 -5
		\pinlabel \scalebox{0.7}{$\textcolor{red}{2}$} at 0 -5
		\pinlabel \scalebox{0.7}{$\textcolor{red}{2}$} at 32 -5
		\pinlabel \scalebox{0.7}{$\textcolor{blue}{1}$} at 24 -5
            \pinlabel \scalebox{0.9}{$\lambda$} at 28 12  
		\endlabellist
		\centering
		\includegraphics[scale=1.3]{./RlCap3BlCap1Big}
	}\endxy - k_1^2(\lambda)\xy (0,0)*{
		\labellist
		\small\hair 2pt
		\pinlabel \scalebox{0.7}{$\textcolor{blue}{1}$} at 0 -5
		\pinlabel \scalebox{0.7}{$\textcolor{red}{2}$} at 8 -5
		\pinlabel \scalebox{0.7}{$\textcolor{red}{2}$} at 24 -5
		\pinlabel \scalebox{0.7}{$\textcolor{blue}{1}$} at 32 -5
            \pinlabel \scalebox{0.9}{$\lambda$} at 28 12  
		\endlabellist
		\centering
		\includegraphics[scale=1.3]{./BlCap3RlCap1Big}
	}\endxy\;
        \right)\\ \nn
        & \\ \nn
        & = -k_1^3(\lambda) k_1^2(\lambda+\alpha_2)\xy (0,0)*{
		\labellist
		\small\hair 2pt
		\pinlabel \scalebox{0.7}{$\textcolor{blue}{1}$} at 8 -5
		\pinlabel \scalebox{0.7}{$\textcolor{red}{2}$} at 0 -5
		\pinlabel \scalebox{0.7}{$\textcolor{red}{2}$} at 32 -5
		\pinlabel \scalebox{0.7}{$\textcolor{blue}{1}$} at 24 -5
            \pinlabel \scalebox{0.9}{$\lambda$} at 28 12  
		\endlabellist
		\centering
		\includegraphics[scale=1.3]{./RlCap3BlCap1Big}
	}\endxy + k_1^2(\lambda)^2 \xy (0,0)*{
		\labellist
		\small\hair 2pt
		\pinlabel \scalebox{0.7}{$\textcolor{blue}{1}$} at 0 -5
		\pinlabel \scalebox{0.7}{$\textcolor{red}{2}$} at 8 -5
		\pinlabel \scalebox{0.7}{$\textcolor{red}{2}$} at 24 -5
		\pinlabel \scalebox{0.7}{$\textcolor{blue}{1}$} at 32 -5
            \pinlabel \scalebox{0.9}{$\lambda$} at 28 12  
		\endlabellist
		\centering
		\includegraphics[scale=1.3]{./BlCap3RlCap1Big}
	}\endxy \\ \nn
   &  \\ \nn
 & = -k_1^3(\lambda)^2 \xy (0,0)*{
		\labellist
		\small\hair 2pt
		\pinlabel \scalebox{0.7}{$\textcolor{blue}{1}$} at 8 -5
		\pinlabel \scalebox{0.7}{$\textcolor{red}{2}$} at 0 -5
		\pinlabel \scalebox{0.7}{$\textcolor{red}{2}$} at 32 -5
		\pinlabel \scalebox{0.7}{$\textcolor{blue}{1}$} at 24 -5
            \pinlabel \scalebox{0.9}{$\lambda$} at 28 12  
		\endlabellist
		\centering
		\includegraphics[scale=1.3]{./RlCap3BlCap1Big}
	}\endxy + k_1^2(\lambda)^2 \xy (0,0)*{
		\labellist
		\small\hair 2pt
		\pinlabel \scalebox{0.7}{$\textcolor{blue}{1}$} at 0 -5
		\pinlabel \scalebox{0.7}{$\textcolor{red}{2}$} at 8 -5
		\pinlabel \scalebox{0.7}{$\textcolor{red}{2}$} at 24 -5
		\pinlabel \scalebox{0.7}{$\textcolor{blue}{1}$} at 32 -5
            \pinlabel \scalebox{0.9}{$\lambda$} at 28 12  
		\endlabellist
		\centering
		\includegraphics[scale=1.3]{./BlCap3RlCap1Big}
	}\endxy \\ \nn
     & \\ \nn
     & = -\xy (0,0)*{
		\labellist
		\small\hair 2pt
		\pinlabel \scalebox{0.7}{$\textcolor{blue}{1}$} at 8 -5
		\pinlabel \scalebox{0.7}{$\textcolor{red}{2}$} at 0 -5
		\pinlabel \scalebox{0.7}{$\textcolor{red}{2}$} at 32 -5
		\pinlabel \scalebox{0.7}{$\textcolor{blue}{1}$} at 24 -5
            \pinlabel \scalebox{0.9}{$\lambda$} at 28 12  
		\endlabellist
		\centering
		\includegraphics[scale=1.3]{./RlCap3BlCap1Big}
	}\endxy + \xy (0,0)*{
		\labellist
		\small\hair 2pt
		\pinlabel \scalebox{0.7}{$\textcolor{blue}{1}$} at 0 -5
		\pinlabel \scalebox{0.7}{$\textcolor{red}{2}$} at 8 -5
		\pinlabel \scalebox{0.7}{$\textcolor{red}{2}$} at 24 -5
		\pinlabel \scalebox{0.7}{$\textcolor{blue}{1}$} at 32 -5
            \pinlabel \scalebox{0.9}{$\lambda$} at 28 12  
		\endlabellist
		\centering
		\includegraphics[scale=1.3]{./BlCap3RlCap1Big}
	}\endxy\\ \nn 
    & \\ \nn
    & = \Ev\left(\xy (0,0)*{
			\labellist
			\small\hair 2pt
			\pinlabel \scalebox{0.7}{$\textcolor{black}{3}$} at 1 -5
			\pinlabel \scalebox{0.9}{$\lambda$} at 16 8
			\endlabellist 
			\centering 
			\includegraphics[scale=1.3]{./KrCap1Thin}
	}\endxy\;\;\right)
       \\
       \gamma\Ev'\delta\gamma\left(\xy (0,0)*{
			\labellist
			\small\hair 2pt
			\pinlabel \scalebox{0.7}{$\textcolor{black}{3}$} at 1 -5
			\pinlabel \scalebox{0.9}{$\lambda$} at 16 8
			\endlabellist 
			\centering 
			\includegraphics[scale=1.3]{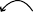}
	}\endxy\;\;
        \right) & = \gamma\left((-1)^{\lambda_3}\left(k_1^1(\lambda)\xy (0,0)*{
		\labellist
		\small\hair 2pt
		\pinlabel \scalebox{0.7}{$\textcolor{blue}{1}$} at 8 -5
		\pinlabel \scalebox{0.7}{$\textcolor{red}{2}$} at 0 -5
		\pinlabel \scalebox{0.7}{$\textcolor{red}{2}$} at 32 -5
		\pinlabel \scalebox{0.7}{$\textcolor{blue}{1}$} at 24 -5
            \pinlabel \scalebox{0.9}{$\lambda$} at 28 12  
		\endlabellist
		\centering
		\includegraphics[scale=1.3]{./RrCap3BrCap1}
	}\endxy -k_1^2(\lambda)\xy (0,0)*{
		\labellist
		\small\hair 2pt
		\pinlabel \scalebox{0.7}{$\textcolor{blue}{1}$} at 0 -5
		\pinlabel \scalebox{0.7}{$\textcolor{red}{2}$} at 8 -5
		\pinlabel \scalebox{0.7}{$\textcolor{red}{2}$} at 24 -5
		\pinlabel \scalebox{0.7}{$\textcolor{blue}{1}$} at 32 -5
            \pinlabel \scalebox{0.9}{$\lambda$} at 28 12  
		\endlabellist
		\centering
		\includegraphics[scale=1.3]{./BrCap3RrCap1}
	}\endxy\right)
        \right)\\ \nn
        & \\ \nn
        & = (-1)^{\lambda_3}\gamma\left(k_1^2(\lambda-\alpha_2)\xy (0,0)*{
		\labellist
		\small\hair 2pt
		\pinlabel \scalebox{0.7}{$\textcolor{blue}{1}$} at 8 -5
		\pinlabel \scalebox{0.7}{$\textcolor{red}{2}$} at 0 -5
		\pinlabel \scalebox{0.7}{$\textcolor{red}{2}$} at 32 -5
		\pinlabel \scalebox{0.7}{$\textcolor{blue}{1}$} at 24 -5
            \pinlabel \scalebox{0.9}{$\lambda$} at 28 12  
		\endlabellist
		\centering
		\includegraphics[scale=1.3]{./RrCap3BrCap1}
	}\endxy - k_1^2(\lambda)\xy (0,0)*{
		\labellist
		\small\hair 2pt
		\pinlabel \scalebox{0.7}{$\textcolor{blue}{1}$} at 0 -5
		\pinlabel \scalebox{0.7}{$\textcolor{red}{2}$} at 8 -5
		\pinlabel \scalebox{0.7}{$\textcolor{red}{2}$} at 24 -5
		\pinlabel \scalebox{0.7}{$\textcolor{blue}{1}$} at 32 -5
            \pinlabel \scalebox{0.9}{$\lambda$} at 28 12  
		\endlabellist
		\centering
		\includegraphics[scale=1.3]{./BrCap3RrCap1}
	}\endxy
        \right)\\ \nn
        & \\ \nn
        & = (-1)^{\lambda_3+1}\left(k_1^2(\lambda-\alpha_2)^2\xy (0,0)*{
		\labellist
		\small\hair 2pt
		\pinlabel \scalebox{0.7}{$\textcolor{blue}{1}$} at 8 -5
		\pinlabel \scalebox{0.7}{$\textcolor{red}{2}$} at 0 -5
		\pinlabel \scalebox{0.7}{$\textcolor{red}{2}$} at 32 -5
		\pinlabel \scalebox{0.7}{$\textcolor{blue}{1}$} at 24 -5
            \pinlabel \scalebox{0.9}{$\lambda$} at 28 12  
		\endlabellist
		\centering
		\includegraphics[scale=1.3]{./RrCap3BrCap1}
	}\endxy - k_1^2(\lambda)^2 \xy (0,0)*{
		\labellist
		\small\hair 2pt
		\pinlabel \scalebox{0.7}{$\textcolor{blue}{1}$} at 0 -5
		\pinlabel \scalebox{0.7}{$\textcolor{red}{2}$} at 8 -5
		\pinlabel \scalebox{0.7}{$\textcolor{red}{2}$} at 24 -5
		\pinlabel \scalebox{0.7}{$\textcolor{blue}{1}$} at 32 -5
            \pinlabel \scalebox{0.9}{$\lambda$} at 28 12  
		\endlabellist
		\centering
		\includegraphics[scale=1.3]{./BrCap3RrCap1}
	}\endxy \right)\\ \nn
     & \\ \nn
     & = (-1)^{\lambda_3+1}\left(\xy (0,0)*{
		\labellist
		\small\hair 2pt
		\pinlabel \scalebox{0.7}{$\textcolor{blue}{1}$} at 8 -5
		\pinlabel \scalebox{0.7}{$\textcolor{red}{2}$} at 0 -5
		\pinlabel \scalebox{0.7}{$\textcolor{red}{2}$} at 32 -5
		\pinlabel \scalebox{0.7}{$\textcolor{blue}{1}$} at 24 -5
            \pinlabel \scalebox{0.9}{$\lambda$} at 28 12  
		\endlabellist
		\centering
		\includegraphics[scale=1.3]{./RrCap3BrCap1}
	}\endxy - \xy (0,0)*{
		\labellist
		\small\hair 2pt
		\pinlabel \scalebox{0.7}{$\textcolor{blue}{1}$} at 0 -5
		\pinlabel \scalebox{0.7}{$\textcolor{red}{2}$} at 8 -5
		\pinlabel \scalebox{0.7}{$\textcolor{red}{2}$} at 24 -5
		\pinlabel \scalebox{0.7}{$\textcolor{blue}{1}$} at 32 -5
            \pinlabel \scalebox{0.9}{$\lambda$} at 28 12  
		\endlabellist
		\centering
		\includegraphics[scale=1.3]{./BrCap3RrCap1}
	}\endxy\right)\\ \nn
    & \\ \nn
    & = \Ev\left(\xy (0,0)*{
			\labellist
			\small\hair 2pt
			\pinlabel \scalebox{0.7}{$\textcolor{black}{3}$} at 1 -5
			\pinlabel \scalebox{0.9}{$\lambda$} at 16 8
			\endlabellist 
			\centering 
			\includegraphics[scale=1.3]{./KlCap1Thin}
	}\endxy\;\;\right)
   \\[1ex]
   \gamma\Ev'\delta\gamma \bigl(\xy (0,0)*{
	\labellist
	\small\hair 2pt
	\pinlabel \scalebox{0.7}{$\textcolor{black}{3}$} at 1 12
        \pinlabel \scalebox{0.9}{$\lambda$} at 16 0 
	\endlabellist
	\centering
	\includegraphics[scale=1.1]{./KrCup1}
}\endxy\,\bigr) & =\gamma\left(k_1^0(\lambda)\xy (0,0)*{
			\labellist
			\small\hair 2pt
			\pinlabel \scalebox{0.7}{$\textcolor{blue}{1}$} at 0 15
			\pinlabel \scalebox{0.7}{$\textcolor{red}{2}$} at 8 15
			\pinlabel \scalebox{0.7}{$\textcolor{red}{2}$} at 24 15
			\pinlabel \scalebox{0.7}{$\textcolor{blue}{1}$} at 32 15
                \pinlabel \scalebox{0.9}{$\lambda$} at 30 2   
			\endlabellist
			\centering
			\includegraphics[scale=1.3]{./BlCup3RlCup1}
		}\endxy\;\; - k_1^0(\lambda-\alpha_2)\xy (0,0)*{
			\labellist
			\small\hair 2pt
			\pinlabel \scalebox{0.7}{$\textcolor{blue}{1}$} at 8 15
			\pinlabel \scalebox{0.7}{$\textcolor{red}{2}$} at 0 15
			\pinlabel \scalebox{0.7}{$\textcolor{red}{2}$} at 32 15
			\pinlabel \scalebox{0.7}{$\textcolor{blue}{1}$} at 24 15
                \pinlabel \scalebox{0.9}{$\lambda$} at 30 2   
			\endlabellist
			\centering
			\includegraphics[scale=1.3]{./RlCup3BlCup1}
		}\endxy\;\right)\\ \nn
        & \\ \nn
        & = -k_1^0(\lambda)^2\xy (0,0)*{
			\labellist
			\small\hair 2pt
			\pinlabel \scalebox{0.7}{$\textcolor{blue}{1}$} at 0 15
			\pinlabel \scalebox{0.7}{$\textcolor{red}{2}$} at 8 15
			\pinlabel \scalebox{0.7}{$\textcolor{red}{2}$} at 24 15
			\pinlabel \scalebox{0.7}{$\textcolor{blue}{1}$} at 32 15
                \pinlabel \scalebox{0.9}{$\lambda$} at 30 2   
			\endlabellist
			\centering
			\includegraphics[scale=1.3]{./BlCup3RlCup1}
		}\endxy\;\; +k_1^0(\lambda-\alpha_2)^2\xy (0,0)*{
			\labellist
			\small\hair 2pt
			\pinlabel \scalebox{0.7}{$\textcolor{blue}{1}$} at 8 15
			\pinlabel \scalebox{0.7}{$\textcolor{red}{2}$} at 0 15
			\pinlabel \scalebox{0.7}{$\textcolor{red}{2}$} at 32 15
			\pinlabel \scalebox{0.7}{$\textcolor{blue}{1}$} at 24 15
                \pinlabel \scalebox{0.9}{$\lambda$} at 30 2   
			\endlabellist
			\centering
			\includegraphics[scale=1.3]{./RlCup3BlCup1}
		}\endxy\\ \nn
        & \\ \nn
        & = -\xy (0,0)*{
			\labellist
			\small\hair 2pt
			\pinlabel \scalebox{0.7}{$\textcolor{blue}{1}$} at 0 15
			\pinlabel \scalebox{0.7}{$\textcolor{red}{2}$} at 8 15
			\pinlabel \scalebox{0.7}{$\textcolor{red}{2}$} at 24 15
			\pinlabel \scalebox{0.7}{$\textcolor{blue}{1}$} at 32 15
                \pinlabel \scalebox{0.9}{$\lambda$} at 30 2   
			\endlabellist
			\centering
			\includegraphics[scale=1.3]{./BlCup3RlCup1}
		}\endxy \;\; + \xy (0,0)*{
			\labellist
			\small\hair 2pt
			\pinlabel \scalebox{0.7}{$\textcolor{blue}{1}$} at 8 15
			\pinlabel \scalebox{0.7}{$\textcolor{red}{2}$} at 0 15
			\pinlabel \scalebox{0.7}{$\textcolor{red}{2}$} at 32 15
			\pinlabel \scalebox{0.7}{$\textcolor{blue}{1}$} at 24 15
                \pinlabel \scalebox{0.9}{$\lambda$} at 30 2   
			\endlabellist
			\centering
			\includegraphics[scale=1.3]{./RlCup3BlCup1}
		}\endxy\\ \nn
        & \\[-2ex] \nn
        & = \Ev\left(\xy (0,0)*{
	\labellist
	\small\hair 2pt
	\pinlabel \scalebox{0.7}{$\textcolor{black}{3}$} at 1 12
        \pinlabel \scalebox{0.9}{$\lambda$} at 16 0 
	\endlabellist
	\centering
	\includegraphics[scale=1.1]{./KrCup1}
}\endxy\;\right)
\\[1ex]
            \gamma\Ev'\delta\gamma \bigl(\xy (0,0)*{
	\labellist
	\small\hair 2pt
	\pinlabel \scalebox{0.7}{$\textcolor{black}{3}$} at 1 12
        \pinlabel \scalebox{0.9}{$\lambda$} at 16 0 
	\endlabellist
	\centering
	\includegraphics[scale=1.1]{./KlCup1}
}\endxy\,\bigr) & =\gamma\left((-1)^{\lambda_3+1}k_1^0(\lambda)\xy (0,0)*{
		\labellist
		\small\hair 2pt
		\pinlabel \scalebox{0.7}{$\textcolor{blue}{1}$} at 0 15
		\pinlabel \scalebox{0.7}{$\textcolor{red}{2}$} at 8 15
		\pinlabel \scalebox{0.7}{$\textcolor{red}{2}$} at 24 15
		\pinlabel \scalebox{0.7}{$\textcolor{blue}{1}$} at 32 15
            \pinlabel \scalebox{0.9}{$\lambda$} at 28 2  
		\endlabellist
		\centering
		\includegraphics[scale=1.3]{./BrCup3RrCup1}
}\endxy \;\; + (-1)^{\lambda_3}k_1^1(\lambda)\xy (0,0)*{
			\labellist
			\small\hair 2pt
			\pinlabel \scalebox{0.7}{$\textcolor{blue}{1}$} at 8 15
			\pinlabel \scalebox{0.7}{$\textcolor{red}{2}$} at 0 15
			\pinlabel \scalebox{0.7}{$\textcolor{red}{2}$} at 32 15
			\pinlabel \scalebox{0.7}{$\textcolor{blue}{1}$} at 24 15
                \pinlabel \scalebox{0.9}{$\lambda$} at 28 2   
			\endlabellist
			\centering
			\includegraphics[scale=1.3]{./RrCup3BrCup1}
	}\endxy\right)\\ \nn
    & \\ \nn 
    & = (-1)^{\lambda_3+1}\left((-1)^{\lambda_3+1}k_1^0(\lambda)\xy (0,0)*{
		\labellist
		\small\hair 2pt
		\pinlabel \scalebox{0.7}{$\textcolor{blue}{1}$} at 0 15
		\pinlabel \scalebox{0.7}{$\textcolor{red}{2}$} at 8 15
		\pinlabel \scalebox{0.7}{$\textcolor{red}{2}$} at 24 15
		\pinlabel \scalebox{0.7}{$\textcolor{blue}{1}$} at 32 15
            \pinlabel \scalebox{0.9}{$\lambda$} at 28 2  
		\endlabellist
		\centering
		\includegraphics[scale=1.3]{./BrCup3RrCup1}
}\endxy \;\; -k_1^0(\lambda+\alpha_2)\xy (0,0)*{
			\labellist
			\small\hair 2pt
			\pinlabel \scalebox{0.7}{$\textcolor{blue}{1}$} at 8 15
			\pinlabel \scalebox{0.7}{$\textcolor{red}{2}$} at 0 15
			\pinlabel \scalebox{0.7}{$\textcolor{red}{2}$} at 32 15
			\pinlabel \scalebox{0.7}{$\textcolor{blue}{1}$} at 24 15
                \pinlabel \scalebox{0.9}{$\lambda$} at 28 2   
			\endlabellist
			\centering
			\includegraphics[scale=1.3]{./RrCup3BrCup1}
	}\endxy
    \right)\\ \nn
    & \\ \nn
    & = (-1)^{\lambda_3+1}\left(-k_1^0(\lambda)^2\;\xy (0,0)*{
		\labellist
		\small\hair 2pt
		\pinlabel \scalebox{0.7}{$\textcolor{blue}{1}$} at 0 15
		\pinlabel \scalebox{0.7}{$\textcolor{red}{2}$} at 8 15
		\pinlabel \scalebox{0.7}{$\textcolor{red}{2}$} at 24 15
		\pinlabel \scalebox{0.7}{$\textcolor{blue}{1}$} at 32 15
            \pinlabel \scalebox{0.9}{$\lambda$} at 28 2  
		\endlabellist
		\centering
		\includegraphics[scale=1.3]{./BrCup3RrCup1}
}\endxy\;\; +k_1^0(\lambda+\alpha_2)^2\;\xy (0,0)*{
			\labellist
			\small\hair 2pt
			\pinlabel \scalebox{0.7}{$\textcolor{blue}{1}$} at 8 15
			\pinlabel \scalebox{0.7}{$\textcolor{red}{2}$} at 0 15
			\pinlabel \scalebox{0.7}{$\textcolor{red}{2}$} at 32 15
			\pinlabel \scalebox{0.7}{$\textcolor{blue}{1}$} at 24 15
                \pinlabel \scalebox{0.9}{$\lambda$} at 28 2   
			\endlabellist
			\centering
			\includegraphics[scale=1.3]{./RrCup3BrCup1}
	}\endxy
    \right)\\ \nn
    & \\ \nn
    & = (-1)^{\lambda_3+1}\left(-\xy (0,0)*{
		\labellist
		\small\hair 2pt
		\pinlabel \scalebox{0.7}{$\textcolor{blue}{1}$} at 0 15
		\pinlabel \scalebox{0.7}{$\textcolor{red}{2}$} at 8 15
		\pinlabel \scalebox{0.7}{$\textcolor{red}{2}$} at 24 15
		\pinlabel \scalebox{0.7}{$\textcolor{blue}{1}$} at 32 15
            \pinlabel \scalebox{0.9}{$\lambda$} at 28 2  
		\endlabellist
		\centering
		\includegraphics[scale=1.3]{./BrCup3RrCup1}
}\endxy\;\; +\xy (0,0)*{
			\labellist
			\small\hair 2pt
			\pinlabel \scalebox{0.7}{$\textcolor{blue}{1}$} at 8 15
			\pinlabel \scalebox{0.7}{$\textcolor{red}{2}$} at 0 15
			\pinlabel \scalebox{0.7}{$\textcolor{red}{2}$} at 32 15
			\pinlabel \scalebox{0.7}{$\textcolor{blue}{1}$} at 24 15
                \pinlabel \scalebox{0.9}{$\lambda$} at 28 2   
			\endlabellist
			\centering
			\includegraphics[scale=1.3]{./RrCup3BrCup1}
	}\endxy
    \right)\\ \nn
    & \\ \nn
    & = \Ev\left(\xy (0,0)*{
	\labellist
	\small\hair 2pt
	\pinlabel \scalebox{0.7}{$\textcolor{black}{3}$} at 1 12
        \pinlabel \scalebox{0.9}{$\lambda$} at 16 0 
	\endlabellist
	\centering
	\includegraphics[scale=1.1]{./KlCup1}
}\endxy\;\right)
        \end{align}\endgroup
    
Recalling that \begin{itemize}[wide,labelindent=0pt]
    \item $\Ev' \left(\xy (0,0)*{
		\labellist
		\small\hair 2pt
		\pinlabel \scalebox{0.7}{$\textcolor{black}{3}$} at 16 -5
		\pinlabel \scalebox{0.7}{$\textcolor{red}{2}$} at 0 -5
       	\pinlabel \scalebox{0.9}{$\lambda$} at 18 9
		\endlabellist
		\centering
		\includegraphics[scale=1]{./RdrKdl}}\endxy\;\right)=
        \left(-k_3^{0,3}k_1^{0,3}\xy (0,0)*{
					\labellist
					\small\hair 2pt
					\pinlabel \scalebox{0.7}{$\textcolor{red}{2}$} at 23 -5
					\pinlabel \scalebox{0.7}{$\textcolor{red}{2}$} at 0 -5
					\pinlabel \scalebox{0.7}{$\textcolor{blue}{1}$} at 12 -5
                 	\pinlabel \scalebox{0.9}{$\lambda$} at 25 8  
					\endlabellist
					\centering
					\includegraphics[scale=1.3]{./RdrrBulRul}
			}\endxy\;\;\;, -k_3^{0,3}k_1^{0,1}\xy (0,0)*{
					\labellist
					\small\hair 2pt
					\pinlabel \scalebox{0.7}{$\textcolor{red}{2}$} at 12 -5
					\pinlabel \scalebox{0.7}{$\textcolor{red}{2}$} at 0 -5
					\pinlabel \scalebox{0.7}{$\textcolor{blue}{1}$} at 23 -5
                 	\pinlabel \scalebox{0.9}{$\lambda$} at 25 8
					\endlabellist
					\centering
                    \includegraphics[scale=1.3]{./RdrrRulBul}
			}\endxy
	\;\;\right)$
\item[]
\item[]
    \item $\Ev' \left(\xy (0,0)*{
		\labellist
		\small\hair 2pt
		\pinlabel \scalebox{0.7}{$\textcolor{black}{3}$} at 0 -5
		\pinlabel \scalebox{0.7}{$\textcolor{red}{2}$} at 16 -5
      	\pinlabel \scalebox{0.9}{$\lambda$} at 18 9
		\endlabellist
		\centering
		\includegraphics[scale=1]{./KdrRdl}}\endxy\;\right) = \left(k_3^{0,3}k_1^{0,3}\left(\xy (0,0)*{
					\labellist
					\small\hair 2pt
					\pinlabel \scalebox{0.7}{$\textcolor{red}{2}$} at 23 -5
					\pinlabel \scalebox{0.7}{$\textcolor{red}{2}$} at 12 -5
					\pinlabel \scalebox{0.7}{$\textcolor{blue}{1}$} at 0 -5
                 	\pinlabel \scalebox{0.9}{$\lambda$} at 25 8
					\endlabellist
					\centering
					\includegraphics[scale=1.3]{./BurRurRodll}
			}\endxy
            \ \;\, -
            \xy (0,0)*{
					\labellist
					\small\hair 2pt
					\pinlabel \scalebox{0.7}{$\textcolor{red}{2}$} at 23 -5
					\pinlabel \scalebox{0.7}{$\textcolor{red}{2}$} at 12 -5
					\pinlabel \scalebox{0.7}{$\textcolor{blue}{1}$} at 0 -5
                 	\pinlabel \scalebox{0.9}{$\lambda$} at 25 8
					\endlabellist
					\centering
					\includegraphics[scale=1.3]{./BurRourRdll}
			}\endxy
            \;\;\right),k_3^{0,3}k_1^{0,1}\left(\xy (0,0)*{
				\labellist
				\small\hair 2pt
				\pinlabel \scalebox{0.7}{$\textcolor{red}{2}$} at 23 -5
				\pinlabel \scalebox{0.7}{$\textcolor{red}{2}$} at 0 -5
				\pinlabel \scalebox{0.7}{$\textcolor{blue}{1}$} at 12 -5
               	\pinlabel \scalebox{0.9}{$\lambda$} at 25 8
				\endlabellist
				\centering
				\includegraphics[scale=1.3]{./RurBurRodll}
	}\endxy\ \;\, - 
    \xy (0,0)*{
				\labellist
				\small\hair 2pt
				\pinlabel \scalebox{0.7}{$\textcolor{red}{2}$} at 23 -5
				\pinlabel \scalebox{0.7}{$\textcolor{red}{2}$} at 0 -5
				\pinlabel \scalebox{0.7}{$\textcolor{blue}{1}$} at 12 -5
               	\pinlabel \scalebox{0.9}{$\lambda$} at 25 8
				\endlabellist
				\centering
				\includegraphics[scale=1.3]{./RourBurRdll}
	}\endxy
    \;\;
    \right)
        \right)$,
        \item[]
        \item[]
\end{itemize}

\noindent It is straightforward to see that $\gamma\Ev'\delta\gamma\left(\xy (0,0)*{
		\labellist
		\small\hair 2pt
		\pinlabel \scalebox{0.7}{$\textcolor{black}{3}$} at 16 -5
		\pinlabel \scalebox{0.7}{$\textcolor{red}{2}$} at 0 -5
       	\pinlabel \scalebox{0.9}{$\lambda$} at 18 9
		\endlabellist
		\centering
		\includegraphics[scale=1]{./RdrKdl}}\endxy\;\right) = \Ev\left(\xy (0,0)*{
		\labellist
		\small\hair 2pt
		\pinlabel \scalebox{0.7}{$\textcolor{black}{3}$} at 16 -5
		\pinlabel \scalebox{0.7}{$\textcolor{red}{2}$} at 0 -5
       	\pinlabel \scalebox{0.9}{$\lambda$} at 18 9
		\endlabellist
		\centering
		\includegraphics[scale=1]{./RdrKdl}}\endxy\;\right)$ and $\gamma\Ev'\delta\gamma\left(\xy (0,0)*{
		\labellist
		\small\hair 2pt
		\pinlabel \scalebox{0.7}{$\textcolor{black}{3}$} at 0 -5
		\pinlabel \scalebox{0.7}{$\textcolor{red}{2}$} at 16 -5
       	\pinlabel \scalebox{0.9}{$\lambda$} at 18 9
		\endlabellist
		\centering
		\includegraphics[scale=1]{./KdrRdl}}\endxy\;\right) = \Ev\left(\xy (0,0)*{
		\labellist
		\small\hair 2pt
		\pinlabel \scalebox{0.7}{$\textcolor{black}{3}$} at 0 -5
		\pinlabel \scalebox{0.7}{$\textcolor{red}{2}$} at 16 -5
       	\pinlabel \scalebox{0.9}{$\lambda$} at 18 9
		\endlabellist
		\centering
		\includegraphics[scale=1]{./KdrRdl}}\endxy\;\right)$. \end{proof}\eqskip

This shows that \autoref{thm:BigThmPrime} implies \autoref{thm:BigThmMain}, and so it suffices to prove the former.

\subsection{Essential uniquenss of the image of the dotted 3-strand}

Finally, let us show that the above choice for the image of a dotted $3$-strand under $\Ev$ is the only one possible, up to multiplication by a scalar. 
We will be using this on occasion in the proof of \autoref{thm:BigThmPrime} and elsewhere. 
\begin{lem}\label{lem:DotShiftBase} Both  $\End_{\cK^b(\naffu{3})}^{*}(\Ev (\E_3\oneid_\lambda))$  and $\End_{\cK^b(\naffu{3})}^{*}(\Ev (\F_3\oneid_\lambda))$ are isomorphic to $\bQ[x]$, where $\deg x=2$.
	\begin{proof} We only prove the result for 
 $\End_{\cK^b(\naffu{3})}^{*}(\Ev (\E_3\oneid_\lambda))$, the 
 proof of the other case being similar. For starters, let us work in the $2$-category of bounded complexes $\cC^b(\naffu{3})$. We claim that  
\[
\End_{\cC^b(\naffu{3})}^{*}(\Ev (\E_3\oneid_\lambda))\cong \bQ[x_1,x_2],
\]
where $\deg x_1=\deg x_2=2$. Note that an element of $\End_{\cC^b(\naffu{3})}^{*}(\Ev (\E_3\oneid_\lambda))$ is a commutative square of 
 the form \newline
\[
 \xymatrix { \underline{\F_{12}\oneid_\lambda}\langle S+r\rangle \ar[rr]^{\xy (0,4)*{
	\labellist
	\small\hair 2pt
	\pinlabel \scalebox{0.7}{$\textcolor{red}{2}$} at 16 -5
	\pinlabel \scalebox{0.7}{$\textcolor{blue}{1}$} at 0 -5
        \pinlabel \scalebox{0.9}{$\lambda$} at 18 8       
        \endlabellist
	\centering
	\includegraphics[scale=1.15]{./BdrRdl}}\endxy} & & \F_{21}\oneid_\lambda\langle S+r+1\rangle \\ \underline{\F_{12}\oneid_\lambda}\langle S\rangle \ar[u]^{f_0} \ar[rr]_{\xy (0,0)*{
	\labellist
	\small\hair 2pt
	\pinlabel \scalebox{0.7}{$\textcolor{red}{2}$} at 16 -5
	\pinlabel \scalebox{0.7}{$\textcolor{blue}{1}$} at 0 -5
        \pinlabel \scalebox{0.9}{$\lambda$} at 18 8       
	\endlabellist
	\centering
	\includegraphics[scale=1.15]{./BdrRdl}}\endxy} & & \F_{21}\oneid_\lambda\langle S+1\rangle. \ar[u]_{f_1}}
\]
\eqskip By \cite[Theorem 2.7]{Kh-L}, the shift $r$ has to be even and $f_0$ is a linear combination of 2-morphisms of the form $g_0=\mspace{15mu}\xy (0,0)*{\xy (0,0)*{
	\labellist
	\small\hair 2pt
        \pinlabel \scalebox{0.9}{$\lambda$} at 18 12       
 \pinlabel \scalebox{0.7}{$\textcolor{red}{2}$} at 9 -5
	\pinlabel \scalebox{0.7}{$\textcolor{blue}{1}$} at 1 -5
 \pinlabel \scalebox{1}{$\textcolor{red}{b}$} at 13 8
	\pinlabel \scalebox{1}{$\textcolor{blue}{a}$} at -5 8
	\endlabellist
	\centering
	\includegraphics[scale=1.15]{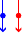}}\endxy}\endxy\mspace{12mu}\;\;$, where $a+b=r/2$.

\noindent The equality 
	\begin{equation*}\xy (0,0)*{
	\labellist
	\small\hair 2pt
	\pinlabel \scalebox{0.7}{$\textcolor{red}{2}$} at 20 -5
	\pinlabel \scalebox{0.7}{$\textcolor{blue}{1}$} at 4 -5
        \pinlabel \scalebox{1}{$f_1$} at 12 26.5
        \pinlabel \scalebox{0.9}{$\lambda$} at 28 24       
        \endlabellist
	\centering
	\includegraphics[scale=1.15]{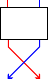}}\endxy
\mspace{14mu} = \ 
\xy (0,0)*{
	\labellist
	\small\hair 2pt
	\pinlabel \scalebox{0.7}{$\textcolor{red}{2}$} at 20 -5
	\pinlabel \scalebox{0.7}{$\textcolor{blue}{1}$} at 4 -5
        \pinlabel \scalebox{1}{$f_0$} at 12 11.5        
        \pinlabel \scalebox{0.9}{$\lambda$} at 26 24       
        \endlabellist
	\centering
	\includegraphics[scale=1.15]{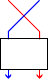}}\endxy
\end{equation*} 
\eqskip implies that, for each summand $g_0$ of $f_0$, there is a corresponding summand $g_1$ of $f_1$ that is determined by the choice of $g_0$, i.e.      
$(g_0,g_1)=\left(\mspace{10mu} \xy (0,1)*{\xy (0,0)*{
	\labellist
	\small\hair 2pt
	\pinlabel \scalebox{0.7}{$\textcolor{red}{2}$} at 9 -5
	\pinlabel \scalebox{0.7}{$\textcolor{blue}{1}$} at 1 -5
 \pinlabel \scalebox{1}{$\textcolor{red}{b}$} at 13 8
	\pinlabel \scalebox{1}{$\textcolor{blue}{a}$} at -5 8
        \pinlabel \scalebox{0.9}{$\lambda$} at 18 12      
 \endlabellist
	\centering
	\includegraphics[scale=1.15]{./BdoRdo}}\endxy}\endxy \mspace{18mu},\mspace{14mu}
 \xy (0,1)*{\xy (0,0)*{
	\labellist
	\small\hair 2pt
	\pinlabel \scalebox{0.7}{$\textcolor{red}{2}$} at 1 -5
	\pinlabel \scalebox{0.7}{$\textcolor{blue}{1}$} at 9 -5
 \pinlabel \scalebox{1}{$\textcolor{red}{b}$} at -5 8
	\pinlabel \scalebox{1}{$\textcolor{blue}{a}$} at 13 8
        \pinlabel \scalebox{0.9}{$\lambda$} at 18 12      
 \endlabellist
	\centering
	\includegraphics[scale=1.15]{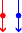}}\endxy}\endxy\mspace{18mu}\right)\vspace{2mm}$, 
 where we use the presentation used in \cite{abram2022categorification} of only giving the vertical 2-morphisms as an ordered pair, for clarity of reading.

\noindent This proves the claim, with 
\[
x_1=\left(\mspace{4mu} \xy (0,0)*{\xy (0,0)*{
	\labellist
	\small\hair 2pt
	\pinlabel \scalebox{0.7}{$\textcolor{red}{2}$} at 9 -5
	\pinlabel \scalebox{0.7}{$\textcolor{blue}{1}$} at 1 -5
       \pinlabel \scalebox{0.9}{$\lambda$} at 14 12      
	\endlabellist
	\centering
	\includegraphics[scale=1.15]{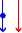}}\endxy}\endxy
 \mspace{14mu},\mspace{6mu} 
\xy (0,0)*{\xy (0,0)*{
	\labellist
	\small\hair 2pt
	\pinlabel \scalebox{0.7}{$\textcolor{red}{2}$} at 1 -5
	\pinlabel \scalebox{0.7}{$\textcolor{blue}{1}$} at 9 -5
        \pinlabel \scalebox{0.9}{$\lambda$} at 15 12   
	\endlabellist
	\centering
	\includegraphics[scale=1.15]{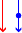}}\endxy}\endxy\mspace{11mu}
\right)
\, , \quad 
x_2=\left(\, \xy (0,0)*{\xy (0,0)*{
	\labellist
	\small\hair 2pt
	\pinlabel \scalebox{0.7}{$\textcolor{red}{2}$} at 9 -5
	\pinlabel \scalebox{0.7}{$\textcolor{blue}{1}$} at 1 -5
        \pinlabel \scalebox{0.9}{$\lambda$} at 15 12   
	\endlabellist
	\centering
	\includegraphics[scale=1.15]{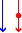}}\endxy}\endxy
 \mspace{14mu},\mspace{6mu}
 \xy (0,0)*{\xy (0,0)*{
	\labellist
	\small\hair 2pt
	\pinlabel \scalebox{0.7}{$\textcolor{red}{2}$} at 1 -5
	\pinlabel \scalebox{0.7}{$\textcolor{blue}{1}$} at 9 -5
        \pinlabel \scalebox{0.9}{$\lambda$} at 14 12   
	\endlabellist
	\centering
	\includegraphics[scale=1.15]{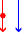}}\endxy}\endxy
\mspace{10mu}\right).
\]
\eqskip
We further claim that $x_1$ and $x_2$ are homotopic. Indeed, consider the diagram
\begin{equation}\label{eq:dotshift}
\xymatrix {\underline{\F_{12} \oneid_\lambda}\langle S+2\rangle \ar[rr]^{\xy (0,4)*{
	\labellist
	\small\hair 2pt
	\pinlabel \scalebox{0.7}{$\textcolor{red}{2}$} at 16 -5
	\pinlabel \scalebox{0.7}{$\textcolor{blue}{1}$} at 0 -5
        \pinlabel \scalebox{0.9}{$\lambda$} at 18 8       
        \endlabellist
	\centering
	\includegraphics[scale=1.15]{./BdrRdl}}\endxy} & & \F_{21} \oneid_\lambda\langle S+3\rangle \\ & & \\ \underline{\F_{12} \oneid_\lambda}\langle S\rangle \ar[uu]^{\xy (0,0)*{
	\labellist
	\small\hair 2pt
	\pinlabel \scalebox{0.7}{$\textcolor{red}{2}$} at 9 -5
	\pinlabel \scalebox{0.7}{$\textcolor{blue}{1}$} at 1 -5
        \pinlabel \scalebox{0.9}{$\lambda$} at 14 8       
        \endlabellist
	\centering
	\includegraphics[scale=1.15]{./BdRdo}}\endxy
    \mspace{18mu} - \
		\xy (0,0)*{
	\labellist
	\small\hair 2pt
	\pinlabel \scalebox{0.7}{$\textcolor{red}{2}$} at 9 -5
	\pinlabel \scalebox{0.7}{$\textcolor{blue}{1}$} at 1 -5
        \pinlabel \scalebox{0.9}{$\lambda$} at 14 8       
	\endlabellist
	\centering
	\includegraphics[scale=1.15]{./BdoRd}}\endxy
\mspace{10mu}} \ar[rr]_{\xy (0,2)*{
	\labellist
	\small\hair 2pt
	\pinlabel \scalebox{0.7}{$\textcolor{red}{2}$} at 16 -5
	\pinlabel \scalebox{0.7}{$\textcolor{blue}{1}$} at 0 -5
        \pinlabel \scalebox{0.9}{$\lambda$} at 18 8       
 \endlabellist
	\centering
	\includegraphics[scale=1.15]{./BdrRdl}}\endxy} & & \F_{21}\oneid_\lambda\langle S+1\rangle. \ar[uu]_{\xy (0,0)*{
	\labellist
	\small\hair 2pt
	\pinlabel \scalebox{0.7}{$\textcolor{red}{2}$} at 1 -5
	\pinlabel \scalebox{0.7}{$\textcolor{blue}{1}$} at 9 -5
        \pinlabel \scalebox{0.9}{$\lambda$} at 14 8       
	\endlabellist
	\centering
	\includegraphics[scale=1.15]{./RdoBd}}\endxy
    \mspace{18mu} - \
   	\xy (0,0)*{
	\labellist
	\small\hair 2pt
	\pinlabel \scalebox{0.7}{$\textcolor{red}{2}$} at 1 -5
	\pinlabel \scalebox{0.7}{$\textcolor{blue}{1}$} at 9 -5
        \pinlabel \scalebox{0.9}{$\lambda$} at 14 8       
	\endlabellist
	\centering
	\includegraphics[scale=1.15]{./RdBdo}}\endxy}
		\ar[lluu]_{\xy (0,4)*{
	\labellist
	\small\hair 2pt
	\pinlabel \scalebox{0.7}{$\textcolor{red}{2}$} at 0 -5
	\pinlabel \scalebox{0.7}{$\textcolor{blue}{1}$} at 16 -5
        \pinlabel \scalebox{0.9}{$\lambda$} at 18 8       
        \endlabellist
	\centering
	\includegraphics[scale=1.15]{./RdrBdl}}\endxy}}
\end{equation}	
One sees that this diagram is commutative, by the downward version of relation \eqref{eq:R2a}, and hence $x_1-x_2\simeq_h 0$, which proves the lemma. 
\end{proof}
\end{lem}

A directly analogous result and proof hold for $\Ev'$.

\section{Categorified braid group action}\label{sec:BraidAction}
To categorify the connection between our desired evaluation functor and Lusztig's algebra automorphisms $T_{1,-1}'$ and $T_{2,1}''$, discussed in \autoref{sec:DecatT}, we need to introduce various 2-functors to deal with some complications. While the automorphisms have already been categorified in \cite{abram2022categorification}, that paper works over $\mathfrak{sl}_3$ and does not cover our choice of scalars and bubble parameters. We therefore adapt their constructions to our setup through composition with 2-isomorphisms. 

\subsection{The braid group actions}
Denote by $\taffu{3}$ the $\mathfrak{gl}_3$ version of the (unsigned version of the) 2-category $\mathcal{U}_Q(\mathfrak{sl}_3)$ defined in \cite[Definition 3.3]{abram2022categorification} with the trivial choice of scalars and bubble parameters. For this section, we will be utilising the 2-functors $\mathcal{T}_{1,1}'$ and $\mathcal{T}_{2,1}''$ as defined in \cite[Section 4]{abram2022categorification}, as well as the 2-isomorphisms $\omega$ and $\psi$ defined in \cite{Kh-L-2} (and generalised in \cite[Section 3.5]{abram2022categorification}, though we do not use the more general setting), the latter of which we recall here. The 2-isomorphism $\omega:\taffu{3}\to\taffu{3}$  is 1- and 2-covariant and degree-preserving, and sends a weight $\lambda$ to $-\lambda$, reverses the orientation of 2-morphisms, and scales the $1,1$- and $2,2$-crossings by a factor of $-1$. Similarly, $\psi:\taffu{3}\to\taffu{3}^{co}$ is a 1-covariant, 2-contravariant 2-isomorphism that is the identity on objects, scales weights of 1-morphisms by a factor of $-1$, and reflects diagrams of 2-morphisms in the horizontal axis and then reverses their orientation. We also remind the reader that $k_i^{a_1,\dots,a_n}(\mu)$, defined in \eqref{eq:kdef}, omits the argument $\mu$ when it is equal to $\lambda$, but retains it otherwise (generally when it is $s_1(\lambda)$ or $s_2(\lambda)$).

We also use the 2-isomorphism $\zeta:\taffu{3}\to\naffu{3}$, first defined as $\Sigma$ in \cite[Section 4.2]{Kh-L} and \cite{Kh-L-err}, which is the identity on objects and 1-morphisms, and the identity on 2-morphisms except for the following generating 2-morphisms (and hence the 2-morphisms derived from them): 
\begin{align}
\xy (0,1)*{
			\labellist
			\small\hair 2pt
			\pinlabel \scalebox{0.7}{$\textcolor{blue}{1}$} at 1 -5
			\pinlabel \scalebox{0.9}{$\lambda$} at 6 8
			\endlabellist 
			\centering 
			\includegraphics[scale=1.3]{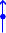}
	}\endxy\;\;&\overset{\zeta}{\mapsto}
    -\,\xy (0,0)*{
			\labellist
			\small\hair 2pt
			\pinlabel \scalebox{0.7}{$\textcolor{blue}{1}$} at 1 -5
			\pinlabel \scalebox{0.9}{$\lambda$} at 6 8
			\endlabellist 
			\centering 
			\includegraphics[scale=1.3]{./Buo}
	}\endxy 
&
\xy (0,1)*{
			\labellist
			\small\hair 2pt
			\pinlabel \scalebox{0.7}{$\textcolor{blue}{1}$} at 0 -5
			\pinlabel \scalebox{0.7}{$\textcolor{blue}{1}$} at 16 -5
			\pinlabel \scalebox{0.9}{$\lambda$} at 18 8
			\endlabellist 
			\centering 
			\includegraphics[scale=1.3]{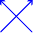}
	}\endxy\;\;&\overset{\zeta}{\mapsto}
    -\xy (0,0)*{
			\labellist
			\small\hair 2pt
			\pinlabel \scalebox{0.7}{$\textcolor{blue}{1}$} at 0 -5
			\pinlabel \scalebox{0.7}{$\textcolor{blue}{1}$} at 16 -5
			\pinlabel \scalebox{0.9}{$\lambda$} at 18 8
			\endlabellist 
			\centering 
			\includegraphics[scale=1.3]{./BurBul}
	}\endxy
\\[2ex]
\xy (0,0)*{
			\labellist
			\small\hair 2pt
			\pinlabel \scalebox{0.7}{$\textcolor{blue}{1}$} at 1 -5
			\pinlabel \scalebox{0.9}{$\lambda$} at 18 8
			\endlabellist 
			\centering 
			\includegraphics[scale=1.3]{./BlCap1}
	}\endxy\;\;&\overset{\zeta}{\mapsto}
    (-1)^{\lambda_1+1}k_1^2\;\;
     \xy (0,0)*{
			\labellist
			\small\hair 2pt
			\pinlabel \scalebox{0.7}{$\textcolor{blue}{1}$} at 1 -5
			\pinlabel \scalebox{0.9}{$\lambda$} at 18 8
			\endlabellist 
			\centering 
			\includegraphics[scale=1.3]{./BlCap1}
	}\endxy
    &\mspace{60mu}
\xy (0,0)*{
			\labellist
			\small\hair 2pt
			\pinlabel \scalebox{0.7}{$\textcolor{blue}{1}$} at 1 10
			\pinlabel \scalebox{0.9}{$\lambda$} at 17 1
			\endlabellist 
			\centering 
			\includegraphics[scale=1.3]{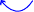}
	}\endxy\;\;&\overset{\zeta}{\mapsto}
    (-1)^{\lambda_1}k_1^0\;\;
    \xy (0,0)*{
			\labellist
			\small\hair 2pt
			\pinlabel \scalebox{0.7}{$\textcolor{blue}{1}$} at 1 10
			\pinlabel \scalebox{0.9}{$\lambda$} at 17 1
			\endlabellist 
			\centering 
			\includegraphics[scale=1.3]{./BlCup1Thin}
	}\endxy    
\\[2ex]
   \xy (0,0)*{
			\labellist
			\small\hair 2pt
			\pinlabel \scalebox{0.7}{$\textcolor{blue}{1}$} at 1 -5
			\pinlabel \scalebox{0.9}{$\lambda$} at 18 8
			\endlabellist 
			\centering 
			\includegraphics[scale=1.3]{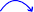}
	}\endxy\;\;&\overset{\zeta}{\mapsto}
    -k_1^0\;\;
    \xy (0,0)*{
			\labellist
			\small\hair 2pt
			\pinlabel \scalebox{0.7}{$\textcolor{blue}{1}$} at 1 -5
			\pinlabel \scalebox{0.9}{$\lambda$} at 18 8
			\endlabellist 
			\centering 
			\includegraphics[scale=1.3]{./BrCap1Thin}
	}\endxy
    &
 \xy (0,0)*{
			\labellist
			\small\hair 2pt
			\pinlabel \scalebox{0.7}{$\textcolor{blue}{1}$} at 1 10
			\pinlabel \scalebox{0.9}{$\lambda$} at 17 1
			\endlabellist 
			\centering 
			\includegraphics[scale=1.3]{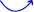}
	}\endxy\;\;&\overset{\zeta}{\mapsto}
    -k_1^2\;\;
    \xy (0,0)*{
			\labellist
			\small\hair 2pt
			\pinlabel \scalebox{0.7}{$\textcolor{blue}{1}$} at 1 10
			\pinlabel \scalebox{0.9}{$\lambda$} at 17 1
			\endlabellist 
			\centering 
			\includegraphics[scale=1.3]{./BrCup1Thin}
	}\endxy   
    \\[2ex]
\xy (0,0)*{
			\labellist
			\small\hair 2pt
			\pinlabel \scalebox{0.7}{$\textcolor{red}{2}$} at 1 -5
			\pinlabel \scalebox{0.9}{$\lambda$} at 18 8
			\endlabellist 
			\centering 
			\includegraphics[scale=1.3]{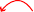}
	}\endxy\;\;&\overset{\zeta}{\mapsto}
    (-1)^{\lambda_3-1}\;\;
    \xy (0,0)*{
			\labellist
			\small\hair 2pt
			\pinlabel \scalebox{0.7}{$\textcolor{red}{2}$} at 1 -5
			\pinlabel \scalebox{0.9}{$\lambda$} at 18 8
			\endlabellist 
			\centering 
			\includegraphics[scale=1.3]{./RlCap1}
	}\endxy
    &
\xy (0,0)*{
			\labellist
			\small\hair 2pt
			\pinlabel \scalebox{0.7}{$\textcolor{red}{2}$} at 1 10
			\pinlabel \scalebox{0.9}{$\lambda$} at 17 1
			\endlabellist 
			\centering 
			\includegraphics[scale=1.3]{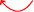}
	}\endxy\;\;&\overset{\zeta}{\mapsto}
    (-1)^{\lambda_3}\;\;
    \xy (0,0)*{
			\labellist
			\small\hair 2pt
			\pinlabel \scalebox{0.7}{$\textcolor{red}{2}$} at 1 10
			\pinlabel \scalebox{0.9}{$\lambda$} at 17 1
			\endlabellist 
			\centering 
			\includegraphics[scale=1.3]{./RlCup1Thin}
	}\endxy    
\end{align}

\eqskip We now define two 2-functors $\tilde{\mathcal{T}}'_{1,-1},\tilde{\mathcal{T}}_{2,1}'':\naffu{3}\to\cK^b(\naffu{3})$ using composites of the above 2-functors. 
Specifically, 
\begin{itemize}[wide,labelindent=0pt,itemsep=5pt]
    \item $\tilde{\mathcal{T}}_{1,-1}':=\zeta\psi\mathcal{T}_{1,1}'\psi\zeta^{-1}=\zeta\mathcal{T}_{1,-1}'\zeta^{-1}$;
    \item $\tilde{\mathcal{T}}_{2,1}'':=\zeta\omega\mathcal{T}_{2,1}'\omega\zeta^{-1}=\zeta\mathcal{T}_{2,1}''\zeta^{-1}$.
\end{itemize}
We let $X[y]\langle z\rangle$ denote the 1-term complex with the 1-morphism at homological degree $-y$ with internal degree shift of $z$. In detail, $\tilde{\mathcal{T}}_{1,-1}'$ acts as follows:
\begin{itemize}[wide,labelindent=0pt,itemsep=5pt]
    \item On objects, $\lambda\xmapsto{\tilde{\mathcal{T}}_{1,-1}'}s_1(\lambda)$.
    \item On 1-morphisms, 

    \begin{gather*}
\E_1 \oneid_\lambda\xmapsto{\tilde{\mathcal{T}}_{1,-1}'}  \F_1\oneid_{s_1(\lambda)}[-1]\langle 2+\bar{\lambda}_1\rangle
\mspace{60mu}
\F_1\oneid_\lambda\xmapsto{\tilde{\mathcal{T}}_{1,-1}'} \E_1\oneid_{s_1(\lambda)}[1]\langle -\bar{\lambda}_1\rangle
\\[1ex]
\E_2\oneid_\lambda \xmapsto{\tilde{\mathcal{T}}_{1,-1}'} 
\Bigg(\xymatrix{\E_{12}\oneid_{s_1(\lambda)}\langle -1\rangle \ar[rr]^(0.55){\xy (0,4.5)*{
		\labellist
		\small\hair 2pt
		\pinlabel \scalebox{0.7}{$\textcolor{red}{2}$} at 16 -5
		\pinlabel \scalebox{0.7}{$\textcolor{blue}{1}$} at 0 -5
		\pinlabel \scalebox{0.9}{$\lambda$} at 16 7
		\endlabellist
		\centering
		\includegraphics[scale=1.3]{./BurRul}
	}\endxy} & & \underline{\E_{21}\oneid_{s_1(\lambda)}}}\Bigg)
\\[1ex]
\F_2\oneid_\lambda \xmapsto{\tilde{\mathcal{T}}_{1,-1}'} 
\Bigg(\xymatrix{\underline{\F_{12}\oneid_{s_1(\lambda)}} \ar[rr]^(0.55){\xy (0,4.5)*{
		\labellist
		\small\hair 2pt
		\pinlabel \scalebox{0.7}{$\textcolor{red}{2}$} at 16 -5
		\pinlabel \scalebox{0.7}{$\textcolor{blue}{1}$} at 0 -5
		\pinlabel \scalebox{0.9}{$\lambda$} at 16 7
		\endlabellist
		\centering
		\includegraphics[scale=1.3]{./BdrRdl}
	}\endxy\;\;\;\;\;} & & \F_{21}\oneid_{s_1(\lambda)}\langle 1\rangle} \Bigg)
\end{gather*}
 
     \item On non-identity generating 2-morphisms:
\begingroup\allowdisplaybreaks
\begin{gather}
    \xy (0,1)*{
			\labellist
			\small\hair 2pt
			\pinlabel \scalebox{0.7}{$\textcolor{blue}{1}$} at 1 -5
			\pinlabel \scalebox{0.9}{$\lambda$} at 6 8
			\endlabellist 
			\centering 
			\includegraphics[scale=1.3]{./Buo}
	}\endxy\;\;\xmapsto{\tilde{\mathcal{T}}_{1,-1}'}
    \,\xy (0,0)*{
			\labellist
			\small\hair 2pt
			\pinlabel \scalebox{0.7}{$\textcolor{blue}{1}$} at 1 -5
			\pinlabel \scalebox{0.9}{$s_1(\lambda)$} at 12 8
			\endlabellist 
			\centering 
			\includegraphics[scale=1.3]{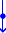}
	}\endxy
\mspace{180mu}
        \xy (0,0)*{
			\labellist
			\small\hair 2pt
			\pinlabel \scalebox{0.7}{$\textcolor{blue}{1}$} at 0 -5
			\pinlabel \scalebox{0.7}{$\textcolor{blue}{1}$} at 16 -5
			\pinlabel \scalebox{0.9}{$\lambda$} at 18 8
			\endlabellist 
			\centering 
			\includegraphics[scale=1.3]{./BurBul}
	}\endxy\;\;\xmapsto{\tilde{\mathcal{T}}_{1,-1}'}
    -\xy (0,0)*{
			\labellist
			\small\hair 2pt
			\pinlabel \scalebox{0.7}{$\textcolor{blue}{1}$} at 0 -5
			\pinlabel \scalebox{0.7}{$\textcolor{blue}{1}$} at 16 -5
			\pinlabel \scalebox{0.9}{$s_1(\lambda)$} at 24 8
			\endlabellist 
			\centering 
			\includegraphics[scale=1.3]{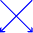}
	}\endxy
\\[2ex]
\xy (0,1)*{
			\labellist
			\small\hair 2pt
			\pinlabel \scalebox{0.7}{$\textcolor{red}{2}$} at 1 -5
			\pinlabel \scalebox{0.9}{$\lambda$} at 6 8
			\endlabellist 
			\centering 
			\includegraphics[scale=1.3]{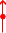}
	}\endxy\;\;\xmapsto{\tilde{\mathcal{T}}_{1,-1}'}
    \,\left(\xy (0,0)*{
		\labellist
		\small\hair 2pt
		\pinlabel \scalebox{0.7}{$\textcolor{blue}{1}$} at 1 -5
		\pinlabel \scalebox{0.7}{$\textcolor{red}{2}$} at 9 -5
       	\pinlabel \scalebox{0.9}{$s_1(\lambda)$} at 20 8
		\endlabellist
		\centering
		\includegraphics[scale=1.3]{./BuRuo}
	}\endxy
    \quad\quad,\xy (0,0)*{
		\labellist
		\small\hair 2pt
		\pinlabel \scalebox{0.7}{$\textcolor{blue}{1}$} at 9 -5
		\pinlabel \scalebox{0.7}{$\textcolor{red}{2}$} at 1 -5
       	\pinlabel \scalebox{0.9}{$s_1(\lambda)$} at 20 8
		\endlabellist
		\centering
		\includegraphics[scale=1.3]{./RuoBu}
	}\endxy\quad\quad
    \right)
\end{gather}\endgroup
%
%
\begin{multline}
\xy (0,1)*{
			\labellist
			\small\hair 2pt
			\pinlabel \scalebox{0.7}{$\textcolor{red}{2}$} at 0 -5
			\pinlabel \scalebox{0.7}{$\textcolor{red}{2}$} at 16 -5
			\pinlabel \scalebox{0.9}{$\lambda$} at 18 8
			\endlabellist 
			\centering 
			\includegraphics[scale=1.3]{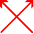}
	}\endxy\;\;\xmapsto{\tilde{\mathcal{T}}_{1,-1}'}
    \\[2ex]
   \Big(\xy (0,0)*{
			\labellist
			\small\hair 2pt 
			\pinlabel \scalebox{0.55}{$\textcolor{blue}{1}$} at 0 -5
			\pinlabel \scalebox{0.55}{$\textcolor{red}{2}$} at 8 -5
			\pinlabel \scalebox{0.55}{$\textcolor{red}{2}$} at 32 -5
			\pinlabel \scalebox{0.55}{$\textcolor{blue}{1}$} at 24 -5
                \pinlabel \scalebox{0.9}{$s_1(\lambda)$} at 40 8   
			\endlabellist
			\centering
			\includegraphics[scale=1]{./BurrRurrBullRull}
	}\endxy
    \quad\quad, 
    -\xy (0,0)*{
			\labellist
			\small\hair 2pt 
			\pinlabel \scalebox{0.55}{$\textcolor{blue}{1}$} at 8 -5
			\pinlabel \scalebox{0.55}{$\textcolor{red}{2}$} at 0 -5
			\pinlabel \scalebox{0.55}{$\textcolor{red}{2}$} at 32 -5
			\pinlabel \scalebox{0.55}{$\textcolor{blue}{1}$} at 24 -5
                \pinlabel \scalebox{0.9}{$s_1(\lambda)$} at 40 8   
			\endlabellist
			\centering
			\includegraphics[scale=1]{./RurrBurrBullRull}
	}\endxy
    \quad\quad -\xy (0,0)*{
			\labellist
			\small\hair 2pt 
			\pinlabel \scalebox{0.55}{$\textcolor{blue}{1}$} at 0 -5
			\pinlabel \scalebox{0.55}{$\textcolor{red}{2}$} at 8 -5
			\pinlabel \scalebox{0.55}{$\textcolor{red}{2}$} at 24 -5
			\pinlabel \scalebox{0.55}{$\textcolor{blue}{1}$} at 32 -5
                \pinlabel \scalebox{0.9}{$s_1(\lambda)$} at 40 8   
			\endlabellist
			\centering
			\includegraphics[scale=1]{./BurrRurrRullBull}
	}\endxy
    \quad\quad + \xy (0,0)*{
				\labellist
				\small\hair 2pt
				\pinlabel \scalebox{0.55}{$\textcolor{blue}{1}$} at 8 -5
				\pinlabel \scalebox{0.55}{$\textcolor{red}{2}$} at 1 -5
				\pinlabel \scalebox{0.55}{$\textcolor{red}{2}$} at 31 -5
				\pinlabel \scalebox{0.55}{$\textcolor{blue}{1}$} at 24 -5
                    \pinlabel \scalebox{0.9}{$s_1(\lambda)$} at 44 8    
				\endlabellist
				\centering
				\includegraphics[scale=1]{./RuBurBulRu}
		}\endxy
        \quad\quad + \xy (0,0)*{
				\labellist
				\small\hair 2pt
				\pinlabel \scalebox{0.55}{$\textcolor{blue}{1}$} at 1 -5
				\pinlabel \scalebox{0.55}{$\textcolor{red}{2}$} at 8 -5
				\pinlabel \scalebox{0.55}{$\textcolor{red}{2}$} at 24 -5
				\pinlabel \scalebox{0.55}{$\textcolor{blue}{1}$} at 31 -5
                    \pinlabel \scalebox{0.9}{$s_1(\lambda)$} at 44 8    
				\endlabellist
				\centering
				\includegraphics[scale=1]{./BuRurRulBu}
		}\endxy
        \quad\quad, -\xy (0,0)*{
			\labellist
			\small\hair 2pt 
			\pinlabel \scalebox{0.55}{$\textcolor{blue}{1}$} at 8 -5
			\pinlabel \scalebox{0.55}{$\textcolor{red}{2}$} at 0 -5
			\pinlabel \scalebox{0.55}{$\textcolor{red}{2}$} at 24 -5
			\pinlabel \scalebox{0.55}{$\textcolor{blue}{1}$} at 32 -5
                \pinlabel \scalebox{0.9}{$s_1(\lambda)$} at 40 8   
			\endlabellist
			\centering
			\includegraphics[scale=1]{./RurrBurrRullBull}
	}\endxy\quad\;\;
    \Big)
\end{multline}
%
%
\begingroup\allowdisplaybreaks
\begin{gather}
    \xy (0,1)*{
			\labellist
			\small\hair 2pt
			\pinlabel \scalebox{0.7}{$\textcolor{red}{2}$} at 0 -5
			\pinlabel \scalebox{0.7}{$\textcolor{blue}{1}$} at 16 -5
			\pinlabel \scalebox{0.9}{$\lambda$} at 18 8
			\endlabellist 
			\centering 
			\includegraphics[scale=1.3]{./RurBul}
	}\endxy\;\;\xmapsto{\tilde{\mathcal{T}}_{1,-1}'}
    \left(k_1^{2,3}
    \xy (0,0)*{
			\labellist
			\small\hair 2pt 
			\pinlabel \scalebox{0.55}{$\textcolor{blue}{1}$} at 24 -5
			\pinlabel \scalebox{0.55}{$\textcolor{blue}{1}$} at 0 -5
			\pinlabel \scalebox{0.55}{$\textcolor{red}{2}$} at 12 -5
                \pinlabel \scalebox{0.9}{$s_1(\lambda)$} at 38 8   
			\endlabellist
			\centering
			\includegraphics[scale=1]{./BurRurBdll}
	}\endxy\quad\quad\;\;,
    k_1^{0,1}
    \xy (0,0)*{
			\labellist
			\small\hair 2pt 
			\pinlabel \scalebox{0.55}{$\textcolor{blue}{1}$} at 24 -5
			\pinlabel \scalebox{0.55}{$\textcolor{blue}{1}$} at 12 -5
			\pinlabel \scalebox{0.55}{$\textcolor{red}{2}$} at 0 -5
                \pinlabel \scalebox{0.9}{$s_1(\lambda)$} at 38 8   
			\endlabellist
			\centering
			\includegraphics[scale=1]{./RurBurBdll}
	}\endxy\quad\;\;\;
    \right)
\\[2ex]
        \xy (0,1)*{
			\labellist
			\small\hair 2pt
			\pinlabel \scalebox{0.7}{$\textcolor{red}{2}$} at 16 -5
			\pinlabel \scalebox{0.7}{$\textcolor{blue}{1}$} at 0 -5
			\pinlabel \scalebox{0.9}{$\lambda$} at 18 8
			\endlabellist 
			\centering 
			\includegraphics[scale=1.3]{./BurRul}
	}\endxy\;\;\xmapsto{\tilde{\mathcal{T}}_{1,-1}'}
     \left(  k_1^{0,1}\left(
    \xy (0,0)*{
			\labellist
			\small\hair 2pt 
			\pinlabel \scalebox{0.55}{$\textcolor{blue}{1}$} at 12 -5
			\pinlabel \scalebox{0.55}{$\textcolor{blue}{1}$} at 0 -5
			\pinlabel \scalebox{0.55}{$\textcolor{red}{2}$} at 24 -5
                \pinlabel \scalebox{0.9}{$s_1(\lambda)$} at 38 8   
			\endlabellist
			\centering
			\includegraphics[scale=1]{./BodrrBulRul}
	}\endxy\quad\quad -
    \xy (0,0)*{
			\labellist
			\small\hair 2pt 
			\pinlabel \scalebox{0.55}{$\textcolor{blue}{1}$} at 12 -5
			\pinlabel \scalebox{0.55}{$\textcolor{blue}{1}$} at 0 -5
			\pinlabel \scalebox{0.55}{$\textcolor{red}{2}$} at 24 -5
                \pinlabel \scalebox{0.9}{$s_1(\lambda)$} at 38 8   
			\endlabellist
			\centering
			\includegraphics[scale=1]{./BdrrBoulRul}
	}\endxy\quad\quad\right),
k_1^{0,1}\left(
    \xy (0,0)*{
			\labellist
			\small\hair 2pt 
			\pinlabel \scalebox{0.55}{$\textcolor{blue}{1}$} at 24 -5
			\pinlabel \scalebox{0.55}{$\textcolor{blue}{1}$} at 0 -5
			\pinlabel \scalebox{0.55}{$\textcolor{red}{2}$} at 12 -5
                \pinlabel \scalebox{0.9}{$s_1(\lambda)$} at 38 8   
			\endlabellist
			\centering
			\includegraphics[scale=1]{./BdrrRulBoul}
	}\endxy\quad\quad -
    \xy (0,0)*{
			\labellist
			\small\hair 2pt 
			\pinlabel \scalebox{0.55}{$\textcolor{blue}{1}$} at 24 -5
			\pinlabel \scalebox{0.55}{$\textcolor{blue}{1}$} at 0 -5
			\pinlabel \scalebox{0.55}{$\textcolor{red}{2}$} at 12 -5
                \pinlabel \scalebox{0.9}{$s_1(\lambda)$} at 38 8   
			\endlabellist
			\centering
			\includegraphics[scale=1]{./BodrrRulBul}
	}\endxy\quad\quad\right)
\right)
\\[2ex]
\xy (0,1)*{
			\labellist
			\small\hair 2pt
			\pinlabel \scalebox{0.7}{$\textcolor{blue}{1}$} at 1 -5
			\pinlabel \scalebox{0.9}{$\lambda$} at 18 8
			\endlabellist 
			\centering 
			\includegraphics[scale=1.3]{./BrCap1}
	}\endxy\;\;\xmapsto{\tilde{\mathcal{T}}_{1,-1}'}
    (-1)^{\lambda_1+1} \xy (0,0)*{
			\labellist
			\small\hair 2pt
			\pinlabel \scalebox{0.7}{$\textcolor{blue}{1}$} at 1 -5
			\pinlabel \scalebox{0.9}{$s_1(\lambda)$} at 22 8
			\endlabellist 
			\centering 
			\includegraphics[scale=1.3]{./BlCap1}
	}\endxy
    \mspace{180mu}
    \xy (0,0)*{
			\labellist
			\small\hair 2pt
			\pinlabel \scalebox{0.7}{$\textcolor{blue}{1}$} at 1 -5
			\pinlabel \scalebox{0.9}{$\lambda$} at 18 8
			\endlabellist 
			\centering 
			\includegraphics[scale=1.3]{./BlCap1}
	}\endxy\;\;\xmapsto{\tilde{\mathcal{T}}_{1,-1}'}
    (-1)^{\lambda_2+1} \xy (0,0)*{
			\labellist
			\small\hair 2pt
			\pinlabel \scalebox{0.7}{$\textcolor{blue}{1}$} at 1 -5
			\pinlabel \scalebox{0.9}{$s_1(\lambda)$} at 22 8
			\endlabellist 
			\centering 
			\includegraphics[scale=1.3]{./BrCap1}
	}\endxy
\\[2ex]
\xy (0,1)*{
			\labellist
			\small\hair 2pt
			\pinlabel \scalebox{0.7}{$\textcolor{blue}{1}$} at 1 9
			\pinlabel \scalebox{0.9}{$\lambda$} at 19 2
			\endlabellist 
			\centering 
			\includegraphics[scale=1.3]{./BrCup1}
	}\endxy\;\;\xmapsto{\tilde{\mathcal{T}}_{1,-1}'}
    (-1)^{\lambda_1} \xy (0,0)*{
			\labellist
			\small\hair 2pt
			\pinlabel \scalebox{0.7}{$\textcolor{blue}{1}$} at 1 9
			\pinlabel \scalebox{0.9}{$s_1(\lambda)$} at 24 2
			\endlabellist 
			\centering 
			\includegraphics[scale=1.3]{./BlCup1}
	}\endxy
   \mspace{190mu}
   \xy (0,0)*{
			\labellist
			\small\hair 2pt
			\pinlabel \scalebox{0.7}{$\textcolor{blue}{1}$} at 1 9
			\pinlabel \scalebox{0.9}{$\lambda$} at 19 2
			\endlabellist 
			\centering 
			\includegraphics[scale=1.3]{./BlCup1}
	}\endxy\;\;\xmapsto{\tilde{\mathcal{T}}_{1,-1}'}
    (-1)^{\lambda_2} \xy (0,0)*{
			\labellist
			\small\hair 2pt
			\pinlabel \scalebox{0.7}{$\textcolor{blue}{1}$} at 1 9
			\pinlabel \scalebox{0.9}{$s_1(\lambda)$} at 24 2
			\endlabellist 
			\centering 
			\includegraphics[scale=1.3]{./BrCup1}
	}\endxy
\\[2ex]
\xy (0,1)*{
			\labellist
			\small\hair 2pt
			\pinlabel \scalebox{0.7}{$\textcolor{red}{2}$} at 1 -5
			\pinlabel \scalebox{0.9}{$\lambda$} at 18 8
			\endlabellist 
			\centering 
			\includegraphics[scale=1.3]{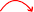}
	}\endxy\;\;\xmapsto{\tilde{\mathcal{T}}_{1,-1}'}
    (-1)^{\overline{\lambda}_2+1}\left(
    k_1^3\;\xy (0,0)*{
			\labellist
			\small\hair 2pt
			\pinlabel \scalebox{0.7}{$\textcolor{red}{2}$} at 0 -5
			\pinlabel \scalebox{0.7}{$\textcolor{blue}{1}$} at 10 -5
			\pinlabel \scalebox{0.9}{$s_1(\lambda)$} at 36 8
			\endlabellist 
			\centering 
			\includegraphics[scale=1.3]{./RrCap3BrCap1}
	}\endxy\quad\;\; -
    k_1^2\;
    \xy (0,0)*{
			\labellist
			\small\hair 2pt
			\pinlabel \scalebox{0.7}{$\textcolor{red}{2}$} at 10 -5
			\pinlabel \scalebox{0.7}{$\textcolor{blue}{1}$} at 0 -5
			\pinlabel \scalebox{0.9}{$s_1(\lambda)$} at 36 8
			\endlabellist 
			\centering 
			\includegraphics[scale=1.3]{./BrCap3RrCap1}
	}\endxy\mspace{27mu}\right)
\\[2ex]
\xy (0,1)*{
			\labellist
			\small\hair 2pt
			\pinlabel \scalebox{0.7}{$\textcolor{red}{2}$} at 1 -5
			\pinlabel \scalebox{0.9}{$\lambda$} at 18 8
			\endlabellist 
			\centering 
			\includegraphics[scale=1.3]{./RlCap1}
	}\endxy\;\;\xmapsto{\tilde{\mathcal{T}}_{1,-1}'}
    (-1)^{\lambda_3}\left(k_1^1\;
    \xy (0,0)*{
			\labellist
			\small\hair 2pt
			\pinlabel \scalebox{0.7}{$\textcolor{red}{2}$} at 0 -5
			\pinlabel \scalebox{0.7}{$\textcolor{blue}{1}$} at 10 -5
			\pinlabel \scalebox{0.9}{$s_1(\lambda)$} at 36 8
			\endlabellist 
			\centering 
			\includegraphics[scale=1.3]{./RlCap3BlCap1Big}
	}\endxy\quad\;\; -
    k_1^2\;
    \xy (0,0)*{
			\labellist
			\small\hair 2pt
			\pinlabel \scalebox{0.7}{$\textcolor{red}{2}$} at 10 -5
			\pinlabel \scalebox{0.7}{$\textcolor{blue}{1}$} at 0 -5
			\pinlabel \scalebox{0.9}{$s_1(\lambda)$} at 36 8
			\endlabellist 
			\centering 
			\includegraphics[scale=1.3]{./BlCap3RlCap1Big}
	}\endxy\mspace{27mu}\right)
\\[2ex]
\xy (0,1)*{
			\labellist
			\small\hair 2pt
			\pinlabel \scalebox{0.7}{$\textcolor{red}{2}$} at 1 10
			\pinlabel \scalebox{0.9}{$\lambda$} at 18 2
			\endlabellist 
			\centering 
			\includegraphics[scale=1.3]{./RlCup1Thin}
	}\endxy\;\;\xmapsto{\tilde{\mathcal{T}}_{1,-1}'}
    (-1)^{\lambda_3}\left(k_1^1\;
    \xy (0,0)*{
			\labellist
			\small\hair 2pt
			\pinlabel \scalebox{0.7}{$\textcolor{red}{2}$} at 0 14
			\pinlabel \scalebox{0.7}{$\textcolor{blue}{1}$} at 10 14
			\pinlabel \scalebox{0.9}{$s_1(\lambda)$} at 36 2
			\endlabellist 
			\centering 
			\includegraphics[scale=1.3]{./RlCup3BlCup1}
	}\endxy\quad\;\; -
   k_1^0\;
    \xy (0,0)*{
			\labellist
			\small\hair 2pt
			\pinlabel \scalebox{0.7}{$\textcolor{red}{2}$} at 10 14
			\pinlabel \scalebox{0.7}{$\textcolor{blue}{1}$} at 0 14
			\pinlabel \scalebox{0.9}{$s_1(\lambda)$} at 36 2
			\endlabellist 
			\centering 
			\includegraphics[scale=1.3]{./BlCup3RlCup1}
	}\endxy\mspace{27mu}\right)
\\[2ex]
\xy (0,1)*{
			\labellist
			\small\hair 2pt
			\pinlabel \scalebox{0.7}{$\textcolor{red}{2}$} at 1 10
			\pinlabel \scalebox{0.9}{$\lambda$} at 18 2
			\endlabellist 
			\centering 
			\includegraphics[scale=1.3]{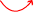}
	}\endxy\;\;\xmapsto{\tilde{\mathcal{T}}_{1,-1}'}
    (-1)^{\overline{\lambda}_2}\left(k_1^3\;
    \xy (0,0)*{
			\labellist
			\small\hair 2pt
			\pinlabel \scalebox{0.7}{$\textcolor{red}{2}$} at 0 14
			\pinlabel \scalebox{0.7}{$\textcolor{blue}{1}$} at 10 14
			\pinlabel \scalebox{0.9}{$s_1(\lambda)$} at 36 2
			\endlabellist 
			\centering 
			\includegraphics[scale=1.3]{./RrCup3BrCup1}
	}\endxy\quad\;\; -
    k_1^0\;
    \xy (0,0)*{
			\labellist
			\small\hair 2pt
			\pinlabel \scalebox{0.7}{$\textcolor{red}{2}$} at 10 14
			\pinlabel \scalebox{0.7}{$\textcolor{blue}{1}$} at 0 14
			\pinlabel \scalebox{0.9}{$s_1(\lambda)$} at 36 2
			\endlabellist 
			\centering 
			\includegraphics[scale=1.3]{./BrCup3RrCup1}
	}\endxy\mspace{27mu}\right)
\end{gather}\endgroup

\end{itemize}

In detail, $\tilde{\mathcal{T}}_{2,1}''$ acts as follows:
\begin{itemize}[wide,labelindent=0pt,itemsep=5pt]
    \item On objects, $\lambda \xmapsto{\tilde{\mathcal{T}}_{2,1}''}  s_2(\lambda)$

    \item On 1-morphisms,
\begin{gather*}
\E_1\oneid_\lambda \xmapsto{\tilde{\mathcal{T}}_{2,1}''} 
\Bigg(\xymatrix{\E_{12}\oneid_{s_2(\lambda)}\langle -1\rangle \ar[rr]^(0.55){\xy (0,4.5)*{
		\labellist
		\small\hair 2pt
		\pinlabel \scalebox{0.7}{$\textcolor{red}{2}$} at 16 -5
		\pinlabel \scalebox{0.7}{$\textcolor{blue}{1}$} at 0 -5
		\pinlabel \scalebox{0.9}{$\lambda$} at 16 7
		\endlabellist
		\centering
		\includegraphics[scale=1.3]{./BurRul}
	}\endxy} & & \underline{\E_{21}\oneid_{s_2(\lambda)}}}\Bigg)
\\[1ex]
\F_1\oneid_\lambda \xmapsto{\tilde{\mathcal{T}}_{2,1}''} 
\Bigg(\xymatrix{\underline{\F_{12}\oneid_{s_2(\lambda)}} \ar[rrrrr]^(0.55){\xy (0,4.5)*{
		\labellist
		\small\hair 2pt
		\pinlabel \scalebox{0.7}{$\textcolor{red}{2}$} at 16 -5
		\pinlabel \scalebox{0.7}{$\textcolor{blue}{1}$} at 0 -5
		\pinlabel \scalebox{0.9}{$\lambda$} at 16 8
            \pinlabel \scalebox{0.9}{$k_3^{2,3}$} at -10 8
            \endlabellist
		\centering
		\includegraphics[scale=1.3]{./BdrRdl}
	}\endxy \mspace{70mu}} & & & & & \F_{21}\oneid_{s_2(\lambda)}\langle 1\rangle}\Bigg)
\\[1ex]
\E_2 \oneid_\lambda\xmapsto{\tilde{\mathcal{T}}_{2,1}''}  \F_2\oneid_{s_2(\lambda)}[-1]\langle\bar{\lambda}_2\rangle
\mspace{60mu}
\F_2\oneid_\lambda\xmapsto{\tilde{\mathcal{T}}_{2,1}''} \E_2\oneid_{s_2(\lambda)}[1]\langle -2-\bar{\lambda}_2\rangle 
\end{gather*}

    \item On non-identity generating 2-morphisms:
\begingroup\allowdisplaybreaks
\begin{gather}
\xy (0,1)*{
			\labellist
			\small\hair 2pt
			\pinlabel \scalebox{0.7}{$\textcolor{blue}{1}$} at 1 -5
			\pinlabel \scalebox{0.9}{$\lambda$} at 6 8
			\endlabellist 
			\centering 
			\includegraphics[scale=1.3]{./Buo}
	}\endxy\;\;\xmapsto{\tilde{\mathcal{T}}_{2,1}''}
    \left(\xy (0,0)*{
		\labellist
		\small\hair 2pt
		\pinlabel \scalebox{0.7}{$\textcolor{blue}{1}$} at 1 -5
		\pinlabel \scalebox{0.7}{$\textcolor{red}{2}$} at 9 -5
       	\pinlabel \scalebox{0.9}{$s_2(\lambda)$} at 20 8
		\endlabellist
		\centering
		\includegraphics[scale=1.3]{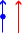}
	}\endxy
    \quad\quad,\xy (0,0)*{
		\labellist
		\small\hair 2pt
		\pinlabel \scalebox{0.7}{$\textcolor{blue}{1}$} at 9 -5
		\pinlabel \scalebox{0.7}{$\textcolor{red}{2}$} at 1 -5
       	\pinlabel \scalebox{0.9}{$s_2(\lambda)$} at 20 8
		\endlabellist
		\centering
		\includegraphics[scale=1.3]{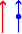}
	}\endxy\quad\quad
    \right)
\mspace{90mu}
\xy (0,0)*{
			\labellist
			\small\hair 2pt
			\pinlabel \scalebox{0.7}{$\textcolor{red}{2}$} at 1 -5
			\pinlabel \scalebox{0.9}{$\lambda$} at 6 8
			\endlabellist 
			\centering 
			\includegraphics[scale=1.3]{./Ruo}
	}\endxy\;\;\xmapsto{\tilde{\mathcal{T}}_{2,1}''}
    \,\xy (0,0)*{
			\labellist
			\small\hair 2pt
			\pinlabel \scalebox{0.7}{$\textcolor{red}{2}$} at 1 -5
			\pinlabel \scalebox{0.9}{$s_2(\lambda)$} at 12 8
			\endlabellist 
			\centering 
			\includegraphics[scale=1.3]{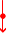}
	}\endxy
\mspace{90mu}
    \xy (0,0)*{
			\labellist
			\small\hair 2pt
			\pinlabel \scalebox{0.7}{$\textcolor{red}{2}$} at 0 -5
			\pinlabel \scalebox{0.7}{$\textcolor{red}{2}$} at 16 -5
			\pinlabel \scalebox{0.9}{$\lambda$} at 18 8
			\endlabellist 
			\centering 
			\includegraphics[scale=1.3]{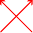}
	}\endxy\;\;\xmapsto{\tilde{\mathcal{T}}_{2,1}''}
    -\xy (0,0)*{
			\labellist
			\small\hair 2pt
			\pinlabel \scalebox{0.7}{$\textcolor{red}{2}$} at 0 -5
			\pinlabel \scalebox{0.7}{$\textcolor{red}{2}$} at 16 -5
			\pinlabel \scalebox{0.9}{$s_2(\lambda)$} at 24 8
			\endlabellist 
			\centering 
			\includegraphics[scale=1.3]{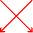}
	}\endxy
\\[2ex] \label{eq:temp-eq}
\xy (0,1)*{
			\labellist
			\small\hair 2pt
			\pinlabel \scalebox{0.7}{$\textcolor{blue}{1}$} at 0 -5
			\pinlabel \scalebox{0.7}{$\textcolor{blue}{1}$} at 16 -5
			\pinlabel \scalebox{0.9}{$\lambda$} at 18 8
			\endlabellist 
			\centering 
			\includegraphics[scale=1.3]{./BurBul}
	}\endxy\;\; \xmapsto{\tilde{\mathcal{T}}_{2,1}''} 
    \bigg(\;\xy (0,0)*{
			\labellist
			\small\hair 2pt 
			\pinlabel \scalebox{0.55}{$\textcolor{blue}{1}$} at 0 -5
			\pinlabel \scalebox{0.55}{$\textcolor{red}{2}$} at 8 -5
			\pinlabel \scalebox{0.55}{$\textcolor{red}{2}$} at 32 -5
			\pinlabel \scalebox{0.55}{$\textcolor{blue}{1}$} at 24 -5
                \pinlabel \scalebox{0.9}{$s_2(\lambda)$} at 40 8   
			\endlabellist
			\centering
			\includegraphics[scale=1]{./BurrRurrBullRull}
	}\endxy
    \quad\quad, 
    k_3^{1,3}\xy (0,0)*{
			\labellist
			\small\hair 2pt 
			\pinlabel \scalebox{0.55}{$\textcolor{blue}{1}$} at 8 -5
			\pinlabel \scalebox{0.55}{$\textcolor{red}{2}$} at 0 -5
			\pinlabel \scalebox{0.55}{$\textcolor{red}{2}$} at 32 -5
			\pinlabel \scalebox{0.55}{$\textcolor{blue}{1}$} at 24 -5
                \pinlabel \scalebox{0.9}{$s_2(\lambda)$} at 40 8   
			\endlabellist
			\centering
			\includegraphics[scale=1]{./RurrBurrBullRull}
	}\endxy
    \quad\quad+ k_3^{1,3}\xy (0,0)*{
			\labellist
			\small\hair 2pt 
			\pinlabel \scalebox{0.55}{$\textcolor{blue}{1}$} at 0 -5
			\pinlabel \scalebox{0.55}{$\textcolor{red}{2}$} at 8 -5
			\pinlabel \scalebox{0.55}{$\textcolor{red}{2}$} at 24 -5
			\pinlabel \scalebox{0.55}{$\textcolor{blue}{1}$} at 32 -5
                \pinlabel \scalebox{0.9}{$s_2(\lambda)$} at 40 8   
			\endlabellist
			\centering
			\includegraphics[scale=1]{./BurrRurrRullBull}
	}\endxy  
\\ \nonumber
      +\;\; \xy (0,0)*{
				\labellist
				\small\hair 2pt
				\pinlabel \scalebox{0.55}{$\textcolor{blue}{1}$} at 8 -5
				\pinlabel \scalebox{0.55}{$\textcolor{red}{2}$} at 1 -5
				\pinlabel \scalebox{0.55}{$\textcolor{red}{2}$} at 31 -5
				\pinlabel \scalebox{0.55}{$\textcolor{blue}{1}$} at 24 -5
                    \pinlabel \scalebox{0.9}{$s_2(\lambda)$} at 44 8    
				\endlabellist
				\centering
				\includegraphics[scale=1]{./RuBurBulRu}
		}\endxy
        \quad\quad + \xy (0,0)*{
				\labellist
				\small\hair 2pt
				\pinlabel \scalebox{0.55}{$\textcolor{blue}{1}$} at 1 -5
				\pinlabel \scalebox{0.55}{$\textcolor{red}{2}$} at 8 -5
				\pinlabel \scalebox{0.55}{$\textcolor{red}{2}$} at 24 -5
				\pinlabel \scalebox{0.55}{$\textcolor{blue}{1}$} at 31 -5
                    \pinlabel \scalebox{0.9}{$s_2(\lambda)$} at 44 8    
				\endlabellist
				\centering
				\includegraphics[scale=1]{./BuRurRulBu}
		}\endxy
        \quad\quad, -\xy (0,0)*{
			\labellist
			\small\hair 2pt 
			\pinlabel \scalebox{0.55}{$\textcolor{blue}{1}$} at 8 -5
			\pinlabel \scalebox{0.55}{$\textcolor{red}{2}$} at 0 -5
			\pinlabel \scalebox{0.55}{$\textcolor{red}{2}$} at 24 -5
			\pinlabel \scalebox{0.55}{$\textcolor{blue}{1}$} at 32 -5
                \pinlabel \scalebox{0.9}{$s_2(\lambda)$} at 40 8   
			\endlabellist
			\centering
			\includegraphics[scale=1]{./RurrBurrRullBull}
	}\endxy\quad\;\;\;
    \bigg)
\\[2ex]
\xy (0,1)*{
			\labellist
			\small\hair 2pt
			\pinlabel \scalebox{0.7}{$\textcolor{red}{2}$} at 0 -5
			\pinlabel \scalebox{0.7}{$\textcolor{blue}{1}$} at 16 -5
			\pinlabel \scalebox{0.9}{$\lambda$} at 18 8
			\endlabellist 
			\centering 
			\includegraphics[scale=1.3]{./RurBul}
	}\endxy\;\;\xmapsto{\tilde{\mathcal{T}}_{2,1}''}
    \left(
    \xy (0,0)*{
			\labellist
			\small\hair 2pt 
			\pinlabel \scalebox{0.55}{$\textcolor{blue}{1}$} at 12 -5
			\pinlabel \scalebox{0.55}{$\textcolor{red}{2}$} at 0 -5
			\pinlabel \scalebox{0.55}{$\textcolor{red}{2}$} at 24 -5
                \pinlabel \scalebox{0.9}{$s_2(\lambda)$} at 38 8   
			\endlabellist
			\centering
			\includegraphics[scale=1]{./RdrrBulRul}
	}\endxy\quad\quad\;,
    -\xy (0,0)*{
			\labellist
			\small\hair 2pt 
			\pinlabel \scalebox{0.55}{$\textcolor{blue}{1}$} at 24 -5
			\pinlabel \scalebox{0.55}{$\textcolor{red}{2}$} at 12 -5
			\pinlabel \scalebox{0.55}{$\textcolor{red}{2}$} at 0 -5
                \pinlabel \scalebox{0.9}{$s_2(\lambda)$} at 38 8   
			\endlabellist
			\centering
			\includegraphics[scale=1]{./RdrrRulBul}
	}\endxy\mspace{37mu}
    \right)
\\[2ex]
       \xy (0,1)*{
			\labellist
			\small\hair 2pt
			\pinlabel \scalebox{0.7}{$\textcolor{red}{2}$} at 16 -5
			\pinlabel \scalebox{0.7}{$\textcolor{blue}{1}$} at 0 -5
			\pinlabel \scalebox{0.9}{$\lambda$} at 18 8
			\endlabellist 
			\centering 
			\includegraphics[scale=1.3]{./BurRul}
	}\endxy\;\;\xmapsto{\tilde{\mathcal{T}}_{2,1}''}
    \left(\;
    \xy (0,0)*{
			\labellist
			\small\hair 2pt 
			\pinlabel \scalebox{0.55}{$\textcolor{red}{2}$} at 12 -5
			\pinlabel \scalebox{0.55}{$\textcolor{blue}{1}$} at 0 -5
			\pinlabel \scalebox{0.55}{$\textcolor{red}{2}$} at 24 -5
                \pinlabel \scalebox{0.9}{$s_2(\lambda)$} at 38 8   
			\endlabellist
			\centering
			\includegraphics[scale=1]{./BurRurRodll}
	}\endxy\quad\quad -
    \xy (0,0)*{
			\labellist
			\small\hair 2pt 
			\pinlabel \scalebox{0.55}{$\textcolor{red}{2}$} at 12 -5
			\pinlabel \scalebox{0.55}{$\textcolor{blue}{1}$} at 0 -5
			\pinlabel \scalebox{0.55}{$\textcolor{red}{2}$} at 24 -5
                \pinlabel \scalebox{0.9}{$s_2(\lambda)$} at 38 8   
			\endlabellist
			\centering
			\includegraphics[scale=1]{./BurRourRdll}
	}\endxy\quad\quad,
    \xy (0,0)*{
			\labellist
			\small\hair 2pt 
			\pinlabel \scalebox{0.55}{$\textcolor{blue}{1}$} at 12 -5
			\pinlabel \scalebox{0.55}{$\textcolor{red}{2}$} at 0 -5
			\pinlabel \scalebox{0.55}{$\textcolor{red}{2}$} at 24 -5
                \pinlabel \scalebox{0.9}{$s_2(\lambda)$} at 38 8   
			\endlabellist
			\centering
			\includegraphics[scale=1]{./RourBurRdll}
	}\endxy\quad\quad -
    \xy (0,0)*{
			\labellist
			\small\hair 2pt 
			\pinlabel \scalebox{0.55}{$\textcolor{blue}{1}$} at 12 -5
			\pinlabel \scalebox{0.55}{$\textcolor{red}{2}$} at 0 -5
			\pinlabel \scalebox{0.55}{$\textcolor{red}{2}$} at 24 -5
                \pinlabel \scalebox{0.9}{$s_2(\lambda)$} at 38 8   
			\endlabellist
			\centering
			\includegraphics[scale=1]{./RurBurRodll}
	}\endxy\quad\quad
\right)
\\[2ex]
\xy (0,1)*{
			\labellist
			\small\hair 2pt
			\pinlabel \scalebox{0.7}{$\textcolor{red}{2}$} at 1 -5
			\pinlabel \scalebox{0.9}{$\lambda$} at 18 8
			\endlabellist 
			\centering 
			\includegraphics[scale=1.3]{./RrCap1Thin}
	}\endxy\;\;\xmapsto{\tilde{\mathcal{T}}_{2,1}''}
    (-1)^{\lambda_2+1} \xy (0,0)*{
			\labellist
			\small\hair 2pt
			\pinlabel \scalebox{0.7}{$\textcolor{red}{2}$} at 1 -5
			\pinlabel \scalebox{0.9}{$s_2(\lambda)$} at 22 8
			\endlabellist 
			\centering 
			\includegraphics[scale=1.3]{./RlCap1}
	}\endxy
\mspace{160mu}
\xy (0,0)*{
			\labellist
			\small\hair 2pt
			\pinlabel \scalebox{0.7}{$\textcolor{red}{2}$} at 1 -5
			\pinlabel \scalebox{0.9}{$\lambda$} at 18 8
			\endlabellist 
			\centering 
			\includegraphics[scale=1.3]{./RlCap1}
	}\endxy\;\;\xmapsto{\tilde{\mathcal{T}}_{2,1}''}
    (-1)^{\lambda_3+1} \xy (0,0)*{
			\labellist
			\small\hair 2pt
			\pinlabel \scalebox{0.7}{$\textcolor{red}{2}$} at 1 -5
			\pinlabel \scalebox{0.9}{$s_2(\lambda)$} at 22 8
			\endlabellist 
			\centering 
			\includegraphics[scale=1.3]{./RrCap1Thin}
	}\endxy
\\[2ex]
\xy (0,1)*{
			\labellist
			\small\hair 2pt
			\pinlabel \scalebox{0.7}{$\textcolor{red}{2}$} at 1 10
			\pinlabel \scalebox{0.9}{$\lambda$} at 19 2
			\endlabellist 
			\centering 
			\includegraphics[scale=1.3]{./RrCup1Thin}
	}\endxy\;\;\xmapsto{\tilde{\mathcal{T}}_{2,1}''}
    (-1)^{\lambda_2} \xy (0,0)*{
			\labellist
			\small\hair 2pt
			\pinlabel \scalebox{0.7}{$\textcolor{red}{2}$} at 1 10
			\pinlabel \scalebox{0.9}{$s_2(\lambda)$} at 24 2
			\endlabellist 
			\centering 
			\includegraphics[scale=1.3]{./RlCup1Thin}
	}\endxy
\mspace{160mu}
\xy (0,0)*{
			\labellist
			\small\hair 2pt
			\pinlabel \scalebox{0.7}{$\textcolor{red}{2}$} at 1 10
			\pinlabel \scalebox{0.9}{$\lambda$} at 19 2
			\endlabellist 
			\centering 
			\includegraphics[scale=1.3]{./RlCup1Thin}
	}\endxy\;\;\xmapsto{\tilde{\mathcal{T}}_{2,1}''}
    (-1)^{\lambda_3} \xy (0,0)*{
			\labellist
			\small\hair 2pt
			\pinlabel \scalebox{0.7}{$\textcolor{red}{2}$} at 1 10
			\pinlabel \scalebox{0.9}{$s_2(\lambda)$} at 24 2
			\endlabellist 
			\centering 
			\includegraphics[scale=1.3]{./RrCup1Thin}
	}\endxy
\\[2ex]
\xy (0,0)*{
			\labellist
			\small\hair 2pt
			\pinlabel \scalebox{0.7}{$\textcolor{blue}{1}$} at 1 -5
			\pinlabel \scalebox{0.9}{$\lambda$} at 14 10
			\endlabellist 
			\centering 
			\includegraphics[scale=1]{./BrCap1}
	}\endxy\xmapsto{\tilde{\mathcal{T}}_{2,1}''}
    -k_1^0\left(k_3^3\;
    \xy (0,0)*{
			\labellist
			\small\hair 2pt
			\pinlabel \scalebox{0.7}{$\textcolor{red}{2}$} at 0 -5
			\pinlabel \scalebox{0.7}{$\textcolor{blue}{1}$} at 10 -5
			\pinlabel \scalebox{0.9}{$s_2(\lambda)$} at 22 16
			\endlabellist 
			\centering 
			\includegraphics[scale=1]{./RrCap3BrCap1}
	}\endxy -
    k_3^2\;
    \xy (0,0)*{
			\labellist
			\small\hair 2pt
			\pinlabel \scalebox{0.7}{$\textcolor{red}{2}$} at 10 -5
			\pinlabel \scalebox{0.7}{$\textcolor{blue}{1}$} at 0 -5
			\pinlabel \scalebox{0.9}{$s_2(\lambda)$} at 22 16
			\endlabellist 
			\centering 
			\includegraphics[scale=1]{./BrCap3RrCap1}
	}\endxy\right)
\\[2ex]
\xy (0,0)*{
			\labellist
			\small\hair 2pt
			\pinlabel \scalebox{0.7}{$\textcolor{blue}{1}$} at 1 -5
			\pinlabel \scalebox{0.9}{$\lambda$} at 14 10
			\endlabellist 
			\centering 
			\includegraphics[scale=1]{./BlCap1}
	}\endxy\xmapsto{\tilde{\mathcal{T}}_{2,1}''}(-1)^{\lambda_1}k_1^2\left(
    k_3^1\;
    \xy (0,0)*{
			\labellist
			\small\hair 2pt
			\pinlabel \scalebox{0.7}{$\textcolor{red}{2}$} at 0 -5
			\pinlabel \scalebox{0.7}{$\textcolor{blue}{1}$} at 10 -5
			\pinlabel \scalebox{0.9}{$s_2(\lambda)$} at 22 16
			\endlabellist 
			\centering 
			\includegraphics[scale=1]{./RlCap3BlCap1Big}
	}\endxy -
    k_3^2\;
    \xy (0,0)*{
			\labellist
			\small\hair 2pt
			\pinlabel \scalebox{0.7}{$\textcolor{red}{2}$} at 10 -5
			\pinlabel \scalebox{0.7}{$\textcolor{blue}{1}$} at 0 -5
			\pinlabel \scalebox{0.9}{$s_2(\lambda)$} at 22 16
			\endlabellist 
			\centering 
			\includegraphics[scale=1]{./BlCap3RlCap1Big}
	}\endxy\;\right)
\\[2ex]
\xy (0,0)*{
			\labellist
			\small\hair 2pt
			\pinlabel \scalebox{0.7}{$\textcolor{blue}{1}$} at 1 9
			\pinlabel \scalebox{0.9}{$\lambda$} at 12 -4
			\endlabellist 
			\centering 
			\includegraphics[scale=1]{./BlCup1}
	}\endxy\xmapsto{\tilde{\mathcal{T}}_{2,1}''}
    (-1)^{\lambda_1}k_1^0\left( k_3^1\;
    \xy (0,0)*{
			\labellist
			\small\hair 2pt
			\pinlabel \scalebox{0.7}{$\textcolor{red}{2}$} at 0 14
			\pinlabel \scalebox{0.7}{$\textcolor{blue}{1}$} at 10 14
			\pinlabel \scalebox{0.9}{$s_2(\lambda)$} at 22 -6
			\endlabellist 
			\centering 
			\includegraphics[scale=1]{./RlCup3BlCup1}
	}\endxy -
    k_3^0\;
    \xy (0,0)*{
			\labellist
			\small\hair 2pt
			\pinlabel \scalebox{0.7}{$\textcolor{red}{2}$} at 10 14
			\pinlabel \scalebox{0.7}{$\textcolor{blue}{1}$} at 0 14
			\pinlabel \scalebox{0.9}{$s_2(\lambda)$} at 22 -6
			\endlabellist 
			\centering 
			\includegraphics[scale=1]{./BlCup3RlCup1}
	}\endxy\;\right)
\\[2ex]
\xy (0,0)*{
			\labellist
			\small\hair 2pt
			\pinlabel \scalebox{0.7}{$\textcolor{blue}{1}$} at 1 9
			\pinlabel \scalebox{0.9}{$\lambda$} at 12 -4
			\endlabellist 
			\centering 
			\includegraphics[scale=1]{./BrCup1}
	}\endxy\xmapsto{\tilde{\mathcal{T}}_{2,1}''}
     k_1^2\left(k_3^3\;
    \xy (0,0)*{
			\labellist
			\small\hair 2pt
			\pinlabel \scalebox{0.7}{$\textcolor{red}{2}$} at 0 14
			\pinlabel \scalebox{0.7}{$\textcolor{blue}{1}$} at 10 14
			\pinlabel \scalebox{0.9}{$s_2(\lambda)$} at 22 -6
			\endlabellist 
			\centering 
			\includegraphics[scale=1]{./RrCup3BrCup1}
	}\endxy -
    k_3^0\;
    \xy (0,0)*{
			\labellist
			\small\hair 2pt
			\pinlabel \scalebox{0.7}{$\textcolor{red}{2}$} at 10 14
			\pinlabel \scalebox{0.7}{$\textcolor{blue}{1}$} at 0 14
			\pinlabel \scalebox{0.9}{$s_2(\lambda)$} at 22 -6
			\endlabellist 
			\centering 
			\includegraphics[scale=1]{./BrCup3RrCup1}
	}\endxy\right)
\end{gather}
\endgroup
 
\end{itemize}
\eqskip

\subsection{Relating the two actions}
We define a 2-automorphism $\beta:\naffu{3}\to\naffu{3}$, which is the identity on objects, 1-morphisms and all generating 2-morphisms with the exception of:
\begin{align*}
\xy (0,0)*{
			\labellist
			\small\hair 2pt
			\pinlabel \scalebox{0.7}{$\textcolor{blue}{1}$} at 1 -5
			\pinlabel \scalebox{0.9}{$\lambda$} at 18 8
			\endlabellist 
			\centering 
			\includegraphics[scale=1.3]{./BrCap1}
	}\endxy\;\;
    &\overset{\beta}{\mapsto}
    -k_1^2\; \xy (0,0)*{
			\labellist
			\small\hair 2pt
			\pinlabel \scalebox{0.7}{$\textcolor{blue}{1}$} at 1 -5
			\pinlabel \scalebox{0.9}{$\lambda$} at 18 8
			\endlabellist 
			\centering 
			\includegraphics[scale=1.3]{./BrCap1}
	}\endxy
    &
\xy (0,0)*{
			\labellist
			\small\hair 2pt
			\pinlabel \scalebox{0.7}{$\textcolor{blue}{1}$} at 1 9
			\pinlabel \scalebox{0.9}{$\lambda$} at 18 2
			\endlabellist 
			\centering 
			\includegraphics[scale=1.3]{./BlCup1}
	}\endxy\;\;
    &\overset{\beta}{\mapsto}
    -k_1^2\; \xy (0,0)*{
			\labellist
			\small\hair 2pt
			\pinlabel \scalebox{0.7}{$\textcolor{blue}{1}$} at 1 9
			\pinlabel \scalebox{0.9}{$\lambda$} at 18 2
			\endlabellist 
			\centering 
			\includegraphics[scale=1.3]{./BlCup1}
	}\endxy
   \\[1ex]
   \xy (0,0)*{
			\labellist
			\small\hair 2pt
			\pinlabel \scalebox{0.7}{$\textcolor{blue}{1}$} at 1 -5
			\pinlabel \scalebox{0.9}{$\lambda$} at 18 8
			\endlabellist 
			\centering 
			\includegraphics[scale=1.3]{./BlCap1}
	}\endxy\;\;
    &\overset{\beta}{\mapsto}
    -k_1^0\; \xy (0,0)*{
			\labellist
			\small\hair 2pt
			\pinlabel \scalebox{0.7}{$\textcolor{blue}{1}$} at 1 -5
			\pinlabel \scalebox{0.9}{$\lambda$} at 18 8
			\endlabellist 
			\centering 
			\includegraphics[scale=1.3]{./BlCap1}
	}\endxy
 &
    \xy (0,0)*{
			\labellist
			\small\hair 2pt
			\pinlabel \scalebox{0.7}{$\textcolor{blue}{1}$} at 1 9
			\pinlabel \scalebox{0.9}{$\lambda$} at 18 2
			\endlabellist 
			\centering 
			\includegraphics[scale=1.3]{./BrCup1}
	}\endxy\;\;&\overset{\beta}{\mapsto}
    -k_1^0\; \xy (0,0)*{
			\labellist
			\small\hair 2pt
			\pinlabel \scalebox{0.7}{$\textcolor{blue}{1}$} at 1 9
			\pinlabel \scalebox{0.9}{$\lambda$} at 18 2
			\endlabellist 
			\centering 
			\includegraphics[scale=1.3]{./BrCup1}
	}\endxy
\end{align*}
\eqskip
It is a straightforward calculation to confirm that this preserves the axioms KM1-3 and KM7-9 and is therefore a well-defined 2-automorphism.\par

We also define a 2-automorphism $\alpha:\naffu{3}\to \naffu{3}$ which is the identity on all objects and 1-morphisms and on all generating 2-morphisms except:
\begin{equation*}
    \xy (0,0)*{
			\labellist
			\small\hair 2pt
			\pinlabel \scalebox{0.7}{$\textcolor{red}{2}$} at 0 -5
			\pinlabel \scalebox{0.7}{$\textcolor{blue}{1}$} at 16 -5
			\pinlabel \scalebox{0.9}{$\lambda$} at 18 8
			\endlabellist 
			\centering 
			\includegraphics[scale=1.3]{./RurBul}
	}\endxy\;\;\overset{\alpha}{\mapsto}
    k_1^{0,1}\;
    \xy (0,0)*{
			\labellist
			\small\hair 2pt
			\pinlabel \scalebox{0.7}{$\textcolor{red}{2}$} at 0 -5
			\pinlabel \scalebox{0.7}{$\textcolor{blue}{1}$} at 16 -5
			\pinlabel \scalebox{0.9}{$\lambda$} at 18 8
			\endlabellist 
			\centering 
			\includegraphics[scale=1.3]{./RurBul}
	}\endxy
\mspace{120mu}
        \xy (0,0)*{
			\labellist
			\small\hair 2pt
			\pinlabel \scalebox{0.7}{$\textcolor{red}{2}$} at 16 -5
			\pinlabel \scalebox{0.7}{$\textcolor{blue}{1}$} at 0 -5
			\pinlabel \scalebox{0.9}{$\lambda$} at 18 8
			\endlabellist 
			\centering 
			\includegraphics[scale=1.3]{./BurRul}
	}\endxy\;\;\overset{\alpha}{\mapsto}
    k_1^{0,1}\;
    \xy (0,0)*{
			\labellist
			\small\hair 2pt
			\pinlabel \scalebox{0.7}{$\textcolor{red}{2}$} at 16 -5
			\pinlabel \scalebox{0.7}{$\textcolor{blue}{1}$} at 0 -5
			\pinlabel \scalebox{0.9}{$\lambda$} at 18 8
			\endlabellist 
			\centering 
			\includegraphics[scale=1.3]{./BurRul}
	}\endxy
\end{equation*}
\eqskip
It is straightforward to see that $\alpha$ preserves axioms KM3-7, and is therefore a 2-automorphism. We also use $\alpha$ to denote its extension to a 2-automorphism of $K^b(\naffu{3})$.

We now define some notation for ease of stating the following lemma. Let $D_i(\lambda)$ denote a 2-morphism diagram with strands mono-coloured in colour $i$ such that the weight to the right of the diagram is $\lambda$.

\begin{lem}\label{lem:DiagMatch}
For any diagram $D_i(\lambda)$, 
$$\tilde{\mathcal{T}}_{1,-1}'(D_2(s_1(\lambda))) \sim_h \alpha\tilde{\mathcal{T}}_{2,1}''(\beta(D_1(s_2(\lambda)))).$$
\end{lem}
\begin{proof}
The proof follows from calculations that, while not strictly complicated, are liable to confuse. We therefore present the calculations below.

\begin{align*}
\alpha\tilde{\mathcal{T}}_{2,1}''\left(\beta\left(\xy (0,0)*{
			\labellist
			\small\hair 2pt
			\pinlabel \scalebox{0.7}{$\textcolor{blue}{1}$} at 0 -5
			\pinlabel \scalebox{0.7}{$\textcolor{blue}{1}$} at 16 -5
			\pinlabel \scalebox{0.9}{$s_2(\lambda)$} at 22 8
			\endlabellist 
			\centering 
			\includegraphics[scale=1.3]{./BurBul}
	}\endxy\quad\;\;\right)\right)  
    = 
    \alpha\tilde{\mathcal{T}}_{2,1}''\left(\xy (0,0)*{
			\labellist
			\small\hair 2pt
			\pinlabel \scalebox{0.7}{$\textcolor{blue}{1}$} at 0 -5
			\pinlabel \scalebox{0.7}{$\textcolor{blue}{1}$} at 16 -5
			\pinlabel \scalebox{0.9}{$s_2(\lambda)$} at 22 8
			\endlabellist 
			\centering 
			\includegraphics[scale=1.3]{./BurBul}
	}\endxy\quad\;\;\right) \mspace{-180mu}& \\[1ex]
     = & \alpha\Big(\xy (0,0)*{
			\labellist
			\small\hair 2pt 
			\pinlabel \scalebox{0.55}{$\textcolor{blue}{1}$} at 0 -5
			\pinlabel \scalebox{0.55}{$\textcolor{red}{2}$} at 8 -5
			\pinlabel \scalebox{0.55}{$\textcolor{red}{2}$} at 32 -5
			\pinlabel \scalebox{0.55}{$\textcolor{blue}{1}$} at 24 -5
                \pinlabel \scalebox{0.9}{$\lambda$} at 36 8   
			\endlabellist
			\centering
			\includegraphics[scale=1]{./BurrRurrBullRull}
	}\endxy
    \;\;\;\;, 
    k_3^{1,3}(s_2(\lambda))\xy (0,0)*{
			\labellist
			\small\hair 2pt 
			\pinlabel \scalebox{0.55}{$\textcolor{blue}{1}$} at 8 -5
			\pinlabel \scalebox{0.55}{$\textcolor{red}{2}$} at 0 -5
			\pinlabel \scalebox{0.55}{$\textcolor{red}{2}$} at 32 -5
			\pinlabel \scalebox{0.55}{$\textcolor{blue}{1}$} at 24 -5
                \pinlabel \scalebox{0.9}{$\lambda$} at 34 8   
			\endlabellist
			\centering
			\includegraphics[scale=1]{./RurrBurrBullRull}
	}\endxy
    \;\;\;+  k_3^{1,3}(s_2(\lambda))\xy (0,0)*{
			\labellist
			\small\hair 2pt 
			\pinlabel \scalebox{0.55}{$\textcolor{blue}{1}$} at 0 -5
			\pinlabel \scalebox{0.55}{$\textcolor{red}{2}$} at 8 -5
			\pinlabel \scalebox{0.55}{$\textcolor{red}{2}$} at 24 -5
			\pinlabel \scalebox{0.55}{$\textcolor{blue}{1}$} at 32 -5
                \pinlabel \scalebox{0.9}{$\lambda$} at 34 8   
			\endlabellist
			\centering
			\includegraphics[scale=1]{./BurrRurrRullBull}
	}\endxy \\
    &  + \xy (0,0)*{
				\labellist
				\small\hair 2pt
				\pinlabel \scalebox{0.55}{$\textcolor{blue}{1}$} at 8 -5
				\pinlabel \scalebox{0.55}{$\textcolor{red}{2}$} at 1 -5
				\pinlabel \scalebox{0.55}{$\textcolor{red}{2}$} at 31 -5
				\pinlabel \scalebox{0.55}{$\textcolor{blue}{1}$} at 24 -5
                    \pinlabel \scalebox{0.9}{$\lambda$} at 40 8    
				\endlabellist
				\centering
				\includegraphics[scale=1]{./RuBurBulRu}
		}\endxy
        \;\;\; + \xy (0,0)*{
				\labellist
				\small\hair 2pt
				\pinlabel \scalebox{0.55}{$\textcolor{blue}{1}$} at 1 -5
				\pinlabel \scalebox{0.55}{$\textcolor{red}{2}$} at 8 -5
				\pinlabel \scalebox{0.55}{$\textcolor{red}{2}$} at 24 -5
				\pinlabel \scalebox{0.55}{$\textcolor{blue}{1}$} at 31 -5
                    \pinlabel \scalebox{0.9}{$\lambda$} at 40 8    
				\endlabellist
				\centering
				\includegraphics[scale=1]{./BuRurRulBu}
		}\endxy
        \;\;\;\;, -\xy (0,0)*{
			\labellist
			\small\hair 2pt 
			\pinlabel \scalebox{0.55}{$\textcolor{blue}{1}$} at 8 -5
			\pinlabel \scalebox{0.55}{$\textcolor{red}{2}$} at 0 -5
			\pinlabel \scalebox{0.55}{$\textcolor{red}{2}$} at 24 -5
			\pinlabel \scalebox{0.55}{$\textcolor{blue}{1}$} at 32 -5
                \pinlabel \scalebox{0.9}{$\lambda$} at 36 8   
			\endlabellist
			\centering
			\includegraphics[scale=1]{./RurrBurrRullBull}
	}\endxy\;\;
    \Big) \\
    & \\
    = &
    \Big(\xy (0,0)*{
			\labellist
			\small\hair 2pt 
			\pinlabel \scalebox{0.55}{$\textcolor{blue}{1}$} at 0 -5
			\pinlabel \scalebox{0.55}{$\textcolor{red}{2}$} at 8 -5
			\pinlabel \scalebox{0.55}{$\textcolor{red}{2}$} at 32 -5
			\pinlabel \scalebox{0.55}{$\textcolor{blue}{1}$} at 24 -5
                \pinlabel \scalebox{0.9}{$\lambda$} at 36 8   
			\endlabellist
			\centering
			\includegraphics[scale=1]{./BurrRurrBullRull}
	}\endxy
    \;\;\;,
    -\xy (0,0)*{
			\labellist
			\small\hair 2pt 
			\pinlabel \scalebox{0.55}{$\textcolor{blue}{1}$} at 8 -5
			\pinlabel \scalebox{0.55}{$\textcolor{red}{2}$} at 0 -5
			\pinlabel \scalebox{0.55}{$\textcolor{red}{2}$} at 32 -5
			\pinlabel \scalebox{0.55}{$\textcolor{blue}{1}$} at 24 -5
                \pinlabel \scalebox{0.9}{$\lambda$} at 36 8   
			\endlabellist
			\centering
			\includegraphics[scale=1]{./RurrBurrBullRull}
	}\endxy
    \;\;\; -\xy (0,0)*{
			\labellist
			\small\hair 2pt 
			\pinlabel \scalebox{0.55}{$\textcolor{blue}{1}$} at 0 -5
			\pinlabel \scalebox{0.55}{$\textcolor{red}{2}$} at 8 -5
			\pinlabel \scalebox{0.55}{$\textcolor{red}{2}$} at 24 -5
			\pinlabel \scalebox{0.55}{$\textcolor{blue}{1}$} at 32 -5
                \pinlabel \scalebox{0.9}{$\lambda$} at 36 8   
			\endlabellist
			\centering
			\includegraphics[scale=1]{./BurrRurrRullBull}
	}\endxy
    \;\;\; + \xy (0,0)*{
				\labellist
				\small\hair 2pt
				\pinlabel \scalebox{0.55}{$\textcolor{blue}{1}$} at 8 -5
				\pinlabel \scalebox{0.55}{$\textcolor{red}{2}$} at 1 -5
				\pinlabel \scalebox{0.55}{$\textcolor{red}{2}$} at 31 -5
				\pinlabel \scalebox{0.55}{$\textcolor{blue}{1}$} at 24 -5
                    \pinlabel \scalebox{0.9}{$\lambda$} at 40 8    
				\endlabellist
				\centering
				\includegraphics[scale=1]{./RuBurBulRu}
		}\endxy
        \;\;\; + \xy (0,0)*{
				\labellist
				\small\hair 2pt
				\pinlabel \scalebox{0.55}{$\textcolor{blue}{1}$} at 1 -5
				\pinlabel \scalebox{0.55}{$\textcolor{red}{2}$} at 8 -5
				\pinlabel \scalebox{0.55}{$\textcolor{red}{2}$} at 24 -5
				\pinlabel \scalebox{0.55}{$\textcolor{blue}{1}$} at 31 -5
                    \pinlabel \scalebox{0.9}{$\lambda$} at 40 8    
				\endlabellist
				\centering
				\includegraphics[scale=1]{./BuRurRulBu}
		}\endxy
        \;\;\;, -\xy (0,0)*{
			\labellist
			\small\hair 2pt 
			\pinlabel \scalebox{0.55}{$\textcolor{blue}{1}$} at 8 -5
			\pinlabel \scalebox{0.55}{$\textcolor{red}{2}$} at 0 -5
			\pinlabel \scalebox{0.55}{$\textcolor{red}{2}$} at 24 -5
			\pinlabel \scalebox{0.55}{$\textcolor{blue}{1}$} at 32 -5
                \pinlabel \scalebox{0.9}{$\lambda$} at 36 8   
			\endlabellist
			\centering
			\includegraphics[scale=1]{./RurrBurrRullBull}
	}\endxy\;\;
    \Big) \\
    & \\
    = & \tilde{\mathcal{T}}_{1,-1}'\left(\xy (0,0)*{
			\labellist
			\small\hair 2pt
			\pinlabel \scalebox{0.7}{$\textcolor{red}{2}$} at 0 -5
			\pinlabel \scalebox{0.7}{$\textcolor{red}{2}$} at 16 -5
			\pinlabel \scalebox{0.9}{$s_1(\lambda)$} at 22 8
			\endlabellist 
			\centering 
			\includegraphics[scale=1.3]{./RurRul}
	}\endxy\quad\;\;\right)
    \end{align*}

\begin{align*}
\alpha\tilde{\mathcal{T}}_{2,1}''\left(\beta\left(\xy (0,0)*{
			\labellist
			\small\hair 2pt
			\pinlabel \scalebox{0.7}{$\textcolor{blue}{1}$} at 1 -5
			\pinlabel \scalebox{0.9}{$s_2(\lambda)$} at 22 7
			\endlabellist 
			\centering 
			\includegraphics[scale=1.3]{./BrCap1}
	}\endxy\quad\;\;\right)\right) 
    =& 
    \alpha\tilde{\mathcal{T}}_{2,1}''\left(-k_1^2(s_2(\lambda))
\xy (0,0)*{
			\labellist
			\small\hair 2pt
			\pinlabel \scalebox{0.7}{$\textcolor{blue}{1}$} at 1 -5
			\pinlabel \scalebox{0.9}{$s_2(\lambda)$} at 24 8
			\endlabellist 
			\centering 
			\includegraphics[scale=1.3]{./BrCap1}
	}\endxy\quad\;\;\;\right)
    \\[1ex]
    =& \alpha\left(-k_1^2(s_2(\lambda))k_1^0(s_2(\lambda)\left(-k_3^3(s_2(\lambda))
    \xy (0,0)*{
			\labellist
			\small\hair 2pt
			\pinlabel \scalebox{0.7}{$\textcolor{red}{2}$} at 0 -5
			\pinlabel \scalebox{0.7}{$\textcolor{blue}{1}$} at 10 -5
			\pinlabel \scalebox{0.9}{$\lambda$} at 32 8
			\endlabellist 
			\centering 
			\includegraphics[scale=1.3]{./RrCap3BrCap1}
	}\endxy\; +
    k_3^2(s_2(\lambda))
    \xy (0,0)*{
			\labellist
			\small\hair 2pt
			\pinlabel \scalebox{0.7}{$\textcolor{red}{2}$} at 10 -5
			\pinlabel \scalebox{0.7}{$\textcolor{blue}{1}$} at 0 -5
			\pinlabel \scalebox{0.9}{$\lambda$} at 32 8
			\endlabellist 
			\centering 
			\includegraphics[scale=1.3]{./BrCap3RrCap1}
	}\endxy\mspace{12mu}\right)\right)
    \\[1ex]
   = & (-1)^{\overline{\lambda}_3+1}\left(k_1^3(s_1(\lambda))
    \xy (0,0)*{
			\labellist
			\small\hair 2pt
			\pinlabel \scalebox{0.7}{$\textcolor{red}{2}$} at 0 -5
			\pinlabel \scalebox{0.7}{$\textcolor{blue}{1}$} at 10 -5
			\pinlabel \scalebox{0.9}{$\lambda$} at 32 8
			\endlabellist 
			\centering 
			\includegraphics[scale=1.3]{./RrCap3BrCap1}
	}\endxy\;\; -
    k_1^2(s_1(\lambda))
    \xy (0,0)*{
			\labellist
			\small\hair 2pt
			\pinlabel \scalebox{0.7}{$\textcolor{red}{2}$} at 10 -5
			\pinlabel \scalebox{0.7}{$\textcolor{blue}{1}$} at 0 -5
			\pinlabel \scalebox{0.9}{$\lambda$} at 32 8
			\endlabellist 
			\centering 
			\includegraphics[scale=1.3]{./BrCap3RrCap1}
	}\endxy\;\;\right) 
    \\[1ex]
     = & \tilde{\mathcal{T}}_{1,-1}'\left(\xy (0,0)*{
			\labellist
			\small\hair 2pt
			\pinlabel \scalebox{0.7}{$\textcolor{red}{2}$} at 1 -5
			\pinlabel \scalebox{0.9}{$s_1(\lambda)$} at 22 7
			\endlabellist 
			\centering 
			\includegraphics[scale=1.3]{./RrCap1Thin}
	}\endxy\quad\;\;
    \right)
    \end{align*}

    \begin{align*}
\alpha\tilde{\mathcal{T}}_{2,1}''\left(\beta\left(\xy (0,0)*{
			\labellist
			\small\hair 2pt
			\pinlabel \scalebox{0.7}{$\textcolor{blue}{1}$} at 1 9
			\pinlabel \scalebox{0.9}{$s_2(\lambda)$} at 22 2
			\endlabellist 
			\centering 
			\includegraphics[scale=1.3]{./BlCup1}
	}\endxy\quad\;\;\right)\right) 
    &= \alpha\tilde{\mathcal{T}}_{2,1}''\left(-k_1^2(s_2(\lambda))
\xy (0,0)*{
			\labellist
			\small\hair 2pt
			\pinlabel \scalebox{0.7}{$\textcolor{blue}{1}$} at 1 9
			\pinlabel \scalebox{0.9}{$s_2(\lambda)$} at 24 2
			\endlabellist 
			\centering 
			\includegraphics[scale=1.3]{./BlCup1}
	}\endxy\quad\;\;\;\right)
    \\[1ex]
    & \mspace{-60mu}= \alpha\left((-1)^{s_2(\lambda)_1+1}k_1^2(s_2(\lambda))k_1^0(s_2(\lambda))\left(k_3^1(s_2(\lambda))
    \xy (0,0)*{
			\labellist
			\small\hair 2pt
			\pinlabel \scalebox{0.7}{$\textcolor{red}{2}$} at 0 14
			\pinlabel \scalebox{0.7}{$\textcolor{blue}{1}$} at 10 14
			\pinlabel \scalebox{0.9}{$\lambda$} at 31 6
			\endlabellist 
			\centering 
			\includegraphics[scale=1.3]{./RlCup3BlCup1}
	}\endxy\;
-
    k_3^0(s_2(\lambda))
    \xy (0,0)*{
			\labellist
			\small\hair 2pt
			\pinlabel \scalebox{0.7}{$\textcolor{red}{2}$} at 10 14
			\pinlabel \scalebox{0.7}{$\textcolor{blue}{1}$} at 0 14
			\pinlabel \scalebox{0.9}{$\lambda$} at 31 6
			\endlabellist 
			\centering 
			\includegraphics[scale=1.3]{./BlCup3RlCup1}
	}\endxy\mspace{10mu}\right)\right)\\
     \\[1ex]
    & \mspace{-60mu}=(-1)^{\lambda_3}\left(k_1^1(s_1(\lambda))
    \xy (0,0)*{
			\labellist
			\small\hair 2pt
			\pinlabel \scalebox{0.7}{$\textcolor{red}{2}$} at 0 14
			\pinlabel \scalebox{0.7}{$\textcolor{blue}{1}$} at 10 14
			\pinlabel \scalebox{0.9}{$\lambda$} at 31 6
			\endlabellist 
			\centering 
			\includegraphics[scale=1.3]{./RlCup3BlCup1}
	}\endxy\;\; -
    k_1^0(s_1(\lambda))
    \xy (0,0)*{
			\labellist
			\small\hair 2pt
			\pinlabel \scalebox{0.7}{$\textcolor{red}{2}$} at 10 14
			\pinlabel \scalebox{0.7}{$\textcolor{blue}{1}$} at 0 14
			\pinlabel \scalebox{0.9}{$\lambda$} at 32 8
			\endlabellist 
			\centering 
			\includegraphics[scale=1.3]{./BlCup3RlCup1}
	}\endxy\mspace{10mu}\right) 
    \\[1ex]
    & \mspace{-60mu}= \tilde{\mathcal{T}}_{1,-1}'\left(\xy (0,0)*{
			\labellist
			\small\hair 2pt
			\pinlabel \scalebox{0.7}{$\textcolor{red}{2}$} at 1 11
			\pinlabel \scalebox{0.9}{$s_1(\lambda)$} at 24 2
			\endlabellist 
			\centering 
			\includegraphics[scale=1.3]{./RlCup1Thin}
	}\endxy\quad\;\;\;
    \right)
    \end{align*}

    \begin{align*}
\alpha\tilde{\mathcal{T}}_{2,1}''\left(\beta\left(\xy (0,0)*{
			\labellist
			\small\hair 2pt
			\pinlabel \scalebox{0.7}{$\textcolor{blue}{1}$} at 1 -5
			\pinlabel \scalebox{0.9}{$s_2(\lambda)$} at 22 7
			\endlabellist 
			\centering 
			\includegraphics[scale=1.3]{./BlCap1}
	}\endxy\quad\;\;\right)\right)  
    &= 
    \alpha\tilde{\mathcal{T}}_{2,1}''\left(-k_1^0(s_2(\lambda))
\xy (0,0)*{
			\labellist
			\small\hair 2pt
			\pinlabel \scalebox{0.7}{$\textcolor{blue}{1}$} at 1 -5
			\pinlabel \scalebox{0.9}{$s_2(\lambda)$} at 24 8
			\endlabellist 
			\centering 
			\includegraphics[scale=1.3]{./BlCap1}
	}\endxy\quad\;\;\;\right) \mspace{-420mu} 
    \\[1ex]
    & \mspace{-80mu}= \alpha\left((-1)^{s_2(\lambda)_1+1}k_1^0(s_2(\lambda))k_1^2(s_2(\lambda))\left(k_3^1(s_2(\lambda))
    \xy (0,0)*{
			\labellist
			\small\hair 2pt
			\pinlabel \scalebox{0.7}{$\textcolor{red}{2}$} at 0 -5
			\pinlabel \scalebox{0.7}{$\textcolor{blue}{1}$} at 10 -5
			\pinlabel \scalebox{0.9}{$\lambda$} at 32 8
			\endlabellist 
			\centering 
			\includegraphics[scale=1.3]{./RlCap3BlCap1Big}
	}\endxy\;\; 
       -
   k_3^2(s_2(\lambda))
    \xy (0,0)*{
			\labellist
			\small\hair 2pt
			\pinlabel \scalebox{0.7}{$\textcolor{red}{2}$} at 10 -5
			\pinlabel \scalebox{0.7}{$\textcolor{blue}{1}$} at 0 -5
			\pinlabel \scalebox{0.9}{$\lambda$} at 32 8
			\endlabellist 
			\centering 
			\includegraphics[scale=1.3]{./BlCap3RlCap1Big}
	}\endxy\mspace{10mu}\right)\right)
     \\[1ex]
    & \mspace{-80mu}=(-1)^{\lambda_3}\left(k_1^1(s_1(\lambda))
    \xy (0,0)*{
			\labellist
			\small\hair 2pt
			\pinlabel \scalebox{0.7}{$\textcolor{red}{2}$} at 0 -5
			\pinlabel \scalebox{0.7}{$\textcolor{blue}{1}$} at 10 -5
			\pinlabel \scalebox{0.9}{$\lambda$} at 32 8
			\endlabellist 
			\centering 
			\includegraphics[scale=1.3]{./RlCap3BlCap1Big}
	}\endxy\;\; -
    k_1^2(s_1(\lambda))
    \xy (0,0)*{
			\labellist
			\small\hair 2pt
			\pinlabel \scalebox{0.7}{$\textcolor{red}{2}$} at 10 -5
			\pinlabel \scalebox{0.7}{$\textcolor{blue}{1}$} at 0 -5
			\pinlabel \scalebox{0.9}{$\lambda$} at 32 8
			\endlabellist 
			\centering 
			\includegraphics[scale=1.3]{./BlCap3RlCap1Big}
	}\endxy\;\;\right) 
    \\[1ex]
    & \mspace{-80mu}= \tilde{\mathcal{T}}_{1,-1}'\left(\xy (0,0)*{
			\labellist
			\small\hair 2pt
			\pinlabel \scalebox{0.7}{$\textcolor{red}{2}$} at 1 -5
			\pinlabel \scalebox{0.9}{$s_1(\lambda)$} at 22 7
			\endlabellist 
			\centering 
			\includegraphics[scale=1.3]{./RlCap1}
	}\endxy\quad\;\;\right)
    \end{align*}

    \begin{align*}
\alpha\tilde{\mathcal{T}}_{2,1}''\left(\beta\left(\xy (0,0)*{
			\labellist
			\small\hair 2pt
			\pinlabel \scalebox{0.7}{$\textcolor{blue}{1}$} at 1 9
			\pinlabel \scalebox{0.9}{$s_2(\lambda)$} at 22 2
			\endlabellist 
			\centering 
			\includegraphics[scale=1.3]{./BrCup1}
	}\endxy\quad\;\;\right)\right)  
    &= 
    \alpha\tilde{\mathcal{T}}_{2,1}''\left(-k_1^0(s_2(\lambda))
\xy (0,0)*{
			\labellist
			\small\hair 2pt
			\pinlabel \scalebox{0.7}{$\textcolor{blue}{1}$} at 1 9
			\pinlabel \scalebox{0.9}{$s_2(\lambda)$} at 24 2
			\endlabellist 
			\centering 
			\includegraphics[scale=1.3]{./BrCup1}
	}\endxy\mspace{32mu}\right) 
     \\[1ex]
    & = \alpha\left(-k_1^0(s_2(\lambda))k_1^2(s_2(\lambda))\left(k_3^3(s_2(\lambda))
    \xy (0,0)*{
			\labellist
			\small\hair 2pt
			\pinlabel \scalebox{0.7}{$\textcolor{red}{2}$} at 0 14
			\pinlabel \scalebox{0.7}{$\textcolor{blue}{1}$} at 10 14
			\pinlabel \scalebox{0.9}{$\lambda$} at 32 2
			\endlabellist 
			\centering 
			\includegraphics[scale=1.3]{./RrCup3BrCup1}
	}\endxy\;\; 
   -
    k_3^0(s_2(\lambda))
    \xy (0,0)*{
			\labellist
			\small\hair 2pt
			\pinlabel \scalebox{0.7}{$\textcolor{red}{2}$} at 10 14
			\pinlabel \scalebox{0.7}{$\textcolor{blue}{1}$} at 0 14
			\pinlabel \scalebox{0.9}{$\lambda$} at 32 8
			\endlabellist 
			\centering 
			\includegraphics[scale=1.3]{./BrCup3RrCup1}
	}\endxy\mspace{10mu}\right)\right)
    \\[1ex]
        & =(-1)^{\overline{\lambda}_3}\left(k_1^3(s_1(\lambda))
    \xy (0,0)*{
			\labellist
			\small\hair 2pt
			\pinlabel \scalebox{0.7}{$\textcolor{red}{2}$} at 0 14
			\pinlabel \scalebox{0.7}{$\textcolor{blue}{1}$} at 10 14
			\pinlabel \scalebox{0.9}{$\lambda$} at 32 2
			\endlabellist 
			\centering 
			\includegraphics[scale=1.3]{./RrCup3BrCup1}
	}\endxy\;\; -
   k_1^0(s_1(\lambda))
    \xy (0,0)*{
			\labellist
			\small\hair 2pt
			\pinlabel \scalebox{0.7}{$\textcolor{red}{2}$} at 10 14
			\pinlabel \scalebox{0.7}{$\textcolor{blue}{1}$} at 0 14
			\pinlabel \scalebox{0.9}{$\lambda$} at 32 2
			\endlabellist 
			\centering 
			\includegraphics[scale=1.3]{./BrCup3RrCup1}
	}\endxy\;\;\right) 
    \\[1ex]
    & = \tilde{\mathcal{T}}_{1,-1}'\left(\xy (0,0)*{
			\labellist
			\small\hair 2pt
			\pinlabel \scalebox{0.7}{$\textcolor{red}{2}$} at 1 9
			\pinlabel \scalebox{0.9}{$s_1(\lambda)$} at 24 2
			\endlabellist 
			\centering 
			\includegraphics[scale=1.3]{./RrCup1Thin}
	}\endxy\quad\;\;\;
    \right)
    \end{align*}

The remaining element of the proof is comparing $$\tilde{\mathcal{T}}_{1,-1}'\left(\xy (0,0)*{
			\labellist
			\small\hair 2pt
			\pinlabel \scalebox{0.7}{$\textcolor{red}{2}$} at 1 -5
			\pinlabel \scalebox{0.9}{$s_1(\lambda)$} at 12 8
			\endlabellist 
			\centering 
			\includegraphics[scale=1.3]{./Ruo}
	}\endxy\quad\;\;\;\right)=\left(\xy (0,0)*{
		\labellist
		\small\hair 2pt
		\pinlabel \scalebox{0.7}{$\textcolor{blue}{1}$} at 1 -5
		\pinlabel \scalebox{0.7}{$\textcolor{red}{2}$} at 9 -5
       	\pinlabel \scalebox{0.9}{$\lambda$} at 14 8
		\endlabellist
		\centering
		\includegraphics[scale=1.3]{./BuRuo}
	}\endxy
    \;\;\;,\xy (0,0)*{
		\labellist
		\small\hair 2pt
		\pinlabel \scalebox{0.7}{$\textcolor{blue}{1}$} at 9 -5
		\pinlabel \scalebox{0.7}{$\textcolor{red}{2}$} at 1 -5
       	\pinlabel \scalebox{0.9}{$\lambda$} at 14 8
		\endlabellist
		\centering
		\includegraphics[scale=1.3]{./RuoBu}
	}\endxy\;\;
    \right)$$ and $$\alpha\tilde{\mathcal{T}}_{2,1}''\left(\beta\left( \xy (0,0)*{
			\labellist
			\small\hair 2pt
			\pinlabel \scalebox{0.7}{$\textcolor{blue}{1}$} at 1 -5
			\pinlabel \scalebox{0.9}{$s_2(\lambda)$} at 12 8
			\endlabellist 
			\centering 
			\includegraphics[scale=1.3]{./Buo}
	}\endxy\quad\quad \right)\right)=\left(\xy (0,0)*{
		\labellist
		\small\hair 2pt
		\pinlabel \scalebox{0.7}{$\textcolor{blue}{1}$} at 1 -5
		\pinlabel \scalebox{0.7}{$\textcolor{red}{2}$} at 9 -5
       	\pinlabel \scalebox{0.9}{$\lambda$} at 14 8
		\endlabellist
		\centering
		\includegraphics[scale=1.3]{./BuoRu}
	}\endxy
    \;\;\;,\xy (0,0)*{
		\labellist
		\small\hair 2pt
		\pinlabel \scalebox{0.7}{$\textcolor{blue}{1}$} at 9 -5
		\pinlabel \scalebox{0.7}{$\textcolor{red}{2}$} at 1 -5
       	\pinlabel \scalebox{0.9}{$\lambda$} at 14 8
		\endlabellist
		\centering
		\includegraphics[scale=1.3]{./RuBuo}
	}\endxy\;\;\;
    \right),$$
\eqskip 
which are equal up to homotopy by \eqref{eq:dotshift}.
\end{proof}

We also mention the following result that we will be using often to prove the main theorem:
\begin{lem}\label{lem:KMPreserve}
    $\tilde{\mathcal{T}}_{1,-1}'$ and $\alpha\tilde{\mathcal{T}}_{2,1}''\beta$ preserve all the KM identities.
\end{lem}
\begin{proof}
    All the component 2-functors of these 2-functors are either 2-isomorphisms (which clearly preserve defining axioms) or preserve the KM identities by \cite[Section 4]{abram2022categorification}.
\end{proof}

\section{The Proof of \autoref{thm:BigThmPrime}}\label{sec:MainProof}

With the various definitions of \autoref{sec:BraidAction} and \autoref{Sec:Eval2Fr}, as well as \autoref{lem:EvRel}, we are now almost prepared to prove \autoref{thm:BigThmPrime} - we merely require some extra 2-functors to account for $\Ev'$ having a different domain 2-category to $\tilde{\mathcal{T}}_{1,-1}'$ and $\tilde{\mathcal{T}}_{2,1}''$, and for categorifying the powers of $q$ and signs found in \autoref{sec:DecatT}. We also remind the reader that $k_i^{a_1,\dots,a_n}(\mu)$ omits the argument $\mu$ when it is equal to $\lambda$, but retains it otherwise (generally when it is $s_1(\lambda)$ or $s_2(\lambda)$).

\subsection{Embeddings}

We define 2-embeddings $\iota, \iota':\naffu{3}\to\affu{3}$ using other component 2-functors. First, we define $\tilde{\omega}=\zeta\omega\zeta^{-1}:\naffu{3}\to\naffu{3}$. Second, we define 2-functors $\eta,\eta':\naffu{3}\to\affu{3}$ as follows:

For $\eta$,\begin{itemize}[wide,labelindent=0pt,itemsep=5pt]
    \item On objects, $\lambda\overset{\eta}{\mapsto} -s_1(\lambda)$.
    \item On 1-morphisms,
\begin{gather*}
\E_1\oneid_\lambda\overset{\eta}{\mapsto} \E_1\oneid_{-s_1(\lambda)} ,
\mspace{35mu}
\F_1\oneid_\lambda\overset{\eta}{\mapsto} \F_1\oneid_{-s_1(\lambda)} ,
\mspace{35mu}
\E_2\oneid_\lambda\overset{\eta}{\mapsto} \E_3\oneid_{-s_1(\lambda)} ,
\mspace{35mu}
\F_2\oneid_\lambda\overset{\eta}{\mapsto} \F_3\oneid_{-s_1(\lambda)} .
\end{gather*}

    \item On 2-morphisms, 
\begingroup\allowdisplaybreaks
\begin{align}
\xy (0,1)*{
			\labellist
			\small\hair 2pt
			\pinlabel \scalebox{0.7}{$\textcolor{blue}{1}$} at 1 -5
			\pinlabel \scalebox{0.9}{$\lambda$} at 6 8
			\endlabellist 
			\centering 
			\includegraphics[scale=1.3]{./Buo}
	}\endxy\;\;
    &\overset{\eta}{\mapsto}
    \,\xy (0,0)*{
			\labellist
			\small\hair 2pt
			\pinlabel \scalebox{0.7}{$\textcolor{blue}{1}$} at 1 -5
			\pinlabel \scalebox{0.9}{$-s_1(\lambda)$} at 16 8
			\endlabellist 
			\centering 
			\includegraphics[scale=1.3]{./Buo}
	}\endxy
    &
\xy (0,1)*{
			\labellist
			\small\hair 2pt
			\pinlabel \scalebox{0.7}{$\textcolor{red}{2}$} at 1 -5
			\pinlabel \scalebox{0.9}{$\lambda$} at 6 8
			\endlabellist 
			\centering 
			\includegraphics[scale=1.3]{./Ruo}
	}\endxy\;\;
    &\overset{\eta}{\mapsto}
    \,\xy (0,0)*{
			\labellist
			\small\hair 2pt
			\pinlabel \scalebox{0.7}{$\textcolor{black}{3}$} at 1 -5
			\pinlabel \scalebox{0.9}{$-s_1(\lambda)$} at 16 8
			\endlabellist 
			\centering 
			\includegraphics[scale=1.3]{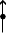}
	}\endxy    
\\[2ex]
\xy (0,1)*{
			\labellist
			\small\hair 2pt
			\pinlabel \scalebox{0.7}{$\textcolor{blue}{1}$} at 0 -5
			\pinlabel \scalebox{0.7}{$\textcolor{blue}{1}$} at 16 -5
			\pinlabel \scalebox{0.9}{$\lambda$} at 18 8
			\endlabellist 
			\centering 
			\includegraphics[scale=1.3]{./BurBul}
	}\endxy\;\; 
    &\overset{\eta}{\mapsto}
    \xy (0,0)*{
			\labellist
			\small\hair 2pt
			\pinlabel \scalebox{0.7}{$\textcolor{blue}{1}$} at 0 -5
			\pinlabel \scalebox{0.7}{$\textcolor{blue}{1}$} at 16 -5
			\pinlabel \scalebox{0.9}{$-s_1(\lambda)$} at 24 8
			\endlabellist 
			\centering 
			\includegraphics[scale=1.3]{./BurBul}
	}\endxy
    &
\xy (0,1)*{
			\labellist
			\small\hair 2pt
			\pinlabel \scalebox{0.7}{$\textcolor{red}{2}$} at 0 -5
			\pinlabel \scalebox{0.7}{$\textcolor{red}{2}$} at 16 -5
			\pinlabel \scalebox{0.9}{$\lambda$} at 18 8
			\endlabellist 
			\centering 
			\includegraphics[scale=1.3]{./RurRul}
	}\endxy\;\;&\overset{\eta}{\mapsto}
    \xy (0,0)*{
			\labellist
			\small\hair 2pt
			\pinlabel \scalebox{0.7}{$\textcolor{black}{3}$} at 0 -5
			\pinlabel \scalebox{0.7}{$\textcolor{black}{3}$} at 16 -5
			\pinlabel \scalebox{0.9}{$-s_1(\lambda)$} at 24 8
			\endlabellist 
			\centering 
			\includegraphics[scale=1.3]{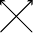}
	}\endxy
\\[2ex]
\xy (0,1)*{
			\labellist
			\small\hair 2pt
			\pinlabel \scalebox{0.7}{$\textcolor{blue}{1}$} at 0 -5
			\pinlabel \scalebox{0.7}{$\textcolor{red}{2}$} at 16 -5
			\pinlabel \scalebox{0.9}{$\lambda$} at 18 8
			\endlabellist 
			\centering 
			\includegraphics[scale=1.3]{./BurRul}
	}\endxy\;\;
    &\overset{\eta}{\mapsto}
    \xy (0,0)*{
			\labellist
			\small\hair 2pt
			\pinlabel \scalebox{0.7}{$\textcolor{blue}{1}$} at 0 -5
			\pinlabel \scalebox{0.7}{$\textcolor{black}{3}$} at 16 -5
			\pinlabel \scalebox{0.9}{$-s_1(\lambda)$} at 24 8
			\endlabellist 
			\centering 
			\includegraphics[scale=1.3]{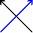}
	}\endxy
    &
    \xy (0,1)*{
			\labellist
			\small\hair 2pt
			\pinlabel \scalebox{0.7}{$\textcolor{blue}{1}$} at 16 -5
			\pinlabel \scalebox{0.7}{$\textcolor{red}{2}$} at 0 -5
			\pinlabel \scalebox{0.9}{$\lambda$} at 18 8
			\endlabellist 
			\centering 
			\includegraphics[scale=1.3]{./RurBul}
	}\endxy\;\;&\overset{\eta}{\mapsto}
    -\xy (0,0)*{
			\labellist
			\small\hair 2pt
			\pinlabel \scalebox{0.7}{$\textcolor{blue}{1}$} at 16 -5
			\pinlabel \scalebox{0.7}{$\textcolor{black}{3}$} at 0 -5
			\pinlabel \scalebox{0.9}{$-s_1(\lambda)$} at 24 8
			\endlabellist 
			\centering 
			\includegraphics[scale=1.3]{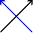}
	}\endxy
\\[2ex]
\xy (0,1)*{
			\labellist
			\small\hair 2pt
			\pinlabel \scalebox{0.7}{$\textcolor{blue}{1}$} at 1 -5
			\pinlabel \scalebox{0.9}{$\lambda$} at 18 8
			\endlabellist 
			\centering 
			\includegraphics[scale=1.3]{./BrCap1}
	}\endxy\;\;
    &\overset{\eta}{\mapsto}
     \xy (0,0)*{
			\labellist
			\small\hair 2pt
			\pinlabel \scalebox{0.7}{$\textcolor{blue}{1}$} at 1 -5
			\pinlabel \scalebox{0.9}{$-s_1(\lambda)$} at 24 8
			\endlabellist 
			\centering 
			\includegraphics[scale=1.3]{./BrCap1}
	}\endxy
    &
    \xy (0,1)*{
			\labellist
			\small\hair 2pt
			\pinlabel \scalebox{0.7}{$\textcolor{blue}{1}$} at 1 -5
			\pinlabel \scalebox{0.9}{$\lambda$} at 18 8
			\endlabellist 
			\centering 
			\includegraphics[scale=1.3]{./BlCap1}
	}\endxy\;\;
    &\overset{\eta}{\mapsto}
    (-1)^{\bar{\lambda}_1} \xy (0,0)*{
			\labellist
			\small\hair 2pt
			\pinlabel \scalebox{0.7}{$\textcolor{blue}{1}$} at 1 -5
			\pinlabel \scalebox{0.9}{$-s_1(\lambda)$} at 24 8
			\endlabellist 
			\centering 
			\includegraphics[scale=1.3]{./BlCap1}
	}\endxy
\\[2ex]
\xy (0,1)*{
			\labellist
			\small\hair 2pt
			\pinlabel \scalebox{0.7}{$\textcolor{blue}{1}$} at 1 9
			\pinlabel \scalebox{0.9}{$\lambda$} at 19 2
			\endlabellist 
			\centering 
			\includegraphics[scale=1.3]{./BrCup1}
	}\endxy\;\;
    &\overset{\eta}{\mapsto}
    \xy (0,0)*{
			\labellist
			\small\hair 2pt
			\pinlabel \scalebox{0.7}{$\textcolor{blue}{1}$} at 1 9
			\pinlabel \scalebox{0.9}{$-s_1(\lambda)$} at 26 2
			\endlabellist 
			\centering 
			\includegraphics[scale=1.3]{./BrCup1}
	}\endxy
    &
    \xy (0,1)*{
			\labellist
			\small\hair 2pt
			\pinlabel \scalebox{0.7}{$\textcolor{blue}{1}$} at 1 9
			\pinlabel \scalebox{0.9}{$\lambda$} at 19 2
			\endlabellist 
			\centering 
			\includegraphics[scale=1.3]{./BlCup1}
	}\endxy\;\;
    &\overset{\eta}{\mapsto}
    (-1)^{\bar{\lambda}_1} \xy (0,0)*{
			\labellist
			\small\hair 2pt
			\pinlabel \scalebox{0.7}{$\textcolor{blue}{1}$} at 1 9
			\pinlabel \scalebox{0.9}{$-s_1(\lambda)$} at 26 2
			\endlabellist 
			\centering 
			\includegraphics[scale=1.3]{./BlCup1}
	}\endxy
\\[2ex]
    \xy (0,1)*{
			\labellist
			\small\hair 2pt
			\pinlabel \scalebox{0.7}{$\textcolor{red}{2}$} at 1 -4
			\pinlabel \scalebox{0.9}{$\lambda$} at 18 8
			\endlabellist 
			\centering 
			\includegraphics[scale=1.3]{./RrCap1Thin}
	}\endxy\;\;
    &\overset{\eta}{\mapsto}
     \xy (0,0)*{
			\labellist
			\small\hair 2pt
			\pinlabel \scalebox{0.7}{$\textcolor{black}{3}$} at 1 -5
			\pinlabel \scalebox{0.9}{$-s_1(\lambda)$} at 24 8
			\endlabellist 
			\centering 
			\includegraphics[scale=1.3]{./KrCap1Thin}
	}\endxy
&
\xy (0,1)*{
			\labellist
			\small\hair 2pt
			\pinlabel \scalebox{0.7}{$\textcolor{red}{2}$} at 1 -4
			\pinlabel \scalebox{0.9}{$\lambda$} at 18 8
			\endlabellist 
			\centering 
			\includegraphics[scale=1.3]{./RlCap1}
	}\endxy\;\;
    &\overset{\eta}{\mapsto}
    (-1)^{\bar{\lambda}_2} \xy (0,0)*{
			\labellist
			\small\hair 2pt
			\pinlabel \scalebox{0.7}{$\textcolor{black}{3}$} at 1 -5
			\pinlabel \scalebox{0.9}{$-s_1(\lambda)$} at 24 8
			\endlabellist 
			\centering 
			\includegraphics[scale=1.3]{./KlCap1Thin}
	}\endxy
\\[2ex]
    \xy (0,1)*{
			\labellist
			\small\hair 2pt
			\pinlabel \scalebox{0.7}{$\textcolor{red}{2}$} at 1 10
			\pinlabel \scalebox{0.9}{$\lambda$} at 19 2
			\endlabellist 
			\centering 
			\includegraphics[scale=1.3]{./RrCup1Thin}
	}\endxy\;\;
    &\overset{\eta}{\mapsto}
    \xy (0,0)*{
			\labellist
			\small\hair 2pt
			\pinlabel \scalebox{0.7}{$\textcolor{black}{3}$} at 1 9
			\pinlabel \scalebox{0.9}{$-s_1(\lambda)$} at 26 2
			\endlabellist 
			\centering 
			\includegraphics[scale=1.3]{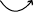}
	}\endxy
&
 \xy (0,1)*{
			\labellist
			\small\hair 2pt
			\pinlabel \scalebox{0.7}{$\textcolor{red}{2}$} at 1 10
			\pinlabel \scalebox{0.9}{$\lambda$} at 19 2
			\endlabellist 
			\centering 
			\includegraphics[scale=1.3]{./RlCup1Thin}
	}\endxy\;\;
    &\overset{\eta}{\mapsto}
    (-1)^{\bar{\lambda}_2} \xy (0,0)*{
			\labellist
			\small\hair 2pt
			\pinlabel \scalebox{0.7}{$\textcolor{black}{3}$} at 1 9
			\pinlabel \scalebox{0.9}{$-s_1(\lambda)$} at 26 2
			\endlabellist 
			\centering 
			\includegraphics[scale=1.3]{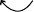}
	}\endxy
\end{align}
\endgroup

\end{itemize}

For $\eta'$:\begin{itemize}[wide,labelindent=0pt,itemsep=5pt]
    \item On objects, $\lambda\overset{\eta'}{\mapsto} -s_2(\lambda)$.
    \item On 1-morphisms,    
\begin{gather*}
\E_2\oneid_\lambda\overset{\eta'}{\mapsto} \E_2\oneid_{-s_2(\lambda)} ,
\mspace{35mu}
\F_2\oneid_\lambda\overset{\eta'}{\mapsto} \F_2\oneid_{-s_2(\lambda)} ,
\mspace{35mu}
\E_1\oneid_\lambda\overset{\eta'}{\mapsto} \E_3\oneid_{-s_2(\lambda)} ,
\mspace{35mu}
\F_1\oneid_\lambda\overset{\eta'}{\mapsto} \F_3\oneid_{-s_2(\lambda)} .
\end{gather*}

    \item On 2-morphisms, 
\begingroup\allowdisplaybreaks
\begin{align}
\xy (0,1)*{
			\labellist
			\small\hair 2pt
			\pinlabel \scalebox{0.7}{$\textcolor{red}{2}$} at 1 -5
			\pinlabel \scalebox{0.9}{$\lambda$} at 6 8
			\endlabellist 
			\centering 
			\includegraphics[scale=1.3]{./Ruo}
	}\endxy\;\;
    &\overset{\eta'}{\mapsto}
    \,\xy (0,0)*{
			\labellist
			\small\hair 2pt
			\pinlabel \scalebox{0.7}{$\textcolor{red}{2}$} at 1 -5
			\pinlabel \scalebox{0.9}{$-s_2(\lambda)$} at 16 8
			\endlabellist 
			\centering 
			\includegraphics[scale=1.3]{./Ruo}
	}\endxy
&
\xy (0,1)*{
			\labellist
			\small\hair 2pt
			\pinlabel \scalebox{0.7}{$\textcolor{blue}{1}$} at 1 -5
			\pinlabel \scalebox{0.9}{$\lambda$} at 6 8
			\endlabellist 
			\centering 
			\includegraphics[scale=1.3]{./Buo}
	}\endxy\;\;
    &\overset{\eta'}{\mapsto}
    \,\xy (0,0)*{
			\labellist
			\small\hair 2pt
			\pinlabel \scalebox{0.7}{$\textcolor{black}{3}$} at 1 -5
			\pinlabel \scalebox{0.9}{$-s_2(\lambda)$} at 16 8
			\endlabellist 
			\centering 
			\includegraphics[scale=1.3]{./KuoThin}
	}\endxy
\\[2ex]
\xy (0,1)*{
			\labellist
			\small\hair 2pt
			\pinlabel \scalebox{0.7}{$\textcolor{red}{2}$} at 0 -5
			\pinlabel \scalebox{0.7}{$\textcolor{red}{2}$} at 16 -5
			\pinlabel \scalebox{0.9}{$\lambda$} at 18 8
			\endlabellist 
			\centering 
			\includegraphics[scale=1.3]{./RurRul}
	}\endxy\;\; 
    &\overset{\eta'}{\mapsto}
    \xy (0,0)*{
			\labellist
			\small\hair 2pt
			\pinlabel \scalebox{0.7}{$\textcolor{red}{2}$} at 0 -5
			\pinlabel \scalebox{0.7}{$\textcolor{red}{2}$} at 16 -5
			\pinlabel \scalebox{0.9}{$-s_2(\lambda)$} at 24 8
			\endlabellist 
			\centering 
			\includegraphics[scale=1.3]{./RurRul}
	}\endxy
&
\xy (0,1)*{
			\labellist
			\small\hair 2pt
			\pinlabel \scalebox{0.7}{$\textcolor{blue}{1}$} at 0 -5
			\pinlabel \scalebox{0.7}{$\textcolor{blue}{1}$} at 16 -5
			\pinlabel \scalebox{0.9}{$\lambda$} at 18 8
			\endlabellist 
			\centering 
			\includegraphics[scale=1.3]{./BurBul}
	}\endxy\;\;
    &\overset{\eta'}{\mapsto}
    \xy (0,0)*{
			\labellist
			\small\hair 2pt
			\pinlabel \scalebox{0.7}{$\textcolor{black}{3}$} at 0 -5
			\pinlabel \scalebox{0.7}{$\textcolor{black}{3}$} at 16 -5
			\pinlabel \scalebox{0.9}{$-s_2(\lambda)$} at 24 8
			\endlabellist 
			\centering 
			\includegraphics[scale=1.3]{./KurKul}
	}\endxy
\\[2ex]
\xy (0,1)*{
			\labellist
			\small\hair 2pt
			\pinlabel \scalebox{0.7}{$\textcolor{blue}{1}$} at 16 -5
			\pinlabel \scalebox{0.7}{$\textcolor{red}{2}$} at 0 -5
			\pinlabel \scalebox{0.9}{$\lambda$} at 18 8
			\endlabellist 
			\centering 
			\includegraphics[scale=1.3]{./RurBul}
	}\endxy\;\;
    &\overset{\eta'}{\mapsto}
    \xy (0,0)*{
			\labellist
			\small\hair 2pt
			\pinlabel \scalebox{0.7}{$\textcolor{red}{2}$} at 0 -5
			\pinlabel \scalebox{0.7}{$\textcolor{black}{3}$} at 16 -5
			\pinlabel \scalebox{0.9}{$-s_2(\lambda)$} at 24 8
			\endlabellist 
			\centering 
			\includegraphics[scale=1.3]{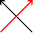}
	}\endxy
&
\xy (0,1)*{
			\labellist
			\small\hair 2pt
			\pinlabel \scalebox{0.7}{$\textcolor{blue}{1}$} at 0 -5
			\pinlabel \scalebox{0.7}{$\textcolor{red}{2}$} at 16 -5
			\pinlabel \scalebox{0.9}{$\lambda$} at 18 8
			\endlabellist 
			\centering 
			\includegraphics[scale=1.3]{./BurRul}
	}\endxy\;\;
    &\overset{\eta'}{\mapsto}
    -\xy (0,0)*{
			\labellist
			\small\hair 2pt
			\pinlabel \scalebox{0.7}{$\textcolor{red}{2}$} at 16 -5
			\pinlabel \scalebox{0.7}{$\textcolor{black}{3}$} at 0 -5
			\pinlabel \scalebox{0.9}{$-s_2(\lambda)$} at 24 8
			\endlabellist 
			\centering 
			\includegraphics[scale=1.3]{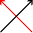}
	}\endxy
\\[2ex]
   \xy (0,1)*{
			\labellist
			\small\hair 2pt
			\pinlabel \scalebox{0.7}{$\textcolor{red}{2}$} at 1 -5
			\pinlabel \scalebox{0.9}{$\lambda$} at 18 8
			\endlabellist 
			\centering 
			\includegraphics[scale=1.3]{./RrCap1Thin}
	}\endxy\;\;
    &\overset{\eta'}{\mapsto}
    \xy (0,0)*{
			\labellist
			\small\hair 2pt
			\pinlabel \scalebox{0.7}{$\textcolor{red}{2}$} at 1 -5
			\pinlabel \scalebox{0.9}{$-s_2(\lambda)$} at 24 8
			\endlabellist 
			\centering 
			\includegraphics[scale=1.3]{./RrCap1Thin}
	}\endxy
&
 \xy (0,1)*{
			\labellist
			\small\hair 2pt
			\pinlabel \scalebox{0.7}{$\textcolor{red}{2}$} at 1 -5
			\pinlabel \scalebox{0.9}{$\lambda$} at 18 8
			\endlabellist 
			\centering 
			\includegraphics[scale=1.3]{./RlCap1}
	}\endxy\;\;
    &\overset{\eta'}{\mapsto}
    (-1)^{\bar{\lambda}_2} \xy (0,0)*{
			\labellist
			\small\hair 2pt
			\pinlabel \scalebox{0.7}{$\textcolor{red}{1}$} at 2 -5
			\pinlabel \scalebox{0.9}{$-s_2(\lambda)$} at 24 8
			\endlabellist 
			\centering 
			\includegraphics[scale=1.3]{./RlCap1}
	}\endxy
\\[2ex]
\xy (0,1)*{
			\labellist
			\small\hair 2pt
			\pinlabel \scalebox{0.7}{$\textcolor{red}{2}$} at 1 10
			\pinlabel \scalebox{0.9}{$\lambda$} at 19 2
			\endlabellist 
			\centering 
			\includegraphics[scale=1.3]{./RrCup1Thin}
	}\endxy\;\;
    &\overset{\eta'}{\mapsto}
    \xy (0,0)*{
			\labellist
			\small\hair 2pt
			\pinlabel \scalebox{0.7}{$\textcolor{red}{2}$} at 1 10
			\pinlabel \scalebox{0.9}{$-s_2(\lambda)$} at 26 2
			\endlabellist 
			\centering 
			\includegraphics[scale=1.3]{./RrCup1Thin}
	}\endxy
&
 \xy (0,1)*{
			\labellist
			\small\hair 2pt
			\pinlabel \scalebox{0.7}{$\textcolor{red}{2}$} at 1 10
			\pinlabel \scalebox{0.9}{$\lambda$} at 19 2
			\endlabellist 
			\centering 
			\includegraphics[scale=1.3]{./RlCup1Thin}
	}\endxy\;\;
    &\overset{\eta'}{\mapsto}
    (-1)^{\bar{\lambda}_2} \xy (0,0)*{
			\labellist
			\small\hair 2pt
			\pinlabel \scalebox{0.7}{$\textcolor{red}{2}$} at 1 10
			\pinlabel \scalebox{0.9}{$-s_2(\lambda)$} at 26 2
			\endlabellist 
			\centering 
			\includegraphics[scale=1.3]{./RlCup1Thin}
	}\endxy
\\[2ex]
\xy (0,1)*{
			\labellist
			\small\hair 2pt
			\pinlabel \scalebox{0.7}{$\textcolor{blue}{1}$} at 1 -5
			\pinlabel \scalebox{0.9}{$\lambda$} at 18 8
			\endlabellist 
			\centering 
			\includegraphics[scale=1.3]{./BrCap1}
	}\endxy\;\;
    &\overset{\eta'}{\mapsto}
    \xy (0,0)*{
			\labellist
			\small\hair 2pt
			\pinlabel \scalebox{0.7}{$\textcolor{black}{3}$} at 1 -5
			\pinlabel \scalebox{0.9}{$-s_2(\lambda)$} at 24 8
			\endlabellist 
			\centering 
			\includegraphics[scale=1.3]{./KrCap1Thin}
	}\endxy
&
 \xy (0,1)*{
			\labellist
			\small\hair 2pt
			\pinlabel \scalebox{0.7}{$\textcolor{blue}{1}$} at 1 -5
			\pinlabel \scalebox{0.9}{$\lambda$} at 18 8
			\endlabellist 
			\centering 
			\includegraphics[scale=1.3]{./BlCap1}
	}\endxy\;\;
    &\overset{\eta'}{\mapsto}
    (-1)^{\bar{\lambda}_1} \xy (0,0)*{
			\labellist
			\small\hair 2pt
			\pinlabel \scalebox{0.7}{$\textcolor{black}{3}$} at 1 -5
			\pinlabel \scalebox{0.9}{$-s_2(\lambda)$} at 24 8
			\endlabellist 
			\centering 
			\includegraphics[scale=1.3]{./KlCap1Thin}
	}\endxy
\\[2ex]
\xy (0,1)*{
			\labellist
			\small\hair 2pt
			\pinlabel \scalebox{0.7}{$\textcolor{blue}{1}$} at 1 9
			\pinlabel \scalebox{0.9}{$\lambda$} at 19 2
			\endlabellist 
			\centering 
			\includegraphics[scale=1.3]{./BrCup1}
	}\endxy\;\;
    &\overset{\eta'}{\mapsto}
     \xy (0,0)*{
			\labellist
			\small\hair 2pt
			\pinlabel \scalebox{0.7}{$\textcolor{black}{3}$} at 1 9
			\pinlabel \scalebox{0.9}{$-s_2(\lambda)$} at 26 2
			\endlabellist 
			\centering 
			\includegraphics[scale=1.3]{./KrCup1Thin}
	}\endxy
&
\xy (0,1)*{
			\labellist
			\small\hair 2pt
			\pinlabel \scalebox{0.7}{$\textcolor{blue}{1}$} at 1 9
			\pinlabel \scalebox{0.9}{$\lambda$} at 19 2
			\endlabellist 
			\centering 
			\includegraphics[scale=1.3]{./BlCup1}
	}\endxy\;\;
    &\overset{\eta'}{\mapsto}
    (-1)^{\bar{\lambda}_1} \xy (0,0)*{
			\labellist
			\small\hair 2pt
			\pinlabel \scalebox{0.7}{$\textcolor{black}{3}$} at 1 9
			\pinlabel \scalebox{0.9}{$-s_2(\lambda)$} at 26 2
			\endlabellist 
			\centering 
			\includegraphics[scale=1.3]{./KlCup1Thin}
	}\endxy
\end{align}\endgroup    

\end{itemize}
It is straightforward to check that $\eta$ and $\eta'$ preserve KM1-9 and are therefore are well-defined.

We now define $\iota=\eta\tilde{\omega}$ and $\iota'=\eta'\tilde{\omega}$. Explicitly, $\iota$ is given by:\begin{itemize}[wide,labelindent=0pt,itemsep=5pt]
    \item On objects, $\lambda\overset{\iota}{\mapsto} s_1(\lambda)$.
    \item On 1-morphisms,
\begin{gather*}
\E_1\oneid_\lambda\overset{\iota}{\mapsto} \F_1\oneid_{s_1(\lambda)} ,
\mspace{35mu}
\F_1\oneid_\lambda\overset{\iota}{\mapsto} \E_1\oneid_{s_1(\lambda)} ,
\mspace{35mu}
\E_2\oneid_\lambda\overset{\iota}{\mapsto} \F_3\oneid_{s_1(\lambda)} ,
\mspace{35mu}
\F_2\oneid_\lambda\overset{\iota}{\mapsto} \E_3\oneid_{s_1(\lambda)} .
\end{gather*}

\item On 2-morphisms, 
\begingroup\allowdisplaybreaks
\begin{align}
\xy (0,1)*{
			\labellist
			\small\hair 2pt
			\pinlabel \scalebox{0.7}{$\textcolor{blue}{1}$} at 1 -5
			\pinlabel \scalebox{0.9}{$\lambda$} at 6 8
			\endlabellist 
			\centering 
			\includegraphics[scale=1.3]{./Buo}
	}\endxy\;\;
    &\overset{\iota}{\mapsto}
    \,\xy (0,0)*{
			\labellist
			\small\hair 2pt
			\pinlabel \scalebox{0.7}{$\textcolor{blue}{1}$} at 1 -5
			\pinlabel \scalebox{0.9}{$s_1(\lambda)$} at 12 8
			\endlabellist 
			\centering 
			\includegraphics[scale=1.3]{./Bdo}
	}\endxy
&
\xy (0,1)*{
			\labellist
			\small\hair 2pt
			\pinlabel \scalebox{0.7}{$\textcolor{red}{2}$} at 1 -5
			\pinlabel \scalebox{0.9}{$\lambda$} at 6 8
			\endlabellist 
			\centering 
			\includegraphics[scale=1.3]{./Ruo}
	}\endxy\;\;
    &\overset{\iota}{\mapsto}
    \,\xy (0,0)*{
			\labellist
			\small\hair 2pt
			\pinlabel \scalebox{0.7}{$\textcolor{black}{3}$} at 1 -5
			\pinlabel \scalebox{0.9}{$s_1(\lambda)$} at 12 8
			\endlabellist 
			\centering 
			\includegraphics[scale=1.3]{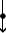}
	}\endxy
\\[2ex]
\xy (0,1)*{
			\labellist
			\small\hair 2pt
			\pinlabel \scalebox{0.7}{$\textcolor{blue}{1}$} at 0 -5
			\pinlabel \scalebox{0.7}{$\textcolor{blue}{1}$} at 16 -5
			\pinlabel \scalebox{0.9}{$\lambda$} at 18 8
			\endlabellist 
			\centering 
			\includegraphics[scale=1.3]{./BurBul}
	}\endxy\;\; 
    &\overset{\iota}{\mapsto}
    -\xy (0,0)*{
			\labellist
			\small\hair 2pt
			\pinlabel \scalebox{0.7}{$\textcolor{blue}{1}$} at 0 -5
			\pinlabel \scalebox{0.7}{$\textcolor{blue}{1}$} at 16 -5
			\pinlabel \scalebox{0.9}{$s_1(\lambda)$} at 22 8
			\endlabellist 
			\centering 
			\includegraphics[scale=1.3]{./BdrBdl}
	}\endxy
&
\xy (0,1)*{
			\labellist
			\small\hair 2pt
			\pinlabel \scalebox{0.7}{$\textcolor{red}{2}$} at 0 -5
			\pinlabel \scalebox{0.7}{$\textcolor{red}{2}$} at 16 -5
			\pinlabel \scalebox{0.9}{$\lambda$} at 18 8
			\endlabellist 
			\centering 
			\includegraphics[scale=1.3]{./RurRul}
	}\endxy\;\;
    &\overset{\iota}{\mapsto}
    -\xy (0,0)*{
			\labellist
			\small\hair 2pt
			\pinlabel \scalebox{0.7}{$\textcolor{black}{3}$} at 0 -5
			\pinlabel \scalebox{0.7}{$\textcolor{black}{3}$} at 16 -5
			\pinlabel \scalebox{0.9}{$s_1(\lambda)$} at 22 8
			\endlabellist 
			\centering 
			\includegraphics[scale=1.3]{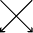}
	}\endxy
\\[2ex]
\xy (0,1)*{
			\labellist
			\small\hair 2pt
			\pinlabel \scalebox{0.7}{$\textcolor{blue}{1}$} at 0 -5
			\pinlabel \scalebox{0.7}{$\textcolor{red}{2}$} at 16 -5
			\pinlabel \scalebox{0.9}{$\lambda$} at 18 8
			\endlabellist 
			\centering 
			\includegraphics[scale=1.3]{./BurRul}
	}\endxy\;\;
    &\overset{\iota}{\mapsto}
   k_1^{2,3}\xy (0,0)*{
			\labellist
			\small\hair 2pt
			\pinlabel \scalebox{0.7}{$\textcolor{blue}{1}$} at 0 -5
			\pinlabel \scalebox{0.7}{$\textcolor{black}{3}$} at 16 -5
			\pinlabel \scalebox{0.9}{$s_1(\lambda)$} at 22 8
			\endlabellist 
			\centering 
			\includegraphics[scale=1.3]{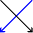}
	}\endxy
&
\xy (0,1)*{
			\labellist
			\small\hair 2pt
			\pinlabel \scalebox{0.7}{$\textcolor{blue}{1}$} at 16 -5
			\pinlabel \scalebox{0.7}{$\textcolor{red}{2}$} at 0 -5
			\pinlabel \scalebox{0.9}{$\lambda$} at 18 8
			\endlabellist 
			\centering 
			\includegraphics[scale=1.3]{./RurBul}
	}\endxy\;\;
    &\overset{\iota}{\mapsto}
    k_1^{0,1}\xy (0,0)*{
			\labellist
			\small\hair 2pt
			\pinlabel \scalebox{0.7}{$\textcolor{blue}{1}$} at 16 -5
			\pinlabel \scalebox{0.7}{$\textcolor{black}{3}$} at 0 -5
			\pinlabel \scalebox{0.9}{$s_1(\lambda)$} at 22 8
			\endlabellist 
			\centering 
			\includegraphics[scale=1.3]{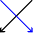}
	}\endxy
\\[2ex]
\xy (0,1)*{
			\labellist
			\small\hair 2pt
			\pinlabel \scalebox{0.7}{$\textcolor{blue}{1}$} at 1 -5
			\pinlabel \scalebox{0.9}{$\lambda$} at 18 8
			\endlabellist 
			\centering 
			\includegraphics[scale=1.3]{./BrCap1}
	}\endxy\;\;
    &\overset{\iota}{\mapsto}
    (-1)^{\lambda_1+1} \xy (0,0)*{
			\labellist
			\small\hair 2pt
			\pinlabel \scalebox{0.7}{$\textcolor{blue}{1}$} at 1 -5
			\pinlabel \scalebox{0.9}{$s_1(\lambda)$} at 22 8
			\endlabellist 
			\centering 
			\includegraphics[scale=1.3]{./BlCap1}
	}\endxy
&
\xy (0,1)*{
			\labellist
			\small\hair 2pt
			\pinlabel \scalebox{0.7}{$\textcolor{blue}{1}$} at 1 -5
			\pinlabel \scalebox{0.9}{$\lambda$} at 18 8
			\endlabellist 
			\centering 
			\includegraphics[scale=1.3]{./BlCap1}
	}\endxy\;\;
    &\overset{\iota}{\mapsto}
    (-1)^{\lambda_2+1} \xy (0,0)*{
			\labellist
			\small\hair 2pt
			\pinlabel \scalebox{0.7}{$\textcolor{blue}{1}$} at 1 -5
			\pinlabel \scalebox{0.9}{$s_1(\lambda)$} at 22 8
			\endlabellist 
			\centering 
			\includegraphics[scale=1.3]{./BrCap1}
	}\endxy
\\[2ex]
\xy (0,1)*{
			\labellist
			\small\hair 2pt
			\pinlabel \scalebox{0.7}{$\textcolor{blue}{1}$} at 1 9
			\pinlabel \scalebox{0.9}{$\lambda$} at 19 2
			\endlabellist 
			\centering 
			\includegraphics[scale=1.3]{./BrCup1}
	}\endxy\;\;
    &\overset{\iota}{\mapsto}
    (-1)^{\lambda_1} \xy (0,0)*{
			\labellist
			\small\hair 2pt
			\pinlabel \scalebox{0.7}{$\textcolor{blue}{1}$} at 1 9
			\pinlabel \scalebox{0.9}{$s_1(\lambda)$} at 24 2
			\endlabellist 
			\centering 
			\includegraphics[scale=1.3]{./BlCup1}
	}\endxy
&
\xy (0,1)*{
			\labellist
			\small\hair 2pt
			\pinlabel \scalebox{0.7}{$\textcolor{blue}{1}$} at 1 9
			\pinlabel \scalebox{0.9}{$\lambda$} at 19 2
			\endlabellist 
			\centering 
			\includegraphics[scale=1.3]{./BlCup1}
	}\endxy\;\;
    &\overset{\iota}{\mapsto}
    (-1)^{\lambda_2} \xy (0,0)*{
			\labellist
			\small\hair 2pt
			\pinlabel \scalebox{0.7}{$\textcolor{blue}{1}$} at 1 9
			\pinlabel \scalebox{0.9}{$s_1(\lambda)$} at 24 2
			\endlabellist 
			\centering 
			\includegraphics[scale=1.3]{./BrCup1}
	}\endxy
\\[2ex]
\xy (0,1)*{
			\labellist
			\small\hair 2pt
			\pinlabel \scalebox{0.7}{$\textcolor{red}{2}$} at 1 -5
			\pinlabel \scalebox{0.9}{$\lambda$} at 18 8
			\endlabellist 
			\centering 
			\includegraphics[scale=1.3]{./RrCap1Thin}
	}\endxy\;\;
    &\overset{\iota}{\mapsto}
    (-1)^{\lambda_2+1} \xy (0,0)*{
			\labellist
			\small\hair 2pt
			\pinlabel \scalebox{0.7}{$\textcolor{black}{3}$} at 1 -5
			\pinlabel \scalebox{0.9}{$s_1(\lambda)$} at 22 8
			\endlabellist 
			\centering 
			\includegraphics[scale=1.3]{./KlCap1Thin}
	}\endxy
&
\xy (0,1)*{
			\labellist
			\small\hair 2pt
			\pinlabel \scalebox{0.7}{$\textcolor{red}{2}$} at 1 -5
			\pinlabel \scalebox{0.9}{$\lambda$} at 18 8
			\endlabellist 
			\centering 
			\includegraphics[scale=1.3]{./RlCap1}
	}\endxy\;\;
    &\overset{\iota}{\mapsto}
    (-1)^{\lambda_3+1} \xy (0,0)*{
			\labellist
			\small\hair 2pt
			\pinlabel \scalebox{0.7}{$\textcolor{black}{3}$} at 1 -5
			\pinlabel \scalebox{0.9}{$s_1(\lambda)$} at 22 8
			\endlabellist 
			\centering 
			\includegraphics[scale=1.3]{./KrCap1Thin}
	}\endxy
\\[2ex]
\xy (0,1)*{
			\labellist
			\small\hair 2pt
			\pinlabel \scalebox{0.7}{$\textcolor{red}{2}$} at 1 9
			\pinlabel \scalebox{0.9}{$\lambda$} at 19 2
			\endlabellist 
			\centering 
			\includegraphics[scale=1.3]{./RrCup1Thin}
	}\endxy\;\;
    &\overset{\iota}{\mapsto}
    (-1)^{\lambda_2} \xy (0,0)*{
			\labellist
			\small\hair 2pt
			\pinlabel \scalebox{0.7}{$\textcolor{black}{3}$} at 1 9
			\pinlabel \scalebox{0.9}{$s_1(\lambda)$} at 24 2
			\endlabellist 
			\centering 
			\includegraphics[scale=1.3]{./KlCup1Thin}
	}\endxy
&
\xy (0,1)*{
			\labellist
			\small\hair 2pt
			\pinlabel \scalebox{0.7}{$\textcolor{red}{2}$} at 1 9
			\pinlabel \scalebox{0.9}{$\lambda$} at 19 2
			\endlabellist 
			\centering 
			\includegraphics[scale=1.3]{./RlCup1Thin}
	}\endxy\;\;
    &\overset{\iota}{\mapsto}
    (-1)^{\lambda_3} \xy (0,0)*{
			\labellist
			\small\hair 2pt
			\pinlabel \scalebox{0.7}{$\textcolor{black}{3}$} at 1 9
			\pinlabel \scalebox{0.9}{$s_1(\lambda)$} at 24 2
			\endlabellist 
			\centering 
			\includegraphics[scale=1.3]{./KrCup1Thin}
	}\endxy
\end{align}\endgroup

\end{itemize}

Explicitly, $\iota'$ is given by:

\begin{itemize}[wide,labelindent=0pt,itemsep=5pt]

    \item On objects, $\lambda\overset{\iota'}{\mapsto} s_2(\lambda)$.

    \item On 1-morphisms,

\begin{gather*}
\E_2\oneid_\lambda\overset{\iota'}{\mapsto} \F_2\oneid_{s_2(\lambda)} ,
\mspace{35mu}
\F_2\oneid_\lambda\overset{\iota'}{\mapsto} \E_2\oneid_{s_2(\lambda)} ,
\mspace{35mu}
\E_1\oneid_\lambda\overset{\iota'}{\mapsto} \F_3\oneid_{s_2(\lambda)} ,
\mspace{35mu}
\F_1\oneid_\lambda\overset{\iota'}{\mapsto} \E_3\oneid_{s_2(\lambda)} .
\end{gather*}   

    \item On 2-morphisms, 

\begingroup\allowdisplaybreaks
\begin{align}
\xy (0,1)*{
			\labellist
			\small\hair 2pt
			\pinlabel \scalebox{0.7}{$\textcolor{red}{2}$} at 1 -5
			\pinlabel \scalebox{0.9}{$\lambda$} at 6 8
			\endlabellist 
			\centering 
			\includegraphics[scale=1.3]{./Ruo}
	}\endxy\;\;
    &\overset{\iota'}{\mapsto}
    \,\xy (0,0)*{
			\labellist
			\small\hair 2pt
			\pinlabel \scalebox{0.7}{$\textcolor{red}{2}$} at 1 -5
			\pinlabel \scalebox{0.9}{$s_2(\lambda)$} at 12 8
			\endlabellist 
			\centering 
			\includegraphics[scale=1.3]{./Rdo}
	}\endxy
&  
\xy (0,1)*{
			\labellist
			\small\hair 2pt
			\pinlabel \scalebox{0.7}{$\textcolor{blue}{1}$} at 1 -5
			\pinlabel \scalebox{0.9}{$\lambda$} at 6 8
			\endlabellist 
			\centering 
			\includegraphics[scale=1.3]{./Buo}
	}\endxy\;\;
    &\overset{\iota'}{\mapsto}
    \,\xy (0,0)*{
			\labellist
			\small\hair 2pt
			\pinlabel \scalebox{0.7}{$\textcolor{black}{3}$} at 1 -5
			\pinlabel \scalebox{0.9}{$s_2(\lambda)$} at 12 8
			\endlabellist 
			\centering 
			\includegraphics[scale=1.3]{./KdoThin}
	}\endxy
\\[2ex]
\xy (0,1)*{
			\labellist
			\small\hair 2pt
			\pinlabel \scalebox{0.7}{$\textcolor{red}{2}$} at 0 -5
			\pinlabel \scalebox{0.7}{$\textcolor{red}{2}$} at 16 -5
			\pinlabel \scalebox{0.9}{$\lambda$} at 18 8
			\endlabellist 
			\centering 
			\includegraphics[scale=1.3]{./RurRul}
	}\endxy\;\; 
    &\overset{\iota'}{\mapsto}
    -\xy (0,0)*{
			\labellist
			\small\hair 2pt
			\pinlabel \scalebox{0.7}{$\textcolor{red}{2}$} at 0 -5
			\pinlabel \scalebox{0.7}{$\textcolor{red}{2}$} at 16 -5
			\pinlabel \scalebox{0.9}{$s_2(\lambda)$} at 22 8
			\endlabellist 
			\centering 
			\includegraphics[scale=1.3]{./RdrRdl}
	}\endxy
&
\xy (0,1)*{
			\labellist
			\small\hair 2pt
			\pinlabel \scalebox{0.7}{$\textcolor{blue}{1}$} at 0 -5
			\pinlabel \scalebox{0.7}{$\textcolor{blue}{1}$} at 16 -5
			\pinlabel \scalebox{0.9}{$\lambda$} at 18 8
			\endlabellist 
			\centering 
			\includegraphics[scale=1.3]{./BurBul}
	}\endxy\;\;
    &\overset{\iota'}{\mapsto}
    -\xy (0,0)*{
			\labellist
			\small\hair 2pt
			\pinlabel \scalebox{0.7}{$\textcolor{black}{3}$} at 0 -5
			\pinlabel \scalebox{0.7}{$\textcolor{black}{3}$} at 16 -5
			\pinlabel \scalebox{0.9}{$s_2(\lambda)$} at 22 8
			\endlabellist 
			\centering 
			\includegraphics[scale=1.3]{./KdrKdl2}
	}\endxy
\\[2ex]
\xy (0,1)*{
			\labellist
			\small\hair 2pt
			\pinlabel \scalebox{0.7}{$\textcolor{blue}{1}$} at 16 -5
			\pinlabel \scalebox{0.7}{$\textcolor{red}{2}$} at 0 -5
			\pinlabel \scalebox{0.9}{$\lambda$} at 18 8
			\endlabellist 
			\centering 
			\includegraphics[scale=1.3]{./RurBul}
	}\endxy\;\;
    &\overset{\iota'}{\mapsto}
   k_1^{2,3}\xy (0,0)*{
			\labellist
			\small\hair 2pt
			\pinlabel \scalebox{0.7}{$\textcolor{red}{2}$} at 0 -5
			\pinlabel \scalebox{0.7}{$\textcolor{black}{3}$} at 16 -5
			\pinlabel \scalebox{0.9}{$s_2(\lambda)$} at 22 8
			\endlabellist 
			\centering 
			\includegraphics[scale=1.3]{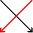}
	}\endxy
&
\xy (0,1)*{
			\labellist
			\small\hair 2pt
			\pinlabel \scalebox{0.7}{$\textcolor{blue}{1}$} at 0 -5
			\pinlabel \scalebox{0.7}{$\textcolor{red}{2}$} at 16 -5
			\pinlabel \scalebox{0.9}{$\lambda$} at 18 8
			\endlabellist 
			\centering 
			\includegraphics[scale=1.3]{./BurRul}
	}\endxy\;\;
    &\overset{\iota'}{\mapsto}
    k_1^{0,1}\xy (0,0)*{
			\labellist
			\small\hair 2pt
			\pinlabel \scalebox{0.7}{$\textcolor{red}{2}$} at 16 -5
			\pinlabel \scalebox{0.7}{$\textcolor{black}{3}$} at 0 -5
			\pinlabel \scalebox{0.9}{$s_2(\lambda)$} at 22 8
			\endlabellist 
			\centering 
			\includegraphics[scale=1.3]{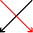}
	}\endxy
\\[2ex]
\xy (0,1)*{
			\labellist
			\small\hair 2pt
			\pinlabel \scalebox{0.7}{$\textcolor{red}{2}$} at 1 -5
			\pinlabel \scalebox{0.9}{$\lambda$} at 18 8
			\endlabellist 
			\centering 
			\includegraphics[scale=1.3]{./RrCap1Thin}
	}\endxy\;\;
    &\overset{\iota'}{\mapsto}
    (-1)^{\lambda_2+1} \xy (0,0)*{
			\labellist
			\small\hair 2pt
			\pinlabel \scalebox{0.7}{$\textcolor{red}{2}$} at 1 -5
			\pinlabel \scalebox{0.9}{$s_2(\lambda)$} at 22 8
			\endlabellist 
			\centering 
			\includegraphics[scale=1.3]{./RlCap1}
	}\endxy
&
\xy (0,1)*{
			\labellist
			\small\hair 2pt
			\pinlabel \scalebox{0.7}{$\textcolor{red}{2}$} at 1 -5
			\pinlabel \scalebox{0.9}{$\lambda$} at 18 8
			\endlabellist 
			\centering 
			\includegraphics[scale=1.3]{./RlCap1}
	}\endxy\;\;
    &\overset{\iota'}{\mapsto}
    (-1)^{\lambda_3+1} \xy (0,0)*{
			\labellist
			\small\hair 2pt
			\pinlabel \scalebox{0.7}{$\textcolor{red}{1}$} at 2 -5
			\pinlabel \scalebox{0.9}{$s_2(\lambda)$} at 22 8
			\endlabellist 
			\centering 
			\includegraphics[scale=1.3]{./RrCap1Thin}
	}\endxy
\\[2ex]
\xy (0,1)*{
			\labellist
			\small\hair 2pt
			\pinlabel \scalebox{0.7}{$\textcolor{red}{2}$} at 1 9
			\pinlabel \scalebox{0.9}{$\lambda$} at 19 2
			\endlabellist 
			\centering 
			\includegraphics[scale=1.3]{./RrCup1Thin}
	}\endxy\;\;
    &\overset{\iota'}{\mapsto}
    (-1)^{\lambda_2} \xy (0,0)*{
			\labellist
			\small\hair 2pt
			\pinlabel \scalebox{0.7}{$\textcolor{red}{2}$} at 1 9
			\pinlabel \scalebox{0.9}{$s_3(\lambda)$} at 24 2
			\endlabellist 
			\centering 
			\includegraphics[scale=1.3]{./RlCup1Thin}
	}\endxy
&
\xy (0,1)*{
			\labellist
			\small\hair 2pt
			\pinlabel \scalebox{0.7}{$\textcolor{red}{2}$} at 1 9
			\pinlabel \scalebox{0.9}{$\lambda$} at 19 2
			\endlabellist 
			\centering 
			\includegraphics[scale=1.3]{./RlCup1Thin}
	}\endxy\;\;
    &\overset{\iota'}{\mapsto}
    (-1)^{\lambda_3} \xy (0,0)*{
			\labellist
			\small\hair 2pt
			\pinlabel \scalebox{0.7}{$\textcolor{red}{2}$} at 1 9
			\pinlabel \scalebox{0.9}{$s_2(\lambda)$} at 24 2
			\endlabellist 
			\centering 
			\includegraphics[scale=1.3]{./RrCup1Thin}
	}\endxy
\\[2ex]
\xy (0,1)*{
			\labellist
			\small\hair 2pt
			\pinlabel \scalebox{0.7}{$\textcolor{blue}{1}$} at 1 -5
			\pinlabel \scalebox{0.9}{$\lambda$} at 18 8
			\endlabellist 
			\centering 
			\includegraphics[scale=1.3]{./BrCap1}
	}\endxy\;\;
    &\overset{\iota'}{\mapsto}
    (-1)^{\lambda_1+1} \xy (0,0)*{
			\labellist
			\small\hair 2pt
			\pinlabel \scalebox{0.7}{$\textcolor{black}{3}$} at 1 -5
			\pinlabel \scalebox{0.9}{$s_2(\lambda)$} at 22 8
			\endlabellist 
			\centering 
			\includegraphics[scale=1.3]{./KlCap1Thin}
	}\endxy
&
\xy (0,1)*{
			\labellist
			\small\hair 2pt
			\pinlabel \scalebox{0.7}{$\textcolor{blue}{1}$} at 1 -5
			\pinlabel \scalebox{0.9}{$\lambda$} at 18 8
			\endlabellist 
			\centering 
			\includegraphics[scale=1.3]{./BlCap1}
	}\endxy\;\;
    &\overset{\iota'}{\mapsto}
    (-1)^{\lambda_2+1} \xy (0,0)*{
			\labellist
			\small\hair 2pt
			\pinlabel \scalebox{0.7}{$\textcolor{black}{3}$} at 1 -5
			\pinlabel \scalebox{0.9}{$s_2(\lambda)$} at 22 8
			\endlabellist 
			\centering 
			\includegraphics[scale=1.3]{./KrCap1Thin}
	}\endxy
\\[2ex]
\xy (0,1)*{
			\labellist
			\small\hair 2pt
			\pinlabel \scalebox{0.7}{$\textcolor{blue}{1}$} at 1 9
			\pinlabel \scalebox{0.9}{$\lambda$} at 19 2
			\endlabellist 
			\centering 
			\includegraphics[scale=1.3]{./BrCup1}
	}\endxy\;\;
    &\overset{\iota'}{\mapsto}
    (-1)^{\lambda_1} \xy (0,0)*{
			\labellist
			\small\hair 2pt
			\pinlabel \scalebox{0.7}{$\textcolor{black}{3}$} at 1 9
			\pinlabel \scalebox{0.9}{$s_2(\lambda)$} at 24 2
			\endlabellist 
			\centering 
			\includegraphics[scale=1.3]{./KlCup1Thin}
	}\endxy
&
\xy (0,1)*{
			\labellist
			\small\hair 2pt
			\pinlabel \scalebox{0.7}{$\textcolor{blue}{1}$} at 1 9
			\pinlabel \scalebox{0.9}{$\lambda$} at 19 2
			\endlabellist 
			\centering 
			\includegraphics[scale=1.3]{./BlCup1}
	}\endxy\;\;
    &\overset{\iota'}{\mapsto}
    (-1)^{\lambda_2} \xy (0,0)*{
			\labellist
			\small\hair 2pt
			\pinlabel \scalebox{0.7}{$\textcolor{black}{3}$} at 1 9
			\pinlabel \scalebox{0.9}{$s_2(\lambda)$} at 24 2
			\endlabellist 
			\centering 
			\includegraphics[scale=1.3]{./KrCup1Thin}
	}\endxy
\end{align}\endgroup

\end{itemize}

\subsection{Degree shifts}
Finally, we introduce two `shift' 2-isomorphisms $\sigma_1,\sigma_2:K^b(\naffu{3})\to K^b(\naffu{3})$. These shift the homological degree and internal degree of complexes, but otherwise act as the identity. We recall that $X[y]\langle z\rangle$ is the 1-term complex with the 1-morphism at homological degree $-y$ with grade shift $z$. Specifically, we define $\sigma_1$ on the generating 1-morphisms by:\begin{itemize}
    \item $\E_1\oneid_\lambda[y]\langle z\rangle \overset{\sigma_1}{\mapsto} \E_1\oneid_\lambda[y-1]\langle z-\overline{\lambda}_1\rangle$
    \item $\F_1\oneid_\lambda[y]\langle z\rangle\overset{\sigma_1}{\mapsto}\F_1\oneid_\lambda[y+1]\langle z-2+\overline{\lambda}_1\rangle$
    \item $\E_2\oneid_\lambda[y]\langle z \rangle\overset{\sigma_1}{\mapsto} \E_2\oneid_\lambda[y+1]\langle z-\lambda_2-\lambda_3-t\rangle$
    \item $\F_2\oneid_\lambda[y]\langle z\rangle \overset{\sigma_1}{\mapsto} \F_2\oneid_\lambda[y-1]\langle z+\lambda_2+\lambda_3+t\rangle$
\end{itemize}
and $\sigma_2$ on the generating 1-morphisms by:\begin{itemize}
    \item $\E_1\oneid_\lambda[y]\langle z\rangle\overset{\sigma_2}{\mapsto}\E_1\oneid_\lambda[y+1]\langle z-2\lambda_3-\overline{\lambda}_1-t+4\rangle$
    \item $\F_1\oneid_\lambda[y]\langle z\rangle\overset{\sigma_2}{\mapsto}\F_1\oneid_\lambda[y-1]\langle z+2\lambda_3+\overline{\lambda}_1+t-4\rangle$
    \item $\E_2\oneid_\lambda[y]\langle z\rangle\overset{\sigma_2}{\mapsto}\E_2\oneid_\lambda[y+1]\langle z+2-\overline{\lambda}_2\rangle$
    \item $\F_2\oneid_\lambda[y]\langle z\rangle\overset{\sigma_2}{\mapsto}\F_2\oneid_\lambda[y-1]\langle z+\overline{\lambda}_2\rangle$,
\end{itemize}
each extended to compositions and complexes in the obvious fashion. We will use these to match the homological degree and grading of the image of 1-morphisms under $\tilde{\mathcal{T}}_{1,-1}'$ and $\tilde{\mathcal{T}}_{2,1}''$ to their image under $\Ev$. 

\subsection{Auxiliary result}
We present here the categorification of \autoref{sec:DecatT}.
\begin{prop}\label{thm:EvIota}\begin{enumerate}
    \item $\Ev'\circ\iota$ and $\sigma_1\circ\tilde{\mathcal{T}}_{1,-1}'$ are equal on objects and generating 1-morphisms in $\naffu{3}$, and are equal up to homotopy on generating 2-morphisms in $\naffu{3}$.
    \item $\Ev'\circ\iota'$ and $\sigma_2\circ\alpha\circ\tilde{\mathcal{T}}_{2,1}''\circ\beta$ are equal on objects and generating 1-morphisms in $\naffu{3}$, and are equal up to homotopy on generating 2-morphisms in $\naffu{3}$.
\end{enumerate}
\end{prop}
\begin{rem}
    We phrase the proposition in this fashion because we have not yet proved that $\Ev'$ is a well-defined 2-functor (and indeed, we will use this proposition to do so), and so it would not be accurate to claim that the two sides are equivalent 2-functors.
\end{rem}
\begin{proof}
It is immediately clear from the definitions that both sides of each equation agree on objects. Further, we note that $\sigma_1$ and $\sigma_2$ only affect 1-morphisms, and do not change the generating 2-morphisms. We recall that $S(\lambda)=\lambda_1+\lambda_3+t-1$.

\smallskip

For $\Ev'\circ\iota$:
\begingroup\allowdisplaybreaks    
\begin{align}
\Ev'\iota(\E_1\oneid_\lambda)
\mspace{-70mu}&\mspace{70mu}=\Ev'(\F_1\oneid_{s_1(\lambda)})=\underline{\F_1\oneid_{s_1(\lambda)}}=\sigma_1\tilde{\mathcal{T}}_{1,-1}'(\E_1\oneid_\lambda)
\\ \Ev'\iota(\F_1\oneid_\lambda)
\mspace{-70mu}&\mspace{70mu}=\Ev'(\E_1\oneid_{s_1(\lambda)})=\underline{\E_1\oneid_{s_1(\lambda)}}=\sigma_1\tilde{\mathcal{T}}_{1,-1}'(\F_1\oneid_\lambda)
\\[2ex]
\Ev'\iota(\E_2\oneid_\lambda)\mspace{-70mu}&\mspace{70mu} =\Ev'(\F_3\oneid_{s_1(\lambda)})\\ \nn
    & =\xymatrix{\E_{12}\oneid_{s_1(\lambda)}\langle -S(s_1(\lambda))-1\rangle 
 \ar[rr]^(0.45){\xy (0,4.5)*{
 		\labellist
		\small\hair 2pt
		\pinlabel \scalebox{0.7}{$\textcolor{red}{2}$} at 0 -5
		\pinlabel \scalebox{0.7}{$\textcolor{blue}{1}$} at 16 -5
		\pinlabel \scalebox{0.9}{$s_1(\lambda)$} at 22 7
		\endlabellist
		\centering
		\includegraphics[scale=1.3]{./RurBul}
	}\endxy} & & \underline{\E_{21}\oneid_{s_1(\lambda)}}\langle -S(s_1(\lambda))\rangle}=\sigma_1\tilde{\mathcal{T}}_{1,-1}'(\E_2\oneid_\lambda)
\\[2ex]
\Ev'\iota(\F_2\oneid_\lambda)\mspace{-70mu}&\mspace{70mu} =\Ev'(\E_3\oneid_{s_1(\lambda)})\\ \nn
    & =\xymatrix{\underline{\F_{12}\oneid_{s_1(\lambda)}}\langle S(s_1(\lambda))\rangle \ar[rr]^(0.55){\xy (0,4.5)*{
		\labellist
		\small\hair 2pt
		\pinlabel \scalebox{0.7}{$\textcolor{red}{2}$} at 0 -5
		\pinlabel \scalebox{0.7}{$\textcolor{blue}{1}$} at 16 -5
		\pinlabel \scalebox{0.9}{$s_1(\lambda)$} at 22 7
		\endlabellist
		\centering
		\includegraphics[scale=1.3]{./RdlBdr}
	}\endxy} & & \F_{21}\oneid_{s_1(\lambda)}\langle S(s_1(\lambda))+1\rangle} = \sigma_1\tilde{\mathcal{T}}_{1,-1}'(\F_2\oneid_\lambda)
 \end{align}\endgroup

\smallskip

\begingroup\allowdisplaybreaks    
\endgroup

\smallskip

For $\Ev'\circ\iota':$ 
\begingroup\allowdisplaybreaks    
\begin{align}    
    \Ev'\iota'(\E_1\oneid_\lambda) \mspace{-70mu}&\mspace{70mu} =\Ev'(\F_3\oneid_{s_2(\lambda)})
    \\ \nn
& =\xymatrix{\E_{12}\oneid_{s_2(\lambda)}\langle -S(s_2(\lambda))-1\rangle 
 \ar[rr]^(0.50){\xy (0,5)*{
 		\labellist
		\small\hair 2pt
		\pinlabel \scalebox{0.7}{$\textcolor{red}{2}$} at 0 -5
		\pinlabel \scalebox{0.7}{$\textcolor{blue}{1}$} at 16 -5
		\pinlabel \scalebox{0.9}{$s_2(\lambda)$} at 22 8
		\endlabellist
		\centering
		\includegraphics[scale=1.3]{./RurBul}
	}\endxy} & & \underline{\E_{21}\oneid_{s_2(\lambda)}}\langle -S(s_2(\lambda))\rangle} = \sigma_2\alpha\tilde{\mathcal{T}}_{2,1}''\beta(\E_1\oneid_\lambda)
\\
\Ev'\iota'(\F_1\oneid_\lambda) \mspace{-70mu}&\mspace{70mu} =\Ev'(\E_3\oneid_{s_2(\lambda)})
\\ \nn
    & =\xymatrix{\underline{\F_{12}\oneid_{s_2(\lambda)}}\langle S(s_2(\lambda))\rangle \ar[rr]^(0.48){\xy (0,5)*{
		\labellist
		\small\hair 2pt
		\pinlabel \scalebox{0.7}{$\textcolor{red}{2}$} at 0 -5
		\pinlabel \scalebox{0.7}{$\textcolor{blue}{1}$} at 16 -5
		\pinlabel \scalebox{0.9}{$s_2(\lambda)$} at 22 8
		\endlabellist
		\centering
		\includegraphics[scale=1.3]{./RdlBdr}
	}\endxy} & & \F_{21}\oneid_{s_2(\lambda)}\langle S(s_2(\lambda))+1\rangle} = \sigma_2\alpha\tilde{\mathcal{T}}_{2,1}''\beta(\F_1\oneid_\lambda)
\\
\Ev'\iota'(\E_2\oneid_\lambda)
\mspace{-70mu}&\mspace{70mu}=\Ev'(\F_2\oneid_\lambda)=\underline{\F_2\oneid_{s_2(\lambda)}}=\sigma_2\alpha\tilde{\mathcal{T}}_{2,1}''\beta(\E_2\oneid_\lambda)
\\ \Ev'\iota'(\F_2\oneid_\lambda)
\mspace{-70mu}&\mspace{70mu}=\Ev'(\E_2\oneid_\lambda)=\underline{\E_2\oneid_{s_2(\lambda)}}=\sigma_2\alpha\tilde{\mathcal{T}}_{2,1}''\beta(\F_2\oneid_\lambda)
\end{align}\endgroup   
%
%
\begingroup
\endgroup   
This finishes the proof.
\end{proof}

\subsection{Proof of \autoref{thm:BigThmPrime}}\label{sec:BigThPrime}
    Because $\iota$ reverses the orientation of the diagrams, we felt that this proof would be clearer to the reader if we proved that $\Ev'$ preserved the 180 degree rotated versions of relations \eqref{eq:adjunction}-\eqref{eq:Grassmannian}. By the cyclicity relations KM2 and KM3, this is equivalent to proving the original relations are preserved.
    
    For any KM relation that only involves strands labelled $1$ and $2$ and does not involve a crossing of a $1$-strand and a $2$-strand, $\Ev'$ acts as the identity and therefore trivially preserves the relation. For relations that do involve these crossings, the calculations are generally straightforward. For example, \begin{align*}\Ev'\left(\xy (0,0)*{
			\labellist
			\small\hair 2pt
			\pinlabel \scalebox{0.7}{$\textcolor{red}{2}$} at 0 -5
                \pinlabel \scalebox{0.7}{$\textcolor{blue}{1}$} at 11 -5
			\pinlabel \scalebox{0.9}{$\lambda$} at 58 25
			\endlabellist 
			\centering 
			\includegraphics[scale=1.3]{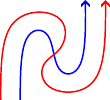}
	}\endxy\;\;\right)
  & = (-1)^{\overline{\lambda}_1(\overline{\lambda}_2+1)}\xy (0,0)*{
			\labellist
			\small\hair 2pt
			\pinlabel \scalebox{0.7}{$\textcolor{red}{2}$} at 0 -5
                \pinlabel \scalebox{0.7}{$\textcolor{blue}{1}$} at 11 -5
			\pinlabel \scalebox{0.9}{$\lambda$} at 58 25
			\endlabellist 
			\centering 
			\includegraphics[scale=1.3]{./RBUpRightCross}
	}\endxy\\
    & \\
    & = 
 (-1)^{\overline{\lambda}_1(\overline{\lambda}_2+1)}\xy (0,0)*{
			\labellist
			\small\hair 2pt
			\pinlabel \scalebox{0.7}{$\textcolor{red}{2}$} at 43 -5
                \pinlabel \scalebox{0.7}{$\textcolor{blue}{1}$} at 53 -5
			\pinlabel \scalebox{0.9}{$\lambda$} at 58 25
			\endlabellist 
			\centering 
			\includegraphics[scale=1.3]{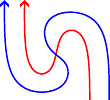}
	}\endxy
 \;\; = \Ev'\left(
 \xy (0,0)*{
			\labellist
			\small\hair 2pt
			\pinlabel \scalebox{0.7}{$\textcolor{red}{2}$} at 43 -5
                \pinlabel \scalebox{0.7}{$\textcolor{blue}{1}$} at 53 -5
			\pinlabel \scalebox{0.9}{$\lambda$} at 58 25
			\endlabellist 
			\centering 
			\includegraphics[scale=1.3]{./RBUpLeftCross}
	}\endxy\;\;\right)\end{align*}
    \eqskip
    with the other cyclicity relations following similarly, as does the relevant KM5 relation (since $(-1)^{\overline{\lambda}_1(\overline{\lambda}_2+1)}\cdot 0 = 0$). For KM4, we have $$\Ev'\left(\xy (0,0)*{
			\labellist
			\small\hair 2pt
			\pinlabel \scalebox{0.7}{$\textcolor{red}{2}$} at 0 -5
			\pinlabel \scalebox{0.7}{$\textcolor{blue}{1}$} at 16 -5
			\pinlabel \scalebox{0.9}{$\lambda$} at 18 16
			\endlabellist 
			\centering 
			\includegraphics[scale=1.3]{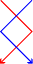}
	}\endxy\;\;\right) = ((-1)^{\overline{\lambda}_1(\overline{\lambda}_2+1)})^2\xy (0,0)*{
			\labellist
			\small\hair 2pt
			\pinlabel \scalebox{0.7}{$\textcolor{red}{2}$} at 0 -5
			\pinlabel \scalebox{0.7}{$\textcolor{blue}{1}$} at 16 -5
			\pinlabel \scalebox{0.9}{$\lambda$} at 18 16
			\endlabellist 
			\centering 
			\includegraphics[scale=1.3]{./RdrdlBdldr}
	}\endxy \;\; = \xy (0,0)*{
			\labellist
			\small\hair 2pt
			\pinlabel \scalebox{0.7}{$\textcolor{red}{2}$} at 0 -5
			\pinlabel \scalebox{0.7}{$\textcolor{blue}{1}$} at 13 -5
			\pinlabel \scalebox{0.9}{$\lambda$} at 17 16
			\endlabellist 
			\centering 
			\includegraphics[scale=1.3]{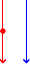}
	}\endxy
\ \ \  - \
 \xy (0,0)*{
			\labellist
			\small\hair 2pt
			\pinlabel \scalebox{0.7}{$\textcolor{red}{2}$} at 0 -5
			\pinlabel \scalebox{0.7}{$\textcolor{blue}{1}$} at 13 -5
			\pinlabel \scalebox{0.9}{$\lambda$} at 17 16
			\endlabellist 
			\centering 
			\includegraphics[scale=1.3]{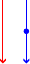}
	}\endxy\;\; = \Ev'\left(\xy (0,0)*{
			\labellist
			\small\hair 2pt
			\pinlabel \scalebox{0.7}{$\textcolor{red}{2}$} at 0 -5
			\pinlabel \scalebox{0.7}{$\textcolor{blue}{1}$} at 13 -5
			\pinlabel \scalebox{0.9}{$\lambda$} at 17 16
			\endlabellist 
			\centering 
			\includegraphics[scale=1.3]{./RdodBdd}
	}\endxy
\ \ \  - \
 \xy (0,0)*{
			\labellist
			\small\hair 2pt
			\pinlabel \scalebox{0.7}{$\textcolor{red}{2}$} at 0 -5
			\pinlabel \scalebox{0.7}{$\textcolor{blue}{1}$} at 13 -5
			\pinlabel \scalebox{0.9}{$\lambda$} at 17 16
			\endlabellist 
			\centering 
			\includegraphics[scale=1.3]{./RddBdod}
	}\endxy\;\;\right).$$
    \eqskip
    KM7 is similar, as is KM6; any (multicolour) cubic KLR diagram consisting only of strands labelled 1 and 2 will have precisely two multicoloured crossings, leading to a similar squaring of the sign.
    It therefore remains to consider only those diagrams with at least one strand labelled 3.
    
For most of the KM identities discussed below, we are able to use that $\iota$ and $\iota'$ are locally faithful 2-functors, and therefore we are able to consider the unique pre-image of any 2-morphism in their images. The results will then follow from liberal use of \autoref{thm:EvIota} (we give an example in first equation below of where it is used). We also implicitly make use of \autoref{lem:DiagMatch} when there is a diagram in the image of both $\iota$ and $\iota'$. We will present a representative sampling of the identities of each KM axiom. The exception is the six instances of KM6 where, using the notation of \eqref{eq:R3}, $\{i,j,k\}=\{1,2,3\}$, since such 2-morphisms are not in the image of either $\iota$ or $\iota'$. In these cases, we will be proving directly that $\Ev'$ preserves KM6.
    \begin{itemize}[wide,labelindent=0pt]
        \item[KM1, 2:] \begin{align*}\Ev'\left(\xy (0,0)*{
			\labellist
			\small\hair 2pt
			\pinlabel \scalebox{0.7}{$3$} at 13 -5
			\pinlabel \scalebox{0.9}{$\lambda$} at 19 12
			\endlabellist 
			\centering 
			\includegraphics[scale=1.4]{./down_right}
	}\endxy\;\;\right)
    &=(-1)^{s_1(\lambda)_2+1}(-1)^{(s_1(\lambda)+\alpha_2)_2}\Ev'\iota\left(\xy (0,0)*{
			\labellist
			\small\hair 2pt
			\pinlabel \scalebox{0.7}{$\textcolor{red}{2}$} at 13 -5
			\pinlabel \scalebox{0.9}{$s_1(\lambda)$} at 25 12
			\endlabellist 
			\centering 
			\includegraphics[scale=1.4]{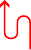}
	}\endxy\qquad
    \right)\overset{\ref{thm:EvIota}}{=}\sigma_1\tilde{\mathcal{T}}_{1,-1}'\left(\xy (0,0)*{
			\labellist
			\small\hair 2pt
			\pinlabel \scalebox{0.7}{$\textcolor{red}{2}$} at 13 -5
			\pinlabel \scalebox{0.9}{$s_1(\lambda)$} at 25 12
			\endlabellist 
			\centering 
			\includegraphics[scale=1.4]{./RUpLeft}
	}\endxy\qquad
    \right)\\
    \\
    &\sim_h \sigma_1\tilde{\mathcal{T}}_{1,-1}'\left(\xy (0,0)*{
			\labellist
			\small\hair 2pt
			\pinlabel \scalebox{0.7}{$\textcolor{red}{2}$} at 2 -5
			\pinlabel \scalebox{0.9}{$s_1(\lambda)$} at 12 12
			\endlabellist 
			\centering 
			\includegraphics[scale=1.4]{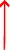}
	}\endxy\quad\;\;
    \right)\overset{\ref{thm:EvIota}}{=}\Ev'\iota\left(\xy (0,0)*{
			\labellist
			\small\hair 2pt
			\pinlabel \scalebox{0.7}{$\textcolor{red}{2}$} at 2 -5
			\pinlabel \scalebox{0.9}{$s_1(\lambda)$} at 12 12
			\endlabellist 
			\centering 
			\includegraphics[scale=1.4]{./RUp}
	}\endxy\quad\;\;
    \right)=\Ev'\left(\xy (0,0)*{
			\labellist
			\small\hair 2pt
			\pinlabel \scalebox{0.7}{$3$} at 2 -5
			\pinlabel \scalebox{0.9}{$\lambda$} at 6 12
			\endlabellist 
			\centering 
			\includegraphics[scale=1.4]{./down}}\endxy\;
    \right)\end{align*} 
    \eqskip 
    where the homotopy (and all future homotopies in this proof) follows from \autoref{lem:KMPreserve} and from the $\sigma_i$ being 2-isomorphisms. The other adjunction relation and dot cyclicity work similarly.
    \item[KM3:]\begin{align*}
       & \Ev'\left(\xy (0,0)*{
			\labellist
			\small\hair 2pt
			\pinlabel \scalebox{0.7}{$\textcolor{blue}{1}$} at 42 -5
                \pinlabel \scalebox{0.7}{$\textcolor{black}{3}$} at 52 -5
			\pinlabel \scalebox{0.9}{$\lambda$} at 58 25
			\endlabellist 
			\centering 
			\includegraphics[scale=1.3]{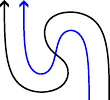}
	}\endxy\;\;\right) 
     \\  &= k_1^{2,3}(s_1(\lambda))(-1)^{s_1(\lambda)_2+1+(s_1(\lambda)-\alpha_2)_1+1+(s_1(\lambda)+\alpha_3)_2+(s_1(\lambda)-\alpha_1)_1}\Ev'\iota\left(\xy (0,0)*{
			\labellist
			\small\hair 2pt
			\pinlabel \scalebox{0.7}{$\textcolor{blue}{1}$} at 40 -5
                \pinlabel \scalebox{0.7}{$\textcolor{red}{2}$} at 50 -5
			\pinlabel \scalebox{0.9}{$s_1(\lambda)$} at 64 25
			\endlabellist 
			\centering 
			\includegraphics[scale=1.3]{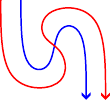}
	}\endxy\qquad
    \right)\\
     \\
     &= k_1^{2,3}(s_1(\lambda))\sigma_1\tilde{\mathcal{T}}_{1,-1}'\left(\xy (0,0)*{
			\labellist
			\small\hair 2pt
			\pinlabel \scalebox{0.7}{$\textcolor{blue}{1}$} at 40 -5
                \pinlabel \scalebox{0.7}{$\textcolor{red}{2}$} at 50 -5
			\pinlabel \scalebox{0.9}{$s_1(\lambda)$} at 64 25
			\endlabellist 
			\centering 
			\includegraphics[scale=1.3]{./BRDownRightCross}
	}\endxy\qquad
    \right)
     \sim_h k_1^{2,3}(s_1(\lambda))\sigma_1\tilde{\mathcal{T}}_{1,-1}'\left(\xy (0,0)*{
			\labellist
			\small\hair 2pt
			\pinlabel \scalebox{0.7}{$\textcolor{red}{2}$} at 16 -5
			\pinlabel \scalebox{0.7}{$\textcolor{blue}{1}$} at 0 -5
			\pinlabel \scalebox{0.9}{$s_1(\lambda)$} at 24 8
			\endlabellist 
			\centering 
			\includegraphics[scale=1.3]{./BdrRdl}
	}\endxy\qquad\right)\\
     \\
     &= k_1^{2,3}(s_1(\lambda))\Ev'\iota\left(\xy (0,0)*{
			\labellist
			\small\hair 2pt
			\pinlabel \scalebox{0.7}{$\textcolor{red}{2}$} at 16 -5
			\pinlabel \scalebox{0.7}{$\textcolor{blue}{1}$} at 0 -5
			\pinlabel \scalebox{0.9}{$s_1(\lambda)$} at 24 8
			\endlabellist 
			\centering 
			\includegraphics[scale=1.3]{./BdrRdl}
	}\endxy\qquad\right)
    = \Ev'\left(\xy (0,0)*{
			\labellist
			\small\hair 2pt
			\pinlabel \scalebox{0.7}{$\textcolor{black}{3}$} at 16 -5
			\pinlabel \scalebox{0.7}{$\textcolor{blue}{1}$} at 0 -5
			\pinlabel \scalebox{0.9}{$\lambda$} at 18 8
			\endlabellist 
			\centering 
			\includegraphics[scale=1.3]{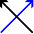}
	}\endxy\;\;
    \right)
    \end{align*}
    \begin{align*}
       & \Ev'\left(\xy (0,0)*{
			\labellist
			\small\hair 2pt
			\pinlabel \scalebox{0.7}{$\textcolor{red}{2}$} at 52 -5
                \pinlabel \scalebox{0.7}{$\textcolor{black}{3}$} at 42 -5
			\pinlabel \scalebox{0.9}{$\lambda$} at 58 25
			\endlabellist 
			\centering 
			\includegraphics[scale=1.3]{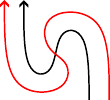}
	}\endxy\;\;\right)
    \\ &= k_1^{0,1}(s_2(\lambda))(-1)^{s_2(\lambda)_2+1+(s_2(\lambda)-\alpha_2)_1+1+(s_2(\lambda)+\alpha_3)_2+(s_2(\lambda)-\alpha_1)_1}\Ev'\iota'\left(\xy (0,0)*{
			\labellist
			\small\hair 2pt
			\pinlabel \scalebox{0.7}{$\textcolor{blue}{1}$} at 40 -5
                \pinlabel \scalebox{0.7}{$\textcolor{red}{2}$} at 50 -5
			\pinlabel \scalebox{0.9}{$s_2(\lambda)$} at 64 25
			\endlabellist 
			\centering 
			\includegraphics[scale=1.3]{./BRDownRightCross}
	}\endxy\qquad
    \right)\\
     \\
     &= k_1^{0,1}(s_2(\lambda))\sigma_2\alpha\tilde{\mathcal{T}}_{2,1}''\beta\left(\xy (0,0)*{
			\labellist
			\small\hair 2pt
			\pinlabel \scalebox{0.7}{$\textcolor{blue}{1}$} at 40 -5
                \pinlabel \scalebox{0.7}{$\textcolor{red}{2}$} at 50 -5
			\pinlabel \scalebox{0.9}{$s_1(\lambda)$} at 64 25
			\endlabellist 
			\centering 
			\includegraphics[scale=1.3]{./BRDownRightCross}
	}\endxy\qquad
    \right)
     \sim_hk_1^{0,1}(s_2(\lambda))\sigma_1\tilde{\mathcal{T}}_{1,-1}'\left(\xy (0,0)*{
			\labellist
			\small\hair 2pt
			\pinlabel \scalebox{0.7}{$\textcolor{red}{2}$} at 16 -5
			\pinlabel \scalebox{0.7}{$\textcolor{blue}{1}$} at 0 -5
			\pinlabel \scalebox{0.9}{$s_1(\lambda)$} at 24 8
			\endlabellist 
			\centering 
			\includegraphics[scale=1.3]{./BdrRdl}
	}\endxy\qquad\right)\\
     \\
     &= k_1^{0,1}(s_2(\lambda))\Ev'\iota'\left(\xy (0,0)*{
			\labellist
			\small\hair 2pt
			\pinlabel \scalebox{0.7}{$\textcolor{red}{2}$} at 16 -5
			\pinlabel \scalebox{0.7}{$\textcolor{blue}{1}$} at 0 -5
			\pinlabel \scalebox{0.9}{$s_1(\lambda)$} at 24 8
			\endlabellist 
			\centering 
			\includegraphics[scale=1.3]{./BdrRdl}
	}\endxy\qquad\right)
    = \Ev'\left(\xy (0,0)*{
			\labellist
			\small\hair 2pt
			\pinlabel \scalebox{0.7}{$\textcolor{black}{3}$} at 16 -5
			\pinlabel \scalebox{0.7}{$\textcolor{blue}{1}$} at 0 -5
			\pinlabel \scalebox{0.9}{$\lambda$} at 18 8
			\endlabellist 
			\centering 
			\includegraphics[scale=1.3]{./BurKul2}
	}\endxy\;\;
    \right)
    \end{align*}
    \begin{align*}
        \Ev'\left(\xy (0,0)*{
			\labellist
			\small\hair 2pt
			\pinlabel \scalebox{0.7}{$\textcolor{black}{3}$} at 42 -5
                \pinlabel \scalebox{0.7}{$\textcolor{black}{3}$} at 52 -5
			\pinlabel \scalebox{0.9}{$\lambda$} at 58 25
			\endlabellist 
			\centering 
			\includegraphics[scale=1.3]{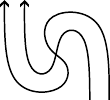}
	}\endxy\;\;\right)  &= (-1)^{s_1(\lambda)_2+1+2(s_1(\lambda)-\alpha_2)_2+1+(s_1(\lambda)-2\alpha_2)_2}\Ev'\iota\left(\xy (0,0)*{
			\labellist
			\small\hair 2pt
			\pinlabel \scalebox{0.7}{$\textcolor{red}{2}$} at 40 -5
                \pinlabel \scalebox{0.7}{$\textcolor{red}{2}$} at 50 -5
			\pinlabel \scalebox{0.9}{$s_1(\lambda)$} at 64 25
			\endlabellist 
			\centering 
			\includegraphics[scale=1.3]{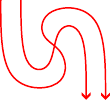}
	}\endxy\qquad
    \right)\\
     \\
     &= \sigma_1\tilde{\mathcal{T}}_{1,-1}'\left(\xy (0,0)*{
			\labellist
			\small\hair 2pt
			\pinlabel \scalebox{0.7}{$\textcolor{red}{2}$} at 40 -5
                \pinlabel \scalebox{0.7}{$\textcolor{red}{2}$} at 50 -5
			\pinlabel \scalebox{0.9}{$s_1(\lambda)$} at 64 25
			\endlabellist 
			\centering 
			\includegraphics[scale=1.3]{./RRDownRightCross}
	}\endxy\qquad
    \right)
     \sim_h\sigma_1\tilde{\mathcal{T}}_{1,-1}'\left(\xy (0,0)*{
			\labellist
			\small\hair 2pt
			\pinlabel \scalebox{0.7}{$\textcolor{red}{2}$} at 16 -5
			\pinlabel \scalebox{0.7}{$\textcolor{red}{2}$} at 0 -5
			\pinlabel \scalebox{0.9}{$s_1(\lambda)$} at 24 8
			\endlabellist 
			\centering 
			\includegraphics[scale=1.3]{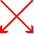}
	}\endxy\qquad\right)\\
     \\
     &= \Ev'\iota\left(\xy (0,0)*{
			\labellist
			\small\hair 2pt
			\pinlabel \scalebox{0.7}{$\textcolor{red}{2}$} at 16 -5
			\pinlabel \scalebox{0.7}{$\textcolor{red}{2}$} at 0 -5
			\pinlabel \scalebox{0.9}{$s_1(\lambda)$} at 24 8
			\endlabellist 
			\centering 
			\includegraphics[scale=1.3]{./RdrRdl2}
	}\endxy\qquad\right)
    = \Ev'\left(\xy (0,0)*{
			\labellist
			\small\hair 2pt
			\pinlabel \scalebox{0.7}{$\textcolor{black}{3}$} at 16 -5
			\pinlabel \scalebox{0.7}{$\textcolor{black}{3}$} at 0 -5
			\pinlabel \scalebox{0.9}{$\lambda$} at 18 8
			\endlabellist 
			\centering 
			\includegraphics[scale=1.3]{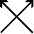}
	}\endxy\;\;
    \right)
    \end{align*}
\eqskip
    The other crossing cyclicity identities are similar.
    \item[KM4:] \begin{align*}
        \Ev'\left(\xy (0,0)*{
			\labellist
			\small\hair 2pt
			\pinlabel \scalebox{0.7}{$\textcolor{black}{3}$} at 0 -5
			\pinlabel \scalebox{0.7}{$\textcolor{blue}{1}$} at 16 -5
			\pinlabel \scalebox{0.9}{$\lambda$} at 18 16
			\endlabellist 
			\centering 
			\includegraphics[scale=1.3]{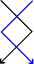}
	}\endxy\;\;
        \right)& = k_1^{0,1}(s_1(\lambda))k_1^{2,3}(s_1(\lambda))\Ev'\iota\left(\xy (0,0)*{
			\labellist
			\small\hair 2pt
			\pinlabel \scalebox{0.7}{$\textcolor{red}{2}$} at 0 -5
			\pinlabel \scalebox{0.7}{$\textcolor{blue}{1}$} at 16 -5
			\pinlabel \scalebox{0.9}{$s_1(\lambda)$} at 26 16
			\endlabellist 
			\centering 
			\includegraphics[scale=1.3]{./RurulBulur}
	}\endxy\qquad
        \right)
         = -\sigma_1\tilde{\mathcal{T}}_{1,-1}'\left(\xy (0,0)*{
			\labellist
			\small\hair 2pt
			\pinlabel \scalebox{0.7}{$\textcolor{red}{2}$} at 0 -5
			\pinlabel \scalebox{0.7}{$\textcolor{blue}{1}$} at 16 -5
			\pinlabel \scalebox{0.9}{$s_1(\lambda)$} at 26 16
			\endlabellist 
			\centering 
			\includegraphics[scale=1.3]{./RurulBulur}
	}\endxy\qquad
        \right) \\
        & \\
        & \sim_h-\sigma_1\tilde{\mathcal{T}}_{1,-1}'\left(\xy (0,0)*{
			\labellist
			\small\hair 2pt
			\pinlabel \scalebox{0.7}{$\textcolor{red}{2}$} at 0 -5
			\pinlabel \scalebox{0.7}{$\textcolor{blue}{1}$} at 13 -5
			\pinlabel \scalebox{0.9}{$s_1(\lambda)$} at 24 16
			\endlabellist 
			\centering 
			\includegraphics[scale=1.3]{./RuouBuu}
	}\endxy\qquad
    -\xy (0,0)*{
			\labellist
			\small\hair 2pt
			\pinlabel \scalebox{0.7}{$\textcolor{red}{2}$} at 0 -5
			\pinlabel \scalebox{0.7}{$\textcolor{blue}{1}$} at 13 -5
			\pinlabel \scalebox{0.9}{$s_1(\lambda)$} at 24 16
			\endlabellist 
			\centering 
			\includegraphics[scale=1.3]{./RuuBuou}
	}\endxy\qquad\right)\\
    & \\
    & = \Ev'\iota\left(\xy (0,0)*{
			\labellist
			\small\hair 2pt
			\pinlabel \scalebox{0.7}{$\textcolor{red}{2}$} at 0 -5
			\pinlabel \scalebox{0.7}{$\textcolor{blue}{1}$} at 13 -5
			\pinlabel \scalebox{0.9}{$s_1(\lambda)$} at 24 16
			\endlabellist 
			\centering 
			\includegraphics[scale=1.3]{./RuuBuou}
	}\endxy\qquad
    -\xy (0,0)*{
			\labellist
			\small\hair 2pt
			\pinlabel \scalebox{0.7}{$\textcolor{red}{2}$} at 0 -5
			\pinlabel \scalebox{0.7}{$\textcolor{blue}{1}$} at 13 -5
			\pinlabel \scalebox{0.9}{$s_1(\lambda)$} at 24 16
			\endlabellist 
			\centering 
			\includegraphics[scale=1.3]{./RuouBuu}
	}\endxy\qquad\right)\\
    & \\
    & =\Ev'\left(\xy (0,0)*{
			\labellist
			\small\hair 2pt
			\pinlabel \scalebox{0.7}{$\textcolor{black}{3}$} at 0 -5
			\pinlabel \scalebox{0.7}{$\textcolor{blue}{1}$} at 13 -5
			\pinlabel \scalebox{0.9}{$\lambda$} at 18 16
			\endlabellist 
			\centering 
			\includegraphics[scale=1.3]{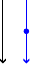}
	}\endxy\;\;
    -\xy (0,0)*{
			\labellist
			\small\hair 2pt
			\pinlabel \scalebox{0.7}{$\textcolor{black}{3}$} at 0 -5
			\pinlabel \scalebox{0.7}{$\textcolor{blue}{1}$} at 13 -5
			\pinlabel \scalebox{0.9}{$\lambda$} at 18 16
			\endlabellist 
			\centering 
			\includegraphics[scale=1.3]{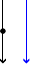}
	}\endxy\;\;\right)
    \end{align*}
    \begin{align*}
        \Ev'\left(\xy (0,0)*{
			\labellist
			\small\hair 2pt
			\pinlabel \scalebox{0.7}{$\textcolor{black}{3}$} at 0 -5
			\pinlabel \scalebox{0.7}{$\textcolor{red}{2}$} at 16 -5
			\pinlabel \scalebox{0.9}{$\lambda$} at 18 16
			\endlabellist 
			\centering 
			\includegraphics[scale=1.3]{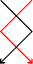}
	}\endxy\;\;
        \right)& = k_1^{0,1}(s_2(\lambda))k_1^{2,3}(s_2(\lambda))\Ev'\iota'\left(\xy (0,0)*{
			\labellist
			\small\hair 2pt
			\pinlabel \scalebox{0.7}{$\textcolor{red}{2}$} at 16 -5
			\pinlabel \scalebox{0.7}{$\textcolor{blue}{1}$} at 0 -5
			\pinlabel \scalebox{0.9}{$s_2(\lambda)$} at 26 16
			\endlabellist 
			\centering 
			\includegraphics[scale=1.3]{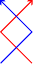}
	}\endxy\qquad
        \right)
         = -\sigma_2\alpha\tilde{\mathcal{T}}_{2,1}''\beta\left(\xy (0,0)*{
			\labellist
			\small\hair 2pt
			\pinlabel \scalebox{0.7}{$\textcolor{red}{2}$} at 16 -5
			\pinlabel \scalebox{0.7}{$\textcolor{blue}{1}$} at 0 -5
			\pinlabel \scalebox{0.9}{$s_2(\lambda)$} at 26 16
			\endlabellist 
			\centering 
			\includegraphics[scale=1.3]{./BurulRulur}
	}\endxy\qquad
        \right) \\
        & \\
        & \sim_h-\sigma_2\alpha\tilde{\mathcal{T}}_{2,1}''\beta\left(\xy (0,0)*{
			\labellist
			\small\hair 2pt
			\pinlabel \scalebox{0.7}{$\textcolor{red}{2}$} at 13 -5
			\pinlabel \scalebox{0.7}{$\textcolor{blue}{1}$} at 0 -5
			\pinlabel \scalebox{0.9}{$s_2(\lambda)$} at 24 16
			\endlabellist 
			\centering 
			\includegraphics[scale=1.3]{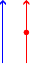}
	}\endxy\qquad
    -\xy (0,0)*{
			\labellist
			\small\hair 2pt
			\pinlabel \scalebox{0.7}{$\textcolor{red}{2}$} at 13 -5
			\pinlabel \scalebox{0.7}{$\textcolor{blue}{1}$} at 0 -5
			\pinlabel \scalebox{0.9}{$s_2(\lambda)$} at 24 16
			\endlabellist 
			\centering 
			\includegraphics[scale=1.3]{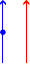}
	}\endxy\qquad\right)\\
    & \\
    & = \Ev'\iota'\left(\xy (0,0)*{
			\labellist
			\small\hair 2pt
			\pinlabel \scalebox{0.7}{$\textcolor{red}{2}$} at 13 -5
			\pinlabel \scalebox{0.7}{$\textcolor{blue}{1}$} at 0 -5
			\pinlabel \scalebox{0.9}{$s_2(\lambda)$} at 24 16
			\endlabellist 
			\centering 
			\includegraphics[scale=1.3]{./BuouRuu}
	}\endxy\qquad
    -\xy (0,0)*{
			\labellist
			\small\hair 2pt
			\pinlabel \scalebox{0.7}{$\textcolor{red}{2}$} at 13 -5
			\pinlabel \scalebox{0.7}{$\textcolor{blue}{1}$} at 0 -5
			\pinlabel \scalebox{0.9}{$s_2(\lambda)$} at 24 16
			\endlabellist 
			\centering 
			\includegraphics[scale=1.3]{./BuuRuou}
	}\endxy\qquad\right)\\
    & \\
    & =\Ev'\left(\xy (0,0)*{
			\labellist
			\small\hair 2pt
			\pinlabel \scalebox{0.7}{$\textcolor{black}{3}$} at 0 -5
			\pinlabel \scalebox{0.7}{$\textcolor{red}{2}$} at 13 -5
			\pinlabel \scalebox{0.9}{$\lambda$} at 18 16
			\endlabellist 
			\centering 
			\includegraphics[scale=1.3]{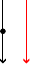}
	}\endxy\;\;
    -\xy (0,0)*{
			\labellist
			\small\hair 2pt
			\pinlabel \scalebox{0.7}{$\textcolor{black}{3}$} at 0 -5
			\pinlabel \scalebox{0.7}{$\textcolor{red}{2}$} at 13 -5
			\pinlabel \scalebox{0.9}{$\lambda$} at 18 16
			\endlabellist 
			\centering 
			\includegraphics[scale=1.3]{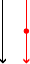}
	}\endxy\;\;\right)
\\[1ex]
        \Ev'\left(\xy (0,0)*{
			\labellist
			\small\hair 2pt
			\pinlabel \scalebox{0.7}{$\textcolor{black}{3}$} at 0 -5
			\pinlabel \scalebox{0.7}{$\textcolor{black}{3}$} at 16 -5
			\pinlabel \scalebox{0.9}{$\lambda$} at 18 16
			\endlabellist 
			\centering 
			\includegraphics[scale=1.3]{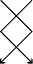}
	}\endxy\;\;
        \right)& = \Ev'\iota\left(\xy (0,0)*{
			\labellist
			\small\hair 2pt
			\pinlabel \scalebox{0.7}{$\textcolor{red}{2}$} at 16 -5
			\pinlabel \scalebox{0.7}{$\textcolor{red}{2}$} at 0 -5
			\pinlabel \scalebox{0.9}{$s_1(\lambda)$} at 26 16
			\endlabellist 
			\centering 
			\includegraphics[scale=1.3]{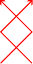}
	}\endxy\qquad
        \right) = \Ev'\iota(0)=0=\Ev'(0)
        \end{align*}
\eqskip
        The other two identities are similar.
        \item[KM5:] \begin{align*}
            \Ev'\left(\xy (0,0)*{
			\labellist
			\small\hair 2pt
			\pinlabel \scalebox{0.7}{$\textcolor{black}{3}$} at 0 -5
			\pinlabel \scalebox{0.7}{$\textcolor{blue}{1}$} at 16 -5
			\pinlabel \scalebox{0.9}{$\lambda$} at 18 7
			\endlabellist 
			\centering 
			\includegraphics[scale=1.3]{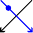}
	}\endxy
 \;\;\; - \ 
    \xy (0,0)*{
			\labellist
			\small\hair 2pt
			\pinlabel \scalebox{0.7}{$\textcolor{black}{3}$} at 0 -5
			\pinlabel \scalebox{0.7}{$\textcolor{blue}{1}$} at 16 -5
			\pinlabel \scalebox{0.9}{$\lambda$} at 18 7
			\endlabellist 
			\centering 
			\includegraphics[scale=1.3]{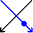}
	}\endxy\;\;\right) & = k_1^{0,1}(s_1(\lambda))\Ev'\iota\left(\xy (0,0)*{
			\labellist
			\small\hair 2pt
			\pinlabel \scalebox{0.7}{$\textcolor{red}{2}$} at 0 -5
			\pinlabel \scalebox{0.7}{$\textcolor{blue}{1}$} at 16 -5
			\pinlabel \scalebox{0.9}{$s_1(\lambda)$} at 24 7
			\endlabellist 
			\centering 
			\includegraphics[scale=1.3]{./RurBulo}
	}\endxy
 \qquad - \ 
    \xy (0,0)*{
			\labellist
			\small\hair 2pt
			\pinlabel \scalebox{0.7}{$\textcolor{red}{2}$} at 0 -5
			\pinlabel \scalebox{0.7}{$\textcolor{blue}{1}$} at 16 -5
			\pinlabel \scalebox{0.9}{$s_1(\lambda)$} at 24 7
			\endlabellist 
			\centering 
			\includegraphics[scale=1.3]{./RurBoul}
	}\endxy\qquad\right) = \Ev'\iota(0)=0=\Ev'(0)
        \end{align*}
        \begin{align*}
            \Ev'\left(\xy (0,0)*{
			\labellist
			\small\hair 2pt
			\pinlabel \scalebox{0.7}{$\textcolor{black}{3}$} at 16 -5
			\pinlabel \scalebox{0.7}{$\textcolor{red}{2}$} at 0 -5
			\pinlabel \scalebox{0.9}{$\lambda$} at 18 7
			\endlabellist 
			\centering 
			\includegraphics[scale=1.3]{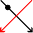}
	}\endxy
 \;\;\; - \ 
    \xy (0,0)*{
			\labellist
			\small\hair 2pt
			\pinlabel \scalebox{0.7}{$\textcolor{black}{3}$} at 16 -5
			\pinlabel \scalebox{0.7}{$\textcolor{red}{2}$} at 0 -5
			\pinlabel \scalebox{0.9}{$\lambda$} at 18 7
			\endlabellist 
			\centering 
			\includegraphics[scale=1.3]{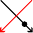}
	}\endxy\;\;\right) & = k_1^{0,1}(s_2(\lambda))\Ev'\iota'\left(\xy (0,0)*{
			\labellist
			\small\hair 2pt
			\pinlabel \scalebox{0.7}{$\textcolor{red}{2}$} at 0 -5
			\pinlabel \scalebox{0.7}{$\textcolor{blue}{1}$} at 16 -5
			\pinlabel \scalebox{0.9}{$s_2(\lambda)$} at 24 7
			\endlabellist 
			\centering 
			\includegraphics[scale=1.3]{./RurBulo}
	}\endxy
 \qquad - \ 
    \xy (0,0)*{
			\labellist
			\small\hair 2pt
			\pinlabel \scalebox{0.7}{$\textcolor{red}{2}$} at 0 -5
			\pinlabel \scalebox{0.7}{$\textcolor{blue}{1}$} at 16 -5
			\pinlabel \scalebox{0.9}{$s_2(\lambda)$} at 24 7
			\endlabellist 
			\centering 
			\includegraphics[scale=1.3]{./RurBoul}
	}\endxy\qquad\right) = \Ev'\iota(0)=0=\Ev'(0)
        \end{align*}

        \begin{align*}
            \Ev'\left(\xy (0,0)*{
			\labellist
			\small\hair 2pt
			\pinlabel \scalebox{0.7}{$\textcolor{black}{3}$} at 16 -5
			\pinlabel \scalebox{0.7}{$\textcolor{black}{3}$} at 0 -5
			\pinlabel \scalebox{0.9}{$\lambda$} at 18 7
			\endlabellist 
			\centering 
			\includegraphics[scale=1.3]{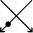}
	}\endxy
 \;\;\; - \ 
    \xy (0,0)*{
			\labellist
			\small\hair 2pt
			\pinlabel \scalebox{0.7}{$\textcolor{black}{3}$} at 16 -5
			\pinlabel \scalebox{0.7}{$\textcolor{black}{3}$} at 0 -5
			\pinlabel \scalebox{0.9}{$\lambda$} at 18 7
			\endlabellist 
			\centering 
			\includegraphics[scale=1.3]{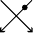}
	}\endxy\;\;\right) = \Ev'\iota\left(\xy (0,0)*{
			\labellist
			\small\hair 2pt
			\pinlabel \scalebox{0.7}{$\textcolor{red}{2}$} at 0 -5
			\pinlabel \scalebox{0.7}{$\textcolor{red}{2}$} at 16 -5
			\pinlabel \scalebox{0.9}{$s_1(\lambda)$} at 24 7
			\endlabellist 
			\centering 
			\includegraphics[scale=1.3]{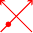}
	}\endxy
 \qquad - \ 
    \xy (0,0)*{
			\labellist
			\small\hair 2pt
			\pinlabel \scalebox{0.7}{$\textcolor{red}{2}$} at 0 -5
			\pinlabel \scalebox{0.7}{$\textcolor{red}{2}$} at 16 -5
			\pinlabel \scalebox{0.9}{$s_1(\lambda)$} at 24 7
			\endlabellist 
			\centering 
			\includegraphics[scale=1.3]{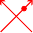}
	}\endxy\qquad\right)\\
    \\
    =\sigma_1\tilde{\mathcal{T}}_{1,-1}'\left(\xy (0,0)*{
			\labellist
			\small\hair 2pt
			\pinlabel \scalebox{0.7}{$\textcolor{red}{2}$} at 0 -5
			\pinlabel \scalebox{0.7}{$\textcolor{red}{2}$} at 16 -5
			\pinlabel \scalebox{0.9}{$s_1(\lambda)$} at 24 7
			\endlabellist 
			\centering 
			\includegraphics[scale=1.3]{./RourRul}
	}\endxy
 \qquad - \ 
    \xy (0,0)*{
			\labellist
			\small\hair 2pt
			\pinlabel \scalebox{0.7}{$\textcolor{red}{2}$} at 0 -5
			\pinlabel \scalebox{0.7}{$\textcolor{red}{2}$} at 16 -5
			\pinlabel \scalebox{0.9}{$s_1(\lambda)$} at 24 7
			\endlabellist 
			\centering 
			\includegraphics[scale=1.3]{./RuroRul}
	}\endxy\qquad\right)\sim_h \sigma_1\tilde{\mathcal{T}}_{1,-1}'\left( \xy (0,0)*{
			\labellist
			\small\hair 2pt
			\pinlabel \scalebox{0.7}{$\textcolor{red}{2}$} at 12 -5
			\pinlabel \scalebox{0.7}{$\textcolor{red}{2}$} at 0 -5
			\pinlabel \scalebox{0.9}{$s_1(\lambda)$} at 25 7
			\endlabellist 
			\centering 
			\includegraphics[scale=1.3]{./RuRu}
	}\endxy\qquad\right)\\
    \\
    =\Ev'\iota\left(\xy (0,0)*{
			\labellist
			\small\hair 2pt
			\pinlabel \scalebox{0.7}{$\textcolor{red}{2}$} at 12 -5
			\pinlabel \scalebox{0.7}{$\textcolor{red}{2}$} at 0 -5
			\pinlabel \scalebox{0.9}{$s_1(\lambda)$} at 25 7
			\endlabellist 
			\centering 
			\includegraphics[scale=1.3]{./RuRu}
	}\endxy\qquad\right)=\Ev'\left(\xy (0,0)*{
			\labellist
			\small\hair 2pt
			\pinlabel \scalebox{0.7}{$\textcolor{black}{3}$} at 12 -5
			\pinlabel \scalebox{0.7}{$\textcolor{black}{3}$} at 0 -5
			\pinlabel \scalebox{0.9}{$\lambda$} at 18 7
			\endlabellist 
			\centering 
			\includegraphics[scale=1.3]{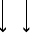}
	}\endxy\;\;\right)
        \end{align*}
        The other two identities are similar.
        \item[KM6:]{\quad}\newline
        $\Ev'\left(\xy (0,0)*{
			\labellist
			\small\hair 2pt
			\pinlabel \scalebox{0.7}{$\textcolor{blue}{1}$} at 0 -5
			\pinlabel \scalebox{0.7}{$\textcolor{black}{3}$} at 16 -5
            \pinlabel \scalebox{0.7}{$\textcolor{blue}{1}$} at 32 -5
			\pinlabel \scalebox{0.9}{$\lambda$} at 34 16
			\endlabellist 
			\centering 
			\includegraphics[scale=1.3]{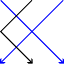}
	}\endxy
 \;\; - \, 
 \xy (0,0)*{
			\labellist
			\small\hair 2pt
			\pinlabel \scalebox{0.7}{$\textcolor{blue}{1}$} at 0 -5
			\pinlabel \scalebox{0.7}{$\textcolor{black}{3}$} at 16 -5
            \pinlabel \scalebox{0.7}{$\textcolor{blue}{1}$} at 32 -5
			\pinlabel \scalebox{0.9}{$\lambda$} at 34 16
			\endlabellist 
			\centering 
			\includegraphics[scale=1.3]{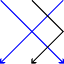}
	}\endxy\;\;\right) $\begin{align*}
         &= \Ev'\iota\left(-k_1^{0,1}(s_1(\lambda)+\alpha_1)k_1^{2,3}(s_1(\lambda)+\alpha_1)\xy (0,0)*{
			\labellist
			\small\hair 2pt
			\pinlabel \scalebox{0.7}{$\textcolor{blue}{1}$} at 0 -5
			\pinlabel \scalebox{0.7}{$\textcolor{red}{2}$} at 16 -5
            \pinlabel \scalebox{0.7}{$\textcolor{blue}{1}$} at 32 -5
			\pinlabel \scalebox{0.9}{$s_1(\lambda)$} at 33 16
			\endlabellist 
			\centering 
			\includegraphics[scale=1.3]{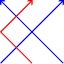}
	}\endxy
 \mspace{20mu} +k_1^{0,1}(s_1(\lambda))k_1^{2,3}(s_1(\lambda)) \, 
 \xy (0,0)*{
			\labellist
			\small\hair 2pt
			\pinlabel \scalebox{0.7}{$\textcolor{blue}{1}$} at 0 -5
			\pinlabel \scalebox{0.7}{$\textcolor{red}{2}$} at 16 -5
            \pinlabel \scalebox{0.7}{$\textcolor{blue}{1}$} at 32 -5
			\pinlabel \scalebox{0.9}{$s_1(\lambda)$} at 40 16
			\endlabellist 
			\centering 
			\includegraphics[scale=1.3]{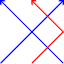}
	}\endxy\mspace{30mu}
    \right)\\
    & \\
    & =\sigma_1\tilde{\mathcal{T}}_{1,-1}'\left(\xy (0,0)*{
			\labellist
			\small\hair 2pt
			\pinlabel \scalebox{0.7}{$\textcolor{blue}{1}$} at 0 -5
			\pinlabel \scalebox{0.7}{$\textcolor{red}{2}$} at 16 -5
            \pinlabel \scalebox{0.7}{$\textcolor{blue}{1}$} at 32 -5
			\pinlabel \scalebox{0.9}{$s_1(\lambda)$} at 40 16
			\endlabellist 
			\centering 
			\includegraphics[scale=1.3]{./BuurrRulurBuull}
	}\endxy
 \qquad - \, 
 \xy (0,0)*{
			\labellist
			\small\hair 2pt
			\pinlabel \scalebox{0.7}{$\textcolor{blue}{1}$} at 0 -5
			\pinlabel \scalebox{0.7}{$\textcolor{red}{2}$} at 16 -5
            \pinlabel \scalebox{0.7}{$\textcolor{blue}{1}$} at 32 -5
			\pinlabel \scalebox{0.9}{$s_1(\lambda)$} at 40 16
			\endlabellist 
			\centering 
			\includegraphics[scale=1.3]{./BuurrRurulBuull}
	}\endxy\qquad
     \right)
     =-\sigma_1\tilde{\mathcal{T}}_{1,-1}'\left(\xy (0,0)*{
			\labellist
			\small\hair 2pt
			\pinlabel \scalebox{0.7}{$\textcolor{blue}{1}$} at 0 -5
			\pinlabel \scalebox{0.7}{$\textcolor{red}{2}$} at 12 -5
            \pinlabel \scalebox{0.7}{$\textcolor{blue}{1}$} at 24 -5
			\pinlabel \scalebox{0.9}{$s_1(\lambda)$} at 34 16
			\endlabellist 
			\centering 
			\includegraphics[scale=1.3]{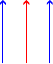}
	}\endxy\qquad\right)\\
  &  \\
   & =-\Ev'\iota\left(\xy (0,0)*{
			\labellist
			\small\hair 2pt
			\pinlabel \scalebox{0.7}{$\textcolor{blue}{1}$} at 0 -5
			\pinlabel \scalebox{0.7}{$\textcolor{red}{2}$} at 12 -5
            \pinlabel \scalebox{0.7}{$\textcolor{blue}{1}$} at 24 -5
			\pinlabel \scalebox{0.9}{$s_1(\lambda)$} at 34 16
			\endlabellist 
			\centering 
			\includegraphics[scale=1.3]{./BuuRuuBuu}
	}\endxy\qquad\right)
    =-\Ev'\left(\xy (0,0)*{
			\labellist
			\small\hair 2pt
			\pinlabel \scalebox{0.7}{$\textcolor{blue}{1}$} at 0 -5
			\pinlabel \scalebox{0.7}{$\textcolor{black}{3}$} at 12 -5
            \pinlabel \scalebox{0.7}{$\textcolor{blue}{1}$} at 24 -5
			\pinlabel \scalebox{0.9}{$s_1(\lambda)$} at 34 16
			\endlabellist 
			\centering 
			\includegraphics[scale=1.3]{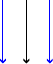}
	}\endxy\qquad\right)
        \end{align*}
\begin{align*}
        \Ev'\left(\xy (0,0)*{
			\labellist
			\small\hair 2pt
			\pinlabel \scalebox{0.7}{$\textcolor{red}{2}$} at 0 -5
			\pinlabel \scalebox{0.7}{$\textcolor{red}{2}$} at 16 -5
            \pinlabel \scalebox{0.7}{$\textcolor{black}{3}$} at 32 -5
			\pinlabel \scalebox{0.9}{$\lambda$} at 34 16
			\endlabellist 
			\centering 
			\includegraphics[scale=1.3]{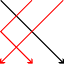}
	}\endxy
 \;\; - \, 
 \xy (0,0)*{
			\labellist
			\small\hair 2pt
			\pinlabel \scalebox{0.7}{$\textcolor{red}{2}$} at 0 -5
			\pinlabel \scalebox{0.7}{$\textcolor{red}{2}$} at 16 -5
            \pinlabel \scalebox{0.7}{$\textcolor{black}{3}$} at 32 -5
			\pinlabel \scalebox{0.9}{$\lambda$} at 34 16
			\endlabellist 
			\centering 
			\includegraphics[scale=1.3]{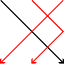}
	}\endxy\;\;\right) \mspace{-50mu} &
    \\[1ex] & = -k_1^{2,3}(s_2(\lambda)+\alpha_2)k_1^{2,3}(s_2(\lambda))\Ev'\iota'\left(\xy (0,0)*{
			\labellist
			\small\hair 2pt
			\pinlabel \scalebox{0.7}{$\textcolor{red}{2}$} at 0 -5
			\pinlabel \scalebox{0.7}{$\textcolor{red}{2}$} at 16 -5
            \pinlabel \scalebox{0.7}{$\textcolor{blue}{1}$} at 32 -5
			\pinlabel \scalebox{0.9}{$s_2(\lambda)$} at 40 16
			\endlabellist 
			\centering 
			\includegraphics[scale=1.3]{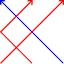}
	}\endxy
 \qquad - \, 
 \xy (0,0)*{
			\labellist
			\small\hair 2pt
			\pinlabel \scalebox{0.7}{$\textcolor{red}{2}$} at 0 -5
			\pinlabel \scalebox{0.7}{$\textcolor{red}{2}$} at 16 -5
            \pinlabel \scalebox{0.7}{$\textcolor{blue}{1}$} at 32 -5
			\pinlabel \scalebox{0.9}{$s_2(\lambda)$} at 40 16
			\endlabellist 
			\centering 
			\includegraphics[scale=1.3]{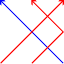}
	}\endxy\qquad
    \right)\\
    & \\
    & = \Ev'\iota'(0)=0=\Ev'(0)
    \end{align*}
    \begin{align*}
        \Ev'\left(\xy (0,0)*{
			\labellist
			\small\hair 2pt
			\pinlabel \scalebox{0.7}{$\textcolor{black}{3}$} at 0 -5
			\pinlabel \scalebox{0.7}{$\textcolor{black}{3}$} at 16 -5
            \pinlabel \scalebox{0.7}{$\textcolor{black}{3}$} at 32 -5
			\pinlabel \scalebox{0.9}{$\lambda$} at 34 16
			\endlabellist 
			\centering 
			\includegraphics[scale=1.3]{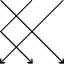}
	}\endxy
 \;\; - \, 
 \xy (0,0)*{
			\labellist
			\small\hair 2pt
			\pinlabel \scalebox{0.7}{$\textcolor{black}{3}$} at 0 -5
			\pinlabel \scalebox{0.7}{$\textcolor{black}{3}$} at 16 -5
            \pinlabel \scalebox{0.7}{$\textcolor{black}{3}$} at 32 -5
			\pinlabel \scalebox{0.9}{$\lambda$} at 34 16
			\endlabellist 
			\centering 
			\includegraphics[scale=1.3]{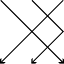}
	}\endxy\;\;\right) &= \Ev'\iota\left(\xy (0,0)*{
			\labellist
			\small\hair 2pt
			\pinlabel \scalebox{0.7}{$\textcolor{red}{2}$} at 0 -5
			\pinlabel \scalebox{0.7}{$\textcolor{red}{2}$} at 16 -5
            \pinlabel \scalebox{0.7}{$\textcolor{red}{2}$} at 32 -5
			\pinlabel \scalebox{0.9}{$s_1(\lambda)$} at 40 16
			\endlabellist 
			\centering 
			\includegraphics[scale=1.3]{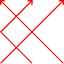}
	}\endxy
 \qquad - \, 
 \xy (0,0)*{
			\labellist
			\small\hair 2pt
			\pinlabel \scalebox{0.7}{$\textcolor{red}{2}$} at 0 -5
			\pinlabel \scalebox{0.7}{$\textcolor{red}{2}$} at 16 -5
            \pinlabel \scalebox{0.7}{$\textcolor{red}{2}$} at 32 -5
			\pinlabel \scalebox{0.9}{$s_1(\lambda)$} at 40 16
			\endlabellist 
			\centering 
			\includegraphics[scale=1.3]{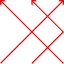}
	}\endxy\qquad
    \right)\\
    & \\
    & = \Ev'\iota'(0)=0=\Ev'(0)
    \end{align*}
    With the exception of the three-coloured identities discussed below, the other cubic KLR relations are similar. We prove three of these identities directly; the other three are similar.
 \begin{align*}
        \Ev'\left(\xy (0,0)*{
	\labellist
	\small\hair 2pt
	\pinlabel \scalebox{0.7}{$\textcolor{black}{3}$} at 12 -5
	\pinlabel \scalebox{0.7}{$\textcolor{blue}{1}$} at 24 -5
	\pinlabel \scalebox{0.7}{$\textcolor{red}{2}$} at 0 -5
        \pinlabel \scalebox{0.9}{$\lambda$} at 24 12
        \endlabellist
	\centering
	\includegraphics[scale=1]{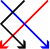} }\endxy\;\;\right)& = k_3^{0,3}(\lambda-\alpha_1)k_1^{0,3}(\lambda-\alpha_1)(-1)^{\overline{\lambda}_1(\overline{\lambda}_2+1)}\left(\xy (0,0)*{
			\labellist
			\small\hair 2pt
			\pinlabel \scalebox{0.7}{$\textcolor{blue}{1}$} at 12 -5
			\pinlabel \scalebox{0.7}{$\textcolor{red}{2}$} at 0 -5
			\pinlabel \scalebox{0.7}{$\textcolor{red}{2}$} at 24 -5
			\pinlabel \scalebox{0.7}{$\textcolor{blue}{1}$} at 35 -5
                \pinlabel \scalebox{0.9}{$\lambda$} at 36 24
                \endlabellist
			\centering
			\includegraphics[scale=1.15]{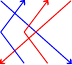}
	}\endxy\;\;\;, \xy (0,0)*{
			\labellist
			\small\hair 2pt
			\pinlabel \scalebox{0.7}{$\textcolor{blue}{1}$} at 24 -5
			\pinlabel \scalebox{0.7}{$\textcolor{red}{2}$} at 0 -5
			\pinlabel \scalebox{0.7}{$\textcolor{red}{2}$} at 12 -5
			\pinlabel \scalebox{0.7}{$\textcolor{blue}{1}$} at 35 -5
                \pinlabel \scalebox{0.9}{$\lambda$} at 36 24
                \endlabellist
			\centering
			\includegraphics[scale=1.15]{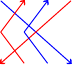}
	}\endxy\;\;\right)
\end{align*}
and 
\begin{align*}
        \Ev'\left(\xy (0,0)*{
	\labellist
	\small\hair 2pt
	\pinlabel \scalebox{0.7}{$\textcolor{black}{3}$} at 12 -5
	\pinlabel \scalebox{0.7}{$\textcolor{blue}{1}$} at 24 -5
	\pinlabel \scalebox{0.7}{$\textcolor{red}{2}$} at 0 -5
        \pinlabel \scalebox{0.9}{$\lambda$} at 28 12
        \endlabellist
	\centering
	\includegraphics[scale=1]{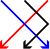} }\endxy\;\;\right)& = k_3^{0,3}(\lambda)k_1^{0,3}(\lambda)(-1)^{(\overline{\lambda-\alpha_3})_1((\overline{\lambda-\alpha_3})_2+1)}\left(\xy (0,0)*{
			\labellist
			\small\hair 2pt
			\pinlabel \scalebox{0.7}{$\textcolor{blue}{1}$} at 12 -5
			\pinlabel \scalebox{0.7}{$\textcolor{red}{2}$} at 0 -5
			\pinlabel \scalebox{0.7}{$\textcolor{red}{2}$} at 24 -5
			\pinlabel \scalebox{0.7}{$\textcolor{blue}{1}$} at 35 -5
                \pinlabel \scalebox{0.9}{$\lambda$} at 36 24
                \endlabellist
			\centering
			\includegraphics[scale=1.15]{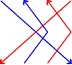}
	}\endxy\;\;\;, \xy (0,0)*{
			\labellist
			\small\hair 2pt
			\pinlabel \scalebox{0.7}{$\textcolor{blue}{1}$} at 24 -5
			\pinlabel \scalebox{0.7}{$\textcolor{red}{2}$} at 0 -5
			\pinlabel \scalebox{0.7}{$\textcolor{red}{2}$} at 12 -5
			\pinlabel \scalebox{0.7}{$\textcolor{blue}{1}$} at 35 -5
                \pinlabel \scalebox{0.9}{$\lambda$} at 36 24
                \endlabellist
			\centering
			\includegraphics[scale=1.15]{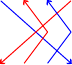}
	}\endxy\;\;\right)
\end{align*}
\eqskip
It is straightforward to check that the signs are both equal to $(-1)^{\overline{\lambda}_1}k_2^{0,1}$, and the diagrams themselves are clearly equal via mixed cubic KLR relations.

\smallskip 

Using KM6 and KM5, it is straightforward to show that 
\begin{align*} &\Ev'\left(\xy (0,0)*{
	\labellist
	\small\hair 2pt
	\pinlabel \scalebox{0.7}{$\textcolor{black}{3}$} at 0 -5
	\pinlabel \scalebox{0.7}{$\textcolor{blue}{1}$} at 12 -5
	\pinlabel \scalebox{0.7}{$\textcolor{red}{2}$} at 24 -5
        \pinlabel \scalebox{0.9}{$\lambda$} at 28 12
        \endlabellist
	\centering
	\includegraphics[scale=1]{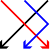} }\endxy\;\;\right)-\Ev'\left(\xy (0,0)*{
	\labellist
	\small\hair 2pt
	\pinlabel \scalebox{0.7}{$\textcolor{black}{3}$} at 0 -5
	\pinlabel \scalebox{0.7}{$\textcolor{blue}{1}$} at 12 -5
	\pinlabel \scalebox{0.7}{$\textcolor{red}{2}$} at 24 -5
        \pinlabel \scalebox{0.9}{$\lambda$} at 28 12
        \endlabellist
	\centering
	\includegraphics[scale=1]{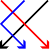} }\endxy\;\;\right)\\
    & \\
    &= (-1)^{\overline{\lambda}_1}k_2^{0,1}\left(\xy (0,0)*{
			\labellist
			\small\hair 2pt
			\pinlabel \scalebox{0.7}{$\textcolor{blue}{1}$} at 24 -5
			\pinlabel \scalebox{0.7}{$\textcolor{red}{2}$} at 35 -5
			\pinlabel \scalebox{0.7}{$\textcolor{red}{2}$} at 12 -5
			\pinlabel \scalebox{0.7}{$\textcolor{blue}{1}$} at 0 -5
                \pinlabel \scalebox{0.9}{$\lambda$} at 36 24
                \endlabellist
			\centering
			\includegraphics[scale=1.15]{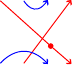}
	}\endxy\;\;-\xy (0,0)*{
			\labellist
			\small\hair 2pt
			\pinlabel \scalebox{0.7}{$\textcolor{blue}{1}$} at 24 -5
			\pinlabel \scalebox{0.7}{$\textcolor{red}{2}$} at 35 -5
			\pinlabel \scalebox{0.7}{$\textcolor{red}{2}$} at 12 -5
			\pinlabel \scalebox{0.7}{$\textcolor{blue}{1}$} at 0 -5
                \pinlabel \scalebox{0.9}{$\lambda$} at 36 24
                \endlabellist
			\centering
			\includegraphics[scale=1.15]{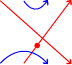}
	}\endxy\;\;\;, \xy (0,0)*{
			\labellist
			\small\hair 2pt
			\pinlabel \scalebox{0.7}{$\textcolor{blue}{1}$} at 24 -5
			\pinlabel \scalebox{0.7}{$\textcolor{red}{2}$} at 35 -5
			\pinlabel \scalebox{0.7}{$\textcolor{red}{2}$} at 0 -5
			\pinlabel \scalebox{0.7}{$\textcolor{blue}{1}$} at 12 -5
                \pinlabel \scalebox{0.9}{$\lambda$} at 36 24
                \endlabellist
			\centering
			\includegraphics[scale=1.15]{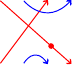}
	}\endxy\;\;-\xy (0,0)*{
			\labellist
			\small\hair 2pt
			\pinlabel \scalebox{0.7}{$\textcolor{blue}{1}$} at 24 -5
			\pinlabel \scalebox{0.7}{$\textcolor{red}{2}$} at 35 -5
			\pinlabel \scalebox{0.7}{$\textcolor{red}{2}$} at 0 -5
			\pinlabel \scalebox{0.7}{$\textcolor{blue}{1}$} at 12 -5
                \pinlabel \scalebox{0.9}{$\lambda$} at 36 24
                \endlabellist
			\centering
			\includegraphics[scale=1.15]{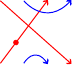}
	}\endxy\;\;
    \right)
\end{align*}
\eqskip
and similarly straightforward to see that this is null-homotopic via the homotopy $$(-1)^{\overline{\lambda}_1}k_2^{0,1}\left(\xy (0,0)*{
			\labellist
			\small\hair 2pt
			\pinlabel \scalebox{0.7}{$\textcolor{blue}{1}$} at 24 -5
			\pinlabel \scalebox{0.7}{$\textcolor{red}{2}$} at 35 -5
			\pinlabel \scalebox{0.7}{$\textcolor{red}{2}$} at 0 -5
			\pinlabel \scalebox{0.7}{$\textcolor{blue}{1}$} at 12 -5
                \pinlabel \scalebox{0.9}{$\lambda$} at 36 24
                \endlabellist
			\centering
			\includegraphics[scale=1.15]{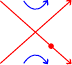}
	}\endxy\;\;- \xy (0,0)*{
			\labellist
			\small\hair 2pt
			\pinlabel \scalebox{0.7}{$\textcolor{blue}{1}$} at 24 -5
			\pinlabel \scalebox{0.7}{$\textcolor{red}{2}$} at 35 -5
			\pinlabel \scalebox{0.7}{$\textcolor{red}{2}$} at 0 -5
			\pinlabel \scalebox{0.7}{$\textcolor{blue}{1}$} at 12 -5
                \pinlabel \scalebox{0.9}{$\lambda$} at 36 24
                \endlabellist
			\centering
			\includegraphics[scale=1.15]{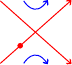}
	}\endxy\;\;\right)$$\eqskip 
    as required. Also,
    \begin{align*} 
    \Ev'\left(\xy (0,0)*{
	\labellist
	\small\hair 2pt
	\pinlabel \scalebox{0.7}{$\textcolor{black}{3}$} at 12 -5
	\pinlabel \scalebox{0.7}{$\textcolor{blue}{1}$} at 0 -5
	\pinlabel \scalebox{0.7}{$\textcolor{red}{2}$} at 24 -5
        \pinlabel \scalebox{0.9}{$\lambda$} at 28 12
        \endlabellist
	\centering
	\includegraphics[scale=1]{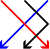} }\endxy\;\;\right)-\Ev'&\left(\xy (0,0)*{
	\labellist
	\small\hair 2pt
	\pinlabel \scalebox{0.7}{$\textcolor{black}{3}$} at 12 -5
	\pinlabel \scalebox{0.7}{$\textcolor{blue}{1}$} at 0 -5
	\pinlabel \scalebox{0.7}{$\textcolor{red}{2}$} at 24 -5
        \pinlabel \scalebox{0.9}{$\lambda$} at 28 12
        \endlabellist
	\centering
	\includegraphics[scale=1]{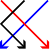} }\endxy\;\;\right) = (-1)^{\overline{\lambda}_1}k_2^{0,1}\left(\xy (0,0)*{
			\labellist
			\small\hair 2pt
			\pinlabel \scalebox{0.7}{$\textcolor{blue}{1}$} at 0 -5
			\pinlabel \scalebox{0.7}{$\textcolor{red}{2}$} at 35 -5
			\pinlabel \scalebox{0.7}{$\textcolor{red}{2}$} at 24 -5
			\pinlabel \scalebox{0.7}{$\textcolor{blue}{1}$} at 12 -5
                \pinlabel \scalebox{0.9}{$\lambda$} at 36 24
                \endlabellist
			\centering
			\includegraphics[scale=1.15]{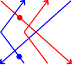}
	}\endxy\;\;+\xy (0,0)*{
			\labellist
			\small\hair 2pt
			\pinlabel \scalebox{0.7}{$\textcolor{blue}{1}$} at 0 -5
			\pinlabel \scalebox{0.7}{$\textcolor{red}{2}$} at 35 -5
			\pinlabel \scalebox{0.7}{$\textcolor{red}{2}$} at 24 -5
			\pinlabel \scalebox{0.7}{$\textcolor{blue}{1}$} at 12 -5
                \pinlabel \scalebox{0.9}{$\lambda$} at 36 24
                \endlabellist
			\centering
			\includegraphics[scale=1.15]{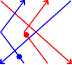}
	}\endxy\;\;+\xy (0,0)*{
			\labellist
			\small\hair 2pt
			\pinlabel \scalebox{0.7}{$\textcolor{blue}{1}$} at 0 -5
			\pinlabel \scalebox{0.7}{$\textcolor{red}{2}$} at 35 -5
			\pinlabel \scalebox{0.7}{$\textcolor{red}{2}$} at 24 -5
			\pinlabel \scalebox{0.7}{$\textcolor{blue}{1}$} at 12 -5
                \pinlabel \scalebox{0.9}{$\lambda$} at 36 24
                \endlabellist
			\centering
			\includegraphics[scale=1.15]{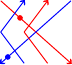}
	}\endxy
    \right.\\
    & \\
    & + \xy (0,0)*{
			\labellist
			\small\hair 2pt
			\pinlabel \scalebox{0.7}{$\textcolor{blue}{1}$} at 0 -5
			\pinlabel \scalebox{0.7}{$\textcolor{red}{2}$} at 35 -5
			\pinlabel \scalebox{0.7}{$\textcolor{red}{2}$} at 24 -5
			\pinlabel \scalebox{0.7}{$\textcolor{blue}{1}$} at 12 -5
                \pinlabel \scalebox{0.9}{$\lambda$} at 36 24
                \endlabellist
			\centering
			\includegraphics[scale=1.15]{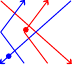}
	}\endxy\;\; - \xy (0,0)*{
			\labellist
			\small\hair 2pt
			\pinlabel \scalebox{0.7}{$\textcolor{blue}{1}$} at 0 -5
			\pinlabel \scalebox{0.7}{$\textcolor{red}{2}$} at 35 -5
			\pinlabel \scalebox{0.7}{$\textcolor{red}{2}$} at 24 -5
			\pinlabel \scalebox{0.7}{$\textcolor{blue}{1}$} at 12 -5
                \pinlabel \scalebox{0.9}{$\lambda$} at 36 24
                \endlabellist
			\centering
			\includegraphics[scale=1.15]{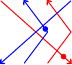}
	}\endxy\;\;-\xy (0,0)*{
			\labellist
			\small\hair 2pt
			\pinlabel \scalebox{0.7}{$\textcolor{blue}{1}$} at 0 -5
			\pinlabel \scalebox{0.7}{$\textcolor{red}{2}$} at 35 -5
			\pinlabel \scalebox{0.7}{$\textcolor{red}{2}$} at 24 -5
			\pinlabel \scalebox{0.7}{$\textcolor{blue}{1}$} at 12 -5
                \pinlabel \scalebox{0.9}{$\lambda$} at 36 24
                \endlabellist
			\centering
			\includegraphics[scale=1.15]{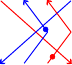}
	}\endxy\;\;-\xy (0,0)*{
			\labellist
			\small\hair 2pt
			\pinlabel \scalebox{0.7}{$\textcolor{blue}{1}$} at 0 -5
			\pinlabel \scalebox{0.7}{$\textcolor{red}{2}$} at 35 -5
			\pinlabel \scalebox{0.7}{$\textcolor{red}{2}$} at 24 -5
			\pinlabel \scalebox{0.7}{$\textcolor{blue}{1}$} at 12 -5
                \pinlabel \scalebox{0.9}{$\lambda$} at 36 24
                \endlabellist
			\centering
			\includegraphics[scale=1.15]{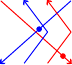}
	}\endxy\;\;-\xy (0,0)*{
			\labellist
			\small\hair 2pt
			\pinlabel \scalebox{0.7}{$\textcolor{blue}{1}$} at 0 -5
			\pinlabel \scalebox{0.7}{$\textcolor{red}{2}$} at 35 -5
			\pinlabel \scalebox{0.7}{$\textcolor{red}{2}$} at 24 -5
			\pinlabel \scalebox{0.7}{$\textcolor{blue}{1}$} at 12 -5
                \pinlabel \scalebox{0.9}{$\lambda$} at 36 24
                \endlabellist
			\centering
			\includegraphics[scale=1.15]{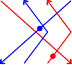}
	}\endxy\;\;\;,\\
    &\\
    & \xy (0,0)*{
			\labellist
			\small\hair 2pt
			\pinlabel \scalebox{0.7}{$\textcolor{blue}{1}$} at 0 -5
			\pinlabel \scalebox{0.7}{$\textcolor{red}{2}$} at 35 -5
			\pinlabel \scalebox{0.7}{$\textcolor{red}{2}$} at 12 -5
			\pinlabel \scalebox{0.7}{$\textcolor{blue}{1}$} at 24 -5
                \pinlabel \scalebox{0.9}{$\lambda$} at 36 24
                \endlabellist
			\centering
			\includegraphics[scale=1.15]{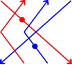}
	}\endxy\;\;+\xy (0,0)*{
			\labellist
			\small\hair 2pt
			\pinlabel \scalebox{0.7}{$\textcolor{blue}{1}$} at 0 -5
			\pinlabel \scalebox{0.7}{$\textcolor{red}{2}$} at 35 -5
			\pinlabel \scalebox{0.7}{$\textcolor{red}{2}$} at 12 -5
			\pinlabel \scalebox{0.7}{$\textcolor{blue}{1}$} at 24 -5
                \pinlabel \scalebox{0.9}{$\lambda$} at 36 24
                \endlabellist
			\centering
			\includegraphics[scale=1.15]{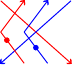}
	}\endxy\;\;+\xy (0,0)*{
			\labellist
			\small\hair 2pt
			\pinlabel \scalebox{0.7}{$\textcolor{blue}{1}$} at 0 -5
			\pinlabel \scalebox{0.7}{$\textcolor{red}{2}$} at 35 -5
			\pinlabel \scalebox{0.7}{$\textcolor{red}{2}$} at 12 -5
			\pinlabel \scalebox{0.7}{$\textcolor{blue}{1}$} at 24 -5
                \pinlabel \scalebox{0.9}{$\lambda$} at 36 24
                \endlabellist
			\centering
			\includegraphics[scale=1.15]{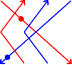}
	}\endxy\;\;+ \xy (0,0)*{
			\labellist
			\small\hair 2pt
			\pinlabel \scalebox{0.7}{$\textcolor{blue}{1}$} at 0 -5
			\pinlabel \scalebox{0.7}{$\textcolor{red}{2}$} at 35 -5
			\pinlabel \scalebox{0.7}{$\textcolor{red}{2}$} at 12 -5
			\pinlabel \scalebox{0.7}{$\textcolor{blue}{1}$} at 24 -5
                \pinlabel \scalebox{0.9}{$\lambda$} at 36 24
                \endlabellist
			\centering
			\includegraphics[scale=1.15]{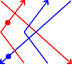}
	}\endxy\;\; - \xy (0,0)*{
			\labellist
			\small\hair 2pt
			\pinlabel \scalebox{0.7}{$\textcolor{blue}{1}$} at 0 -5
			\pinlabel \scalebox{0.7}{$\textcolor{red}{2}$} at 35 -5
			\pinlabel \scalebox{0.7}{$\textcolor{red}{2}$} at 12 -5
			\pinlabel \scalebox{0.7}{$\textcolor{blue}{1}$} at 24 -5
                \pinlabel \scalebox{0.9}{$\lambda$} at 36 24
                \endlabellist
			\centering
			\includegraphics[scale=1.15]{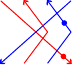}
	}\endxy\\
    & \\
    & \left. -\xy (0,0)*{
			\labellist
			\small\hair 2pt
			\pinlabel \scalebox{0.7}{$\textcolor{blue}{1}$} at 0 -5
			\pinlabel \scalebox{0.7}{$\textcolor{red}{2}$} at 35 -5
			\pinlabel \scalebox{0.7}{$\textcolor{red}{2}$} at 12 -5
			\pinlabel \scalebox{0.7}{$\textcolor{blue}{1}$} at 24 -5
                \pinlabel \scalebox{0.9}{$\lambda$} at 36 24
                \endlabellist
			\centering
			\includegraphics[scale=1.15]{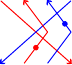}
	}\endxy\;\;-\xy (0,0)*{
			\labellist
			\small\hair 2pt
			\pinlabel \scalebox{0.7}{$\textcolor{blue}{1}$} at 0 -5
			\pinlabel \scalebox{0.7}{$\textcolor{red}{2}$} at 35 -5
			\pinlabel \scalebox{0.7}{$\textcolor{red}{2}$} at 12 -5
			\pinlabel \scalebox{0.7}{$\textcolor{blue}{1}$} at 24 -5
                \pinlabel \scalebox{0.9}{$\lambda$} at 36 24
                \endlabellist
			\centering
			\includegraphics[scale=1.15]{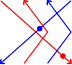}
	}\endxy\;\;-\xy (0,0)*{
			\labellist
			\small\hair 2pt
			\pinlabel \scalebox{0.7}{$\textcolor{blue}{1}$} at 0 -5
			\pinlabel \scalebox{0.7}{$\textcolor{red}{2}$} at 35 -5
			\pinlabel \scalebox{0.7}{$\textcolor{red}{2}$} at 12 -5
			\pinlabel \scalebox{0.7}{$\textcolor{blue}{1}$} at 24 -5
                \pinlabel \scalebox{0.9}{$\lambda$} at 36 24
                \endlabellist
			\centering
			\includegraphics[scale=1.15]{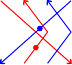}
	}\endxy\;\;
    \right)
    \end{align*}
\eqskip
    We note that, with generous use of KM5, all the dots in the above can be moved to either the top or bottom of their arrow in a consistent fashion. It then follows from mixed direction cubic KLR that the difference is equal to zero as required.
    
        \item[KM7] 
        \begin{align*}
        \Ev'\left(\xy (0,0)*{
			\labellist
			\small\hair 2pt
			\pinlabel \scalebox{0.7}{$\textcolor{black}{3}$} at 0 -5
			\pinlabel \scalebox{0.7}{$\textcolor{blue}{1}$} at 16 -5
			\pinlabel \scalebox{0.9}{$\lambda$} at 18 16
			\endlabellist 
			\centering 
			\includegraphics[scale=1.3]{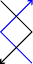}
	}\endxy\;\;
        \right)& =-k_1^{0,1}(s_1(\lambda))k_1^{2,3}(s_1(\lambda))\Ev'\iota\left(\xy (0,0)*{
			\labellist
			\small\hair 2pt
			\pinlabel \scalebox{0.7}{$\textcolor{red}{2}$} at 0 -5
			\pinlabel \scalebox{0.7}{$\textcolor{blue}{1}$} at 16 -5
			\pinlabel \scalebox{0.9}{$s_1(\lambda)$} at 24 16
			\endlabellist 
			\centering 
			\includegraphics[scale=1.3]{./RurulBdldr}
	}\endxy\qquad
        \right)\\
        & \\
        &=\sigma_1\tilde{\mathcal{T}}_{1,-1}'\left(\xy (0,0)*{
			\labellist
			\small\hair 2pt
			\pinlabel \scalebox{0.7}{$\textcolor{red}{2}$} at 0 -5
			\pinlabel \scalebox{0.7}{$\textcolor{blue}{1}$} at 16 -5
			\pinlabel \scalebox{0.9}{$s_1(\lambda)$} at 24 16
			\endlabellist 
			\centering 
			\includegraphics[scale=1.3]{./RurulBdldr}
	}\endxy\qquad
        \right)\sim_h\sigma_1\tilde{\mathcal{T}}_{1,-1}'\left(\xy (0,0)*{
			\labellist
			\small\hair 2pt
			\pinlabel \scalebox{0.7}{$\textcolor{red}{2}$} at 0 -5
			\pinlabel \scalebox{0.7}{$\textcolor{blue}{1}$} at 12 -5
			\pinlabel \scalebox{0.9}{$s_1(\lambda)$} at 24 16
			\endlabellist 
			\centering 
			\includegraphics[scale=1.3]{./RuuBdd}
	}\endxy\qquad\right)\\
    & \\
    & =\Ev'\iota\left(\xy (0,0)*{
			\labellist
			\small\hair 2pt
			\pinlabel \scalebox{0.7}{$\textcolor{red}{2}$} at 0 -5
			\pinlabel \scalebox{0.7}{$\textcolor{blue}{1}$} at 12 -5
			\pinlabel \scalebox{0.9}{$s_1(\lambda)$} at 24 16
			\endlabellist 
			\centering 
			\includegraphics[scale=1.3]{./RuuBdd}
	}\endxy\qquad\right) = \Ev'\left(\xy (0,0)*{
			\labellist
			\small\hair 2pt
			\pinlabel \scalebox{0.7}{$\textcolor{black}{3}$} at 0 -5
			\pinlabel \scalebox{0.7}{$\textcolor{blue}{1}$} at 12 -5
			\pinlabel \scalebox{0.9}{$\lambda$} at 18 16
			\endlabellist 
			\centering 
			\includegraphics[scale=1.3]{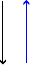}
	}\endxy\;\;\right)
    \end{align*}
    \begin{align*}
        \Ev'\left(\xy (0,0)*{
			\labellist
			\small\hair 2pt
			\pinlabel \scalebox{0.7}{$\textcolor{black}{3}$} at 16 -5
			\pinlabel \scalebox{0.7}{$\textcolor{red}{2}$} at 0 -5
			\pinlabel \scalebox{0.9}{$\lambda$} at 18 16
			\endlabellist 
			\centering 
			\includegraphics[scale=1.3]{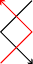}
	}\endxy\;\;
        \right)& =-k_1^{0,1}(s_1(\lambda))k_1^{2,3}(s_1(\lambda))\Ev'\iota'\left(\xy (0,0)*{
			\labellist
			\small\hair 2pt
			\pinlabel \scalebox{0.7}{$\textcolor{red}{2}$} at 0 -5
			\pinlabel \scalebox{0.7}{$\textcolor{blue}{1}$} at 16 -5
			\pinlabel \scalebox{0.9}{$s_2(\lambda)$} at 24 16
			\endlabellist 
			\centering 
			\includegraphics[scale=1.3]{./RdrdlBulur}
	}\endxy\qquad
        \right)\\
        & \\
        &=\sigma_2\alpha\tilde{\mathcal{T}}_{2,1}''\beta\left(\xy (0,0)*{
			\labellist
			\small\hair 2pt
			\pinlabel \scalebox{0.7}{$\textcolor{red}{2}$} at 0 -5
			\pinlabel \scalebox{0.7}{$\textcolor{blue}{1}$} at 16 -5
			\pinlabel \scalebox{0.9}{$s_2(\lambda)$} at 24 16
			\endlabellist 
			\centering 
			\includegraphics[scale=1.3]{./RdrdlBulur}
	}\endxy\qquad
        \right)\sim_h\left(\xy (0,0)*{
			\labellist
			\small\hair 2pt
			\pinlabel \scalebox{0.7}{$\textcolor{red}{2}$} at 0 -5
			\pinlabel \scalebox{0.7}{$\textcolor{blue}{1}$} at 12 -5
			\pinlabel \scalebox{0.9}{$s_2(\lambda)$} at 24 16
			\endlabellist 
			\centering 
			\includegraphics[scale=1.3]{./RddBuu}
	}\endxy\qquad\right)\\
    & \\
    & = \Ev'\iota'\left(\xy (0,0)*{
			\labellist
			\small\hair 2pt
			\pinlabel \scalebox{0.7}{$\textcolor{red}{2}$} at 0 -5
			\pinlabel \scalebox{0.7}{$\textcolor{blue}{1}$} at 12 -5
			\pinlabel \scalebox{0.9}{$s_2(\lambda)$} at 24 16
			\endlabellist 
			\centering 
			\includegraphics[scale=1.3]{./RddBuu}
	}\endxy\qquad\right) = \Ev'\left(\xy (0,0)*{
			\labellist
			\small\hair 2pt
			\pinlabel \scalebox{0.7}{$\textcolor{red}{2}$} at 0 -5
			\pinlabel \scalebox{0.7}{$\textcolor{black}{3}$} at 12 -5
			\pinlabel \scalebox{0.9}{$\lambda$} at 18 16
			\endlabellist 
			\centering 
			\includegraphics[scale=1.3]{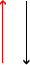}
	}\endxy\;;\right)
    \end{align*}
    \begin{align*}
    \Ev'\left(\xy (0,0)*{
			\labellist
			\small\hair 2pt
			\pinlabel \scalebox{0.7}{$\textcolor{black}{3}$} at 0 -5
			\pinlabel \scalebox{0.7}{$\textcolor{black}{3}$} at 16 -5
			\pinlabel \scalebox{0.9}{$\lambda$} at 18 16
			\endlabellist 
			\centering 
			\includegraphics[scale=1.3]{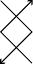}
	}\endxy\;\;
        \right)& =\Ev'\iota\left(\xy (0,0)*{
			\labellist
			\small\hair 2pt
			\pinlabel \scalebox{0.7}{$\textcolor{red}{2}$} at 0 -5
			\pinlabel \scalebox{0.7}{$\textcolor{red}{2}$} at 16 -5
			\pinlabel \scalebox{0.9}{$s_1(\lambda)$} at 24 16
			\endlabellist 
			\centering 
			\includegraphics[scale=1.3]{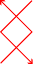}
	}\endxy\qquad
        \right)\\
        & \\
        &=\sigma_1\tilde{\mathcal{T}}_{1,-1}'\left(\xy (0,0)*{
			\labellist
			\small\hair 2pt
			\pinlabel \scalebox{0.7}{$\textcolor{red}{2}$} at 0 -5
			\pinlabel \scalebox{0.7}{$\textcolor{red}{2}$} at 16 -5
			\pinlabel \scalebox{0.9}{$s_1(\lambda)$} at 24 16
			\endlabellist 
			\centering 
			\includegraphics[scale=1.3]{./RurulRdldr}
	}\endxy\qquad
        \right)\sim_h \sigma_1\tilde{\mathcal{T}}_{1,-1}'\left(\xy (0,0)*{
			\labellist
			\small\hair 2pt
			\pinlabel \scalebox{0.7}{$\textcolor{red}{2}$} at 0 -5
			\pinlabel \scalebox{0.7}{$\textcolor{red}{2}$} at 12 -5
			\pinlabel \scalebox{0.9}{$s_1(\lambda)$} at 24 18
			\endlabellist 
			\centering 
			\includegraphics[scale=1.3]{./RuuRdd}
	}\endxy
 \qquad - \ 
 \sum\limits_{a+b+c=\overline{s_1(\lambda)}_2-1} \xy (0,0)*{
			\labellist
			\small\hair 2pt
			\pinlabel \scalebox{0.7}{$\textcolor{red}{2}$} at 0 -5
			\pinlabel \scalebox{0.7}{$\textcolor{red}{2}$} at 0 31
            \pinlabel \scalebox{0.7}{$\textcolor{red}{2}$} at 19 22
            \pinlabel \scalebox{0.7}{$\textcolor{red}{+c}$} at 34 5
			\pinlabel \scalebox{0.9}{$s_1(\lambda)$} at 43 16
            \pinlabel \scalebox{0.7}{$\textcolor{red}{a}$} at 4 8
            \pinlabel \scalebox{0.7}{$\textcolor{red}{b}$} at 4 18
			\endlabellist 
			\centering 
			\includegraphics[scale=1.3]{./RAClockSum}
	}\endxy\qquad\right)\\
    & \\
    & = \Ev'\iota\left(\xy (0,0)*{
			\labellist
			\small\hair 2pt
			\pinlabel \scalebox{0.7}{$\textcolor{red}{2}$} at 0 -5
			\pinlabel \scalebox{0.7}{$\textcolor{red}{2}$} at 12 -5
			\pinlabel \scalebox{0.9}{$s_1(\lambda)$} at 24 18
			\endlabellist 
			\centering 
			\includegraphics[scale=1.3]{./RuuRdd}
	}\endxy
 \qquad - \ 
 \sum\limits_{a+b+c=\overline{s_1(\lambda)}_2-1} \xy (0,0)*{
			\labellist
			\small\hair 2pt
			\pinlabel \scalebox{0.7}{$\textcolor{red}{2}$} at 0 -5
			\pinlabel \scalebox{0.7}{$\textcolor{red}{2}$} at 0 31
            \pinlabel \scalebox{0.7}{$\textcolor{red}{2}$} at 19 22
            \pinlabel \scalebox{0.7}{$\textcolor{red}{+c}$} at 34 5
			\pinlabel \scalebox{0.9}{$s_1(\lambda)$} at 43 16
            \pinlabel \scalebox{0.7}{$\textcolor{red}{a}$} at 4 8
            \pinlabel \scalebox{0.7}{$\textcolor{red}{b}$} at 4 18
			\endlabellist 
			\centering 
			\includegraphics[scale=1.3]{./RAClockSum}
	}\endxy\qquad\right)\\
    & \\
    & = \Ev'\left(\xy (0,0)*{
			\labellist
			\small\hair 2pt
			\pinlabel \scalebox{0.7}{$\textcolor{black}{3}$} at 0 -5
			\pinlabel \scalebox{0.7}{$\textcolor{black}{3}$} at 12 -5
			\pinlabel \scalebox{0.9}{$\lambda$} at 18 18
			\endlabellist 
			\centering 
			\includegraphics[scale=1.3]{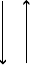}
	}\endxy
 \;\; - \ 
 \sum\limits_{a+b+c=-\olambda_3-1} \xy (0,0)*{
			\labellist
			\small\hair 2pt
			\pinlabel \scalebox{0.7}{$\textcolor{black}{3}$} at 0 -5
			\pinlabel \scalebox{0.7}{$\textcolor{black}{3}$} at 0 31
            \pinlabel \scalebox{0.7}{$\textcolor{black}{3}$} at 19 22
            \pinlabel \scalebox{0.7}{$\textcolor{black}{+c}$} at 34 5
			\pinlabel \scalebox{0.9}{$s_1(\lambda)$} at 43 16
            \pinlabel \scalebox{0.7}{$\textcolor{black}{a}$} at 4 8
            \pinlabel \scalebox{0.7}{$\textcolor{black}{b}$} at 4 18
			\endlabellist 
			\centering 
			\includegraphics[scale=1.3]{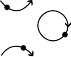}
	}\endxy\qquad\right)
    \end{align*}
\eqskip
    The other two identities are similar.
    \item[KM8:] \begin{align*}\Ev'\left(\xy (0,0)*{
			\labellist
			\small\hair 2pt
			\pinlabel \scalebox{0.7}{$\textcolor{black}{3}$} at 0 14
			\pinlabel \scalebox{0.7}{$\textcolor{black}{+m}$} at 16 -1
			\pinlabel \scalebox{0.9}{$\lambda$} at 18 12
			\endlabellist 
			\centering 
			\includegraphics[scale=1.3]{./KClockClub}
	}\endxy\;\;\right) &=(-1)^{\overline{\lambda}_3+1}\Ev'\iota\left(\xy (0,0)*{
			\labellist
			\small\hair 2pt
			\pinlabel \scalebox{0.7}{$\textcolor{red}{2}$} at 0 14
			\pinlabel \scalebox{0.7}{$\textcolor{red}{+m}$} at 16 -1
			\pinlabel \scalebox{0.9}{$s_1(\lambda)$} at 24 12
			\endlabellist 
			\centering 
			\includegraphics[scale=1.3]{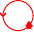}
	}\endxy\qquad\right)=(-1)^{\overline{\lambda}_3+1}\sigma_1\tilde{\mathcal{T}}_{1,-1}'\left(\xy (0,0)*{
			\labellist
			\small\hair 2pt
			\pinlabel \scalebox{0.7}{$\textcolor{red}{2}$} at 0 14
			\pinlabel \scalebox{0.7}{$\textcolor{red}{+m}$} at 16 -1
			\pinlabel \scalebox{0.9}{$s_1(\lambda)$} at 24 12
			\endlabellist 
			\centering 
			\includegraphics[scale=1.3]{./RAClockClub}
	}\endxy\qquad\right)\\
    & \\
    & = \sigma_1\tilde{\mathcal{T}}_{1,-1}'\left(\begin{cases}
        (-1)^{s_1(\lambda)_2} & m=0\\
        0 & m<0
    \end{cases}\right)=\Ev'\left(\begin{cases}
        (-1)^{\lambda_1} & m=0\\
        0 &m<0
    \end{cases}\right)\end{align*}
    and the anti-clockwise bubble is similar.
    \item[KM9:]\begin{align*}\Ev'&\left(\left(\;\xy (0,0)*{
			\labellist
			\small\hair 2pt
			\pinlabel \scalebox{0.7}{$\textcolor{black}{3}$} at 0 14
			\pinlabel \scalebox{0.7}{$\textcolor{black}{+0}$} at 16 -2
			\pinlabel \scalebox{0.9}{$\lambda$} at 18 12
			\endlabellist 
			\centering 
			\includegraphics[scale=1.3]{./KAClockClub}
	}\endxy
 \ +
 \xy (0,0)*{
			\labellist
			\small\hair 2pt
			\pinlabel \scalebox{0.7}{$\textcolor{black}{3}$} at 0 14
			\pinlabel \scalebox{0.7}{$\textcolor{black}{+1}$} at 16 -2
			\pinlabel \scalebox{0.9}{$\lambda$} at 18 12
			\endlabellist 
			\centering 
			\includegraphics[scale=1.3]{./KAClockClub}
	}\endxy
 \ \ t + \dots \right)\!\!
 \left(\;\xy (0,0)*{
			\labellist
			\small\hair 2pt
			\pinlabel \scalebox{0.7}{$\textcolor{black}{3}$} at 0 14
			\pinlabel \scalebox{0.7}{$\textcolor{black}{+0}$} at 16 -2
			\pinlabel \scalebox{0.9}{$\lambda$} at 18 12
			\endlabellist 
			\centering 
			\includegraphics[scale=1.3]{./KClockClub}
	}\endxy
 \ +
 \xy (0,0)*{
			\labellist
			\small\hair 2pt
			\pinlabel \scalebox{0.7}{$\textcolor{black}{3}$} at 0 14
			\pinlabel \scalebox{0.7}{$\textcolor{black}{+1}$} at 16 -2
			\pinlabel \scalebox{0.9}{$\lambda$} at 18 12
			\endlabellist 
			\centering 
			\includegraphics[scale=1.3]{./KClockClub}
	}\endxy
 \ \ t + \dots\right)\right)\\
 &\\
 &=(-1)^{2\overline{\lambda}_3+2}\Ev'\iota\left(\left(\;\xy (0,0)*{
			\labellist
			\small\hair 2pt
			\pinlabel \scalebox{0.7}{$\textcolor{red}{2}$} at 0 14
			\pinlabel \scalebox{0.7}{$\textcolor{red}{+0}$} at 16 -2
			\pinlabel \scalebox{0.9}{$s_1(\lambda)$} at 24 12
			\endlabellist 
			\centering 
			\includegraphics[scale=1.3]{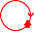}
	}\endxy
 \qquad +
 \xy (0,0)*{
			\labellist
			\small\hair 2pt
			\pinlabel \scalebox{0.7}{$\textcolor{red}{2}$} at 0 14
			\pinlabel \scalebox{0.7}{$\textcolor{red}{+1}$} at 16 -2
			\pinlabel \scalebox{0.9}{$s_1(\lambda)$} at 24 12
			\endlabellist 
			\centering 
			\includegraphics[scale=1.3]{./RClockClub}
	}\endxy
 \qquad t + \dots \right)\!\!
 \left(\;\xy (0,0)*{
			\labellist
			\small\hair 2pt
			\pinlabel \scalebox{0.7}{$\textcolor{red}{2}$} at 0 14
			\pinlabel \scalebox{0.7}{$\textcolor{red}{+0}$} at 16 -2
			\pinlabel \scalebox{0.9}{$s_1(\lambda)$} at 24 12
			\endlabellist 
			\centering 
			\includegraphics[scale=1.3]{./RAClockClub}
	}\endxy
 \qquad +
 \xy (0,0)*{
			\labellist
			\small\hair 2pt
			\pinlabel \scalebox{0.7}{$\textcolor{red}{2}$} at 0 14
			\pinlabel \scalebox{0.7}{$\textcolor{red}{+1}$} at 16 -2
			\pinlabel \scalebox{0.9}{$s_1(\lambda)$} at 24 12
			\endlabellist 
			\centering 
			\includegraphics[scale=1.3]{./RAClockClub}
	}\endxy
 \qquad t + \dots\right)\right)\\
 & \\
 & = \sigma_1\tilde{\mathcal{T}}_{1,-1}'\left(\left(\;\xy (0,0)*{
			\labellist
			\small\hair 2pt
			\pinlabel \scalebox{0.7}{$\textcolor{red}{2}$} at 0 14
			\pinlabel \scalebox{0.7}{$\textcolor{red}{+0}$} at 16 -2
			\pinlabel \scalebox{0.9}{$s_1(\lambda)$} at 24 12
			\endlabellist 
			\centering 
			\includegraphics[scale=1.3]{./RClockClub}
	}\endxy
 \qquad +
 \xy (0,0)*{
			\labellist
			\small\hair 2pt
			\pinlabel \scalebox{0.7}{$\textcolor{red}{2}$} at 0 14
			\pinlabel \scalebox{0.7}{$\textcolor{red}{+1}$} at 16 -2
			\pinlabel \scalebox{0.9}{$s_1(\lambda)$} at 24 12
			\endlabellist 
			\centering 
			\includegraphics[scale=1.3]{./RClockClub}
	}\endxy
 \qquad t + \dots \right)\!\!
 \left(\;\xy (0,0)*{
			\labellist
			\small\hair 2pt
			\pinlabel \scalebox{0.7}{$\textcolor{red}{2}$} at 0 14
			\pinlabel \scalebox{0.7}{$\textcolor{red}{+0}$} at 16 -2
			\pinlabel \scalebox{0.9}{$s_1(\lambda)$} at 24 12
			\endlabellist 
			\centering 
			\includegraphics[scale=1.3]{./RAClockClub}
	}\endxy
 \qquad +
 \xy (0,0)*{
			\labellist
			\small\hair 2pt
			\pinlabel \scalebox{0.7}{$\textcolor{red}{2}$} at 0 14
			\pinlabel \scalebox{0.7}{$\textcolor{red}{+1}$} at 16 -2
			\pinlabel \scalebox{0.9}{$s_1(\lambda)$} at 24 12
			\endlabellist 
			\centering 
			\includegraphics[scale=1.3]{./RAClockClub}
	}\endxy
 \qquad t + \dots\right)\right)\\
 & \\
 & = \sigma_1\tilde{\mathcal{T}}_{1,-1}'(-1)=-1=\Ev'(-1)
    \end{align*}
    \end{itemize}

\appendix

\section{A remark on 2-isomorphism classes}\label{sec:No2Iso}

In \cite[Definition 3.1]{Kh-L}, Khovanov and Lauda chose a different set of scalars and bubble parameters for the Kac--Moody 2-categories than those we use in this paper. We were unable to define an evaluation 2-functor for their choice and we conjecture that no such evaluation 2-functor exists. Although we have yet to prove this conjecture, this would be consistent with the first paragraph of \cite[Page 2699]{Kh-L-2}, which mentions the existence of a one-parameter family of mutually non-2-isomorphic sub-2-categories of the affine type A Kac--Moody 2-categories categorifying the Borel sub-algebra in affine type A. This contrasts with the finite type A case, where \cite[Theorem 3.5]{lauda2020parameters} proves that any two choices of scalars and bubble parameters yield 2-isomorphic Kac--Moody 2-categories. Since Khovanov and Lauda did not include a proof in the above paper, we prove here that the two different choices discussed do lead to non-2-isomorphic Kac--Moody 2-categories in affine type A.

First, a small comment on the choice of weights for the 2-categories in question. The objects of the cyclic $2$-categories $\naffu{n}$ and $\affu{n}$ from Definition~\ref{defn:KLR-2cats} are $\mathfrak{gl}_n$-weights and level-zero $\widehat{\mathfrak{gl}}_n$-weights, respectively, which coincide (as explained in \ref{sec:decat}). Moreover, the relations satisfied by the generating $2$-morphisms of those two $2$-categories really depend on those weights, and not on the induced $\mathfrak{sl}_n$-weights and level-zero $\widehat{\mathfrak{sl}}_n$-weights, respectively. Specifically, the degree-zero $i$-coloured bubbles, in a region labelled by $\lambda\in \mathbb{Z}^n$, are equal to $(-1)^{\lambda_{i+1}}$ or $(-1)^{\lambda_{i+1}-1}$ (depending on orientation), so they cannot be expressed in terms of $\overline{\lambda}$ (unless we choose and fix a certain Schur level). 

On the other hand, the cyclic $2$-categories $\mathcal{U}_Q(\mathfrak{sl}_n)$ and $\mathcal{U}_Q(\widehat{\mathfrak{sl}}_n)$, defined in~\cite[Definition 1.3]{Beliakova-Habiro-Lauda-Webster16} (generalizing~\cite[Definition 3.1]{Kh-L}), trivially induce cyclic $2$-categories $\mathcal{U}_Q(\mathfrak{gl}_n)$ and $\mathcal{U}_Q(\widehat{\mathfrak{gl}}'_n)$ whose objects are $\mathfrak{gl}_n$-weights and 
level-zero $\widehat{\mathfrak{gl}}_n$-weights, respectively: Simply label the regions of the string diagrams by $\lambda\in \mathbb{Z}^n$ and 
let the relations be those for $\overline{\lambda}$, see Section~\ref{sec:decat} for the notation. We write  
$\mathcal{U}_Q(\widehat{\mathfrak{gl}}'_n)$ to indicate that it is actually an extended version of 
$\mathcal{U}_Q(\widehat{\mathfrak{sl}}_n)$ rather than the full $\mathcal{U}_Q(\widehat{\mathfrak{gl}}_n)$ (whatever that would be), see remarks 
below Definition~\ref{d:dotun}.

Recall that, following \cite{lauda2020parameters}, $\mathcal{U}_Q(\mathfrak{sl}_n)$ and $\mathcal{U}_Q(\widehat{\mathfrak{sl}}_n)$ depend on a choice of scalars $t_{ij}\in \mathbb{Q}^\times$ satisfying $t_{ii}=1$ and $t_{ij}=t_{ji}$ when $j\ne i\pm 1 \bmod n$, and bubble parameters $\beta_i=\beta_{i,\lambda}$, $c^+_{i,\lambda},c^-_{i,\lambda}\in \mathbb{Q}^\times$ satisfying \begin{itemize}
    \item $c^+_{i,\lambda} c^-_{i,\lambda}=\frac{-1}{\beta_i}=\frac{1}{t_{ii}}$.
    \item $c^\pm_{i,\lambda+\alpha_j}=t_{ij}c^\pm_{i,\lambda}$.
\end{itemize}
Here $i,j\in 1,\ldots, n-1$ and $\lambda\in \mathbb{Z}^{n-1}$, for $\mathfrak{sl}_n$, and $i,j\in 1,\ldots, n$ and $\lambda\in \mathbb{Z}^{n}$, for $\widehat{\mathfrak{sl}}_n$. For Khovanov and Lauda's original choice in~\cite[Definition 3.1]{Kh-L}, with 
all scalars and bubble parameters equal to one, we will follow their notation and denote the corresponding 2-categories by $\mathcal{U}(\mathfrak{sl}_n)$ and $\mathcal{U}(\widehat{\mathfrak{sl}}_n)$, and the trivially induced $\mathfrak{gl}_n$ versions of these by $\mathcal{U}(\mathfrak{gl}_n)$ and $\mathcal{U}(\widehat{\mathfrak{gl}}'_n)$, respectively. The $2$-categories $\naffu{n}$ and $\affu{n}$ correspond to the 
choice $t_{ii}=-1=t_{i,i+1}=-1$ and $t_{ij}=1$ for all $i$ and $j\ne i,i+1$ in the respective ranges, and $c^+_{i,\lambda}=(-1)^{\lambda_{i+1}}=-c^-_{i,\lambda}$ for all $i$ in the respective ranges.

For any $n\in \mathbb{N}_{\geq 2}$, the $2$-categories $\naffu{n}$ and $\mathcal{U}(\mathfrak{gl}_n)$ are $2$-isomorphic, with the $2$-isomorphism being obtained by composing the $2$-isomorphism from \cite[(6)]{msv-schur} and the $2$-isomorphism $\Sigma$ from \cite[Section 4.2.1]{Kh-L} (see also \cite{Kh-L-err}). When $n\in \mathbb{N}_{\geq 2}$ is even, that composite $2$-isomorphism extends to a $2$-isomorphism between $\affu{n}$ and 
$\mathcal{U}_Q(\widehat{\mathfrak{gl}}'_n)$. When $n$ is odd, it does not extend to the affine 2-categories, because Khovanov and Lauda's $2$-isomorphism $\Sigma$ is no longer well-defined in that case. The reason is that in the definition of $\Sigma$ occur factors like $(-1)^i$, for $i=1,\ldots, n-1$, which are not well-defined for $i\in \mathbb{Z}/n\mathbb{Z}$ when $n$ is odd. 

We are now going to show that there is no $2$-isomorphism between $\affu{n}$ and $\mathcal{U}_Q(\widehat{\mathfrak{gl}}'_n)$ for odd $n$, 
for any choice of scalars and bubble parameters satisfying the above conditions.

\begin{lem}\label{Lem:2IsoNice}
    Let $Q$ be a choice of scalars and bubble parameters for $\widehat{\sln}$ and let $\Xi:\mathcal{U}_Q(\widehat{\mathfrak{gl}}'_n)\to \affu{n}$ be a 2-isomorphism which is the identity on objects and 1-morphisms. Then 
    \begin{equation}\label{eq:2IsoNice}
    \Xi\left(\xy (0,0)*{
			\labellist
			\small\hair 2pt
                \pinlabel \scalebox{0.7}{$i$} at 2 -5
			\pinlabel \scalebox{0.9}{$\lambda$} at 10 8
			\endlabellist 
			\centering 
			\includegraphics[scale=1.3]{./Kuo}
	}\endxy\quad\right) 
 = o_i(\lambda)\;\xy (0,0)*{
			\labellist
			\small\hair 2pt
                \pinlabel \scalebox{0.7}{$i$} at 2 -5
			\pinlabel \scalebox{0.9}{$\lambda$} at 10 8
			\endlabellist 
			\centering 
			\includegraphics[scale=1.3]{./Kuo}
	}\endxy
 \qquad\text{and}\qquad  
 \Xi\left(\xy (0,0)*{
			\labellist
			\small\hair 2pt
   \pinlabel \scalebox{0.7}{$i$} at -2 -5
   \pinlabel \scalebox{0.7}{$j$} at 15 -5
			\pinlabel \scalebox{0.9}{$\lambda$} at 22 9
			\endlabellist 
			\centering 
			\includegraphics[scale=1.3]{./RurBul}
	}\endxy\quad\right) 
 = f_{ij}(\lambda)\;\xy (0,0)*{
			\labellist
			\small\hair 2pt
   \pinlabel \scalebox{0.7}{$i$} at -2 -5
   \pinlabel \scalebox{0.7}{$j$} at 15 -5
			\pinlabel \scalebox{0.9}{$\lambda$} at 22 9
			\endlabellist 
			\centering 
			\includegraphics[scale=1.3]{./RurBul}
	}\endxy
 \end{equation}
 \eqskip for some $o_i(\lambda), f_{ij}(\lambda)\in\mathbb{Q}^\times$ and for all $i,j\in\{1,\dots,n\}$ and all $\lambda\in \mathbb{Z}^n$. Moreover, 
 these scalars satisfy $o_i(\lambda) f_{ii}(\lambda)=1$ for all $i\in \widehat{I}$ and all $\lambda\in \mathbb{Z}^n$.
\end{lem}
\begin{proof} For degree reasons, the second equality in \eqref{eq:2IsoNice} is immediate, but the first one requires an argument. A priori, 
we have 
\[
 \Xi\left(\xy (0,0)*{
			\labellist
			\small\hair 2pt
                \pinlabel \scalebox{0.7}{$i$} at 2 -5
			\pinlabel \scalebox{0.9}{$\lambda$} at 10 8
			\endlabellist 
			\centering 
			\includegraphics[scale=1.3]{./Kuo}
	}\endxy\quad\right) 
 = o_i(\lambda)\;\xy (0,0)*{
			\labellist
			\small\hair 2pt
                \pinlabel \scalebox{0.7}{$i$} at 2 -5
			\pinlabel \scalebox{0.9}{$\lambda$} at 10 8
			\endlabellist 
			\centering 
			\includegraphics[scale=1.3]{./Kuo}
	}\endxy 
 \quad+
 \sum_{j=1}^n b_{ij}(\lambda) \xy (0,0)*{
			\labellist
			\small\hair 2pt
			\pinlabel \scalebox{0.7}{$\textcolor{black}{i}$} at 24 -5
			\pinlabel \scalebox{0.7}{$\textcolor{red}{j}$} at 1 24
            \pinlabel \scalebox{0.7}{$\textcolor{red}{+1}$} at 15 6
			\pinlabel \scalebox{0.9}{$\lambda$} at 27 18
			\endlabellist 
			\centering 
			\includegraphics[scale=1.3]{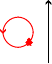}
	}\endxy
\]

Now consider the image of the nil-Hecke relation:
\begin{gather*}
\Xi\left(\xy (0,0)*{
			\labellist
			\small\hair 2pt
			\pinlabel \scalebox{0.7}{${i}$} at 0 -5
			\pinlabel \scalebox{0.7}{${i}$} at 16 -5
			\pinlabel \scalebox{0.9}{$\lambda$} at 18 7
			\endlabellist 
			\centering 
			\includegraphics[scale=1.3]{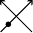}
	}\endxy
 \quad - \ 
    \xy (0,0)*{
			\labellist
			\small\hair 2pt
			\pinlabel \scalebox{0.7}{${i}$} at 0 -5
			\pinlabel \scalebox{0.7}{${i}$} at 16 -5
			\pinlabel \scalebox{0.9}{$\lambda$} at 18 7
			\endlabellist 
			\centering 
			\includegraphics[scale=1.3]{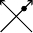}
	}\endxy\quad\right)
 =\\[2ex]
 o_i(\lambda)f_{ii}(\lambda) \left(\xy (0,0)*{
			\labellist
			\small\hair 2pt
			\pinlabel \scalebox{0.7}{${i}$} at 0 -5
			\pinlabel \scalebox{0.7}{${i}$} at 16 -5
			\pinlabel \scalebox{0.9}{$\lambda$} at 18 7
			\endlabellist 
			\centering 
			\includegraphics[scale=1.3]{./KourKul}
	}\endxy
 \quad - \ 
    \xy (0,0)*{
			\labellist
			\small\hair 2pt
			\pinlabel \scalebox{0.7}{${i}$} at 0 -5
			\pinlabel \scalebox{0.7}{${i}$} at 16 -5
			\pinlabel \scalebox{0.9}{$\lambda$} at 18 7
			\endlabellist 
			\centering 
			\includegraphics[scale=1.3]{./KuroKul}
	}\endxy\quad\right) + 
 f_{ii}(\lambda)\sum_{j=1}^n b_{ij}(\lambda) \left(\xy (0,0)*{
			\labellist
			\small\hair 2pt
			\pinlabel \scalebox{0.7}{${i}$} at 24 -5
			\pinlabel \scalebox{0.7}{${i}$} at 40 -5
			\pinlabel \scalebox{0.9}{$\lambda$} at 42 7
            \pinlabel \scalebox{0.7}{$\textcolor{red}{j}$} at 1 16
            \pinlabel \scalebox{0.7}{$\textcolor{red}{+1}$} at 15 -2
			\endlabellist 
			\centering 
			\includegraphics[scale=1.3]{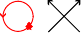}
	}\endxy\;\; - \xy (0,0)*{
			\labellist
			\small\hair 2pt
			\pinlabel \scalebox{0.7}{${i}$} at 0 -5
			\pinlabel \scalebox{0.7}{${i}$} at 32 -5
			\pinlabel \scalebox{0.9}{$\lambda$} at 36 16
            \pinlabel \scalebox{0.7}{$\textcolor{red}{j}$} at 7 29
            \pinlabel \scalebox{0.7}{$\textcolor{red}{+1}$} at 24 14
			\endlabellist 
			\centering 
			\includegraphics[scale=1.3]{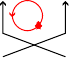}
	}\endxy\;\;\right)
 = \\[2ex]
 o_i(\lambda) f_{ii}(\lambda) \xy (0,0)*{
			\labellist
			\small\hair 2pt
			\pinlabel \scalebox{0.7}{${i}$} at 0 -5
			\pinlabel \scalebox{0.7}{${i}$} at 16 -5
			\pinlabel \scalebox{0.9}{$\lambda$} at 18 7
			\endlabellist 
			\centering 
			\includegraphics[scale=1.3]{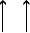}
	}\endxy\;\; + 2f_{ii}(\lambda) 
 \left(b_{ii}(\lambda)\;\xy (0,0)*{
			\labellist
			\small\hair 2pt
            \pinlabel \scalebox{0.7}{$\textcolor{black}{i}$} at 1 16
            \pinlabel \scalebox{0.7}{$\textcolor{black}{+1}$} at 15 -2
			\endlabellist 
			\centering 
			\includegraphics[scale=1.3]{./KAClockClub}
   }\endxy\;\; + b_{i,i+1}(\lambda)\; \xy (0,0)*{
			\labellist
			\small\hair 2pt
            \pinlabel \scalebox{0.7}{$\textcolor{blue}{i+1}$} at -1 16
            \pinlabel \scalebox{0.7}{$\textcolor{blue}{+1}$} at 15 -2
			\endlabellist 
			\centering 
			\includegraphics[scale=1.3]{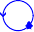}
   }\endxy
\;\;\right)\xy (0,0)*{
			\labellist
			\small\hair 2pt
			\pinlabel \scalebox{0.7}{${i}$} at 0 -5
			\pinlabel \scalebox{0.7}{${i}$} at 16 -5
			\pinlabel \scalebox{0.9}{$\lambda$} at 18 7
			\endlabellist 
			\centering 
			\includegraphics[scale=1.3]{./KurKul}
	}\endxy 
 \;\; +\\[2ex]
 f_{ii}(\lambda)(2b_{ii}(\lambda)+b_{i,i-1}(\lambda)-b_{i,i+1}(\lambda)\left(\xy (0,0)*{
			\labellist
			\small\hair 2pt
			\pinlabel \scalebox{0.7}{${i}$} at 0 -5
			\pinlabel \scalebox{0.7}{${i}$} at 16 -5
			\pinlabel \scalebox{0.9}{$\lambda$} at 18 7
			\endlabellist 
			\centering 
			\includegraphics[scale=1.3]{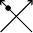}
	}\endxy\;\;\right) . 
 \end{gather*}
\eqskip
The fact that $\Xi$ has to preserve the nil-Hecke relation implies that $o_i(\lambda)f_{ii}(\lambda)=1$ and $b_{ii}(\lambda)=b_{i,i-1}(\lambda)=b_{i,i+1}(\lambda)=0$ for all $i\in \widehat{I}$ and all $\lambda\in \mathbb{Z}^n$. 

To see this, first note that 
\[
\xy (0,0)*{
			\labellist
			\small\hair 2pt
			\pinlabel \scalebox{0.7}{${i}$} at 1 -4
			\pinlabel \scalebox{0.7}{${i}$} at 12.5 -4
			\pinlabel \scalebox{0.9}{$\lambda$} at 18 7
			\endlabellist 
			\centering 
			\includegraphics[scale=1.3]{./KuKu}
	}\endxy
 \qquad , \qquad 
\xy (0,0)*{
			\labellist
			\small\hair 2pt
            \pinlabel \scalebox{0.7}{$\textcolor{black}{i}$} at 1 16
            \pinlabel \scalebox{0.7}{$\textcolor{black}{+1}$} at 15 -2
			\endlabellist 
			\centering 
			\includegraphics[scale=1.3]{./KAClockClub}
   }\endxy \;\; 
   \xy (0,0)*{
			\labellist
			\small\hair 2pt
			\pinlabel \scalebox{0.7}{${i}$} at 0 -4
			\pinlabel \scalebox{0.7}{${i}$} at 15 -4
			\pinlabel \scalebox{0.9}{$\lambda$} at 18 7
			\endlabellist 
			\centering 
			\includegraphics[scale=1.3]{./KurKul}
	}\endxy
 \qquad,\qquad 
    \xy (0,0)*{
			\labellist
			\small\hair 2pt
            \pinlabel \scalebox{0.7}{$\textcolor{blue}{i+1}$} at -1 16
            \pinlabel \scalebox{0.7}{$\textcolor{blue}{+1}$} at 15 -2
			\endlabellist 
			\centering 
			\includegraphics[scale=1.3]{./BAClockClub}
   }\endxy \;\;
   \xy (0,0)*{
			\labellist
			\small\hair 2pt
			\pinlabel \scalebox{0.7}{${i}$} at 0 -4
			\pinlabel \scalebox{0.7}{${i}$} at 15 -4
			\pinlabel \scalebox{0.9}{$\lambda$} at 18 7
			\endlabellist 
			\centering 
			\includegraphics[scale=1.3]{./KurKul}
	}\endxy
 \qquad , \qquad
 \xy (0,0)*{
			\labellist
			\small\hair 2pt
			\pinlabel \scalebox{0.7}{${i}$} at 0 -4
			\pinlabel \scalebox{0.7}{${i}$} at 15 -4
			\pinlabel \scalebox{0.9}{$\lambda$} at 18 7
			\endlabellist 
			\centering 
			\includegraphics[scale=1.3]{./KurKulo}
	}\endxy
\]
\eqskip are linearly independent in $\mathrm{Hom}_{\affu{n}}(\E_{ii}\oneid_\lambda, \E_{ii}\oneid_\lambda)$. Just as in the proof of \cite[Lemma 6.16]{Kh-L}, this follows from looking at their images under the $2$-representation $\mathcal{F}_{\mathrm{Bim}}$ from \cite[Section 4.2]{msv-schur} (in particular, see (45) in that paper), and its extension for the affine case in \cite[Definition 5.6]{mt-affine-schur}.  

The condition $o_i(\lambda)f_{ii}(\lambda)=1$ is therefore immediate. Further, for each $i\in \widehat{I}$, linear independence of the 
two degree-two bubbles above, coloured $i-1$ and $i$ implies that 
$b_{i,i+1}=b_{ii}=0$. 
Using this and 
\[
2b_{ii}(\lambda)+b_{i,i-1}(\lambda)-b_{i,i+1}(\lambda)=0, 
\]
we see that $b_{i,i-1}=0$ as well.  

Next we are going to show that $b_{ij}(\lambda)=0$ for all $i,j\in \widehat{I}$, using the fact that $\Xi$ has to satisfy
\[
            \Xi\left(\;\xy (0,0)*{
			\labellist
			\small\hair 2pt
			\pinlabel \scalebox{0.7}{$\textcolor{red}{i}$} at 0 -5
			\pinlabel \scalebox{0.7}{$\textcolor{blue}{i+1}$} at 16 -5
			\pinlabel \scalebox{0.9}{$\lambda$} at 18 16
			\endlabellist 
			\centering 
			\includegraphics[scale=1.3]{./RurulBulur}
	}\endxy\;\;\right)
= \Xi\left(t_{i,i+1}\;\xy (0,0)*{
			\labellist
			\small\hair 2pt
			\pinlabel \scalebox{0.7}{$\textcolor{red}{i}$} at 0 -5
			\pinlabel \scalebox{0.7}{$\textcolor{blue}{i+1}$} at 13 -5
			\pinlabel \scalebox{0.9}{$\lambda$} at 17 16
			\endlabellist 
			\centering 
			\includegraphics[scale=1.3]{./RuouBuu}
	}\endxy
\quad +  \
 t_{i+1,i}\;\xy (0,0)*{
			\labellist
			\small\hair 2pt
			\pinlabel \scalebox{0.7}{$\textcolor{red}{i}$} at 0 -5
			\pinlabel \scalebox{0.7}{$\textcolor{blue}{i+1}$} at 13 -5
			\pinlabel \scalebox{0.9}{$\lambda$} at 17 16
			\endlabellist 
			\centering 
			\includegraphics[scale=1.3]{./RuuBuou}
	}\endxy\;\;
\;\;
 \right)
\]
\eqskip for $j=i+1$. To shorten notation, put $g_{ij}(\lambda):=f_{ij}(\lambda)f_{ji}(\lambda)$ for all $i,j\in \widehat{I}$. On the 
one hand, we have 
\begin{equation*}
\Xi\left(\;\xy (0,0)*{
			\labellist
			\small\hair 2pt
			\pinlabel \scalebox{0.7}{$\textcolor{red}{i}$} at 0 -5
			\pinlabel \scalebox{0.7}{$\textcolor{blue}{j}$} at 16 -5
			\pinlabel \scalebox{0.9}{$\lambda$} at 18 16
			\endlabellist 
			\centering 
			\includegraphics[scale=1.3]{./RurulBulur}
	}\endxy\;\;\right)
  =g_{i,i+1}(\lambda)\;\;\xy (0,0)*{
			\labellist
			\small\hair 2pt
			\pinlabel \scalebox{0.7}{$\textcolor{red}{i}$} at 0 -5
			\pinlabel \scalebox{0.7}{$\textcolor{blue}{i+1}$} at 16 -5
			\pinlabel \scalebox{0.9}{$\lambda$} at 18 16
			\endlabellist 
			\centering 
			\includegraphics[scale=1.3]{./RurulBulur}
	}\endxy\mspace{10mu}
 =g_{i,i+1}(\lambda)\left(-\;\;\xy (0,0)*{
			\labellist
			\small\hair 2pt
			\pinlabel \scalebox{0.7}{$\textcolor{red}{i}$} at 0 -5
			\pinlabel \scalebox{0.7}{$\textcolor{blue}{i+1}$} at 13 -5
			\pinlabel \scalebox{0.9}{$\lambda$} at 17 16
			\endlabellist 
			\centering 
			\includegraphics[scale=1.3]{./RuouBuu}
	}\endxy
\quad +  \
 \xy (0,0)*{
			\labellist
			\small\hair 2pt
			\pinlabel \scalebox{0.7}{$\textcolor{red}{i}$} at 0 -5
			\pinlabel \scalebox{0.7}{$\textcolor{blue}{i+1}$} at 13 -5
			\pinlabel \scalebox{0.9}{$\lambda$} at 17 16
			\endlabellist 
			\centering 
			\includegraphics[scale=1.3]{./RuuBuou}
	}\endxy\;\;\right)\\
\end{equation*}
\eqskip
and on the other hand, we have 
        \begin{gather*}
       \resizebox{0.3\textwidth}{!}{$\Xi\left(t_{i,i+1}\;\xy (0,0)*{
			\labellist
			\small\hair 2pt
			\pinlabel \scalebox{0.7}{$\textcolor{red}{i}$} at 0 -5
			\pinlabel \scalebox{0.7}{$\textcolor{blue}{i+1}$} at 13 -5
			\pinlabel \scalebox{0.9}{$\lambda$} at 17 16
			\endlabellist 
			\centering 
			\includegraphics[scale=1.3]{./RuouBuu}
	}\endxy
\quad +  \; 
 t_{i+1,i}\;\xy (0,0)*{
			\labellist
			\small\hair 2pt
			\pinlabel \scalebox{0.7}{$\textcolor{red}{i}$} at 0 -5
			\pinlabel \scalebox{0.7}{$\textcolor{blue}{i+1}$} at 13 -5
			\pinlabel \scalebox{0.9}{$\lambda$} at 17 16
			\endlabellist 
			\centering 
			\includegraphics[scale=1.3]{./RuuBuou}
	}\endxy\;\;
\;\;
 \right)=$}
 \\[2ex]
\resizebox{0.99\textwidth}{!}{$t_{i,i+1} o_i(\lambda)\;\xy (0,0)*{
			\labellist
			\small\hair 2pt
			\pinlabel \scalebox{0.7}{$\textcolor{red}{i}$} at 0 -5
			\pinlabel \scalebox{0.7}{$\textcolor{blue}{i+1}$} at 13 -5
			\pinlabel \scalebox{0.9}{$\lambda$} at 17 16
			\endlabellist 
			\centering 
			\includegraphics[scale=1.3]{./RuouBuu}
	}\endxy
\;\; +t_{i+1,i} o_{i+1}(\lambda)\;  \
 \xy (0,0)*{
			\labellist
			\small\hair 2pt
			\pinlabel \scalebox{0.7}{$\textcolor{red}{i}$} at 0 -5
			\pinlabel \scalebox{0.7}{$\textcolor{blue}{j}$} at 13 -5
			\pinlabel \scalebox{0.9}{$\lambda$} at 17 16
			\endlabellist 
			\centering 
			\includegraphics[scale=1.3]{./RuuBuou}
	}\endxy
 \; \; +
 t_{i,i+1} \displaystyle\sum_{k\ne i, i\pm 1} b_{ik}(\lambda)\; \xy (0,0)*{
			\labellist
			\small\hair 2pt
			\pinlabel \scalebox{0.7}{$\textcolor{red}{i}$} at 22 -5
                \pinlabel \scalebox{0.7}{$\textcolor{blue}{i+1}$} at 38 -5
			\pinlabel \scalebox{0.7}{$\textcolor{black}{k}$} at 5 25
            \pinlabel \scalebox{0.7}{$\textcolor{black}{+1}$} at 15 5
			\pinlabel \scalebox{0.9}{$\lambda$} at 42 16
			\endlabellist 
			\centering 
            \includegraphics[scale=1.3]{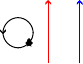}
	}\endxy
 \;\;+
t_{i+1,i} \displaystyle\sum_{\ell\ne i, i+1, i+2} b_{i+1,\ell}(\lambda)\; \xy (0,0)*{
			\labellist
			\small\hair 2pt
			\pinlabel \scalebox{0.7}{$\textcolor{red}{i}$} at 2 -5
                \pinlabel \scalebox{0.7}{$\textcolor{blue}{i+1}$} at 33 -5
			\pinlabel \scalebox{0.7}{$\textcolor{black}{\ell}$} at 11 25
            \pinlabel \scalebox{0.7}{$\textcolor{black}{+1}$} at 23 5
			\pinlabel \scalebox{0.9}{$\lambda$} at 36 18
			\endlabellist 
			\centering 
\includegraphics[scale=1.3]{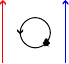}
	}\endxy\;\;=$}
 \\[2ex] 
\resizebox{.82\textwidth}{!}{$t_{i,i+1} o_i(\lambda)\;\;\xy (0,0)*{
			\labellist
			\small\hair 2pt
			\pinlabel \scalebox{0.7}{$\textcolor{red}{i}$} at 0 -5
			\pinlabel \scalebox{0.7}{$\textcolor{blue}{i+1}$} at 13 -5
			\pinlabel \scalebox{0.9}{$\lambda$} at 17 16
			\endlabellist 
			\centering 
			\includegraphics[scale=1.3]{./RuouBuu}
	}\endxy
\;\; +t_{i+1,i}o_{i+1}(\lambda)  \
 \xy (0,0)*{
			\labellist
			\small\hair 2pt
			\pinlabel \scalebox{0.7}{$\textcolor{red}{i}$} at 0 -5
			\pinlabel \scalebox{0.7}{$\textcolor{blue}{j}$} at 13 -5
			\pinlabel \scalebox{0.9}{$\lambda$} at 17 16
			\endlabellist 
			\centering 
			\includegraphics[scale=1.3]{./RuuBuou}
	}\endxy
 \; \;+
 \displaystyle\sum_{k\ne i, i\pm 1, i+2} (t_{i,i+1}b_{ik}(\lambda) + t_{i+1,i}b_{i+1,k})\; \xy (0,0)*{
			\labellist
			\small\hair 2pt
			\pinlabel \scalebox{0.7}{$\textcolor{red}{i}$} at 22 -5
                \pinlabel \scalebox{0.7}{$\textcolor{blue}{i+1}$} at 38 -5
			\pinlabel \scalebox{0.7}{$\textcolor{black}{k}$} at 5 25
            \pinlabel \scalebox{0.7}{$\textcolor{black}{+1}$} at 15 5
			\pinlabel \scalebox{0.9}{$\lambda$} at 42 16
			\endlabellist 
			\centering 
            \includegraphics[scale=1.3]{./KAClockClubRuuBuu}
	}\endxy \;\;+$}
\\[2ex]
\resizebox{0.82\textwidth}{!}{$t_{i,i+1}b_{i,i+2}(\lambda)\; \xy (0,0)*{
			\labellist
			\small\hair 2pt
			\pinlabel \scalebox{0.7}{$\textcolor{red}{i}$} at 22 -5
                \pinlabel \scalebox{0.7}{$\textcolor{blue}{i+1}$} at 38 -5
			\pinlabel \scalebox{0.7}{$\textcolor{myorange}{i+2}$} at 5 25
            \pinlabel \scalebox{0.7}{$\textcolor{myorange}{+1}$} at 15 5
			\pinlabel \scalebox{0.9}{$\lambda$} at 42 16
			\endlabellist 
			\centering 
            \includegraphics[scale=1.3]{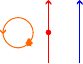}
	}\endxy
 \; \;+
  t_{i+1,i}b_{i+1,i-1}(\lambda)\; \xy (0,0)*{
			\labellist
			\small\hair 2pt
			\pinlabel \scalebox{0.7}{$\textcolor{red}{i}$} at 1 -5
                \pinlabel \scalebox{0.7}{$\textcolor{blue}{i+1}$} at 17 -5
			\pinlabel \scalebox{0.9}{$\lambda$} at 21 16
			\endlabellist 
			\centering 
            \includegraphics[scale=1.3]{./RuouBuu}
	}\endxy
\;\; +
 t_{i+1,i} b_{i+1,i-1}(\lambda)\; \xy (0,0)*{
			\labellist
			\small\hair 2pt
			\pinlabel \scalebox{0.7}{$\textcolor{red}{i}$} at 22 -5
                \pinlabel \scalebox{0.7}{$\textcolor{blue}{i+1}$} at 38 -5
			\pinlabel \scalebox{0.7}{$\textcolor{mypurple}{i-1}$} at 5 25
            \pinlabel \scalebox{0.7}{$\textcolor{mypurple}{+1}$} at 15 5
			\pinlabel \scalebox{0.9}{$\lambda$} at 42 16
			\endlabellist 
			\centering 
            \includegraphics[scale=1.3]{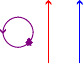}
	}\endxy$}
\end{gather*}
\eqskip By linear independence of the different terms of each expression and comparing corresponding terms in both expressions, as above, we get 
\[
b_{i,i+2}(\lambda)=0,\quad b_{i+1,i-1}(\lambda)=0, \quad t_{i,i+1}b_{ik}(\lambda) + t_{i+1,i}b_{i+1,k}(\lambda)=0
\]
for all $i,k\in \widehat{I}$ such that $k\ne i, i\pm 1, i+2$. Together with the previous results, these equations 
imply that $b_{ij}(\lambda)=0$ for all $i,j\in \widehat{I}$ and all $\lambda\in \mathbb{Z}^n$.
\end{proof}

Note that the above proof also shows that the following equations must hold:
 \begin{equation}\label{eq:2Iso1} g_{i,i+1}(\lambda)=-t_{i,i+1}o_i(\lambda)=t_{i+1,i}o_{i+1}(\lambda)
\end{equation} 
for all $i\in \widehat{I}$.\par

\begin{thm}\label{thm:NoIsomorphism} When $n$ is odd, there does not exist a 2-isomorphism $\Xi:\mathcal{U}_Q(\widehat{\mathfrak{gl}}'_n)\to \affu{n}$ which is the identity on objects and 1-morphisms, for any 
choice of scalars $Q$ with compatible bubble parameters.
\end{thm}

\begin{proof}
    Assume for contradiction that such a 2-isomorphism $\Xi$ exists for some choice of scalars with compatible bubble parameters, with $\Xi$ having parameters $\{o_i,f_{ij}|i,j=1,\dots,n\}$ as above. Recall that $o_i(\lambda),g_{ij}(\lambda),t_{ij}\in\mathbb{Q}^\times$ 
   and that $g_{ij}=g_{ji}$ for all $i,j\in \widehat{I}$. Thus, suppressing $\lambda$ for readability reasons, we get  
   \begin{align*}
  o_1 &=-g_{12}t_{12}^{-1}= 
  -o_2t_{21}t_{12}^{-1}=(-1)^2g_{23}t_{23}^{-1}t_{21}t_{12}^{-1}=\dots=(-1)^n g_{n,1}t_{n,1}^{-1}\dots t_{21}t_{12}^{-1}
  \\[1ex]
  &
  =(-1)^n o_1\prod_{i\in \widehat{I}} t_{i+1,i}t_{i,i+1}^{-1}.
   \end{align*}   
    This implies that $\prod_{i\in \widehat{I}} t_{i+1,i}t_{i,i+1}^{-1} =(-1)^n$ has to hold. But by the definition of $Q$ and the fact that $\sum_{k=1}^n \alpha_k=0$ in the (level zero) $\widehat{\sln}$-root lattice, we have the following:
    \begin{itemize}[wide,labelindent=0pt]
        \item $t_{ii}=1$ for all $i=1,\dots,n$, so in particular $\prod_{i=1}^n t_{ii}=1$,
        \item $t_{ij}=t_{ji}$ whenever $|i-j|>1$ $\operatorname{mod} n$, so in particular $\displaystyle\prod_{\substack{i,j=1,\dots,n\\|i-j|>1}} t_{ij}=x^2$ for some $x\in \mathbb{Q}^\times$,
        \item $1=\frac{c_{i,\olambda}}{c_{i,\olambda}}=\frac{c_{i,\olambda+\sum_{k=1}^n \alpha_k}}{c_{i,\olambda}}=\prod_{j=1}^n \frac{c_{i,\olambda+\sum_{k=1}^j \alpha_k}}{c_{i,\olambda+\sum_{k=1}^{j-1}\alpha_k}}=\prod_{j=1}^n t_{ij}$ for any $\lambda\in\mathbb{Z}^{n}$ and any $i=1,\dots,n$.
    \end{itemize}
    Therefore, for any $\lambda\in\mathbb{Z}^{n}$ $$\prod_{i=1}^n\frac{c_{i,\olambda}}{c_{i,\olambda}}=\prod_{i,j=1,\dots,n} t_{ij}=1.$$ But $$\prod_{i,j=1,\dots,n} t_{ij}=(\prod_{i=1}^n t_{ii})(\prod_{\substack{i,j=1,\dots,n\\|i-j|>1}} t_{ij})(\prod_{\substack{i=1,\dots,n\\ |i-j|=1}} t_{ij})= x^2 \prod_{\substack{i=1,\dots,n\\ |i-j|=1}} t_{ij},$$ for some $x\in \mathbb{Q}^\times$, by the above remarks. Multiplying this by 
    $\prod_{i\in \widehat{I}} t_{i+1,i}t_{i,i+1}^{-1}$ yields 
    \[
    \prod_{i=1}^n t_{i+1,i}^2=(-1)^n x^{-2},
    \]
    which implies $n$ has to be even, completing our proof.
\end{proof}



\vspace*{1cm}


\end{document}